\documentclass[a4paper,11pt]{article}
 \usepackage[plainpages=false]{hyperref}
 \usepackage{amsfonts,amsmath,amssymb,amsthm,mathtools}
 \usepackage{latexsym,lscape,rawfonts,mathrsfs,slashed}
\usepackage{enumitem}
\usepackage{scalerel}[2014/03/10]
\usepackage[usestackEOL]{stackengine}
\usepackage{CJKutf8}
\usepackage{fmtcount}
\usepackage{imakeidx}
\usepackage[thinc]{esdiff}
\usepackage{xcolor}

 \usepackage[all]{xy}
 \usepackage{graphicx,psfrag,float}
 \usepackage{pstool}
 \usepackage{tikz-cd}

 \usepackage{array,tabularx}

 \usepackage{setspace}

 \usepackage[titletoc,title]{appendix}
\usepackage{txfonts}

 \newcommand{\E}{\ensuremath{\mathbb{E}}}
 \newcommand{\F}{\ensuremath{\mathbb{F}}}
 
 \newcommand{\N}{\ensuremath{\mathbb{N}}}
 
 \newcommand{\R}{\ensuremath{\mathbb{R}}}
 
 \def\III{\mathbb{I}}
 \def\XX{\mathcal{X}}
 \def\RR{\mathcal{R}}
 \def\MM{\mathcal{M}}

 \def\PP{\mathcal{P}}
 
 \def\Var{\mathrm{Var}}

 \def\scal{\mathrm{R}}

 \def\rcd{\mathrm{RCD}}
 
 \def\spt{\mathrm{supp}}

 \newcommand{\ba}{\begin{align*}}
 \newcommand{\ea}{\end{align*}}
 \newcommand{\flow}{\mathcal{X}=\{M^n,(g(t))_{t\in I}\}}
 \newcommand{\na}{\nabla}
\newcommand{\la}{\langle}
\newcommand{\ra}{\rangle}

\newcommand{\lc}{\left(}
\newcommand{\rc}{\right)}
 
\newcommand{\ep}{\epsilon}

\def\MMM{\mathscr{M}}

\def\NN{\mathcal{N}}

\def\WW{\mathcal{W}}

\def\delf{\mathrm{div}_f}

\def\Lip{\mathrm{Lip}}

\newcommand{\Rm}{\ensuremath{\mathrm{Rm}}}
\newcommand{\Ric}{\ensuremath{\mathrm{Ric}}}

\renewcommand{\t}{\mathfrak{t}}
\newcommand{\CF}{\mathfrak{C}}

\def\MS{\mathcal{S}}
\newcommand{\IF}{\ensuremath{\mathbb{F}}}

 \makeatletter
 \def\ExtendSymbol#1#2#3#4#5{\ext@arrow 0099{\arrowfill@#1#2#3}{#4}{#5}}

 \makeatother
 
\def\aint{\,\ThisStyle{\ensurestackMath{%
  \stackinset{c}{.2\LMpt}{c}{.5\LMpt}{\SavedStyle-}{\SavedStyle\phantom{\int}}}%
  \setbox0=\hbox{$\SavedStyle\int\,$}\kern-\wd0}\int}

\DeclarePairedDelimiter\abs{\lvert}{\rvert}%
\makeatletter
\let\oldabs\abs
\def\abs{\@ifstar{\oldabs}{\oldabs*}}
\makeatother

\numberwithin{equation}{section}

\newtheorem{thm}{Theorem}[section]
\newtheorem{cor}[thm]{Corollary}
\newtheorem{prop}[thm]{Proposition}
\newtheorem{lem}[thm]{Lemma}

\newtheorem{rem}[thm]{Remark}

\newtheorem{defn}[thm]{Definition}

\newtheorem{exmp}[thm]{Example}

\newtheorem{notn}[thm]{Notation}

\makeindex
\title{On the structure of noncollapsed Ricci flow limit spaces}
\author{Hanbing Fang \quad and \quad Yu Li} 
\date{\today}

\begin{document}
	\begin{CJK}{UTF8}{gbsn}

\maketitle
	\begin{abstract}
We establish a weak compactness theorem for the moduli space of closed Ricci flows, each equipped with a natural spacetime distance, under pointed Gromov--Hausdorff convergence. For the subspace of flows with uniformly bounded entropy, we further develop a structure theory for the corresponding noncollapsed Ricci flow limit spaces, showing that the regular part, where convergence is smooth, admits the structure of a Ricci flow spacetime, while the singular set has codimension at least four.
\end{abstract}
	\tableofcontents

\section{Introduction}

A Ricci flow solution $(g(t))_{t \in I}$ on a closed Riemannian manifold $M^n$ is given by the evolution equation:
		\begin{align*}
\partial_t g(t)=-2 \Ric(g(t))
	\end{align*}	
for any $t \in I$, where $I$ is a closed time interval. Ricci flow was introduced by Hamilton in his pioneering 1982 paper \cite{hamilton1982}, where he used it to prove that a closed 3-manifold with positive Ricci curvature evolves under Ricci flow to a manifold with constant curvature. This result was a major breakthrough in the use of geometric evolution equations to study the topology of manifolds. In the early 2000s, building on Hamilton's program, Perelman introduced several new ideas that revolutionized the understanding of Ricci flow and finally resolved the Poincar\'e Conjecture and the more general Geometrization Conjecture \cite{perelman2002entropy, perelman2003b, perelman2003a}.

Compactness theory plays a central role in the analysis of geometric flows, particularly in Ricci flow, where understanding the behavior of sequences of solutions is essential for studying singularity formation, convergence, and geometric limits. The classical compactness theorem for the Ricci flow, established by Hamilton \cite{hamilton1995compactness}, asserts that a sequence of Ricci flows with uniform curvature bounds and noncollapsing conditions admits a subsequence converging in the Cheeger--Gromov sense. Another example is the compactness of $\kappa$-solutions to the Ricci flow, which were introduced by Perelman as local models for singularities after appropriate blow-up procedures. In three dimensions, Perelman used this compactness result to essentially classify all 3-dimensional $\kappa$-solutions, leading to a detailed understanding of singularity models and enabling the implementation of Ricci flow with surgery. 

In general dimensions, the weak compactness theory of Ricci flows has been developed under the additional assumption of a uniform scalar curvature bound; see, for instance, \cite{chen2012space, tianzhang16, chen2017space, chen2020space, bamler2018convergence}. In the case of K\"ahler Ricci flow on Fano manifolds, this scalar curvature bound is automatic due to Perelman's crucial estimate \cite{sesumtian08}. These weak compactness theories focus on the convergence of the time slices of Ricci flows in the Gromov--Hausdorff sense. A key observation under the scalar curvature bound is that the distance functions at different time slices are mutually comparable; see \cite[Lemma 4.21]{chen2020space} and \cite[Theorem 1.1]{bamler2017heat}. Consequently, the weak compactness theory implies that the time slices converge in the Gromov--Hausdorff sense to a singular metric space, whose singular set has codimension at least 4. Moreover, in \cite{chen2020space} (see also \cite{bamler2018convergence}), the authors further established the convergence of Ricci flows as \textbf{spacetimes}---a perspective that already appeared in Perelman's work.

The convergence theory of Ricci flows can be viewed as a natural generalization of the convergence theory for Einstein manifolds, developed by Cheeger, Colding, Naber, and others; see \cite{cheeger1997structure, cheeger2013lower, cheeger2015regularity}. However, in the case of general Ricci flows without any curvature assumptions, the lack of distance comparability prevents one from establishing Gromov--Hausdorff convergence for individual time slices.

In a series of works \cite{bamler2020entropy, bamler2023compactness, bamler2020structure}, Bamler introduced a number of ideas to develop the theory of $\F$-convergence. Within this framework, Bamler proved that for almost every time, the time slices of a Ricci flow, when equipped with conjugate heat kernel measures, converge in the Gromov--$W_1$-Wasserstein distance (see Definition \ref{defnGWpdistance}). Moreover, he established that Ricci flows $\F$-converge to a limit known as a metric flow (see Definition \ref{defnmetricflow}), and that the family of time slices in this limiting metric flow is almost continuous in the $GW_1$-sense. A metric flow can be regarded as a weak notion of Ricci flow; see also alternative formulations in \cite{haslhofer2018chara, choihas24}. In dimension three, an example of a metric flow is a branch of a weak Ricci flow, as established in \cite{kleiner2017singular}, in which each time slice remains connected.

In general, a limiting metric flow may carry limited geometric information due to potential collapsing phenomena. However, when a uniform bound on the Nash entropy at the base point is imposed, Bamler showed in \cite{bamler2020structure} that the limiting metric flow exhibits favorable geometric properties. Notably, the limit space admits a regular--singular decomposition: the regular part forms a Ricci flow spacetime (see Definition \ref{def_RF_spacetime}), while the singular part has codimension at least 4 with respect to coverings by $P^*$-balls (see Definition \ref{def:pstar}). Moreover, several key results originally established in the context of Einstein manifolds---such as the stratification of the singular set, volume estimates for the quantitative singular strata, and integral curvature radius bounds from \cite{cheeger2013lower, cheeger2015regularity}---continue to hold in the setting of Ricci flow.

In this paper, we first consider the moduli space of all closed Ricci flows on a fixed time interval.

\begin{defn}[Moduli space]
For a fixed constant $T\in(0,+\infty]$, the moduli space $\MM(n,T)$ consists of all $n$-dimensional closed Ricci flows
\begin{align*}
\XX=\{M^n,(g(t))_{t\in\III^{++}}\},\qquad \III^{++}:=[-T,0].
\end{align*}
For $Y>0$, we denote by $\MM(n,Y,T)\subset\MM(n,T)$ the subspace consisting of flows satisfying
\begin{align*}
\inf_{\tau>0}\NN_{x^*}(\tau)\geq-Y
\end{align*}
for every $x^*\in M\times\III^{++}$, where the infimum is taken over all $\tau>0$ for which the Nash entropy is well-defined.
\end{defn}

It is clear that any closed Ricci flow defined on a closed time interval of length $T$ can, via a time translation, be assumed to be defined on $\III^{++}$. The definition of the Nash entropy $\NN_{x^*}$ based at a spacetime point $x^*$ is given in Definition \ref{defnentropy}. The condition defining $\MM(n,Y,T)$ is equivalent to a uniform noncollapsing condition. By Perelman's celebrated monotonicity formula, any closed Ricci flow $\{M^n,(g(t))_{t\in\III^{++}}\}$ satisfying
\begin{align*}
\inf_{\tau\in(0,T]}\boldsymbol{\mu}(g(-T),\tau)\geq-Y
\end{align*}
belongs to $\MM(n,Y,T)$.

For any Ricci flow $\XX=\{M^n,(g(t))_{t\in\III^{++}}\}$, the absence of curvature bounds makes it difficult to define a natural distance between spacetime points. Nevertheless, a key result---proved in \cite[Theorem 2]{mccann2010ricci} and \cite[Theorem 3.1]{ctop2012} (see also \cite[Lemma 2.7]{bamler2020entropy})---states that
\begin{align*}
d_{W_1}^t(\nu_{x^*;t},\nu_{y^*;t})\quad\text{and}\quad d_{W_2}^t(\nu_{x^*;t},\nu_{y^*;t})
\end{align*}
are nondecreasing in $t$ for any spacetime points $x^*,y^*\in\XX$, where $d^t_{W_p}$ denotes the $W_p$-Wasserstein distance with respect to the metric $g(t)$ (see Definition \ref{defnvariance}), and $\nu_{z^*;t}$ denotes the conjugate heat kernel measure based at $z^*$ (see Definition \ref{def:chkm}). In fact, this monotonicity is equivalent to the notion of a super Ricci flow \cite{mccann2010ricci}, and is closely related to weak formulations of super Ricci flows; see \cite{sturm2018super, kopfer2018heat}.

For $\XX=\{M^n,(g(t))_{t \in \III^{++}}\}$, we can use this monotonicity to define a spacetime distance by restricting to the slightly smaller interval $\III^{+}:=[-(1-\sigma)T, 0]$, where $\sigma$ is a small parameter in $(0, 1/100]$. Specifically, we have the following definition:

\begin{defn}\label{def:introd-dstar}
For $x^*,y^*\in M\times\III^+$, set
\begin{align*}
t_+:=\max\{\t(x^*),\t(y^*)\},\qquad t_-:=\min\{\t(x^*),\t(y^*)\}.
\end{align*}
We define
\begin{align}\label{introdef1}
d^*(x^*,y^*):=\inf_{-(1-\sigma)T\leq\tau\leq t_-}
\max\left\{\sqrt{t_+-\tau},\ d_{W_1}^{\tau}(\nu_{x^*;\tau},\nu_{y^*;\tau})\right\}.
\end{align}
If $T=+\infty$, the infimum is taken over $\tau\in(-\infty,t_-]$.
\end{defn}

Definition \ref{def:introd-dstar} ensures that the natural time-function $\t$, defined as the projection of a spacetime point onto its time component, is 2-H\"older continuous, that is, 
	\begin{align*}
	|\t(x^*)-\t(y^*)| \le d^*(x^*, y^*)^2 \quad \text{for all} \quad x^*, y^* \in M \times \III^+.
	\end{align*}
It can be shown---see Lemma \ref{lem:000a}---that $d^*$ is indeed a distance function. Moreover, the topology induced by $d^*$ coincides with the standard topology on $M \times \III^+$ (see Corollary \ref{cor:topo}). In addition, the metric balls $B^*$, defined via $d^*$, are comparable to the parabolic balls $P^*$ introduced by Bamler (see Proposition \ref{propdistance2}). Furthermore, if the scalar curvature is locally bounded, $B^*$ is comparable to the standard parabolic balls; see Proposition \ref{equivalenceofballs}.

There is some flexibility in the construction in Definition \ref{def:introd-dstar}. For example, one may define a similar spacetime distance using the monotonicity of $d_{W_2}^t$, as in \eqref{introdef1}. Nonetheless, all such spacetime distances are equivalent in the sense that they are bi-Lipschitz to one another; see Appendix \ref{app:B} for details.

Our first main result establishes pointed Gromov--Hausdorff convergence, with respect to the $d^*$-distance, for a sequence of Ricci flows in $\MM(n,T)$ restricted to the smaller time interval $\III:=[-(1-2\sigma)T,0]$.

\begin{thm}[Weak compactness]\label{thm:intro1}
Given any sequence $\XX^i=\{M_i^n,(g_i(t))_{t\in\III^{++}}\}\in\MM(n,T)$ with base points $p_i^*\in M_i\times\III$ \emph{(}when $T=+\infty$, we additionally assume $\limsup_{i\to\infty}\t_i(p_i^*)>-\infty$\emph{)}, after passing to a subsequence, we obtain
\begin{align*}
(M_i\times\III,d_i^*,p_i^*,\t_i)\xrightarrow[i\to\infty]{\quad\mathrm{pGH}\quad}(Z,d_Z,p_\infty,\t),
\end{align*}
where $d_i^*$ denotes the restriction of the $d^*$-distance to $M_i\times\III$. The limit $(Z,d_Z,\t)$ is a complete, separable, locally compact metric space coupled with a $2$-H\"older continuous time-function $\t:Z\to\III$.
\end{thm}

In Propositions \ref{propdistance3} and \ref{uppervolumebd}, we establish entropy-weighted lower and upper bounds for spacetime balls $B^*(x^*, r)$. Since the common entropy factor is compared by the Harnack estimate, Theorem \ref{thm:intro1} follows from a standard ball-packing argument, analogous to the convergence theory for sequences of Riemannian manifolds with uniform Ricci curvature lower bounds. The full proof is given in Theorem \ref{thm:GHlimit-dstar}.

The limit space $(Z, d_Z, p_\infty, \t)$ is referred to as a \textbf{Ricci flow limit space} over $\III$. It is called \textbf{noncollapsed} if the approximating sequence lies in $\MM(n,Y,T)$ for some fixed $Y$. A natural question arises: what is the relationship between the space $Z$ and the $\F$-limits obtained from the sequence $\XX^i$? 

To investigate this, consider a sequence of points $z_i^* \in M_i \times \III$ converging to $z \in Z$ in the Gromov--Hausdorff sense. By the theory of $\F$-convergence (see Section \ref{sec:f}), there exists a correspondence $\CF$ such that 
\begin{equation}\label{thm:intro2}
	(\XX^i, (\nu_{z_i^*;t})_{t \in [-T, \t_i(z_i^*)]}) \xrightarrow[i \to \infty]{\quad \IF, \CF\quad} (\XX^z, (\nu_{z;t})_{t \in  [-T, \t(z)]}),
\end{equation}
where the metric flow $\XX^z$ is future continuous for all $t \in [-T, \t(z)]$, except possibly at $t=-(1-\sigma)T$, at which we require that the convergence \eqref{thm:intro2} is uniform. The metric flow $\XX^z$ is referred to as the metric flow associated with $z$. On $\XX^z_{\III^+}$, one can define a spacetime function $d_z^*$ as in Definition \ref{def:introd-dstar} (see Definition \ref{defn:dstar-limit}). In general, $d_z^*$ is only a pseudo-distance on $\XX^z_{\III^+}$. However, by passing to the corresponding quotient space $\widetilde{\XX^{z}_\III}$, one obtains an isometric embedding into the limit space $Z$ (see Theorem \ref{thm:idenproof} for the proof).

\begin{thm}\label{thm:iden}
For any $z \in Z$, there exists an isometric embedding
	\begin{align*}
\iota_z: (\widetilde{\XX^{z}_\III}, d^*_z) \longrightarrow (Z, d_Z)
	\end{align*}
such that $\iota_z(z)=z$ and $\t \circ \iota_z=\t^z$, where $\t^z$ is the time-function on $\widetilde{\XX^{z}_\III}$. Moreover, for any $y_i^* \in \XX^i_\III$ and $y_{\infty} \in \XX^{z}_\III$, $y_i^*$ converge to $y_{\infty}$ within $\CF$ if and only if $y_i^* \to \iota_z(\tilde y_{\infty})$ in the Gromov--Hausdorff sense, where $\tilde y_{\infty}$ is the quotient image of $y_{\infty}$ from $\XX^{z}_\III$ to $\widetilde{\XX^{z}_\III}$.	
\end{thm}

From this point through the structure theory, we assume that the approximating sequence lies in $\MM(n,Y,T)$ for a fixed $Y$; hence, the Ricci flow limit space $Z$ is noncollapsed.

The space $Z$ contains a regular part $\RR$. Its restriction to $\III^-$ is a dense open subset of $Z_{\III^-}$ (see Corollary \ref{cor:comple1}), and $\RR$ carries the structure of a Ricci flow spacetime $(\RR, \t, \partial_\t, g^Z)$. On this regular part, the convergence described in Theorem \ref{thm:intro1} is smooth, in the following sense.

\begin{thm}[Smooth convergence]\label{thm:intro3}
There exist an increasing sequence $U_1 \subset U_2 \subset \ldots \subset \RR$ of open subsets with $\bigcup_{i=1}^\infty U_i = \RR$, open subsets $V_i \subset M_i \times \III$, time-preserving diffeomorphisms $\phi_i : U_i \to V_i$ and a sequence $\ep_i \to 0$ such that the following properties hold:
		\begin{enumerate}[label=\textnormal{(\alph{*})}]
			\item We have
			\begin{align*}
				\Vert \phi_i^* g^i - g^Z \Vert_{C^{[\ep_i^{-1}]}(U_i)} & \leq \ep_i, \\
				\Vert \phi_i^* \partial_{\t_i} - \partial_{\t} \Vert_{C^{[\ep_i^{-1}]}(U_i)} &\leq \ep_i,
			\end{align*}
			where $g^i$ is the spacetime metric induced by $g_i(t)$, and $\partial_{\t_i}$ is the standard time vector field induced by $\t_i$.
			
		\item Let $y \in \RR$ and $y_i^* \in M_i \times \III$. Then $y_i^* \to y$ in the Gromov--Hausdorff sense if and only if $y_i^* \in V_i$ for large $i$ and $\phi_i^{-1}(y_i^*) \to y$ in $\RR$.

			\item For $U_i^{(2)}=\{(x,y) \in U_i \times U_i \mid \t(x)> \t(y)+\ep_i\}$, $V_i^{(2)}=\{(x^*,y^*) \in V_i \times V_i \mid \t_i(x^*)> \t_i(y^*)+\ep_i\}$ and $\phi_i^{(2)}:=(\phi_i, \phi_i): U_i^{(2)} \to V_i^{(2)}$, we have
					\begin{align*}
\Vert  (\phi_i^{(2)})^* K^i-K_Z \Vert_{C^{[\ep_i^{-1}]}(U_i^{(2)})} \le \ep_i,
			\end{align*}	
where $K^i$ and $K_Z$ denote the heat kernels on $(M_i \times \III, g_i(t))$ and $(\RR, g^Z)$, respectively.
		
		\item If $z_i^* \in M_i \times \III$ converge to $z \in Z$ in the Gromov--Hausdorff sense, then
\begin{align*}
K^i(z_i^*;\phi_i(\cdot)) \xrightarrow[i \to \infty]{C^{\infty}_\mathrm{loc}} K_Z(z;\cdot) \quad \text{on} \quad \RR_{(-\infty,\t(z))}.
\end{align*}	

		\item For each $t \in \III$, the time slice $\RR_t$ has at most countably many connected components.
				\end{enumerate}
\end{thm}

Theorem \ref{thm:intro3} is proved in Theorems \ref{thm:smooth1} and \ref{thm:convextra} and Proposition \ref{prop:connectnumber}.

The proof of Theorem \ref{thm:intro3} is similar to the proof of smooth convergence for an $\F$-limit (see \cite[Section 9]{bamler2023compactness}). Roughly speaking, the approach involves constructing a product domain $U_z$ for each $z \in \RR$, such that $U_z$ is realized as a Ricci flow spacetime satisfying the required properties. These local pieces are then glued together using a standard patching procedure. Special care must be taken in the case where $\t(z)=0$. 

The associated metric flow $\XX^z$ also contains a regular part $\RR^z$, which admits the structure of a Ricci flow spacetime $(\RR^z, \t^z, \partial_{\t^z}, g^z)$; see Theorem \ref{Fconvergence}. It can be shown---see Proposition \ref{prop:embed2}---that the isometric embedding $\iota_z$ from Theorem \ref{thm:iden} is, in fact, an isometric embedding of Ricci flow spacetimes. As a result, the regular part $\RR$ can be viewed as the union of the pieces $\iota_z(\RR^z)$.

In general, the regular part $\RR$ may not be connected in the spacetime. We provide a sufficient condition (see Corollary \ref{cor:connected}) under which two points in $\RR$ lie in the same connected component. In particular, $\RR$ is connected if $T=+\infty$. We emphasize that this stands in sharp contrast to the regular part of an $\F$-limit, whose time slices are connected. For example, as illustrated in Figure \ref{embedding}, the slice $\RR_{t_3}$ consists of two components, namely $\iota_x(\RR^x_{t_3})$ and $\iota_y(\RR^y_{t_3})$.

For each $z \in Z$ we can assign a conjugate heat kernel measure $\nu_{z;s}$ based at $z$ for $s \le \t(z)$, which is a probability measure on $\RR_s$. All these probability measures together satisfy the reproduction formula (see \eqref{eq:generalrep}). With the help of conjugate heat kernel measures, we can define a distance $d_t^Z$ at the time-slice $Z_t$ for any $t \in \III^-:=(-(1-2\sigma)T, 0]$.

\begin{defn}\label{intro:timedis}
	For each $t \in \III^-$, we define the distance at the time-slice $Z_t$ by
	\begin{align*}
		d^{Z}_t(x,y):=\lim_{s \nearrow t} d_{W_1}^{\RR_s}(\nu_{x;s},\nu_{y;s}) \in [0,\infty]
	\end{align*}
	for any $x,y \in Z_t$, where $d_{W_1}^{\RR_s}$ denotes the $W_1$-Wasserstein distance on $(\RR_s, g^Z_s)$.
\end{defn}

It can be proved that the limit in Definition \ref{intro:timedis} must exist, since $d_{W_1}^{\RR_s}$ is nondecreasing (see Lemma \ref{lem:limitmono}). 

\begin{thm}\label{thm:intro4}
For the distance $d^Z_t$ defined in Definition \ref{intro:timedis}, the following properties hold.
		\begin{enumerate}[label=\textnormal{(\alph{*})}]
			\item For any $t \in \III^-$, $(Z_t, d^Z_t)$ is a complete extended metric space.
			
		\item $\lc Z, \t, (d^Z_t)_{t \in \III^-}, (\nu_{z; s})_{s\in \III^-, s\le \t(z)} \rc$ is an $H_n$-concentrated extended metric flow over $\III^-$, in the sense of Definition \ref{defn:emf}.
		
		\item For any $w \in \RR_t$, there exists a sufficiently small constant $r>0$ such that for any $x,y \in B_{g^Z_t}(w, r)$,
	\begin{align*}
d^Z_t(x, y)=d_{g^Z_t}(x, y).
	\end{align*}

			\item For all but countably many times $t \in \III^-$, on each connected component of $\RR_t$, we have
	\begin{align*}
d^Z_t=d_{g^Z_t}.
	\end{align*}
		
		\item For any $x,y\in Z_{\III^-}$ with $t_0=\t(x)\geq\t(y)$, set $r=d_Z(x,y)$ and $q=t_0-r^2$. If $q\in\III^-$, then
\begin{align*}
\lim_{t\nearrow q}d_{W_1}^{Z_t}(\nu_{x;t},\nu_{y;t})\leq r.
\end{align*}
If in addition $q<\t(y)$, equivalently $r>\sqrt{\t(x)-\t(y)}$, then
\begin{align*}
r\leq\lim_{t\searrow q}d_{W_1}^{Z_t}(\nu_{x;t},\nu_{y;t}).
\end{align*}
Here $d_{W_1}^{Z_t}$ denotes the $W_1$-Wasserstein distance on $(Z_t,d_t^Z)$ (see Definition \ref{defnvarianceextended}).
				\end{enumerate}
\end{thm}

Theorem \ref{thm:intro4} is proved in Lemma \ref{lem:extend1}, Propositions \ref{prop:com}, \ref{prop:dismatch}, and \ref{prop:005bb}, Theorem \ref{thm:em}, and Propositions \ref{prop:chara} and \ref{prop:dismatch1}.

In general, any conjugate heat kernel measure $\nu_{z;s}$ has full measure on a single connected component of $\RR_s$. Parts (c) and (d) of Theorem \ref{thm:intro4} show that, locally, the distance function $d^Z_t$ agrees with the Riemannian distance $d_{g^Z_t}$ induced by the metric $g^Z_t$, and for almost every $t \in \III^-$, the two coincide on each connected component of $\RR_t$. However, one should not expect this agreement to hold globally on all of $\RR_t$, as it is possible for $d^Z_t(x, y)$ to be finite even when $x$ and $y$ lie in different components of $\RR_t$ (see Figure \ref{embedding} at $t_3$). Part (e) of Theorem \ref{thm:intro4} further clarifies the relationship between the spacetime distance $d_Z$ and the time-slice distance $d^Z_t$, in alignment with Definition \ref{def:introd-dstar}.

\begin{defn}[Tangent flow]
	For any $z \in Z_{\III^-}$, a \textbf{tangent flow} $(Z',d_{Z'},z',\t')$ at $z$ is a pointed Gromov--Hausdorff limit of $(Z, r_j^{-1} d_Z, z, r_j^{-2}(\t-\t(z)))$ for a sequence $r_j  \searrow 0$.
	\end{defn}
	
It can be shown (see Section \ref{sec:tangent}) that any tangent flow is a noncollapsed Ricci flow limit space. We now introduce a broader class of Ricci flow limit spaces, called \textbf{Ricci shrinker spaces}, which encompass all tangent flows. Roughly speaking, a Ricci shrinker space $(Z', d_{Z'}, z', \t')$ is a noncollapsed Ricci flow limit space with $\R_- \subset \mathrm{image}(\t')$ such that the base point $z'$ has constant Nash entropy (see Definition \ref{def:rss}).

For Ricci shrinker spaces, we have the following:

\begin{thm}[Characterization of Ricci shrinker spaces]\label{thm:intro5}
Let $(Z', d_{Z'}, z', \t')$ be a Ricci shrinker space so that its regular part is given by a Ricci flow spacetime $(\RR', \t', \partial_{\t'}, g^{Z'}_t)$. Then the following statements hold.
		\begin{enumerate}[label=\textnormal{(\alph{*})}]
			\item On $\RR'_{(-\infty,0)}$, the following equation holds:
	\begin{align*}
\Ric(g^{Z'})+\na^2 f_{z'}=\frac{g^{Z'}}{2 |\t'|},
	\end{align*}
where $f_{z'}$ is the potential function at $z'$.
			
		\item For any $t<0$, the slice $\RR'_t$ is connected. Moreover, the distance $d^{Z'}_t$, when restricted to $\RR'_t$, coincides with the Riemannian distance induced by the metric $g^{Z'}_t$.
		
		\item $Z'_{(0, \infty)}=\emptyset$ if $(Z', d_{Z'}, z', \t')$ is \textbf{collapsed} (see Definition \ref{def:ncf}).

			\item The space $Z'_{(-\infty, 0]}$ is \textbf{self-similar} in the following sense: there exists a flow $\boldsymbol{\psi}^s$ on $Z'_{(-\infty, 0]}$ such that, when restricted to $\RR'_{(-\infty,0)}$, it is generated by $\tau(\partial_{\t'}-\na f_{z'})$, with $\boldsymbol{\psi}^0=\mathrm{id}$. Moreover, for any $x, y \in Z'_{(-\infty, 0]}$ and $s \in \R$, we have
		\begin{align*}
d_{Z'}(\boldsymbol{\psi}^s(x), \boldsymbol{\psi}^s(y))=e^{-\frac s 2} d_{Z'}(x, y).
	\end{align*}		
In addition, for any $x, y \in Z'_t$ with $t\le 0$ and $s \in \R$, the time-slice distance satisfies
		\begin{align*}
d^{Z'}_{e^{-s}t}(\boldsymbol{\psi}^s(x), \boldsymbol{\psi}^s(y))=e^{-\frac{s}{2}}d^{Z'}_{t}(x,y).
	\end{align*}	
	\item For any $t<0$, $Z'_{t} \setminus \RR'_t$ has Minkowski dimension at most $n-4$ with respect to $d^{Z'}_t$.
				\end{enumerate}
\end{thm}

Theorem \ref{thm:intro5} is proved in Proposition \ref{prop:004}, Corollary \ref{cor:agree2}, Theorem \ref{dichotomy}, and Proposition \ref{selfsimilarall}.

We will show (see Section \ref{sec:tangent}) that the metric flow $\XX^{z'}$ associated with $z'$ is a metric soliton in the sense of Definition \ref{def:met_soliton}, and that its regular part $\RR^{z'}$, under the embedding $\iota_{z'}$, coincides with $\RR'_{(-\infty,0)}$. In general, however, it is not known whether $\iota_{z'}(\XX^{z'}_{(-\infty,0)})=Z'_{(-\infty, 0)}$ holds unconditionally. We will prove in Theorem \ref{metricsolitonbdscalcomplete} that this equality does hold if the scalar curvature on $\RR'_{-1}$ is uniformly bounded.

On $Z_{\III^-}$, we have the following regular--singular decomposition:
	\begin{align*}
Z_{\III^-}=\RR_{\III^-} \sqcup \MS,
	\end{align*}
where $\RR_{\III^-}$ denotes the restriction of $\RR$ to $\III^-$. It can be proved (see Theorem \ref{thm:tworegular}) that a point $z$ is regular if and only if any of its tangent flows is isometric to $(\R^n\times\R,d_E^*,(\vec 0^n,0),\t)$ or $(\R^n\times\R_-,d_E^*,(\vec 0^n,0),\t)$, where $d_E^*$ denotes the spacetime distance induced by Definition \ref{defn:dstar-distance} on the static Euclidean flow (see Example \ref{ex:euclidean}). Here the concept of isometry is given in Definition \ref{def:iso}. Equivalently, $z$ is regular if and only if $\NN_z(0)\geq-\ep_n$ (see Proposition \ref{prop:limitepsilon}).

The singular set $\MS$ admits a natural stratification:
	\begin{equation*}
		\mathcal S^0 \subset \mathcal S^1 \subset \cdots \subset \mathcal S^{n+1}=\mathcal S,
	\end{equation*}
	where a point $z \in \MS^k$ lies in the stratum if and only if no tangent flow at $z$ is $(k+1)$-symmetric. Here, a tangent flow $(Z', d_{Z'}, z', \t')$ is said to be $k$-symmetric if one of the following holds:
		\begin{enumerate}[label=\textnormal{(\arabic*)}]
		\item  $(Z',d_{Z'},z',\t')$ is $k$-splitting and is not a static cone.
	
		\item $(Z',d_{Z'},z',\t')$ is a static cone that is $(k-2)$-splitting.
	\end{enumerate} 
	
Roughly speaking, a \textbf{static cone} is characterized by $\mathrm{image}(\t')=\R$ and vanishing Ricci curvature on $\RR'$. The singular strata $\MS^k$ that we introduce differ from those considered by Bamler (see \cite[Theorem 2.8]{bamler2020structure}). Notably, our framework allows for the presence of a \textbf{quasi-static cone}, which satisfies the vanishing Ricci curvature condition only on $\RR'_{(-\infty, t_a]}$ for some constant $t_a \in [0, \infty)$, but not beyond. Such quasi-static cones do not arise in Bamler's setting. For precise definitions and related properties of static and quasi-static cones, we refer to Definition \ref{def:stacone}, Theorem \ref{staticcone}, and Proposition \ref{prop:staticcone1}.

\begin{thm}\label{thm:introsingular}
In the same setting as above, we have
	\begin{align*}
\MS=\MS^{n-2}.
	\end{align*}
\end{thm}

Theorem \ref{thm:introsingular} is derived from \cite[Theorem 2.8]{bamler2020structure}, where the corresponding metric solitons are excluded (see Theorem \ref{thm:singular} for details). We can also formulate the following quantitative singular strata as in \cite{cheeger2013lower} and \cite{bamler2020structure}.

\begin{defn} 
	For $\ep > 0$ and $0<r_1<r_2<\infty$, the quantitative singular strata
	\[  \MS^{\ep,0}_{r_1,r_2} \subset \MS^{\ep,1}_{r_1,r_2} \subset   \ldots \subset  \MS^{\ep,n-2}_{r_1,r_2} \subset Z_{\III^-} \]
	are defined as follows:
	$z \in  \MS^{\ep,k}_{r_1,r_2}$ if and only if $\t(z)-\ep^{-1} r_2^2 \in \III^-$ and for all $r \in [r_1, r_2]$, $z$ is not $(k+1,\ep,r)$-symmetric. 
	Here, the precise definition of a point being $(k, \ep, r)$-symmetric can be found in Definition \ref{def:epsym}.
\end{defn}

The following identity is clear from the above definitions: for any $L>1$,
\begin{align}\label{eq:introsiden1}
\MS^{k}=\bigcup_{\ep \in (0, L^{-1})} \bigcap_{0<r<\ep L} \MS^{\ep,k}_{r, \ep L}.
\end{align}

\begin{thm}\label{thm:intro6}
Given $x_0 \in Z$, $\ep>0$ and $r>0$ with $\t(x_0)-2 r^2 \in \III^-$, the following statements hold. 
		\begin{enumerate}[label=\textnormal{(\alph{*})}]
			\item For any $\delta \in (0, \ep )$,
	\begin{align*}
\abs{B^*_{Z} \lc \MS^{\ep, n-2}_{\delta r, \ep r}, \delta r \rc \cap B_Z^* (x_0, r)} \le C(n,Y,\sigma, \ep) \delta^{4-\ep} r^{n+2},
	\end{align*}
where $B^*_{Z} \lc A, s \rc$ denotes the $s$-neighborhood of a subset $A$ with respect to $d_Z$. Moreover, for any $t\in \R$,
	\begin{align*}
\abs{B^*_{Z} \lc \MS^{\ep,n-2}_{\delta r, \ep r}, \delta r \rc \cap B_Z^* (x_0, r) \cap Z_t}_t \le C(n,Y,\sigma, \ep) \delta^{2-\ep} r^{n}.
	\end{align*}
			
		\item For any $\delta \in (0, \ep )$,
	\begin{align*}
	\abs{\{r_{\Rm} <  \delta r \} \cap B_Z^* (x_0, r)} \le  C(n,Y,\sigma, \ep) \delta^{4-\ep} r^{n+2},
	\end{align*}		
	where $r_{\Rm}$ denotes the curvature radius; see Definition \ref{curvatureradiuslimit}. Moreover, for any $t\in \R$,
	\begin{align*} 
	\abs{\{r_{\Rm} <  \delta r \} \cap B_Z^* (x_0, r) \cap Z_t}_t \le C(n,Y,\sigma, \ep) \delta^{2-\ep} r^{n}.
	\end{align*}

		\item For any $\ep>0$, we have
	\begin{align*}
		\int_{B^*_Z(x_0,r)\cap \RR}|\Rm|^{2-\ep} \, \mathrm{d}V_{g^Z_t}\mathrm{d}t \le \int_{B^*_Z(x_0,r)\cap \RR}r_{\Rm}^{-4+2\ep} \, \mathrm{d}V_{g^Z_t}\mathrm{d}t\leq C(n,Y,\sigma,\ep)r^{n-2+2\ep}.
	\end{align*} 
	Moreover, for any $t \in \R$,
	\begin{align*}
		\int_{B^*_Z(x_0,r)\cap \RR_t}|\Rm|^{1-\ep} \, \mathrm{d}V_{g^Z_t} \leq \int_{B^*_Z(x_0,r)\cap \RR_t}r_{\Rm}^{-2+2\ep} \, \mathrm{d}V_{g^Z_t} \leq C(n,Y,\sigma,\ep)r^{n-2+2\ep}.		
	\end{align*} 
				\end{enumerate}
\end{thm}

Theorem \ref{thm:intro6} is proved in Corollary \ref{cor:006} and Theorem \ref{thm:integral-Rm-dstar-limit}. Together with \eqref{eq:introsiden1}, it implies the following result.

\begin{thm}\label{thm:minkow}
The Minkowski dimension with respect to $d_Z$ satisfies
	\begin{align*}
		\dim_{\MMM} \mathcal S \le n-2.
	\end{align*}
\end{thm}

As an application, we consider a closed Ricci flow $\XX=\{M^n,(g(t))_{t\in[-T,0)}\}$ such that $0$ is the first singular time. We assume that $T<\infty$ and that the entropy of $\XX$ is bounded below by $-Y$.

We consider the $d^*$-distance on $\XX_{[-0.99T,0)}$, defined as in Definition \ref{def:introd-dstar}. For simplicity, we fix $\sigma=1/100$ throughout.

We then define
\begin{align*}
(Z,d_Z,\t)
\end{align*}
to be the \textbf{metric completion} of $\XX_{[-0.98T,0)}$ with respect to $d^*$. By construction,
\[
(Z_{[-0.98T,0)},d_Z)=(\XX_{[-0.98T,0)},d^*),
\]
that is, the completion only adds the points in $Z_0$. As shown in Section \ref{sec:first}, $(Z,d_Z,\t)$ is a noncollapsed Ricci flow limit space.

\begin{thm}\label{thm:intro7}
With the above assumptions, there exists a constant $C_\ep$ depending on $\ep$ and the Ricci flow $\XX$ such that the following statements hold.
		\begin{enumerate}[label=\textnormal{(\alph{*})}]
		\item For any sufficiently small $\ep>0$,
		\begin{align*}
\int_{-T}^0 \int_M |\Rm|^{2-\ep}  \, \mathrm{d}V_{g(t)} \mathrm{d}t \le \int_{-T}^0 \int_M r_{\Rm}^{-4+2\ep}  \, \mathrm{d}V_{g(t)} \mathrm{d}t \le C_{\ep}.
	\end{align*} 
Moreover, for any $t \in [-T, 0)$,
		\begin{align*}
\int_M |\Rm|^{1-\ep}  \, \mathrm{d}V_{g(t)} \le  \int_M r_{\Rm}^{-2+2\ep}  \, \mathrm{d}V_{g(t)}  \le C_{\ep}.
	\end{align*} 
	
	\item The limit $V_0:=\lim_{t \nearrow 0} |M|_t \in [0, \infty)$ exists. $V_0=0$ if and only if $\RR_0 =\emptyset$. In this case, we have
			\begin{align*}
|M|_t \le C_\ep |t|^{1-\ep}
	\end{align*} 
	for any $t \in [-T, 0)$ and any small $\ep>0$.
	
	\item For every sufficiently small $\delta>0$ and $\ep>0$, we have
	\begin{align*}
\abs{\left\{y\in Z_0 \mid d_0^Z(y,\MS)<\delta \right\} }_0\leq C_\ep \delta^{2-\ep}.
	\end{align*} 
\end{enumerate}
\end{thm}

Theorem \ref{thm:intro7} is proved in Theorem \ref{intefirst}, Proposition \ref{prop:volume0}, Corollary \ref{cor:volume}, and Theorem \ref{thm:slicecodim2}.

\begin{figure}[H]
	\centering
	\includegraphics[width=0.45\linewidth]{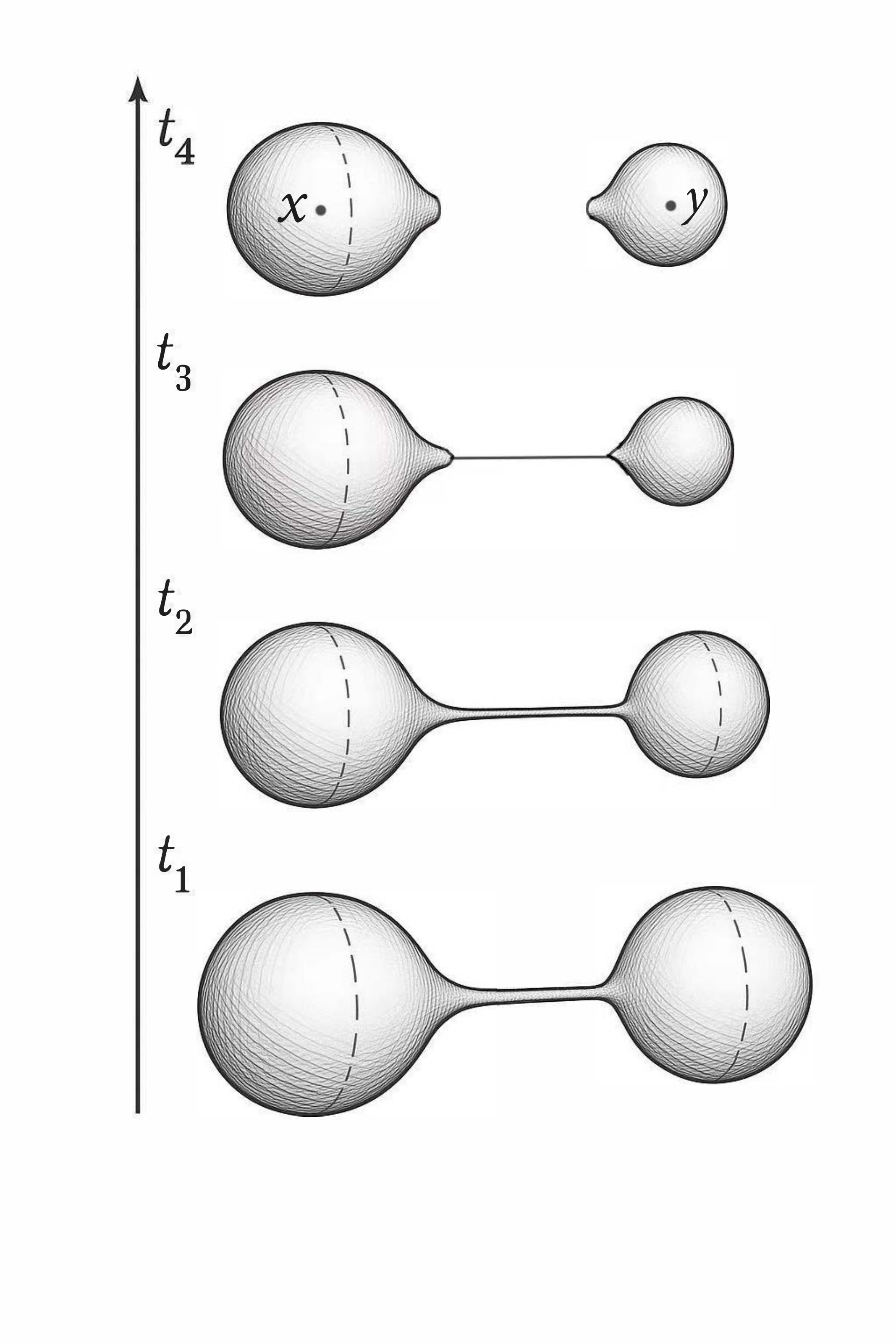}
	\caption{The singular set is a segment at $t_3$; $\iota_x(\RR^x_t) = \iota_y(\RR^y_t)$ for $t < t_3$.}
	\label{embedding}
\end{figure}

\subsection*{Organization of the paper}

This paper is organized as follows.

In Section \ref{secpreliminary}, we introduce the necessary definitions and basic properties of related concepts in metric measure spaces. We also review known results for closed Ricci flows, including estimates for the Nash entropy, heat kernel bounds, and volume bounds.

In Section \ref{secRicciflowlimitspace}, we define the spacetime $d^*$-distance and establish many of its fundamental properties. This section also contains the proof of Theorem \ref{thm:intro1}.
 
In Section \ref{sec:f}, we review Bamler's theory of $\F$-convergence and explain how $\F$-limits relate to the Ricci flow limit space $Z$. Theorem \ref{thm:iden} is also proved in this section.

Section \ref{sec:smooth} focuses on the regular part of the Ricci flow limit space. We detail the construction of the Ricci flow spacetime and analyze the associated conjugate heat kernel measures. The proof of Theorem \ref{thm:intro3} is presented here.
 
In Section \ref{sec:extend}, we define the time-slice distance $d^Z_t$ on $Z_t$ and prove several of its key properties, including Theorem \ref{thm:intro4}.
 
Section \ref{sec:tangent} is devoted to the study of tangent flows of the Ricci flow limit space. We prove Theorem \ref{thm:intro5} in this section.
 
In Section \ref{sec:singular}, we investigate the singular set and the quantitative singular strata, providing estimates on their size. Theorems \ref{thm:intro6} and \ref{thm:minkow} are proved here.

In Section \ref{sec:first}, we apply the results established earlier to the first singular time of a closed Ricci flow. Theorem \ref{thm:intro7} is proved in this section.

Section \ref{sec:almost} focuses on almost splitting maps. We establish their basic properties and show how they relate to the splitting of the limit space.

Finally, in Section \ref{sec:further}, we extend the main results of this paper to Ricci flows with bounded curvature on each compact time interval. We also study the noncollapsed Ricci flow limit spaces arising as limits of sequences of K\"ahler Ricci flows.

In Appendix \ref{app:A}, we derive two versions of estimates for the conjugate heat kernel measures. Appendix \ref{app:B} establishes the equivalence of various spacetime distances. In Appendix \ref{app:C}, we explore the relationship between eigenvalues and almost splitting, a result that may be of independent interest. Appendix \ref{app:D} introduces the notion of the spine of a Ricci shrinker space and investigates its basic properties. Finally, we include a list of notation for reference.
\\
\\
\textbf{Acknowledgments}. Hanbing Fang would like to thank his advisor, Prof. Xiuxiong Chen, for his encouragement and support. Hanbing Fang is supported by the Simons Foundation. Yu Li is supported by the National Key R\&D Program of China (2025YFA1018200), NSFC-12522105, YSBR-001, and research funds from the University of Science and Technology of China and the Chinese Academy of Sciences.

\section{Preliminaries}\label{secpreliminary}
In this section, we review some basic concepts for metric measure spaces and useful results for closed Ricci flows.

\subsection*{Probability measures on metric spaces}\label{secmetricmeasurespace}

Let $(X,d)$ be a complete, separable metric space. Denote by $\PP(X)$ the space of all probability measures on $X$. In particular, we denote by $\delta_x\in \PP(X)$ the Dirac measure at $x \in X$. A tuple $(X,d,\mu)$ with $\mu\in \PP(X)$ is called a \textbf{metric measure space}.

\begin{defn}[Variance and Wasserstein distance]\label{defnvariance}
The \textbf{variance} between two probability measures $\mu_1,\mu_2\in \PP(X)$ is defined by
\begin{equation*}
	\Var_X(\mu_1,\mu_2):= \int_X\int_X d^2(x_1,x_2)\, \mathrm{d}\mu_1(x_1)\, \mathrm{d}\mu_2(x_2).
\end{equation*}
For simplicity, we set $\Var_X(\mu)=\Var_X(\mu,\mu)$. For $p\geq 1$, the \textbf{$W_p$-Wasserstein} distance between $\mu_1,\mu_2 \in \PP(X)$ is defined by
\begin{align*}
	d_{W_p}^X(\mu_1,\mu_2):=\inf_{\Pi}\lc\int_{X\times X}d^p(x_1,x_2)\, \mathrm{d}\Pi(x_1,x_2)\rc^{1/p},
\end{align*}
\index{$\ensuremath{W_p}$-Wasserstein distance}
where the infimum is taken over all couplings $\Pi\in\PP(X\times X)$ between $\mu_1,\mu_2$, that is, any such $\Pi$ satisfies $(\pi_i)_{\#} \Pi=\mu_i$ for $i=1, 2$, where $\pi_i$ is the projection from $X \times X$ to the $i$-th copy of $X$.
\end{defn}

The following result is immediate from the Kantorovich-Rubinstein duality, see \cite[Chapter 5]{villani2009optimal}.

\begin{lem}\label{lem:duality}
For any $\mu_1,\mu_2 \in \PP(X)$, we have
\begin{align*}
	d_{W_1}^X(\mu_1,\mu_2)=\sup_{f \in C_b(X),\ \lVert f\rVert_{\Lip}\leq 1}\lc\int_X f \, \mathrm{d}\mu_1-\int_X f \, \mathrm{d}\mu_2\rc,
\end{align*}
where $C_b(X)$ denotes the space of bounded continuous functions on $X$.
\end{lem}

\begin{defn}\label{defnGWpdistance}
	Given two metric measure spaces $(X_1,d_1,\mu_1)$ and $(X_2,d_2,\mu_2)$, the \textbf{Gromov--$W_p$-Wasserstein distance} for $p\geq 1$ is defined as
	\begin{align*}
		d_{GW_p}\lc (X_1,d_1,\mu_1), (X_2,d_2,\mu_2)\rc:=\inf d_{W_p}^A\lc (\varphi_1)_*\mu_1, (\varphi_2)_*\mu_2\rc,
	\end{align*}
	where the infimum is taken over all isometric embeddings $\varphi_i:(X_i,d_i)\to (A,d_A)$ for $i=1,2$.
\end{defn}
\index{$\ensuremath{d_{GW_p}}$}
The following lemma from \cite[Lemma 3.2]{bamler2020entropy} gives basic properties of variance:
\begin{lem}\label{lem:var}
	For any $\mu_1,\mu_2,\mu_3\in \PP(X)$, we have
	\begin{align*}
		\sqrt{\Var_X(\mu_1,\mu_3)}\leq& \sqrt{\Var_X(\mu_1,\mu_2)}+\sqrt{\Var_X(\mu_2,\mu_3)},\\
		d_{W_1}^X(\mu_1,\mu_2)\leq& \sqrt{\Var_X(\mu_1,\mu_2)}\leq d_{W_1}^X(\mu_1,\mu_2)+\sqrt{\Var_X(\mu_1)}+\sqrt{\Var_X(\mu_2)}.
	\end{align*}
\end{lem}

Next, we recall that a sequence of $\mu_i\in \PP(X)$ converges \textbf{weakly} to $\mu_{\infty}\in \PP(X)$ if, for any $f \in C_b(X)$, 
	\begin{align*}
\lim_{i \to \infty}\int_X f \, \mathrm{d}\mu_i=\int_X f \, \mathrm{d}\mu_{\infty}.
	\end{align*}

\begin{prop}\label{takelimitVarW1}
	Suppose that a sequence of $\mu_i\in \PP(X)$ converges weakly to $\mu_{\infty}\in \PP(X)$. Then the following conclusions hold.
	\begin{enumerate}[label=\textnormal{(\roman{*})}]
	\item We have
	\begin{align*}
\Var_X(\mu_\infty)\leq \liminf_{i \to \infty}\Var_X(\mu_i).
	\end{align*}	
	
\item If $\Var_X(\mu_i)\leq C$ for a uniform constant $C$, then $\mu_i\to\mu_\infty$ in $d_{W_p}^X$ for any $p \in [1, 2)$.
	\end{enumerate}
\end{prop}

\begin{proof}
(i): This is immediate from the definition of weak convergence, since $\mu_i \otimes \mu_i$ converges weakly to $\mu_{\infty} \otimes \mu_{\infty}$ in $\PP(X \times X)$.

(ii): By the definition of the variance, for each $i \in \mathbb N \cup \{\infty\}$, there exists $x_i \in X$ such that
	\begin{align}\label{eq:x20100}
\int_X d^2(x_i, x) \, \mathrm{d}\mu_i(x) \le C.
	\end{align}	
In particular, this implies that for any $L >0$,
	\begin{align} \label{eq:x20101}
\mu_i \lc \{x \in X \mid d(x, x_i) \ge L\}  \rc \le C L^{-2}.
	\end{align}	
We claim that there exists $C_1>0$ such that $d(x_{\infty}, x_i) \le C_1$ for any $i \in \mathbb N$. Suppose otherwise. Then, after passing to a subsequence, we have $\lim_{i \to \infty}d(x_{\infty}, x_i)=+\infty$. Thus, it follows from the weak convergence and \eqref{eq:x20101} that for any $L >0$,
	\begin{align*}
\mu_{\infty} \lc \{x \in X \mid d(x, x_{\infty}) < L\}  \rc \le  \liminf_{i \to \infty} \mu_i \lc \{x \in X \mid d(x, x_{\infty}) < L\} \rc=0,
	\end{align*}	
which is impossible.

By the claim and \eqref{eq:x20100}, we obtain for any $L \ge 2C_1$,
	\begin{align*}
\mu_i \lc \{x \in X \mid d(x, x_\infty) \ge L\}  \rc \le 4C L^{-2}.
	\end{align*}	
Moreover, we have
		\begin{align*}
\int_X d^2(x, x_\infty) \, \mathrm{d}\mu_i(x) \le C_2.
	\end{align*}	
		
Given $p \in [1, 2)$, we have for any $L \ge 2C_1$,
		\begin{align} \label{eq:x20103}
\int_{d(x, x_\infty) \ge L} d^p(x, x_\infty) \, \mathrm{d}\mu_i(x) \le \lc \int_X d^2(x, x_\infty) \, \mathrm{d}\mu_i(x) \rc^{\frac p 2} \lc \mu_i \lc \{x \in X \mid d(x, x_\infty) \ge L\} \rc \rc^{1-\frac p 2} \le C_3 L^{p-2}.
	\end{align}	
		
Consequently, the conclusion follows from \eqref{eq:x20103} and \cite[Theorem 6.9]{villani2009optimal}.
\end{proof}

\subsection*{Preliminary results on the Ricci flow}\label{secgeneralRF}

In this section, we consider a closed Ricci flow solution $\XX=\{M^n,(g(t))_{t\in I}\}$, where $M$ is an $n$-dimensional closed manifold, $I$ is a closed interval, and $(g(t))_{t\in I}$ is a family of smooth metrics on $M$ satisfying the Ricci flow equation for all $t \in I$:
		\begin{align*}
\partial_t g(t)=-2 \Ric(g(t)).
	\end{align*}	

For convenience, we use $x^* \in \XX$\index{$\ensuremath{x^*}$} to denote a spacetime point $x^* \in M \times I$ and define $\t(x^*) \in I$\index{$\ensuremath{\t}$} as its time component. For any subinterval $I' \subset I$, we set $\XX_{I'}=\{M^n,(g(t))_{t\in I'}\}$. We denote by $d_t$ the distance function on $M$ and by $dV_{g(t)}$ the volume form induced by $g(t)$. For any $x^*=(x, t) \in \XX$, let $B_t(x,r)$ denote the geodesic ball centered at $x$ with radius $r$ with respect to $g(t)$. The Riemannian curvature, Ricci curvature, and scalar curvature of $g(t)$ are denoted by $\Rm$, $\Ric$, and $\scal$, respectively, with the time parameter $t$ omitted when there is no ambiguity. Additionally, we define $R_{\min}$ as a lower bound of the scalar curvature. In general, for any $t> t_0$ with $t,t_0\in I$, the scalar curvature satisfies the bound
		\begin{align} \label{eq:lowerscal}
\scal(\cdot,t)\geq -\frac{n}{2(t-t_0)}
	\end{align}	
as shown, for instance, in \cite[Corollary 3.3.5]{topping2006lectures}.

For the closed Ricci flow $\XX$, we denote by $d_{W_p}^t(\mu_1,\mu_2)$\index{$\ensuremath{d_{W_p}^t}$} the $W_p$-Wasserstein distance and by $\Var_t(\mu_1,\mu_2)$\index{$\ensuremath{\Var_t}$} the variance between two probability measures $\mu_1$ and $\mu_2$ on $M$ with respect to $g(t)$.

We define the heat operator as $\square:=\partial_t-\Delta$\index{$\ensuremath{\square}$} and its conjugate operator as $\square^*:=-\partial_t-\Delta+\scal$\index{$\ensuremath{\square^*}$}. Let $K(x, t;y, s)$\index{$K(x, t;y,s)$} be the heat kernel on the Ricci flow for $t>s$. More precisely, it satisfies the following system:
  \begin{align*}
  \begin{cases}
    &\square K(\cdot,\cdot;y,s) =0,\\
    &\square^{*} K(x,t;\cdot,\cdot) =0,   \\
    &\lim_{t\searrow s} K(\cdot,t;y,s)=\delta_{y},\\
    &\lim_{s\nearrow t} K(x,t;\cdot,s)=\delta_{x}.
        \end{cases}
  \end{align*}

\begin{defn}\label{def:chkm}
	The \textbf{conjugate heat kernel measure} $\nu_{x^*;s}$\index{$\ensuremath{\nu_{x^*;s}}$} based at $x^*=(x, t)$ is defined as
		\begin{align*}
\mathrm{d}\nu_{x^*;s}=\mathrm{d}\nu_{x,t;s}:=K(x,t;\cdot,s) \, \mathrm{d}V_{g(s)}.
	\end{align*}		
	It is clear that $\nu_{x^*;s}$ is a probability measure on $M$. If we set 
			\begin{align*}
\mathrm{d}\nu_{x^*;s}=(4\pi(t-s))^{-n/2}e^{-f_{x^*}(\cdot,s)} \, \mathrm{d}V_{g(s)},
	\end{align*}	
	then the function $f_{x^*}$ is called the \textbf{potential function} at $x^*$ which satisfies:
\begin{equation*}
	-\partial_s f_{x^*}=\Delta f_{x^*}-|\nabla f_{x^*}|^2+\scal-\frac{n}{2(t-s)}.
\end{equation*}	
\end{defn}

Next, we recall the definitions of the Nash entropy and $\WW$-entropy based at $x^* \in \XX$.

\begin{defn}\label{defnentropy}
	The \textbf{Nash entropy} based at $x^*\in \XX$ is defined by
	\begin{equation*}
		\NN_{x^*}(\tau):= \int_M f_{x^*} \, \mathrm{d}\nu_{x^*;\t(x^*)-\tau}-\frac{n}{2}\index{$\ensuremath{\NN_{x^*}(\tau)}$}
	\end{equation*}
	for any $\tau>0$ with $\t(x^*)-\tau \in I$, where $f_{x^*}$ is the potential function at $x^*$. Moreover, the $\WW$-entropy based at $x^*$ is defined by
	\begin{equation*}
		\WW_{x^*}(\tau):= \int_M \tau(2\Delta f_{x^*}-|\nabla f_{x^*}|^2+\scal)+f_{x^*}-n \, \mathrm{d}\nu_{x^*;\t(x^*)-\tau}.\index{$\ensuremath{\WW_{x^*}(\tau)}$}
	\end{equation*}
\end{defn}

The following proposition follows from a direct calculation; see \cite{hein2014new} and \cite[Section 5]{bamler2020entropy}.

\begin{prop}\label{propNashentropy}
	For any $x^* \in \XX$ with $\t(x^*)-\tau \in I$ and $\scal(\cdot, \t(x^*)-\tau) \ge R_{\min}$, we have the following inequalities.
	\begin{enumerate}[label=\textnormal{(\roman{*})}]
	\item $\displaystyle-\frac{n}{2\tau}+R_{\min} \leq \diff{}{\tau}\NN_{x^*}(\tau)\leq 0$.
	\item $\displaystyle \diff{}{\tau}\lc\tau \NN_{x^*}(\tau)\rc=\WW_{x^*}(\tau)\leq 0$.
	\item $\displaystyle \diff{^2}{\tau^2}\lc\tau\NN_{x^*}(\tau)\rc=-2\tau\int_M \abs{\Ric+\nabla^2 f_{x^*}-\frac{1}{2\tau}g}^2\,\mathrm{d}\nu_{x^*;\t(x^*)-\tau}\leq 0$.
	\end{enumerate}	
\end{prop}

We also use the notation $\NN_s^*(x^*)=\NN_{x^*}(\t(x^*)-s)$\index{$\ensuremath{\NN_s^*(x^*)}$} as in \cite[Section 5]{bamler2020entropy}. The following result is \cite[Corollary 5.11]{bamler2020entropy}.

\begin{prop}\label{propNashentropy1}
For any $x_1^*, x_2^* \in \XX$ and $s<t \le \min\{\t(x_1^*), \t(x_2^*)\}$ with $s \in I$ and $\scal(\cdot, s) \ge R_{\min}$, we have
	\begin{align*}
		\NN_s^*(x_1^*)-\NN_s^*(x_2^*)\leq \lc \frac{n}{2(t-s)}-R_{\min} \rc^{\frac 1 2} d_{W_1}^{t}( \nu_{x_1^*;t},\nu_{x_2^*;t}) +\frac{n}{2}\log\lc\frac{\t(x_2^*)-s}{t-s}\rc.
	\end{align*}
\end{prop}

\begin{defn}\label{defncurvatureradius}
	For $x^*=(x,t)\in\XX$, the \textbf{curvature radius} $r_{\Rm}$\index{$\ensuremath{r_{\Rm}}$} is defined to be the supremum over all $r>0$ such that $|\Rm|\leq r^{-2}$ on the parabolic ball $B_t(x,r)\times [t-r^2,t+r^2] \cap I$. 
\end{defn}

The following $\ep$-regularity from \cite[Theorem 10.2]{bamler2020entropy} will be useful later:
\begin{thm}\label{epregularityNash}
	There exists a dimensional constant $\ep_n>0$ such that the following holds. If $x^* \in\XX$ satisfies $\t(x^*)-r^2 \in I$ and $\NN_{x^*}(r^2)\geq -\ep_n$, then $r_{\Rm}(x^*)\geq \ep_n r$.
\end{thm}

Now we recall some monotonicity formulas from \cite[Lemma 2.7, Corollary 3.7]{bamler2020entropy} and their consequences (see also \cite{mccann2010ricci} and \cite{topping2014ricci}).
\begin{prop}\label{monotonicityW1}
	Let $\nu_1,\nu_2 \in C^{\infty}(M \times I'),\, I' \subset I$ be two nonnegative solutions of the conjugate heat equation $\square^* \nu_1=\square^* \nu_2=0$ with $\int_M \nu_i(\cdot,t) \, \mathrm{d}V_{g(t)}=1$ for $i=1,2$ and any $t\in I'$. If we set $\mathrm{d}\mu_{i,t}=\nu_i(\cdot,t) \, \mathrm{d}V_{g(t)}$, then 
	\begin{align*}
		t\mapsto d_{W_1}^t(\mu_{1,t},\mu_{2,t}) \quad \text{and} \quad t\mapsto \Var_t(\mu_{1,t},\mu_{2,t})+H_nt
	\end{align*}
	are nondecreasing for $t \in I'$, where $H_n:=(n-1)\pi^2/2+4$\index{$H_n$}. In particular, for any $x^*_1,x^*_2\in \XX$, $d_{W_1}^t(\nu_{x^*_1;t},\nu_{x^*_2;t})$ and $\Var_t(\nu_{x^*_1;t},\nu_{x^*_2;t})+H_nt$ are nondecreasing for $t \in I$ and $t \le \min\{\t(x_1^*), \t(x_2^*)\}$.
\end{prop}

\begin{defn} \label{defh-center}
A point $(z,t) \in \XX$ is called an \textbf{$H$-center}\index{$\ensuremath{H}$-center} of $x_0^* \in \XX$ for a constant $H>0$ if $t \in I$, $t<\t(x_0^*)$ and 
\begin{align*}
	\Var_t(\delta_z,\nu_{x_0^*;t})\leq H(\t(x_0^*)-t).
\end{align*}
By Proposition \ref{monotonicityW1}, there exists an $H_n$-center for every $t \in I$ with $t<\t(x_0^*)$.
\end{defn}

We have the following result from \cite[Propositions 3.12, 3.13]{bamler2020entropy}.

\begin{prop}\label{existenceHncenter}
Any two $H_n$-centers $(z_1, t)$ and $(z_2, t)$ of $x_0^*$ satisfy $d_{t}(z_1,z_2)\leq 2\sqrt{H_n(\t(x_0^*)-t)}$. Moreover, if $(z,t)$ is an $H_n$-center of $x_0^*$, then for any $L>0$, we have
	\begin{align*}
		\nu_{x^*_0;t}\lc B_t \lc z,\sqrt{LH_n(\t(x_0^*)-t)} \rc \rc\geq 1-L^{-1}.
	\end{align*}
\end{prop}

The following theorem gives a sharp upper bound of the heat kernel, which improves \cite[Theorem 7.2]{bamler2020entropy}; see also \cite[Theorem 14]{li2020heat} and \cite[Theorem 4.15]{li2024heat}.

\begin{thm}\label{heatkernelupperbdgeneral}
	Let $\XX=\{M^n,(g(t))_{t\in I}\}$ be a closed Ricci flow with $[t, t_0] \subset I$. Then for any $\epsilon>0,L>0$ and $(x_0,t_0) \in \XX$, the following statements hold.
		\begin{enumerate}[label=\textnormal{(\roman{*})}]
	\item We have
	\begin{align}\label{sharpconcentration}
		\nu_{x_0,t_0;t}\lc M\setminus B_t(z,L)\rc\leq C(n,\ep) \exp\lc-\frac{L^2}{(4+\epsilon)(t_0-t)} \rc. 
	\end{align}
	
\item If $\scal(\cdot, t) \ge R_{\min}$, then for any $(y, t) \in \XX$,
	\begin{equation}\label{sharpheatkernelupper}
		K(x_0,t_0;y,t)\leq \frac{C\lc n,R_{\min}(t_0-t),\ep\rc}{(t_0-t)^{n/2}}\exp\lc -\frac{d_t^2(z,y)}{(4+\epsilon)(t_0-t)}-\NN_{x_0,t_0}(t_0-t) \rc.
	\end{equation}
\end{enumerate}	
Here, $(z,t)$ is any $H_n$-center of $(x_0,t_0)$.
\end{thm}
\begin{proof}
	The upper bounds in \eqref{sharpconcentration} and \eqref{sharpheatkernelupper} are analogous to those in \cite[Theorems 3.14, 7.2]{bamler2020entropy}, but the constant $8+\epsilon$ is replaced by the sharp constant $4+\epsilon$. For simplicity, we set $\mathrm{d}\nu_t=\mathrm{d}\nu_{x_0,t_0;t}$ and define the Laplace transform as
	\begin{equation}\label{defnlaplacefunctional}
		U_t(\lambda)=\sup\int_{M}e^{\lambda h}\, \mathrm{d}\nu_t,
	\end{equation}
	where the supremum runs over all bounded integrable $1$-Lipschitz functions $h$ on $(M,g(t))$ satisfying $\int_M h\,\mathrm{d}\nu_t=0$. By the proof of \cite[Theorem 1.30]{hein2014new}, the following bound for the Laplace transform holds:
	\begin{equation}\label{laplacefunctionalbd}
		U_t(\lambda)\leq e^{(t_0-t)\lambda^2}.
	\end{equation}
	For every integrable $1$-Lipschitz function $F:M\to \R$ and for every $r\geq 0$, applying \eqref{defnlaplacefunctional} to $F-\int_M F\,\mathrm{d}\nu_t$ (first to its bounded truncations and then passing to the limit), we obtain from \eqref{laplacefunctionalbd} that
	\begin{align*}
		\int_M e^{\lambda\lc F-\int_M F\,\mathrm{d}\nu_t\rc}\, \mathrm{d}\nu_t\leq e^{(t_0-t)\lambda^2}.
	\end{align*}
	Thus,
	\begin{align*}
		\nu_t\lc\left\{F\geq \int_M F\,\mathrm{d}\nu_t+r \right\}\rc\leq \inf_{\lambda>0}\lc e^{(t_0-t)\lambda^2-\lambda r}\rc = e^{-\frac{r^2}{4(t_0-t)}}.
	\end{align*}
	Now we take $F(x)=d_t(x,z)$, where $(z,t)$ is the $H_n$-center of $(x_0,t_0)$. Then
	\begin{equation}\label{sharpconcen1}
		\nu_t\lc\left\{x \,\big|\, d_t(x,z)\geq \int_M d_t(z,\cdot) \, \mathrm{d}\nu_t+r \right\}\rc\leq e^{-\frac{r^2}{4(t_0-t)}}.
	\end{equation}
	Recall that, by the definition of an $H_n$-center, $\lc\int_M d^2_t(z,\cdot) \, \mathrm{d}\nu_t\rc^{\frac{1}{2}}\leq \sqrt{H_n(t_0-t)}$, and thus by the Cauchy--Schwarz inequality, we have
	\begin{align*}
		\int_M d_t(z, \cdot) \, \mathrm{d}\nu_t\leq \sqrt{H_n(t_0-t)}.
	\end{align*}
	This implies $\{x \mid d_t(x,z)\geq \sqrt{H_n(t_0-t)}+r\}\subset \{x \mid d_t(x,z)\geq \int_M d_t(z, \cdot) \, \mathrm{d}\nu_t +r\}$. Combining with \eqref{sharpconcen1}, we obtain
	\begin{equation}\label{gaussianmeasureestimate}
		\nu_t\lc\left\{ x \mid d_t(x,z)\geq \sqrt{H_n(t_0-t)}+r \right\}\rc\leq e^{-\frac{r^2}{4(t_0-t)}}.
	\end{equation}
	For any $L>0$, \eqref{gaussianmeasureestimate} implies 
	\begin{align*}
		\nu_t\lc M\setminus B_t(z,L)\rc\leq \exp\lc-\frac{\lc L-\sqrt{H_n(t_0-t)}\rc_+^2}{4(t_0-t)} \rc\leq C(n,\ep) \exp \lc -\frac{L^2}{(4+\epsilon)(t_0-t)} \rc,
	\end{align*}
	which gives \eqref{sharpconcentration}.
	
	We can now follow the argument in the proof of \cite[Theorem 7.2]{bamler2020entropy} or \cite[Theorem 4.15]{li2024heat} to conclude \eqref{sharpheatkernelupper}.
\end{proof}

We have the following gradient bound from \cite[Theorem 7.5]{bamler2020entropy}:

\begin{thm}\label{gradientheatkernel}
If $[s,t]\subset I$ and $\scal\geq R_{\min}$ on $\XX$, then there exists a constant $C=C(n,R_{\min}(t-s))<\infty$ such that
	\begin{align*}
		\frac{|\na_x K|(x,t;y,s)}{K(x,t;y,s)}\leq \frac{C}{(t-s)^{1/2}}\sqrt{\log\lc \frac{C\exp\lc-\NN_{x,t}(t-s)\rc}{(t-s)^{n/2}K(x,t;y,s)}\rc}.
	\end{align*}
\end{thm}

We also need the following volume estimates from \cite[Theorems 6.1, 6.2, 8.1]{bamler2020entropy}.

\begin{prop}\label{prop:volume}
	Assume $[t-r^2,t]\subset I$ and $\scal(\cdot, t-r^2) \ge R_{\min}$. 
	\begin{enumerate}[label=\textnormal{(\roman{*})}]
	\item For any $1\leq A<\infty$, 
	\begin{align*}
		|B_t(x,Ar)|_t\leq C(n,R_{\min}r^2) \exp\lc\NN_{x,t}(r^2)+C(n)A^2 \rc r^n.
	\end{align*}
\item If $(z, t-r^2)$ is an $H_n$-center of $(x, t)$, then
	\begin{align*}
		|B_{t-r^2}(z, \sqrt{2H_n}r)|_{t-r^2}\geq C(n, R_{\min}r^2) \exp \lc\NN_{x,t}(r^2) \rc r^n>0.
	\end{align*}
\item If $\scal\leq r^{-2}$ on $B_t(x,r)$, then 
	\begin{align*}
		|B_t(x,r)|_t\geq C(n) \exp \lc\NN_{x,t}(r^2) \rc r^n>0.
	\end{align*}
\end{enumerate}
Here, $\abs{\cdot}_t$ denotes the volume with respect to $g(t)$.	
\end{prop}

For later applications, we need the following $L^p$-Poincar\'e inequality, proved by \cite[Theorem 1.10]{hein2014new} and \cite[Theorem 11.1]{bamler2020entropy}.

\begin{thm}[Poincar\'e inequality]\label{poincareinequ}
	Let $\flow$ be a closed Ricci flow with $x_0^*=(x_0,t_0)\in\XX$. Suppose $\tau>0$ with $t_0-\tau \in I$, and $h\in C^1(M)$ with $\int_M h \, \mathrm{d}\nu_{x_0^*;t_0-\tau}=0$. Then for any $p\geq 1$, 
	\begin{align*}
		\int_M |h|^p\,\mathrm{d}\nu_{x_0^*;t_0-\tau}\leq C(p)\tau^{p/2}\int_M |\nabla h|^p\,\mathrm{d}\nu_{x_0^*;t_0-\tau}.
	\end{align*}
	Here, we can choose $C(1)=\sqrt{\pi}$ and $C(2)=2$. 
\end{thm}

We also have the following hypercontractivity estimate from \cite[Theorem 12.1]{bamler2020entropy}.

\begin{thm}[Hypercontractivity]\label{hypercontractivity}
	Let $\flow$ be a closed Ricci flow with $x_0^*=(x_0,t_0)\in\XX$. Suppose $0<\tau_1<\tau_2$ with $t_0-\tau_2 \in I$, and $u\in C^2(M \times [t_0-\tau_2, t_0-\tau_1])$ is a solution of $\square u=0$ or $u\geq 0$ with $\square u\leq 0$. If $1<q\leq p<\infty$ satisfies
		\begin{align*}
	\frac{\tau_2}{\tau_1}\geq\frac{p-1}{q-1},
	\end{align*}
	then we have
	\begin{align*}
		\lc \int_{M}|u|^p \, \mathrm{d}\nu_{x_0^*;t_0-\tau_1} \rc^{1/p}\leq \lc \int_{M}|u|^q\,\mathrm{d}\nu_{x_0^*;t_0-\tau_2} \rc^{1/q}.
	\end{align*}	
\end{thm}

In this paper, we mainly focus on the case where a Ricci flow $\XX$ has entropy bounded below. To formalize this, we introduce the following definition:
\begin{defn}\label{def:entropybound}
A closed Ricci flow $\flow$ is said to have entropy bounded below by $-Y$ at $x^* \in \XX$ if
\begin{align}\label{entropybd-Y}
	\inf_{\tau>0}\NN_{x^*}(\tau)\geq -Y,
\end{align}
where the infimum is taken over all $\tau>0$ for which the Nash entropy $\NN_{x^*}(\tau)$ is well-defined. 

Moreover, we say that the Ricci flow $\XX$ has entropy bounded below by $-Y$ if \eqref{entropybd-Y} holds for all $x^* \in \XX$.
\end{defn}

Under the assumption of a local scalar curvature bound, we have the following distance distortion estimates.

\begin{prop}\label{HncenterScal}
Let $\flow$ be a closed Ricci flow with entropy bounded below by $-Y$. Let $x^*=(x,t_0)\in \XX$ with $[t_0-2r^2, t_0]\subset I$. For any constant $R_0>0$, there exists a constant $C=C(n, Y, R_0)>0$ such that the following statements hold.
	\begin{enumerate}[label=\textnormal{(\roman{*})}]
		\item Assume $|\scal|\leq R_0 r^{-2}$ on $\{x\} \times [t_0-r^2, t_0]$. If $(z, t)$ is an $H_n$-center of $x^*$ with $t \in [t_0-r^2, t_0]$, then
		\begin{align}\label{Hncenterscal1}
		d_{t}(x,z)\leq C\sqrt{t_0-t}.
	\end{align}
	
	\item Assume $|\scal|\leq R_0 r^{-2}$ on $\{x\} \times [t_0-r^2, t_0]$ and $\{y\} \times [t_0-r^2, t_0]$. Then for any $t \in [t_0-r^2, t_0]$,
	\begin{align*}
		d_{t}(x, y) \le d_{t_0}(x, y)+C \sqrt{t_0-t}.
	\end{align*}

	\item Assume $|\scal|\leq R_0 r^{-2}$ on $B_{t_0}(x, r) \times [t_0-r^2, t_0+r^2] \cap I$. Then any $y \in M$ with $d_{t}(x, y) \le C^{-1} r$ for some $t \in [t_0-C^{-1} r^2, t_0+C^{-1} r^2] \cap I$ satisfies
	\begin{equation*}
		d_{t_0}(x,y) \le r.
	\end{equation*}
	\end{enumerate}	
\end{prop}
\begin{proof}
(i): Equation \eqref{Hncenterscal1} can be established using the same argument as in \cite[Proposition 4.4]{LW26}; see also \cite[Proposition 3.1]{jian2023improved}.

(ii): Let $(z_1, t)$ and $(z_2, t)$ be $H_n$-centers of $x^*$ and $y^*:=(y, t_0)$, respectively. From \eqref{Hncenterscal1}, we have
	\begin{align*}
		d_t(x, z_1)+d_t(y, z_2) \le C(n, Y, R_0)\sqrt{t_0-t}.
	\end{align*}	
Thus, it follows from Proposition \ref{monotonicityW1} that
	\begin{align*}
		d_t(x, y) \le & d_t(z_1, z_2)+C(n, Y, R_0)\sqrt{t_0-t} \\
		 \le & d_{W_1}^t(\nu_{x^*;t}, \nu_{y^*;t}) +C(n, Y, R_0)\sqrt{t_0-t} \le d_{t_0}(x, y)+C(n, Y, R_0)\sqrt{t_0-t}.
	\end{align*}	

(iii): This follows from the local distance-distortion estimate in \cite[Theorem 1.1]{bamler2017heat}; see also \cite[Lemma 4.21]{chen2020space} and \cite[Lemma 5.8]{li2024heat}. Although \cite[Theorem 1.1]{bamler2017heat} assumes a lower bound on Perelman's $\nu$-entropy, this condition can, in fact, be relaxed. Applying Theorem \ref{heatkernelupperbdgeneral} and Proposition \ref{prop:volume}, we can verify that it suffices to assume only a lower bound on the Nash entropy.
\end{proof}

We also need the following integral estimates from \cite[Proposition 6.2]{bamler2020structure}.

\begin{prop}\label{integralbound}
There exists a constant $\bar \alpha=\bar \alpha(n)>0$ such that the following holds. Let $\flow$ be a closed Ricci flow. Suppose $x_0^*=(x_0,t_0)\in\XX$ with $[t_0-2r^2,t_0]\subset I$, and define $\mathrm{d}\nu_t=\mathrm{d}\nu_{x_0^*;t}=(4\pi\tau)^{-n/2} e^{-f}\,\mathrm{d}V_{g(t)}$, where $\tau=t_0-t$. Assume that $\NN_{x_0^*}(2r^2)\geq -Y$ for some $r>0$. Then, for any $0<\theta\leq 1/2$ and $\alpha\in [0,\bar \alpha]$, the following estimates hold:
	\begin{align*}
		\int_{t_0-r^2}^{t_0-\theta r^2}\int_M\lc\tau|\Ric|^2+\tau |\na^2 f|^2+|\na f|^2+\tau |\na f|^4+\tau^{-1}e^{\alpha f}+\tau^{-1}\rc e^{\alpha f}\, \mathrm{d}\nu_t \mathrm{d}t&\leq C(n,Y)|\log\theta|,\\
		\int_M\lc \tau |\scal|+\tau |\Delta f|+\tau|\na f|^2+e^{\alpha f}+1\rc e^{\alpha f}\, \mathrm{d}\nu_{t_0-r^2}&\leq C(n,Y).
	\end{align*}
\end{prop}

We end this subsection with the following two-sided pseudolocality theorem from \cite[Theorem 10.1]{perelman2002entropy} and \cite[Theorem 2.47]{bamler2020structure}:

\begin{thm}[Two-sided pseudolocality theorem]\label{thm:twoside}
Let $\flow$ be a closed Ricci flow. For any $\alpha > 0$, there is an $\ep (n, \alpha) > 0$ such that the following holds.

Given $x_0^*=(x_0, t_0) \in \XX$ and $r > 0$ with $[t_0-r^2, t_0] \subset I$, if $|B_{t_0}(x_0,r)|_{t_0} \geq \alpha r^n$ and $|\Rm| \leq (\alpha r)^{-2}$ on $B_{t_0}(x_0,r)$, then
\begin{align*}
r_{\Rm}(x_0^*) \ge \ep r.
\end{align*}
\end{thm}

\section{Spacetime distance and Ricci flow limit spaces}\label{secRicciflowlimitspace}

We begin by fixing the time intervals. For a given constant $T \in (0,+\infty]$ and a parameter $\sigma\in(0,1/100]$, define
\begin{align*}
\III^-=(-(1-2\sigma)T,0],\quad \III=[-(1-2\sigma)T,0],\quad
\III^+=[-(1-\sigma)T,0],\quad \III^{++}=[-T,0].\index{$\ensuremath{\III^-,\III,\III^+,\III^{++}}$}
\end{align*}
If $T=+\infty$, we set $\III^-=\III=\III^+=\III^{++}=(-\infty,0]$.

\begin{defn}[Moduli space]\label{defnmoduli}
For a fixed constant $T\in(0,+\infty]$, the moduli space $\MM(n,T)$\index{$\ensuremath{\MM(n,T)}$} consists of all $n$-dimensional closed Ricci flows
\begin{align*}
\XX=\{M^n,(g(t))_{t\in\III^{++}}\}.
\end{align*}
For $Y>0$, the subspace $\MM(n,Y,T)\subset\MM(n,T)$ consists of all such flows with entropy bounded below by $-Y$ (see Definition \ref{def:entropybound}).
\end{defn}

Next, we define the spacetime distance on $\XX_{\III^+}$.

\begin{defn}\label{defn:dstar-distance}
Let $\XX\in\MM(n,T)$. For $x^*,y^*\in\XX_{\III^+}$, set
\begin{align*}
t_+:=\max\{\t(x^*),\t(y^*)\},\qquad t_-:=\min\{\t(x^*),\t(y^*)\}.
\end{align*}
We define
\begin{align}\label{eq:defn-dstar}
d^*(x^*,y^*):=\inf_{-(1-\sigma)T\leq\tau\leq t_-}
\max\left\{\sqrt{t_+-\tau},\ d_{W_1}^{\tau}(\nu_{x^*;\tau},\nu_{y^*;\tau})\right\}.\index{$\ensuremath{d^*}$}
\end{align}
If $T=+\infty$, the infimum is taken over $\tau\in(-\infty,t_-]$.
Equivalently, if $t=\t(x^*)\geq s=\t(y^*)$ and $T<+\infty$, then
\begin{align}\label{eq:defn-dstar-r}
d^*(x^*,y^*)=\min_{r\in[\sqrt{t-s},\sqrt{t+(1-\sigma)T}]}
\max\left\{r,\ d_{W_1}^{t-r^2}(\nu_{x^*;t-r^2},\nu_{y^*;t-r^2})\right\}.
\end{align}
\end{defn}

\begin{lem}\label{continuitydW1}
	Assume $x^*,y^*\in\XX_{\III^{+}}$ for $\XX \in \MM(n,T)$. Then for $t \in [-(1-\sigma)T,\min\{\t(x^*),\t(y^*)\}]$, $t\mapsto d_{W_1}^{t}(\nu_{x^*;t},\nu_{y^*;t})$ is continuous.
\end{lem}

\begin{proof}
Fix $t_0 \in [-(1-\sigma)T,\min\{\t(x^*),\t(y^*)\}]$. We first show
	\begin{align}\label{continuity1}
\lim_{t \searrow t_0}d_{W_1}^{t}(\nu_{x^*;t},\nu_{y^*;t})=d_{W_1}^{t_0}(\nu_{x^*;t_0},\nu_{y^*;t_0}).
	\end{align}
	
If $t_0=\min\{\t(x^*),\t(y^*)\}$, then \eqref{continuity1} is immediate. Hence assume $t_0<\min\{\t(x^*),\t(y^*)\}$. Since $t\mapsto d_{W_1}^{t}(\nu_{x^*;t},\nu_{y^*;t})$ is increasing by Proposition \ref{monotonicityW1}, if \eqref{continuity1} fails, we can find, by Lemma \ref{lem:duality}, $t_i \searrow t_0$ and $f_i \in C^1(M)$ with $|\na_{g(t_i)} f_i| \le 1$ so that
	\begin{align}\label{continuity2}
\int_M f_i \, \mathrm{d}\nu_{x^*;t_i}-\int_M f_i \, \mathrm{d}\nu_{y^*;t_i} \ge d_{W_1}^{t_0}(\nu_{x^*;t_0},\nu_{y^*;t_0})+\delta_0
	\end{align}
for a constant $\delta_0>0$. Without loss of generality, we may assume $f_i(p)=0$ for a fixed point $p \in M$. After passing to a subsequence, $f_i$ converges to a continuous function $f$ on $M$ with $\Lip_{g(t_0)} f \le 1$. Using the continuities of the conjugate heat kernel measures and the corresponding Riemannian metrics, we conclude from \eqref{continuity2} that
	\begin{align*}
\int_M f \, \mathrm{d}\nu_{x^*;t_0}-\int_M f \, \mathrm{d}\nu_{y^*;t_0} \ge d_{W_1}^{t_0}(\nu_{x^*;t_0},\nu_{y^*;t_0})+\delta_0.
	\end{align*}
However, this contradicts Lemma \ref{lem:duality}, and hence \eqref{continuity1} holds.

Next, we show 
	\begin{align}\label{continuity3}
\lim_{t \nearrow t_0}d_{W_1}^{t}(\nu_{x^*;t},\nu_{y^*;t})=d_{W_1}^{t_0}(\nu_{x^*;t_0},\nu_{y^*;t_0}).
	\end{align}
For any $\ep>0$, it follows from Lemma \ref{lem:duality} that there exists a continuous function $f$ on $M$ with $\Lip_{g(t_0)} f \le 1$ such that
	\begin{align*}
\int_M f \, \mathrm{d}\nu_{x^*;t_0}-\int_M f \, \mathrm{d}\nu_{y^*;t_0} \ge d_{W_1}^{t_0}(\nu_{x^*;t_0},\nu_{y^*;t_0})-\ep.
	\end{align*}
Note that $\Lip_{g(t)} f \le 1+\eta(t)$, where $\eta(t) \to 0$ if $t \to t_0$. By the continuities of the conjugate heat kernel measures and the corresponding Riemannian metrics, we conclude
	\begin{align*}
\lim_{t \nearrow t_0}d_{W_1}^{t}(\nu_{x^*;t},\nu_{y^*;t}) \ge & \lim_{t \nearrow t_0} \frac{1}{1+\eta(t)} \lc\int_M f \, \mathrm{d}\nu_{x^*;t}-\int_M f \, \mathrm{d}\nu_{y^*;t} \rc \ge d_{W_1}^{t_0}(\nu_{x^*;t_0},\nu_{y^*;t_0})-\ep.
	\end{align*}
Since $t\mapsto d_{W_1}^{t}(\nu_{x^*;t},\nu_{y^*;t})$ is increasing and $\ep$ is arbitrary, the proof of \eqref{continuity3} is complete.
\end{proof}

By Definition \ref{defn:dstar-distance} and Lemma \ref{continuitydW1}, if $x^*,y^*\in\XX_{\III^+}$, $t=\max\{\t(x^*),\t(y^*)\}$ and $r=d^*(x^*,y^*)$, then
\begin{align}\label{eq:dstar-equality1}
d_{W_1}^{\max\{t-r^2,-(1-\sigma)T\}}
\left(\nu_{x^*;\max\{t-r^2,-(1-\sigma)T\}},
\nu_{y^*;\max\{t-r^2,-(1-\sigma)T\}}\right)\leq r.
\end{align}
If $r>\sqrt{|\t(x^*)-\t(y^*)|}$ and $t-r^2>-(1-\sigma)T$, then
\begin{align}\label{eq:dstar-equality2}
d_{W_1}^{t-r^2}(\nu_{x^*;t-r^2},\nu_{y^*;t-r^2})=r.
\end{align}
For every $R>0$, one also has the open-ball criterion
\begin{align}\label{eq:dstar-ball-criterion}
d^*(x^*,y^*)<R
\quad\Longleftrightarrow\quad
\left\{
\begin{array}{l}
|\t(x^*)-\t(y^*)|<R^2,\\[1mm]
d_{W_1}^{\max\{t-R^2,-(1-\sigma)T\}}
(\nu_{x^*;\max\{t-R^2,-(1-\sigma)T\}},
\nu_{y^*;\max\{t-R^2,-(1-\sigma)T\}})<R.
\end{array}\right.
\end{align}
Indeed, the forward implication follows by backward monotonicity from an almost minimizing comparison time, while the converse follows from Lemma \ref{continuitydW1} by moving the comparison time slightly forward when necessary.

\begin{lem}\label{lem:000a}
For any $\XX\in\MM(n,T)$, $d^*$ defines a distance function on $\XX_{\III^+}$.
\end{lem}
\begin{proof}
Suppose $x^*=(x,t)$ and $y^*=(y,s)$ satisfy $d^*(x^*,y^*)=0$. Definition \ref{defn:dstar-distance} gives $t=s$. Choose comparison times $\tau_i\nearrow t$ such that
\begin{align*}
d_{W_1}^{\tau_i}(\nu_{x^*;\tau_i},\nu_{y^*;\tau_i})\longrightarrow0.
\end{align*}
Since $\nu_{x^*;\tau_i}\to\delta_x$ and $\nu_{y^*;\tau_i}\to\delta_y$, we obtain $d_t(x,y)=0$, and hence $x^*=y^*$.

It remains to verify the triangle inequality. Take $x_i^*=(x_i,t_i)\in\XX_{\III^+}$ for $i=1,2,3$. Without loss of generality, assume $t_1\geq t_2\geq t_3$ and set
\begin{align*}
r=d^*(x_1^*,x_2^*),\qquad s=d^*(x_2^*,x_3^*),\qquad l=d^*(x_1^*,x_3^*).
\end{align*}
We first prove $l\leq r+s$. Set
\begin{align*}
\tau:=\max\{t_1-(r+s)^2,-(1-\sigma)T\}.
\end{align*}
By \eqref{eq:dstar-equality1}, Proposition \ref{monotonicityW1}, and the inequalities
\begin{align*}
\tau\leq\max\{t_1-r^2,-(1-\sigma)T\},\qquad
\tau\leq\max\{t_2-s^2,-(1-\sigma)T\},
\end{align*}
we obtain
\begin{align*}
d_{W_1}^{\tau}(\nu_{x_1^*;\tau},\nu_{x_3^*;\tau})
&\leq d_{W_1}^{\tau}(\nu_{x_1^*;\tau},\nu_{x_2^*;\tau})
+d_{W_1}^{\tau}(\nu_{x_2^*;\tau},\nu_{x_3^*;\tau})\\
&\leq r+s.
\end{align*}
Moreover, $r\geq\sqrt{t_1-t_2}$ and $s\geq\sqrt{t_2-t_3}$, so $t_1-(r+s)^2\leq t_3$; thus $\tau$ is an admissible comparison time for $x_1^*,x_3^*$. Since $t_1-\tau\leq(r+s)^2$, Definition \ref{defn:dstar-distance} gives $l\leq r+s$. The inequalities $r\leq l+s$ and $s\leq r+l$ follow from the same argument, with the roles of the three points interchanged.
\end{proof}

\begin{defn}[$d^*$-balls]
For any $\XX\in\MM(n,T)$, $x^*\in\XX_{\III^+}$ and $r>0$, define
\begin{align*}
B^*(x^*,r):=\{y^*\in\XX_{\III^+}\mid d^*(x^*,y^*)<r\}.\index{$\ensuremath{B^*}$}
\end{align*}
In particular, for any $y^*\in B^*(x^*,r)$, equation \eqref{eq:dstar-ball-criterion} gives
\begin{align}\label{pointinmetricball}
d_{W_1}^{\max\{\t(x^*)-r^2,\t(y^*)-r^2,-(1-\sigma)T\}}
\left(\nu_{x^*;\max\{\t(x^*)-r^2,\t(y^*)-r^2,-(1-\sigma)T\}},
\nu_{y^*;\max\{\t(x^*)-r^2,\t(y^*)-r^2,-(1-\sigma)T\}}\right)<r.
\end{align}
\end{defn}

\begin{exmp}\label{ex:euclidean}
Let
\[
\XX^\E=\{\R^n,(g(t)=g_E)_{t\in\R}\}
\]
be the standard static Ricci flow on Euclidean space. We compute the
$d^*$-distance between arbitrary spacetime points
\[
x^*=(x,t),\qquad y^*=(y,s).
\]
Assume without loss of generality that $t\geq s$, and set
\[
\Delta:=t-s,\qquad \ell:=|x-y|.
\]
Since $T=+\infty$, we have
\[
d_E^*(x^*,y^*)
=
\inf_{\tau\leq s}
\max\left\{
\sqrt{t-\tau},
d_{W_1}^{\tau}(\nu_{x^*;\tau},\nu_{y^*;\tau})
\right\}.
\]

We first compute the Wasserstein term. For $\tau\leq s$, set
\[
a:=t-\tau,\qquad b:=s-\tau,
\]
so that $a\geq b\geq0$. Consider the affine map
\[
F(w):=y+\sqrt{\frac{b}{a}}\,(w-x).
\]
Then
\[
F_\#\nu_{x^*;\tau}=\nu_{y^*;\tau}.
\]
If $a>b$, the map $F$ is a homothety with fixed point
\[
c=\frac{y-\sqrt{b/a}\,x}{1-\sqrt{b/a}}.
\]
The function $\varphi(w)=|w-c|$ is $1$-Lipschitz and satisfies
\[
\varphi(w)-\varphi(F(w))=|w-F(w)|.
\]
It follows from the Kantorovich--Rubinstein duality that $F$ is optimal. Consequently,
\[
d_{W_1}^{\tau}(\nu_{x^*;\tau},\nu_{y^*;\tau})
=
\frac{1}{(4\pi)^{n/2}}
\int_{\R^n}
\left|
x-y+
\left(\sqrt{t-\tau}-\sqrt{s-\tau}\right)z
\right|
e^{-|z|^2/4}\,\mathrm{d}z.
\]
If $t=s$, the two measures differ only by translation, and the same
formula gives
\[
d_{W_1}^{\tau}(\nu_{x^*;\tau},\nu_{y^*;\tau})=|x-y|.
\]

For $\rho,q\geq0$, define
\[
G_n(\rho,q)
:=
\frac{1}{(4\pi)^{n/2}}
\int_{\R^n}
|\rho e_1+qz|e^{-|z|^2/4}\,\mathrm{d}z,
\]
where $e_1=(1,0,\ldots,0)\in\R^n$. By the rotational invariance of the
Gaussian density,
\[
d_{W_1}^{\tau}(\nu_{x^*;\tau},\nu_{y^*;\tau})
=
G_n\left(
\ell,\sqrt{t-\tau}-\sqrt{s-\tau}
\right).
\]

Setting
\[
u:=s-\tau\in[0,+\infty),
\]
we obtain
\[
d_E^*(x^*,y^*)
=
\inf_{u\geq0}
\max\left\{
\sqrt{\Delta+u},
G_n\left(\ell,\sqrt{\Delta+u}-\sqrt u\right)
\right\}.
\]

Suppose $\Delta>0$. The function
\[
u\longmapsto\sqrt{\Delta+u}
\]
is strictly increasing, whereas
\[
u\longmapsto
G_n\left(\ell,\sqrt{\Delta+u}-\sqrt u\right)
\]
is nonincreasing. Moreover,
\[
G_n(\ell,\sqrt{\Delta})
\geq
\frac{1}{\sqrt{4\pi}}
\int_{\R}
|\ell+\sqrt{\Delta}\,z|e^{-z^2/4}\,\mathrm{d}z
\geq
\frac{2}{\sqrt{\pi}}\sqrt{\Delta}
>
\sqrt{\Delta},
\]
whereas
\[
G_n\left(\ell,\sqrt{\Delta+u}-\sqrt u\right)
\longrightarrow\ell
\]
as $u\to+\infty$. It follows that there exists a unique $u_*\geq0$
such that
\[
\sqrt{\Delta+u_*}
=
G_n\left(\ell,\sqrt{\Delta+u_*}-\sqrt{u_*}\right),
\]
and
\[
d_E^*(x^*,y^*)=\sqrt{\Delta+u_*}.
\]

Equivalently, $r=d_E^*(x^*,y^*)$ is the unique number
\[
r\geq\max\{\ell,\sqrt{\Delta}\}
\]
satisfying
\[
r
=
G_n\left(
\ell,r-\sqrt{r^2-\Delta}
\right),
\]
or, explicitly,
\[
r
=
\frac{1}{(4\pi)^{n/2}}
\int_{\R^n}
\left|
\ell e_1+
\left(r-\sqrt{r^2-\Delta}\right)z
\right|
e^{-|z|^2/4}\,\mathrm{d}z.
\]

In particular, if $t=s$, then
\[
d_E^*((x,t),(y,t))=|x-y|.
\]
If $x=y$, define
\[
c_n
:=
\frac{1}{(4\pi)^{n/2}}
\int_{\R^n}|z|e^{-|z|^2/4}\,\mathrm{d}z
=
2\frac{\Gamma\left(\frac{n+1}{2}\right)}
{\Gamma\left(\frac n2\right)}.
\]
Then
\[
d_E^*((x,t),(x,s))
=
\frac{c_n}{\sqrt{2c_n-1}}\sqrt{|t-s|}.
\]
\end{exmp}

Next, we show that $d^*$ is locally comparable to the standard parabolic distance.

\begin{prop}\label{equivalenceofballs}
Given $\XX=\{M^n,(g(t))_{t\in \III^{++}}\} \in \MM(n, Y, T)$, suppose $x^* \in \XX_{\III^+}$, $r \in (0, \sqrt{\sigma T}]$ and $|\scal|\leq R_0 r^{-2}$ on $P(x^*,r)$. Then there exists a constant $\rho=\rho(n,Y, R_0) \in (0, 1)$ such that
	\begin{equation}\label{eq:dstar-pstar-equivalence}
		P(x^*, \rho r)\subset B^*(x^*, r) \quad \text{and} \quad B^*(x^*, \rho r)\subset P(x^*,r).
	\end{equation}
	Here, $P(x^*,s):= B_{t_0}(x,s)\times [t_0-s^2,t_0+s^2]\bigcap \III^{+}$\index{$\ensuremath{P(x^*,s)}$} and $x^*=(x,t_0)$.
\end{prop}
\begin{proof}
Without loss of generality, we assume $r=1$.

First, we show that for any $l \in (0, 1)$, if $y^*=(y,s)\in P(x^*,l)$, then $d^*(x^*,y^*)\leq C(n, Y, R_0)l$. This will give the first inclusion in \eqref{eq:dstar-pstar-equivalence}.
	
	We set $t_1=\max\{-(1-\sigma)T, t_0-l^2\}$ and choose an $H_n$-center $(z, t_1)$ of $x^*$. By Theorem \ref{heatkernelupperbdgeneral} (i) and Proposition \ref{HncenterScal} (i), we know that
		\begin{align}\label{standardestimate1}
		\int_M d_{t_1}(x,\cdot)\, \mathrm{d}\nu_{x^*;t_1} \leq & C(n, Y, R_0) \sqrt{t_0-t_1}+\int_M d_{t_1}(z,\cdot)\, \mathrm{d}\nu_{x^*;t_1} \notag \\
		=&	C(n, Y, R_0)\sqrt{t_0-t_1}+\sum_{k=0}^\infty\int_{\{k\sqrt{t_0-t_1}\leq d_{t_1}(z,\cdot)\leq (k+1)\sqrt{t_0-t_1}\}}d_{t_1}(z,\cdot) \, \mathrm{d}\nu_{x^*;t_1}\nonumber\\
		\leq& C(n, Y, R_0)\sqrt{t_0-t_1}+\sqrt{t_0-t_1}\sum_{k=0}^\infty (k+1)\nu_{x^*;t_1}\lc\{d_{t_1}(z,\cdot) \ge k\sqrt{t_0-t_1}\}\rc\nonumber\\
		\leq&  C(n, Y, R_0)\sqrt{t_0-t_1}+C(n)\sqrt{t_0-t_1}\sum_{k=0}^\infty (k+1) e^{-\frac{k^2}{5}} \notag\\
		\leq & C(n, Y, R_0) \sqrt{t_0-t_1} \le C(n, Y, R_0) l.
	\end{align}
	Similarly, we have
\begin{align}\label{standardestimate1x}
		\int_M d_{t_1}(y,\cdot)\, \mathrm{d}\nu_{y^*;t_1} \leq C(n, Y, R_0) l.
	\end{align}
Now, by Definition \ref{defnvariance}, we estimate
	\begin{align*}
		d_{W_1}^{t_1}(\nu_{x^*;t_1},\nu_{y^*;t_1})&\leq \int_M\int_M d_{t_1}(z_1,z_2) \, \mathrm{d}\nu_{x^*;t_1}(z_1)\, \mathrm{d}\nu_{y^*;t_1}(z_2)\\
		&\leq \int_M\int_M\big( d_{t_1}(z_1,x)+d_{t_1}(z_2,y)+d_{t_1}(x,y)\big) \, \mathrm{d}\nu_{x^*;t_1}(z_1) \, \mathrm{d}\nu_{y^*;t_1}(z_2) \leq C(n, Y,  R_0) l,
	\end{align*}
	where in the last inequality, we have used \eqref{standardestimate1}, \eqref{standardestimate1x}, and the fact that $d_{t_1}(x, y) \le C(n, Y, R_0) l$ by Proposition \ref{HncenterScal} (ii).
	
	By Definition \ref{defn:dstar-distance}, this gives
	\begin{align*}
		d^*(x^*,y^*)\leq C(n, Y, R_0)l.
	\end{align*}
	
Next, we prove the second inclusion in \eqref{eq:dstar-pstar-equivalence}. Given $y^*=(y, s) \in B^*(x^*, \rho)$, where $\rho=\rho(n, Y, R_0, \sigma) \in (0, 1)$ will be chosen later.

We set $t_2:=\max\{-(1-\sigma)T, t_0-\rho^2\}$. Then, by our assumption and \eqref{eq:dstar-ball-criterion}, we have
		\begin{align} \label{eq:torsionext1}
d_{W_1}^{t_2}(\nu_{x^*;t_2},\nu_{y^*;t_2}) <\rho.
	\end{align}

Let $(z,t_2)$ be an $H_n$-center of $y^*$. It follows from Proposition \ref{HncenterScal}(i) and \eqref{eq:torsionext1} that
		\begin{align*}
d_{t_2}(x, z) \le C(n, Y, R_0) \rho.
	\end{align*}

Thus, it follows from the same argument as in \cite[Proposition 5.13]{li2024heat} that if $\rho \le \rho(n, Y, R_0)$, then $d_{t_0}(x, y)<1$, which finishes the proof.
\end{proof}

An immediate consequence of Proposition \ref{equivalenceofballs} is the following:

\begin{cor}\label{cor:topo}
Given $\XX\in\MM(n,T)$, the topology on $\XX_{\III^+}$ induced by the $d^*$-distance agrees with the standard topology.
\end{cor}
\begin{proof}
The assertion is local. On every compact parabolic neighborhood of a fixed smooth flow, the curvature and the Nash entropy are bounded. Proposition \ref{equivalenceofballs}, applied after shrinking the neighborhood and the scale, therefore gives the equivalence of the two topologies.
\end{proof}

\begin{prop}\label{distancefunction1}
For $\XX=\{M^n,(g(t))_{t\in \III^{++}}\} \in \MM(n,T)$, the following properties hold:
	\begin{enumerate}[label=\textnormal{(\arabic{*})}]
		\item For any $x,y\in M$ and $t\in \III^{+}$, $d^*\big((x,t),(y,t)\big)\leq d_t(x,y)$;
		\item For any $x^* \in \XX_{\III^{+}}$, $\mathfrak t(B^*(x^*,r))\subset (t-r^2,t+r^2)\bigcap \III^{+}$. Moreover, the time-function $\t$ is a $2$-H\"older function, i.e., for any $x^*,y^* \in \XX_{\III^{+}}$,
		\begin{align*}
			|\t(x^*)-\t(y^*)|\leq d^*(x^*,y^*)^2.
		\end{align*}
	\end{enumerate}
\end{prop}

\begin{proof}
For (1), take $\tau=t$ in Definition \ref{defn:dstar-distance}. Part (2) follows directly from Definition \ref{defn:dstar-distance}, since every comparison time satisfies
\begin{align*}
\sqrt{t_+-\tau}\geq\sqrt{|\t(x^*)-\t(y^*)|}.
\end{align*}
\end{proof}

Next, we prove

\begin{lem}\label{lem:complete0}
Given $\XX\in\MM(n,T)$, the $d^*$-distance on $\XX_{\III^{+}}$ is complete.
\end{lem}

\begin{proof}
We set $\XX=\{M, (g(t))_{t \in \III^{++}}\}$. Given a Cauchy sequence $x_i^*=(x_i, t_i) \in M \times \III^{+}$ with respect to $d^*$, it follows from Proposition \ref{distancefunction1} (2) that $\{t_i\}$ is a Cauchy sequence in $\R$. Without loss of generality, we assume $t_i \to t_{\infty} \in \III^{+}$. 

Moreover, since $M$ is closed, we can take a subsequence (if necessary) such that $x_i \to x_{\infty}$ with respect to $g(0)$. Then, by Corollary \ref{cor:topo}, we conclude that $x_i^*$ converges to $(x_{\infty}, t_{\infty})$ with respect to $d^*$.
\end{proof}

Next, we recall the following definition of parabolic neighborhoods from \cite[Definition 9.2]{bamler2020entropy}, slightly adapted to our setting.

\begin{defn}[$P^*$-neighborhoods] \label{def:pstar}
For any $\XX\in\MM(n,T)$, $x^*=(x, t)\in\XX_{\III^{+}}$, $A,T^+,T^-\geq 0$, $P^*(x,t;A,-T^-,T^+)\subset \XX_{\III^{+}}$ is defined as the set of points $y^*=(y,s)\in \XX_{\III^{+}}$ with $s\in [t-T^-,t+T^+] \cap \III^+$ and 
\begin{align*}
	d_{W_1}^{\max\{t-T^-,-(1-\sigma)T\}}\big(\nu_{x^*;\max\{t-T^-,-(1-\sigma)T\}},\nu_{y^*;\max\{t-T^-,-(1-\sigma)T\}}\big)<A.
\end{align*}
Moreover, we set $P^*(x^*;r)=P^*(x,t;r,-r^2,r^2)$\index{$\ensuremath{P^*(x^*;r)}$}.
\end{defn}

The following proposition shows that $P^*$-neighborhoods are essentially equivalent to $d^*$-balls:
\begin{prop}\label{propdistance2}
For $x^*=(x,t)\in\XX_{\III^+}$ and $r>0$,
\begin{align*}
P^*(x,t;r,-r^2/2,r^2/2)\subset B^*(x^*,r)\subset P^*(x,t;r,-r^2,r^2).
\end{align*}
In particular,
\begin{align*}
P^*(x^*;r/\sqrt2)\subset B^*(x^*,r)\subset P^*(x^*;r).
\end{align*}
\end{prop}
\begin{proof}
Given $y^*=(y,s)\in B^*(x^*,r)$, equation \eqref{eq:dstar-ball-criterion} and backward monotonicity imply
\begin{align*}
d_{W_1}^{\max\{t-r^2,-(1-\sigma)T\}}
\left(\nu_{x^*;\max\{t-r^2,-(1-\sigma)T\}},
\nu_{y^*;\max\{t-r^2,-(1-\sigma)T\}}\right)<r.
\end{align*}
Together with $|s-t|<r^2$, this gives $y^*\in P^*(x,t;r,-r^2,r^2)$.

Conversely, suppose $y^*=(y,s)\in P^*(x,t;r,-r^2/2,r^2/2)$. If $s\leq t$, the defining Wasserstein estimate and $t-s\leq r^2/2$ give $d^*(x^*,y^*)<r$ by Definition \ref{defn:dstar-distance}. If $s\geq t$, then $s-r^2\leq t-r^2/2$, so backward monotonicity gives the same conclusion. The final inclusion follows by replacing $r$ with $r/\sqrt2$.
\end{proof}

\begin{lem}\label{lem:Hcenterdis}
For $x^*\in\XX_{\III^+}$ and $s\in[-(1-\sigma)T,\t(x^*)]$, let $z^*=(z,s)$ be an $H$-center of $x^*$. Then
\begin{align}\label{eq:Hcenter}
d^*(x^*,z^*)\leq\max\{1,\sqrt H\}\sqrt{\t(x^*)-s}.
\end{align}
In particular, for an $H_n$-center,
\begin{align*}
d^*(x^*,z^*)\leq\sqrt{H_n}\sqrt{\t(x^*)-s}.
\end{align*}
\end{lem}
\begin{proof}
By the definition of an $H$-center,
\begin{align*}
d_{W_1}^{s}(\nu_{x^*;s},\delta_z)\leq\sqrt{\Var_s(\nu_{x^*;s},\delta_z)}\leq\sqrt{H(\t(x^*)-s)}.
\end{align*}
Taking $\tau=s$ in Definition \ref{defn:dstar-distance} proves \eqref{eq:Hcenter}.
\end{proof}

\begin{prop}\label{propdistance3}
For $x^*\in\XX_{\III^+}$ and $r>0$ with $\t(x^*)-r^2\geq-(1-\sigma)T$, the following conclusions hold.
\begin{enumerate}[label=\textnormal{(\roman{*})}]
\item For any $t\in\R$,
\begin{align*}
\abs{B^*(x^*,r)\cap M\times\{t\}}_t
\leq C(n,\sigma)e^{\NN_{x^*}(r^2)}r^n,
\end{align*}
where $|\cdot|_t$ denotes the volume with respect to $\mathrm{d}V_{g(t)}$.

\item We have
\begin{align}\label{eq:entropyweightedlocalvolume}
0<c(n,\sigma)e^{\NN_{x^*}(r^2)}r^{n+2}
\leq\abs{B^*(x^*,r)}
\leq C(n,\sigma)e^{\NN_{x^*}(r^2)}r^{n+2},
\end{align}
	where $|\cdot|$ denotes the spacetime volume with respect to $\mathrm{d}V_{g(t)} \mathrm{d}t$.
\end{enumerate}
\end{prop}
\begin{proof}
The assertion in (i) and the upper bound in (ii) follow from \cite[Theorem 9.8]{bamler2020entropy} and Proposition \ref{propdistance2}. 
	
	For the lower bound in (ii), we set $t=\t(x^*)$ and take any $s$ with $r/2 \le 3\sqrt{H_n}s \le r$. Moreover, let $z^*=(z,t-s^2)$ be an $H_n$-center of $x^*$. By Proposition \ref{prop:volume} (ii), we have
	\begin{align*}
		\big|B_{t-s^2}(z,\sqrt{2H_n} s)\big|_{t-s^2}\geq C(R_{\min} s^2)\exp \lc\mathcal{N}_{x^*}(s^2) \rc s^n,
	\end{align*}
	where $\scal(\cdot, t-s^2) \ge R_{\min}$. Since $\XX$ is defined on $\III^{++}=[-T, 0]$ and $s^2 \le (1-\sigma)T$, it follows from \eqref{eq:lowerscal} and Proposition \ref{propNashentropy} that
		\begin{align}\label{lowerbdvolume1}
		\big|B_{t-s^2}(z,\sqrt{2H_n}s)\big|_{t-s^2}\geq c(n, \sigma) \exp \lc\mathcal{N}_{x^*}(r^2) \rc s^n>0.
	\end{align}

By Lemma \ref{lem:Hcenterdis}, we have
	\begin{align*}
		d^*(x^*, z^*)\leq \sqrt{H_n} s.
	\end{align*}
In addition, by Proposition \ref{distancefunction1} (1), we conclude that
	\begin{align*}
		B_{t-s^2}(z, \sqrt{2H_n}s)\subset B^*\big(z^*,2\sqrt{H_n}s\big) \subset B^*\big(x^*,3\sqrt{H_n}s\big).
	\end{align*}
	Combining with \eqref{lowerbdvolume1}, we get
	\begin{align*}
\big|B^*\big(x^*,r\big)\bigcap \XX_{t-s^2}\big|_{t-s^2} \ge&	\big|B^*\big(x^*,3\sqrt{H_n}s\big)\bigcap \XX_{t-s^2}\big|_{t-s^2} \\
\geq & \big|B_{t-s^2}(z, \sqrt{2H_n}s)\big|_{t-s^2}\geq c(n,\sigma) \exp \lc\mathcal{N}_{x^*}(r^2) \rc s^n.
	\end{align*}
Consequently, the conclusion follows by integrating with respect to $2s\,\mathrm{d}s$ from $r/(6\sqrt{H_n})$ to $r/(3\sqrt{H_n})$.
\end{proof}

\begin{prop}\label{uppervolumebd}
For any $x^*\in\XX_{\III^+}$, the following statements hold.
\begin{enumerate}[label=\textnormal{(\roman{*})}]
\item Suppose $T<+\infty$. Then for every $L>0$,
\begin{align}\label{eq:largeballupperfinite}
\abs{B^*(x^*,L\sqrt T)}
\leq C(n,\sigma,L)\exp \lc \NN^*_{-(1-\sigma/2)T}(x^*) \rc T^{\frac n2+1}.
\end{align}
\item Suppose $T=+\infty$. For every $L>0$, we have
\begin{align}\label{eq:largeballupperancient}
\abs{B^*(x^*,L)}
\leq C(n)\exp \lc\NN_{x^*}(2 L^2)\rc L^{n+2}.
\end{align}
\end{enumerate}
\end{prop}
\begin{proof}
We first prove (i), following the original large-ball estimate while retaining the entropy factor. Set
\begin{align*}
a:=-(1-\sigma)T,\qquad s_0:=-(1-\sigma/2)T.
\end{align*}
For any $y^*\in B^*(x^*,L\sqrt T)$, Definition \ref{defn:dstar-distance} and backward monotonicity give
\begin{align}\label{eq:bottomWcontrol}
d_{W_1}^{a}(\nu_{x^*;a},\nu_{y^*;a})\leq L\sqrt T.
\end{align}
Let $z_0^*:=(z_0,a)$ and $z^*:=(z,a)$ be $H_n$-centers of $x^*$ and $y^*$, respectively. Then
\begin{align*}
d_a(z_0,z)\leq(L+2\sqrt{H_n})\sqrt T.
\end{align*}
Choose
\begin{align*}
B:=B_a(z_0,C(n,L)\sqrt T)
\end{align*}
so that $\nu_{y^*;a}(B)\geq1/2$ for every $y^*\in B^*(x^*,L\sqrt T)$. Let $u$ solve the heat equation with $u(\cdot,a)=\chi_B$. Then $u(y^*)\geq1/2$. Thus, for any $t \ge -(1-\sigma)T$,
	\begin{align}\label{eq:heatmassupper}
		\frac{1}{2} \big|B^*(x^*,L \sqrt{T}) \cap \XX_t \big|_t \le \int_M u(\cdot, t) \, \mathrm{d}V_{g(t)} \leq C(n,\sigma) |B|_{a}.
	\end{align}
Here, the second inequality holds since $\scal\geq-n/(2\sigma T)$ on $\XX_{\III^{+}}$ and hence for $t \ge -(1-\sigma)T$,
	\begin{align*}
		\diff{}{t}\int_M u  \, \mathrm{d}V_{g(t)}=-\int_M u \scal  \, \mathrm{d}V_{g(t)}\leq \frac{n}{2\sigma T}\int_M u  \, \mathrm{d}V_{g(t)}. 
	\end{align*}

On the other hand, Proposition \ref{prop:volume} (i) gives
\begin{align*}
\abs B_a\leq C(n,\sigma,L)e^{\NN^*_{s_0}(z_0^*)}T^{n/2}.
\end{align*}
Since $\scal(\cdot,s_0)\geq -n/(\sigma T)$, Proposition \ref{propNashentropy1}, applied with lower time $s_0$ and comparison time $a$, we have
\begin{align*}
\NN^*_{s_0}(z_0^*)-\NN^*_{s_0}(x^*)
&\leq\left(\frac{n}{2(a-s_0)}+\frac{n}{\sigma T}\right)^{1/2}
 d_{W_1}^{a}(\delta_{z_0},\nu_{x^*;a})
 +\frac n2\log\frac{\t(x^*)-s_0}{a-s_0}\\
&\leq C(n,\sigma).
\end{align*}
Integrating \eqref{eq:heatmassupper} over $t\in\III^+$ proves \eqref{eq:largeballupperfinite}.

For (ii), set
\begin{align*}
b:=\t(x^*)-L^2,\qquad s_0:=\t(x^*)-2L^2.
\end{align*}
If $y^*\in B^*(x^*,L)$, then $|\t(y^*)-\t(x^*)|\le L^2$ and $d_{W_1}^{b}(\nu_{x^*;b},\nu_{y^*;b})\leq L$.

Let $z_0^*:=(z_0,b)$ be an $H_n$-center of $x^*$. Repeating the preceding heat-equation argument gives a ball $B_b(z_0,C(n)L)$ satisfying $\nu_{y^*;b}(B)\geq1/2$. Moreover, Proposition \ref{prop:volume} (i), at scale $L$, gives
\begin{align*}
\abs B_b\leq C(n)e^{\NN^*_{s_0}(z_0^*)}L^n.
\end{align*}
Proposition \ref{propNashentropy1} and the $H_n$-center property imply
\begin{align*}
\NN^*_{s_0}(z_0^*)\leq\NN^*_{s_0}(x^*)+C(n)=\NN_{x^*}(2L^2)+C(n).
\end{align*}
The relevant time interval has length at most $CL^2$, and \eqref{eq:largeballupperancient} follows.
\end{proof}

\begin{prop}\label{nashlip}
Given $\XX=\{M^n,(g(t))_{t\in \III^{++}}\}$ and $s \in [-(1-\sigma)T, 0)$, $\NN_s^*(\cdot)$ is locally uniformly Lipschitz on $\XX_{(s,0]}$ with respect to $d^*$ in the sense that for any $x^*, y^* \in \XX_{(s,0]}$ with $r:=d^*(x^*, y^*)$ and $\max \{\t(x^*), \t(y^*)\}-r^2 \ge s/2$, we have 
	\begin{align}\label{eq:extra01001}
\abs{\NN_s^*(x^*)-\NN_s^*(y^*)} \le C r,
	\end{align}
	where $C=C(n, s, \sigma T)$. In particular, if $T=+\infty$, we can choose $C=2n/\sqrt{|s|}$.
\end{prop}

\begin{proof}
We set $t_1:=\max\{\t(x^*), \t(y^*)\}$. By our assumption, $t_1-r^2>s/2$, which implies by \eqref{eq:dstar-equality1} that
	\begin{align*}
d_{W_1}^{t_1-r^2} \lc \nu_{x^*;t_1-r^2},\nu_{y^*;t_1-r^2} \rc\leq r.
	\end{align*}
It follows from Proposition \ref{propNashentropy1} that
	\begin{align*}
\abs{\NN_s^*(x^*)-\NN_s^*(y^*)}\leq & \lc \frac{n}{2(t_1-r^2-s)}-R_{\min} \rc^{\frac 1 2} d_{W_1}^{t_1-r^2}( \nu_{x^*;t_1-r^2},\nu_{y^*;t_1-r^2}) +\frac{n}{2}\log\lc\frac{t_1-s}{t_1-r^2-s}\rc \\
\le & \lc \frac{n}{2(t_1-r^2-s)}-R_{\min} \rc^{\frac 1 2} r+\frac{n}{2} \frac{r^2}{t_1-r^2-s}.
	\end{align*}
By our assumption, we have $t_1-r^2-s \ge |s|/2$ and $r^2 \le |s|/2$. Thus, \eqref{eq:extra01001} holds. The last conclusion follows since $R_{\min}=0$.
\end{proof}

Similarly, we have

\begin{prop}\label{nashlip1}
Let $\XX=\{M^n,(g(t))_{t\in \III^{++}}\}$ and $x^*,y^*\in\XX_{\III^+}$, and set $t_1=\t(x^*)\geq\t(y^*)=t_2$ and $r=d^*(x^*,y^*)$. If $t_2-\tau>-(1-\sigma)T$ and $t_1-r^2\geq(t_2-\tau)/2$, then
	\begin{align*}
\abs{\NN_{x^*}(\tau)-\NN_{y^*}(\tau)} \le C r
	\end{align*}
where $C=C(n, \tau, \sigma T)$. In particular, if $T=+\infty$, we can choose $C=C(n)/\sqrt{\tau}$.
\end{prop}

\begin{proof}
We set $s=t_2-\tau$. By Proposition \ref{nashlip}, we have
	\begin{align} \label{eq:nash1}
\abs{\NN_s^*(x^*)-\NN_s^*(y^*)} \le C(n, s, \sigma T) r=C(n, \tau, \sigma T) r.
	\end{align}

On the other hand, if we set $\tau_1=\tau+t_1-t_2$, then it follows from Proposition \ref{propNashentropy} (i) (see also \cite[(5.7)]{bamler2020entropy}) that
	\begin{align*}
\NN_{x^*}(\tau) \ge \NN_s^*(x^*) \ge \NN_{x^*}(\tau)-\frac{n}{2} \log  \lc \frac{\tau_1}{\tau}\lc 1-\frac{2}{n}R_{\min}(\tau_1-\tau) \rc \rc.
	\end{align*}
Since $\tau \le \tau_1 \le \tau+r^2$, we conclude that
	\begin{align*}
\abs{\NN_{x^*}(\tau)-\NN_s^*(x^*)}  \le C(n, \tau, \sigma T) r.
	\end{align*}
Combined with \eqref{eq:nash1}, we have
	\begin{align*}
\abs{\NN_{x^*}(\tau)-\NN_{y^*}(\tau)} \le C r.
	\end{align*}
The case $T=+\infty$ is similar. This completes the proof.
\end{proof}

Next, we prove the Lipschitz property of the curvature radius (see Definition \ref{defncurvatureradius}) with respect to $d^*$.

\begin{prop}\label{pro:LiprRm}
	There exists a constant $C=C(n,Y, \sigma T)>0$ such that for $\XX=\{M^n,(g(t))_{t\in \III^{++}}\} \in \MM(n, Y, T)$, we have
	\begin{align*}
		|r_{\Rm}(x^*)-r_{\Rm}(y^*)|\leq C d^*(x^*,y^*), \quad \forall x^*, y^*\in \XX_{\III^+}.
	\end{align*}
If $T=+\infty$, the last constant $C$ depends only on $n$ and $Y$.
\end{prop}
\begin{proof}
We set $x^*=(x,t), y^*=(y,s)$, and assume without loss of generality that $d^*(x^*,y^*)=1$. It suffices to show that
	\begin{align}\label{equ:lip}
		r_{\Rm}(x^*)\leq L+r_{\Rm}(y^*),
	\end{align}
	where $L=L(n,Y, \sigma T)\gg1$ will be chosen below.
	
If $r_{\Rm}(x^*)\leq L$, the estimate is trivial. Therefore, assume $r_{\Rm}(x^*)> L$. Then on the parabolic neighborhood
	\begin{align*}
P:=B_{t}(x,L)\times [t-L^2,t+L^2] \cap \III^{++},
	\end{align*}
	we have $|\Rm|\leq L^{-2}< 1$. By Proposition \ref{distancefunction1} (2), it follows that $|t-s|\leq 1$. Let
	\begin{align*}
t_1:=\max\{\max\{t,s\}-1,-(1-\sigma)T\}
	\end{align*}	
and let $z_1^*=(z_1,t_1)$ and $z_2^*=(z_2,t_1)$ be $H_n$-centers of $x^*,y^*$, respectively. By the definition of $d^*$, we obtain
	\begin{align*}
		d_{W_1}^{t_1}(\nu_{x^*;t_1},\nu_{y^*;t_1}) \le 1,
	\end{align*} 
which gives
	\begin{align*}
		d_{t_1}(z_1,z_2)\leq d_{W_1}^{t_1}(\nu_{x^*;t_1},\nu_{y^*;t_1})+d_{W_1}^{t_1}(\nu_{x^*;t_1},\delta_{z_1})+d_{W_1}^{t_1}(\nu_{y^*;t_1},\delta_{z_2})\leq C_1(n).
	\end{align*}
By Proposition \ref{HncenterScal} (i), we also have
	\begin{align*}
		d_{t_1}(x,z_1)\leq C_2(n, Y, \sigma T),
	\end{align*}
so that
	\begin{align*}
d_{t_1}(x,z_2)\leq C_1(n)+C_2(n, Y, \sigma T).
	\end{align*}

Choose $L\geq e^{100n}(C_1(n)+C_2(n, Y, \sigma T))$. We claim:
	\begin{align*}
d_{t}(x,z_2)\leq e^{(n-1)}(C_1(n)+C_2(n, Y, \sigma T)).
	\end{align*}
To see this, let $\gamma:[0,d_{t_1}(x,z_2)] \to M$ be a unit-speed minimizing geodesic between $x$ and $z_2$ with respect to $g(t_1)$, so that
	\begin{align*}
\mathrm{Length}_{t_1}(\gamma)\leq C_1(n)+C_2(n, Y, \sigma T).
	\end{align*}
Define
	\begin{align*}
\bar r:=\sup\{r\in [0,d_{t_1}(x,z_2)] \mid \gamma \vert_{[0,r]}\times [t_1,t]\subset P\}.
	\end{align*}
If $\bar r< d_{t_1}(x,z_2)$, then the standard distance distortion estimate yields
	\begin{align*}
		\mathrm{Length}_{t}(\gamma|_{[0,\bar r]})\leq e^{|t-t_1|(n-1)}\mathrm{Length}_{t_1}(\gamma|_{[0,\bar r]})\leq e^{(n-1)}(C_1(n)+C_2(n, Y, \sigma T))\leq L/10,
	\end{align*}
which contradicts the definition of $\bar r$. Therefore, $\bar r=d_{t_1}(x,z_2)$, and $\gamma\times [t_1,t]\subset P$, implying
	\begin{align*}
d_{t'}(x,z_2)\leq e^{(n-1)}(C_1(n)+C_2(n, Y, \sigma T)) \le e^{-99n}L, \quad \forall t'\in [t_1,\max\{t,s\}].
	\end{align*}

Similarly, we obtain
\begin{align} \label{equ:2lip6}
B_s(z_2, e^{-n}L) \times [t_1, s] \subset P.
\end{align}

Now choose 
	\begin{align*}
L\geq  \max\{100c(H_n)^{-1/2},e^{100n}(C_1(n)+C_2(n, Y, \sigma T))\},
	\end{align*}
where $c(H_n)$ is the constant from \cite[Proposition 9.16 (b)]{bamler2023compactness}. Then, by \eqref{equ:2lip6} and that proposition, we obtain
	\begin{align*}
		d_s(z_2,y)\leq C_3(n),
	\end{align*}
	and hence
	\begin{align*}
		d_s(x,y)\leq e^{(n-1)}(C_1(n)+C_2(n, Y, \sigma T))+C_3(n).
	\end{align*}	

Next, choose
	\begin{align*}
L\geq \max\{100c(H_n)^{-1/2},e^{100n^2}(C_1(n)+C_2(n, Y, \sigma T)+C_3(n))\}.
	\end{align*}
Then for any $t'\in [t_1,\max\{t,s\}]$, another distance distortion argument gives
	\begin{align*}
		d_{t'}(x,y)\leq e^{2(n-1)}(C_1(n)+C_2(n, Y, \sigma T)+C_3(n)).
	\end{align*}
Thus, we have
	\begin{align*}
B_s(y,r_{\Rm}(x^*)-L/2) \subset B_s(x, r_{\Rm}(x^*)-L/3).
	\end{align*}
On the other hand, for any $z^*=(z,s)$ with $z\in B_s(x,r_{\Rm}(x^*)-L/3)$, the distance distortion estimate gives
	\begin{align*}
		d_{t}(z,x)\leq & e^{(n-1)|t-s|r_{\Rm}^{-2}(x^*)}d_{s}(z,x)	\leq e^{(n-1)|t-s|r_{\Rm}^{-2}(x^*)}(r_{\Rm}(x^*)-L/3)\\
		\leq& \left(1+2(n-1)r_{\Rm}^{-2}(x^*)\right)(r_{\Rm}(x^*)-L/3)\leq r_{\Rm}(x^*)-L/4.
	\end{align*}
	This implies
	\begin{align*}
B_s(x,r_{\Rm}(x^*)-L/3)\subset B_t(x,r_{\Rm}(x^*)-L/4).
	\end{align*}	
Therefore, $B_s(y,r_{\Rm}(x^*)-L/2) \subset B_t(x,r_{\Rm}(x^*)-L/4)$, which shows that the curvature radius at $y^*$ satisfies
	\begin{align*}
r_{\Rm}(y^*) \ge r_{\Rm}(x^*)-L/2.
	\end{align*}	
	This proves \eqref{equ:lip} and completes the proof.
\end{proof}

\begin{defn}[Parabolic metric space]
	A \textbf{parabolic metric space}\index{parabolic metric space} $(Z,d_Z, \t)$ over an interval $I \subset \R$ is a metric space $(Z,d_Z)$ coupled with a \textbf{time-function} $\t: Z \to I$, which satisfies for any $x,y\in Z$,
	\begin{align}\label{timeHolder}
		|\t(x)-\t(y)|\leq d_Z^2(x,y).
	\end{align}
Additionally, for any set $J \subset I$, we define $Z_J:=\t^{-1}(J)$. A sequence of parabolic metric spaces is said to converge if the underlying metric spaces converge in the (pointed) Gromov--Hausdorff sense and the corresponding time functions also converge.
\end{defn}

In the following, we focus on $\XX_{\III}$ for $\XX\in\MM(n,T)$ and restrict the $d^*$-distance from $\XX_{\III^+}$ to $\XX_{\III}$.

\begin{thm}\label{thm:GHlimit-dstar}
Consider a sequence
\begin{align*}
\XX^i=\{M_i^n,(g_i(t))_{t\in\III^{++}}\}\in\MM(n,T)
\end{align*}
with base points $p_i^*\in\XX^i_{\III}$. When $T=+\infty$, we additionally assume $\limsup_{i\to\infty}\t_i(p_i^*)>-\infty$.
After passing to a subsequence, we obtain pointed Gromov--Hausdorff convergence
\begin{align*}
(M_i\times\III,d_i^*,p_i^*,\t_i)
\xrightarrow[i\to\infty]{\quad\mathrm{pGH}\quad}
(Z,d_Z,p_\infty,\t),
\end{align*}
where $d_i^*$ is the distance associated with $\XX^i_{\III^+}$, restricted to $\XX^i_{\III}$. The limit $(Z,d_Z,\t)$ is a complete, separable, locally compact parabolic metric space over $\III$ such that each bounded closed subset is compact. Moreover, there exists $t_e\in[\t(p_\infty),0]$ such that $\operatorname{image}(\t)=[-(1-2\sigma)T,t_e]$ or $[-(1-2\sigma)T,t_e)$; when $T=+\infty$, these intervals are understood as $(-\infty,t_e]$ or $(-\infty,t_e)$.
\end{thm}

\begin{proof}
First suppose $T<+\infty$. Set
\begin{align*}
a:=-(1-\sigma)T,\qquad s_0:=-(1-\sigma/2)T.
\end{align*}
Fix $L>0$ and $\ep\in(0,\sqrt{\sigma T})$. Let $\{x_j^*\}_{1\leq j\leq N}\subset\XX^i_{\III}$ be a maximal $\ep$-separated set in $B^*(p_i^*,L)\cap\XX^i_{\III}$. Then the balls $B^*(x_j^*,\ep/3)$ are pairwise disjoint and are contained in $B^*(p_i^*,L+\ep/3)$.

For each $j$, backward monotonicity and Definition \ref{defn:dstar-distance} give
\begin{align*}
d_{W_1}^{a}(\nu_{p_i^*;a},\nu_{x_j^*;a})\leq d_i^*(p_i^*,x_j^*)<L.
\end{align*}
Applying Proposition \ref{propNashentropy1} in both directions with lower time $s_0$ and comparison time $a$, and using the scalar lower bound at $s_0$, gives
\begin{align}\label{eq:entropyHarnackpacking}
|\NN^*_{s_0}(x_j^*)-\NN^*_{s_0}(p_i^*)|
\leq C(n,\sigma,T,L).
\end{align}
Indeed, the Wasserstein term is bounded by $L$, while all logarithmic time-ratio terms are bounded because $p_i^*,x_j^*\in\XX^i_{\III}$.
Since $x_j^*\in\XX^i_{\III}$ and $(\ep/3)^2<\sigma T/2$, monotonicity of the Nash entropy gives
\begin{align*}
\NN_{x_j^*}((\ep/3)^2)\geq \NN^*_{s_0}(x_j^*).
\end{align*}
Proposition \ref{propdistance3} therefore yields
\begin{align*}
\abs{B^*(x_j^*,\ep/3)}
\geq c(n,\sigma)e^{\NN^*_{s_0}(x_j^*)}\ep^{n+2}
\geq c(n,\sigma,T,L)e^{\NN^*_{s_0}(p_i^*)}\ep^{n+2}.
\end{align*}
On the other hand, Proposition \ref{uppervolumebd} gives
\begin{align*}
\abs{B^*(p_i^*,L+\ep/3)}
\leq C(n,\sigma,T,L)e^{\NN^*_{s_0}(p_i^*)}.
\end{align*}
Summing the disjoint lower bounds and canceling the common entropy factor, we obtain
\begin{align}\label{eq:packingfinite}
N\leq C(n,\sigma,T,L)\ep^{-n-2}.
\end{align}

Now suppose $T=+\infty$. Fix $L>1$ and $\ep\in(0,1)$, and let $\{x_j^*\}_{1\leq j\leq N}\subset\XX^i_{\III}$ be a maximal $\ep$-separated set in $B^*(p_i^*,L)\cap\XX^i_{\III}$. The balls $B^*(x_j^*,\ep/3)$ are pairwise disjoint and are contained in $B^*(p_i^*,L+\ep/3)$. Set
\begin{align*}
b_i:=\t_i(p_i^*)-L^2,\qquad s_i:=\t_i(p_i^*)-2L^2.
\end{align*}
For $x_j^*\in B^*(p_i^*,L)$, we have
\begin{align*}
d_{W_1}^{b_i}(\nu_{p_i^*;b_i},\nu_{x_j^*;b_i})\leq L
\end{align*}
and
\begin{align*}
|\NN^*_{s_i}(x_j^*)-\NN^*_{s_i}(p_i^*)|\leq C(n,L).
\end{align*}
Since $\NN_{x_j^*}((\ep/3)^2)\geq \NN^*_{s_i}(x_j^*)$, the lower estimate in Proposition \ref{propdistance3} and the upper estimate \eqref{eq:largeballupperancient} then give, exactly as above,
\begin{align}\label{eq:packingancient}
N\leq C(n,L)\ep^{-n-2}.
\end{align}

Thus, \eqref{eq:packingfinite} and \eqref{eq:packingancient} imply pointed Gromov--Hausdorff precompactness by \cite[Theorem 8.1.10]{burago2001course}. Since
\begin{align*}
\sqrt{|\t_i(x^*)-\t_i(y^*)|}\leq d_i^*(x^*,y^*),
\end{align*}
a subsequence of the time-functions converges to a function $\t$ satisfying \eqref{timeHolder}. Completeness, separability, local compactness, and the property that each bounded closed subset is compact follow from Corollary \ref{cor:topo}, Lemma \ref{lem:complete0}, and the uniform local covering estimates.

It remains to identify the image of the time-function. If $t\in\operatorname{image}(\t)$ and $s\in\III$ satisfies $s<t$, choose $z\in Z$ with $\t(z)=t$ and $z_i^*\to z$. Let $w_i^*\in M_i\times\{s\}$ be an $H_n$-center of $z_i^*$. Lemma \ref{lem:Hcenterdis} gives
\begin{align*}
d_i^*(z_i^*,w_i^*)\leq\sqrt{H_n}\sqrt{\t_i(z_i^*)-s},
\end{align*}
so the points $w_i^*$ remain in a uniformly bounded pointed ball. Passing to a subsequence, $w_i^*\to w\in Z$ with $\t(w)=s$. This proves the final assertion.
\end{proof}

\begin{defn}\label{defnRicciLimitSpace}
Any pointed Gromov--Hausdorff limit $(Z,d_Z,p_\infty,\t)$ from Theorem \ref{thm:GHlimit-dstar} is called a \textbf{Ricci flow limit space}. If it is obtained from a sequence in $\MM(n,Y,T)$ for a fixed $Y$, it is called a \textbf{noncollapsed Ricci flow limit space}.
\end{defn}
For a Ricci flow limit space $(Z,d_Z,\t)$, we always use $x,y,z$, etc., to denote spacetime points and $\t(x),\t(y),\t(z)$ to represent their time components. We denote metric balls in $Z$ by $B_Z^*(x,r)$\index{$B_Z^*(x,r)$}.

\begin{rem} \label{rem:general}
One can also consider a more general setting. Let $T_i>0$ with $T_i\to T\in(0,+\infty]$, and consider a sequence of Ricci flows
\[
\XX^i=\{M_i^n,(g_i(t))_{t\in\III_i^{++}}\}\in\MM(n,T_i)
\]
with base points $p_i^*\in\XX^i_{\III_i}$, where $\III_i^{++}=[-T_i,0]$ and $\III_i=[-(1-2\sigma)T_i,0]$.

If $T<+\infty$, after a time translation we may assume that, for all sufficiently large $i$, each $\XX^i$ is defined on
\[
\III^{\prime,++}=[-(1-\sigma)T-\sigma T_i, (1-\sigma)(T_i-T)].
\]
Define the distance $d_i^*$ on $M_i\times\III^{\prime,+}$ as in Definition~\ref{defn:dstar-distance}, where
\[
\III^{\prime,+}=[-(1-\sigma)T,(1-\sigma)(T_i-T)].
\]
Arguing as in Theorem~\ref{thm:GHlimit-dstar}, after passing to a subsequence, we obtain
\[
(M_i\times\III'_i,\ d_i^*,\ p_i^*,\ \t_i)\xrightarrow[i\to\infty]{\ \mathrm{pGH}\ }(Z,d_Z,p_\infty,\t),
\]
where $\III'_i=[-(1-\sigma)T+\sigma T_i,(1-\sigma)(T_i-T)]$ and $(Z, d_Z,\t)$ is a parabolic space over $[-(1-2\sigma)T, 0]$.

If $T=+\infty$, after passing to a subsequence, we assume 
	\begin{align*}
\lim_{i \to \infty}\t_i(p_i^*)=t_0 \in [-\infty, 0] \quad \text{and} \quad \lim_{i \to \infty} \lc \t_i(p_i^*)+(1-2\sigma)T_i \rc=a \in [0, +\infty].
	\end{align*}

We then consider the following cases.

\begin{itemize}
\item If $a<+\infty$, then we consider the shifted time functions $\t_i-\t_i(p_i^*)$, and obtain
	\begin{align*}
		(M_i \times \III_i, d^*_i, p_i^*,\t_i-\t_i(p_i^*)) \xrightarrow[i \to \infty]{\quad \mathrm{pGH} \quad} (Z, d_Z, p_{\infty},\t),
	\end{align*}
so that $\t(p_{\infty})=0$ and $(Z, d_Z,\t)$ is a parabolic space with $[-a, 0] \subset \mathrm{image}(\t)$. 

\item If $a=+\infty$ and $t_0>-\infty$, then a similar argument yields
	\begin{align*}
		(M_i \times \III_i, d^*_i, p_i^*,\t_i) \xrightarrow[i \to \infty]{\quad \mathrm{pGH} \quad} (Z, d_Z, p_{\infty},\t),
	\end{align*}
so that $(Z, d_Z,\t)$ is a parabolic space with $(-\infty,t_0]\subset \mathrm{image}(\t)$.

\item If $a=+\infty$ and $t_0=-\infty$, then a similar argument yields
	\begin{align}\label{specialcov}
		(M_i \times \III_i,  d^*_i, p_i^*,\t_i-\t_i(p_i^*)) \xrightarrow[i \to \infty]{\quad \mathrm{pGH} \quad} (Z, d_Z, p_{\infty},\t),
	\end{align}
so that $\t(p_{\infty})=0$ and $(Z, d_Z,\t)$ is a parabolic space with $\R_-:=(-\infty, 0] \subset \mathrm{image}(\t)$. \index{$\R_-$}
\end{itemize}
\end{rem}

Now we introduce the following notation.

\begin{notn}\label{not:1}
We write $(Z,d_Z,z,\t) \in \overline{\MM(n, Y, T)}$\index{$\overline{\MM(n, Y, T)}$} if it arises as the pointed Gromov--Hausdorff limit of a sequence in $\MM(n, Y, T)$. In general, we write $(Z,d_Z,z,\t) \in \overline{\MM(n, Y)}$\index{$\overline{\MM(n, Y)}$} if $(Z,d_Z,z,\t)$ is a noncollapsed Ricci flow limit space obtained as the pointed Gromov--Hausdorff limit of a sequence in $\MM(n, Y, T_i)$ for some sequence $\{T_i\}$ with a finite or infinite limit.
\end{notn}

In this paper, all results concerning noncollapsed Ricci flow limit spaces in $\overline{\MM(n, Y, T)}$ remain valid even in $\overline{\MM(n, Y)}$.

\section{\texorpdfstring{$\mathbb{F}$-limits}{F-limits} of Ricci flows} \label{sec:f}

In this section, we relate our Ricci flow limit spaces to Bamler's $\mathbb F$-limits developed in \cite{bamler2023compactness} and \cite{bamler2020structure}.

We first recall the following definition of a metric flow (see \cite[Definition 3.2]{bamler2023compactness}).

\begin{defn}[Metric flow] \label{defnmetricflow}
A metric flow\index{metric flow} over a subset $I$ of $\R$ is a tuple of the form 
	\begin{align*}
		(\XX , \t, (d_t)_{t \in I} , (\nu_{x;s})_{x \in \XX, s \in I, s \leq \t (x)}) 
	\end{align*}
	with the following properties:
	\begin{enumerate}[label=\textnormal{(\arabic{*})}]
		\item $\XX$ is a set consisting of points.
		
		\item $\t : \XX \to I$ is a map called the time-function.
		Its level sets $\XX_t := \t^{-1} (t)$ are called time slices and the preimages $\XX_{I'} := \t^{-1} (I')$, $I' \subset I$, are called time-slabs.
		\item $(\XX_t, d_t)$ is a complete and separable metric space for all $t \in I$.
		
		\item $\nu_{x;s}$ is a probability measure on $\XX_s$ for all $x \in \XX$, $s \in I$, $s \leq \t (x)$. For any $x \in \XX$, the family $(\nu_{x;s})_{s \in I, s \leq \t (x)}$ is called the conjugate heat kernel at $x$.
		
		\item $\nu_{x; \t (x)} = \delta_x$ for all $x \in \XX$.
		
		\item For all $s, t \in I$, $s< t$, $L \geq 0$ and any measurable function $u_s : \XX_s \to [0,1]$ with the property that if $L > 0$, then $u_s = \Phi \circ f_s$ for some $L^{-1/2}$-Lipschitz function $f_s : \XX_s \to \R$ \emph{(}if $L=0$, then there is no additional assumption on $u_s$\emph{)}, the following holds. 
		The function
		\begin{align*}
			u_t :\XX_t \longrightarrow \R, \qquad x \longmapsto \int_{\XX_s} u_s \, \mathrm{d}\nu_{x;s} 
		\end{align*}
		is either constant or of the form $u_t = \Phi \circ f_t$, where $f_t : \XX_t \to \R$ is $(t-s+L)^{-1/2}$-Lipschitz. Here, $\Phi$ is given by
\[
\diff{}{x}\Phi(x)=(4\pi)^{-1/2}e^{-x^2/4}, \qquad \lim_{x\to-\infty}\Phi(x)=0, \qquad \lim_{x\to\infty}\Phi(x)=1.
\]
		
		\item For any $t_1,t_2,t_3 \in I$, $t_1 \leq t_2 \leq t_3$, $x \in \XX_{t_3}$ we have the reproduction formula
		\[ \nu_{x; t_1} = \int_{\XX_{t_2}} \nu_{\cdot; t_1} \, \mathrm{d}\nu_{x; t_2}, \]
		meaning that for any Borel set $S\subset \XX_{t_1}$
		\[ \nu_{x;t_1} (S) = \int_{\XX_{t_2}} \nu_{y ; t_1} (S) \, \mathrm{d}\nu_{x; t_2}(y). \]
	\end{enumerate}
\end{defn}

Given a metric flow $\XX$ over $I$, we recall the following definitions from \cite[Definition 3.13]{bamler2023compactness}.

\begin{defn}[Conjugate heat flow] \label{def:conju}
	A family of probability measures $(\mu_t \in \mathcal{P} (\XX_t))_{t \in I'}$ over $I' \subset I$ is called a \textbf{conjugate heat flow}\index{conjugate heat flow} if for all $s, t \in I'$, $s \leq t$ we have
	\begin{align*}
		\mu_s = \int_{\XX_t} \nu_{x;s} \, \mathrm{d}\mu_t (x). 
	\end{align*}
\end{defn}

Next, we recall the definition of the metric flow pair from \cite[Definitions 5.1, 5.2]{bamler2023compactness}. Roughly speaking, two metric flow pairs are equivalent if they are the same in the metric measure sense almost everywhere.

\begin{defn}[Metric flow pair]
	A pair $(\XX, (\mu_t)_{t \in I'})$ is called a metric flow pair over $I \subset \R$\index{metric flow pair} if:
	\begin{enumerate}[label=\textnormal{(\arabic*)}]
		\item $I' \subset I$ with $|I \setminus I'| = 0$.
		\item $\XX$ is a metric flow over $I'$.
		\item $(\mu_t \in \mathcal{P} (\XX_t))_{t \in I'}$ is a conjugate heat flow on $\XX$ with $\text{supp}\,\mu_t = \XX_t$ for all $t \in I'$. That is, for all $s, t \in I'$, $s \leq t$ we have
	\begin{align*}
		\mu_s = \int_{\XX_t} \nu_{x;s} \, \mathrm{d}\mu_t (x). 
	\end{align*}		
	\end{enumerate}
	If $J \subset I'$, then we say that $(\XX, (\mu_t)_{t \in I'})$ is fully defined over $J$. We denote by $\IF_I^J$ the set of equivalence classes of metric flow pairs over $I$ that are fully defined over $J$. Here, two metric flow pairs $(\XX^i, (\mu^i_t)_{t \in I^{\prime, i}})$, $i = 1,2$, that are fully defined over $J$ are equivalent if there exist an almost always isometry $\phi$ between $\XX^1$ and $\XX^2$ \emph{(cf. \cite[Definition 5.1]{bamler2023compactness})} and a common set $I' \subset I^{\prime,1} \cap I^{\prime,2}$ such that $|I^{\prime, 1} \setminus I'| = |I^{\prime, 2} \setminus I'| = 0$, $(\phi_t )_* \mu^1_t = \mu^2_t$ for all $t \in I'$, and $J \subset I'$.
\end{defn}

Next, for a sequence of metric flow pairs, we recall the following definition of a correspondence from \cite[Definition 5.4]{bamler2023compactness}, which can be regarded as an embedding into an ambient space.

\begin{defn}[Correspondence]\label{defncorrespondence}
	Let $(\XX^i, (\mu^i_t)_{t \in I^{\prime,i}})$ be metric flow pairs over $I$, indexed by some $i \in \mathcal{I}$.
	A correspondence\index{correspondence} between these metric flows over $I''$ is a pair of the form
	\begin{align*}
		\CF := \big( (A_t, d^A_t)_{t \in I''},(\varphi^i_t)_{t \in I^{\prime\prime,i}, i \in \mathcal{I}} \big), 
	\end{align*}
	where:
	\begin{enumerate}[label=\textnormal{(\arabic*)}]
		\item $(A_t, d^A_t)$ is a metric space for any $t \in I''$.
		\item $I^{\prime\prime,i} \subset I'' \cap I^{\prime,i}$ for any $i \in \mathcal{I}$.
		\item $\varphi^i_t : (\XX^i_t, d^i_t) \to (A_t, d^A_t)$ is an isometric embedding for any $i \in \mathcal{I}$ and $t \in I^{\prime\prime,i}$.
	\end{enumerate}
	If $J \subset I^{\prime\prime,i}$ for all $i \in \mathcal{I}$, we say that $\CF$ is fully defined over $J$.
\end{defn}

Given a correspondence, one can define the corresponding $\mathbb F$-distance. In general, the $\IF$-distance between metric flow pairs is the infimum over all correspondences (see \cite[Definitions 5.6, 5.8]{bamler2023compactness}).

\begin{defn}[$\IF$-distance]
	We define the $\IF$-distance between two metric flow pairs within $\CF$ (uniform over $J$),
	\[ d_{\IF}^{\,\CF, J} \big( (\XX^1, (\mu^1_t)_{t \in I^{\prime,1}}), (\XX^2, (\mu^2_t)_{t \in I^{\prime,2}}) \big), \] 
	to be the infimum over all $r > 0$ with the property that there is a measurable subset $E \subset I''$ with
	\[ J \subset I'' \setminus E \subset I^{\prime\prime,1} \cap I^{\prime\prime,2} \]
	and a family of couplings $(q_t)_{t \in I'' \setminus E}$ between $\mu^1_t, \mu^2_t$ such that:
	\begin{enumerate}[label=\textnormal{(\arabic*)}]
		\item $|E| \leq r^2$.
		\item For all $s, t \in I'' \setminus E$, $s \leq t$, we have
		\[ \int_{\XX^1_t \times \XX_t^2} d_{W_1}^{A_s} ( (\varphi^1_s)_* \nu^1_{x^1; s}, (\varphi^2_s)_* \nu^2_{x^2; s} ) \,\mathrm{d}q_t (x^1, x^2) \leq r. \]
	\end{enumerate}
	The $\IF$-distance between two metric flow pairs (uniform over $J$),
	\[ d_{\IF}^{ J} \big( (\XX^1, (\mu^1_t)_{t \in I^{\prime,1}}), (\XX^2, (\mu^2_t)_{t \in I^{\prime,2}}) \big), \] 
	is defined as the infimum of
	\[ d_{\IF}^{\,\CF, J} \big( (\XX^1, (\mu^1_t)_{t \in I^{\prime,1}}), (\XX^2, (\mu^2_t)_{t \in I^{\prime,2}}) \big), \] 
	over all correspondences $\CF$ between $\XX^1, \XX^2$ over $I''$ that are fully defined over $J$.
\end{defn}

With the $\IF$-distance, one can define the $\IF$-convergence of a sequence of metric flow pairs. In general, $\IF$-convergence implies $\IF$-convergence within a correspondence; see \cite[Theorems 6.5, 6.6]{bamler2023compactness}. More precisely,

\begin{thm} 
	Let $(\XX^i, (\mu^i_t)_{t \in I^{\prime,i}})$, $i \in \N \cup \{ \infty \}$, be metric flow pairs over $I $ that are fully defined over some $J \subset I$.
	Suppose that for any compact subinterval $I_0 \subset I$
	\[ d_{\IF}^{J \cap I_0} \big( (\XX^i , (\mu^i_t)_{t \in I_0 \cap I^{\prime,i}}), (\XX^{\infty} , (\mu^\infty_t)_{t \in I_0\cap I^{\prime,\infty}}) \big) \to 0. \] 
	Then there is a correspondence $\CF$ between the metric flows $\XX^i$, $i \in \N \cup \{ \infty \}$, over $I$ such that
	\begin{align*} 
		(\XX^i, (\mu^i_t)_{t \in I^{\prime,i}}) \xrightarrow[i \to \infty]{\quad \IF, \CF, J \quad} (\XX^{\infty}, (\mu^\infty_t)_{t \in I^{\prime,\infty}})
	\end{align*}
	on compact time intervals, in the sense that
	\begin{align*} 
		d_{\IF}^{\,\CF, J \cap I_0} \big( (\XX^i, (\mu^i_t)_{t \in I_0 \cap I^{\prime,i}}), (\XX^{\infty}, (\mu^\infty_t)_{t \in I_0\cap I^{\prime,\infty}}) \big) \to 0
	\end{align*}
	for any compact subinterval $I_0 \subset I$.
\end{thm}

Next, we recall the notion of convergence of points within a correspondence, see \cite[Definitions 6.7, 6.10, 6.12]{bamler2023compactness}.

\begin{defn}\label{convergenceheatflow2}
	Let $\XX^i$ be metric flows over $I$ and consider a correspondence $\CF$ as in Definition \ref{defncorrespondence} between $\XX^i$ over $I''$. 
	
Let $(\mu_t^i)_{t\in I_*^i}, i\in \N\bigcup \{\infty\}$ be conjugate heat flows on $\XX^i$, where $I_*^i=I^{\prime,i}\bigcap (-\infty, T_i)$ or $I^{\prime,i}\bigcap (-\infty, T_i]$ for some $T_i\in (-\infty,+\infty]$. We say that the conjugate heat flows $(\mu_t^i)_{t\in I_*^i}, i\in \N\bigcup \{\infty\}$ \textbf{converge} to $(\mu_t^\infty)_{t\in I_*^\infty}$ within $\CF$ and that the convergence is uniform over $J$ if $J\subset I_*^i$ for all $i\in\N\bigcup\{\infty\}$ and there exist measurable subsets $E_i\subset I'',i\in \N$ such that:
\begin{enumerate}[label=\textnormal{(\arabic*)}]
		\item $J\bigcap I_*^\infty\subset (I_*^i\bigcap I'')\setminus E_i=(I_*^\infty\bigcap I'')\setminus E_i\subset I^{\prime\prime,i}\bigcap I^{\prime\prime,\infty}$ for large $i$.
		\item $|E_i|\to 0$.
		\item $\sup_{t\in(I_*^\infty\bigcap I'')\setminus E_i}d_{W_1}^{A_t}((\varphi_t^i)_*\mu_t^i,(\varphi_t^\infty)_*\mu_t^\infty)\to 0$. 
	\end{enumerate}
	We write this convergence as
	\begin{align}\label{convergenceheatflow1}
		(\mu_t^i)_{t\in I_*^i}\xrightarrow[i\to\infty]{\quad \CF,J\quad}(\mu_t^\infty)_{t\in I_*^\infty}.
	\end{align}
	We say that \eqref{convergenceheatflow1} holds on compact intervals and is uniform over $J$ if for any compact subinterval $I_0\subset I_*^\infty$, \eqref{convergenceheatflow1} holds after replacing $\CF,J$ by $\CF|_{I''\bigcap I_0}, J\bigcap I_0$. We say that \eqref{convergenceheatflow1} is uniform at time $t\in I''$ if \eqref{convergenceheatflow1} holds after replacing $J$ by $J\bigcup \{t\}$. Let $T_i\in I^{\prime,i}$ and $\mu^i\in \PP(\XX^i_{T_i})$. We say that \textbf{$\mu^i$ converge to $\mu^\infty$ within $\CF$ (and uniformly over $J$)}, and write 
	\begin{align*}
		\mu^i\xrightarrow[i\to\infty]{\quad \CF,J\quad}\mu^\infty,
	\end{align*}
	if $T_i\to T_{\infty}$ and if for the conjugate heat flows $(\tilde\mu_t^i)_{t\in I^{\prime,i}\bigcap (-\infty, T_i]}$, $i\in\N\bigcup\{\infty\}$ with initial condition $\tilde\mu_{T_i}^i=\mu^i$, we have the following convergence on compact time intervals
	\begin{align*}
		(\tilde\mu_t^i)_{t\in I_*^i}\xrightarrow[i\to\infty]{\quad \CF,J\quad}(\tilde\mu_t^\infty)_{t\in I_*^\infty}.
	\end{align*}
	Fix some $T\in I''$ and $\mu^i\in \PP(\XX_T^i)$. We say that $\mu^i$ \textbf{strictly converge} to $\mu^\infty$ within $\CF$ if 
	\begin{align*}
		(\varphi_T^i)_*\mu^i\xrightarrow[i\to\infty]{W_1}(\varphi_T^\infty)_*\mu^\infty.
	\end{align*}

For a sequence of points $x_i\in \XX^i_{T_i},i\in\N\bigcup\{\infty\}$, we say that $x_i$ \textbf{converge} to $x_\infty$ within $\CF$ (and uniformly over $J$) if $ \delta_{x_i}\xrightarrow[i\to\infty]{\quad \CF,J\quad} \delta_{x_{\infty}}$. We write this convergence as
	\begin{align*}
x_i \xrightarrow[i\to\infty]{\quad \CF,J\quad} x_{\infty}.
	\end{align*}
For any sequence of points $x_i\in \XX^i_T, i\in\N\bigcup \{\infty\}$, we say that $x_i$ \textbf{strictly converge} to $x_\infty$ within $\CF$ if 
	\begin{align*}
		(\varphi_T^i)(x_i)\xrightarrow[i\to\infty]{}(\varphi_T^\infty)(x_\infty).
	\end{align*}	
\end{defn}

Next, we recall the following definition from \cite[Definition 3.21]{bamler2023compactness}.

\begin{defn}[$H$-concentration] 
	Given a constant $H>0$, a metric flow $\XX$ is called \textbf{$H$-concentrated} if for any $s \leq t$, $s,t \in I$, $x_1, x_2 \in \XX_t$
	\begin{align*}
		\Var (\nu_{x_1; s}, \nu_{x_2; s} ) \leq d^2_t (x_1, x_2) + H (t-s).
	\end{align*}
\end{defn}

We note that Definition \ref{convergenceheatflow2} has defined two notions of convergence of measures or points. Strict convergence is useful if the $\CF$-convergence is timewise at time $T$, see \cite[Theorems 6.13, 6.15]{bamler2023compactness}. The following theorem from \cite[Theorem 6.19]{bamler2023compactness} shows how to represent points as limits of sequences:
\begin{thm}\label{representlimit}
	Let $\XX^i$ be metric flows over a subset $I^{\prime,i}\subset \R, i\in \N\bigcup \{\infty\}$ and consider a correspondence $\CF$ as in Definition \ref{defncorrespondence} between $\XX^i$. Suppose for some $J\subset \R$, we have on compact time intervals, 
	\begin{align*} 
		(\XX^i, (\mu^i_t)_{t \in I^{\prime,i}}) \xrightarrow[i \to \infty]{\quad \IF, \CF, J \quad} (\XX^{\infty}, (\mu^\infty_t)_{t \in I^{\prime,\infty}})
	\end{align*}
	and all $\XX^i$ are $H$-concentrated for some uniform constant $H$. Consider some $x_\infty\in \XX^\infty_{t_\infty}$ with $t_\infty> \inf\,I_\infty$ and a sequence of times $t_i\in I^{\prime,i}$ with $t_i\to t_\infty$. Then there exist points $x_i\in \XX^i_{t_i}$ such that
	\begin{align*}
		x_i\xrightarrow[i\to\infty]{\quad \CF,J\quad}x_\infty.
	\end{align*}
	In particular, if $t_\infty\in I^{\prime,i}$ for all $i\in\N$, then we can choose all $x_i\in \XX^i_{t_\infty}$.
\end{thm}

In this paper, we will focus on metric flows induced by closed Ricci flows and their limits.  For any pointed Ricci flow $\{M^n,(g(t))_{t \in [-L,0]},p^*=(p,0)\}$, one can define $(\XX, (\mu_t)_{t \in [-L,0]})$ as follows:
\begin{align} \label{eq:example}
	\lc \XX:=M \times [-L,0) \sqcup p^*, \t:=\text{proj}_{[-L,0]}, (d_t)_{t \in [-L,0]}, (\nu_{x^*;s})_{x^* \in \XX, s\in [-L,0], s\le \t(x^*)}, \mu_t:=\nu_{p^*;t}\rc.
\end{align}
Here, if $L=\infty$, we set $[-L, 0]=(-\infty,0]$.

Then, by Proposition \ref{monotonicityW1}, we have:

\begin{prop}
	The pair $(\XX, (\mu_t)_{t \in [-L,0]})$ defined in \eqref{eq:example} is an $H_n$-concentrated metric flow pair that is fully defined over $[-L,0]$.
\end{prop}

For a sequence of closed Ricci flows, we have the following compactness theorem from \cite[Theorem 7.4, Corollary 7.5, Theorem 7.6]{bamler2023compactness}.

\begin{thm}[$\IF$-limit] \label{thm:601}
	Let $\{M_i^n,(g_i(t))_{t \in [-L, 0]},p_i^*=(p_i, 0)\}$ be a sequence of closed Ricci flows with the corresponding metric flow pairs $(\XX^i,(\mu_t^i)_{t \in [-L,0]})$ as described in \eqref{eq:example}.

For any finite set $J \subset [-L, 0]$ containing $0$, after passing to a subsequence, there exist an $H_n$-concentrated metric flow pair $(\XX^\infty, (\mu^\infty_t)_{t \in [-L,0]})$ and a correspondence $\CF$ between the metric flows $\XX^i$, $i \in \N \cup \{ \infty \}$, over $[-L, 0]$ such that (on compact time intervals if $L=+\infty$)
	\begin{equation} \label{Fconv}
		(\XX^i, (\mu^i_t)_{t \in [-L,0]}) \xrightarrow[i \to \infty]{\quad \IF, \CF,J \quad} (\XX^\infty, (\mu^\infty_t)_{t \in [-L,0]}).
	\end{equation}
where $\XX^\infty_{0}$ consists of a single point $p_{\infty}$, and $\mu^\infty_t=\nu_{p_{\infty};t}$. Moreover, the convergence \eqref{Fconv} is uniform over any compact $J' \subset [-L,0]$ that only contains times at which $\XX^\infty$ is continuous. The limit metric flow pair $(\XX^\infty, (\mu^\infty_t)_{t \in [-L,0]})$ is called an \textbf{$\IF$-limit} of the sequence.

In addition, after passing to a subsequence, there exists a unique $\IF$-limit $(\XX^{\IF},(\nu_{p_{\infty};t})_{t \in [-L,0]})$ such that $\XX^{\IF}$ is future continuous in the sense of \cite[Definition 4.7]{bamler2023compactness}. 
\end{thm}

\begin{rem}
In general, the set of discontinuous times of $\XX^\infty$ is countable; see \cite[Corollary 4.11]{bamler2023compactness}.
\end{rem}

Generally speaking, an $\IF$-limit $(\XX^\infty, (\nu_{p_{\infty};t})_{t \in [-L,0]})$ carries limited geometric information. However, if we assume that the Nash entropies of all closed Ricci flows are uniformly bounded, an assumption equivalent to a certain noncollapsing condition, then the $\IF$-limit $(\XX^\infty, (\mu^\infty_t)_{t \in [-L,0]})$ reveals significantly richer structural properties.

Let us first recall the following definitions from \cite[Definitions 6.22, 3.42, 3.40, 3.46]{bamler2023compactness}:

\begin{defn}[Tangent metric flow] 
	Let $\XX$ be a metric flow over $I$ and $x_0 \in \XX_{t_0}$ a point.
	We say that a metric flow pair $(\XX', (\nu'_{x_{\max};t})_{t \in (-\infty,0]})$ is a tangent metric flow of $\XX$ at $x_0$ if there is a sequence of scales $\lambda_k> 0$ with $\lambda_k \to \infty$ such that for any $L > 0$ the parabolic rescalings (see \cite[Lemma 3.4]{bamler2023compactness} for the notations) 
	\begin{align*}
		\big(\XX^{-t_0, \lambda_k}_{[-L,0]}, (\nu_{x_0;t}^{-t_0, \lambda_k})_{\lambda_k^{-2} t + t_0 \in I , t \in [-L,0]}\big)
	\end{align*}
	 $\IF$-converge to $\big(\XX'_{[-L,0]}, (\nu'_{x_{\max};t})_{t \in [-L,0]}\big)$.
\end{defn}

\begin{defn}[Metric soliton] \label{def:met_soliton}\index{metric soliton}
	A metric flow pair $(\XX, (\mu_t)_{t \in (-\infty,0]})$ is called a \textbf{metric soliton} if there is a tuple
	\[ \big( X, d, \mu, (\nu'_{x;t} )_{x \in X; t < 0} \big) \]
	and a map $\phi : \XX \to X$ such that the following holds:
	\begin{enumerate}[label=\textnormal{(\arabic{*})}]
		\item $(X,d,\mu)$ is a metric measure space and for any $t<0$, the map $\phi_t : (\XX_t, d_t, \mu_t) \to (X, \sqrt{|t|} d, \mu)$ is an isometry between metric measure spaces.
		\item For any $x \in \XX_t$ and $s< t$, we have $(\phi_s)_* \nu_{x;s} = \nu'_{\phi_t(x); \log(s/t)}$.
	\end{enumerate}
\end{defn}

\begin{defn}[Static cone]
	A metric flow $\XX$ over $(-\infty,0]$ is called a \textbf{static cone} if there is a tuple
	\[ \big( X, d,  (\nu'_{x;t} )_{x \in X; t \leq 0} \big) \]
	and a map $\phi : \XX \to X$ such that the following holds:
	\begin{enumerate}[label=\textnormal{\arabic{*}.}]
	\item $(X, d)$ is a metric cone with vertex $q$ such that for any $\lambda \in (0, 1]$, if $\psi_{\lambda}:X \to X$ is the radial dilation by $\lambda$ preserving $q$, then $(\psi_{\lambda})_* \nu'_{x;t}=\nu'_{\psi_{\lambda}(x);\lambda^2t}$ for any $x \in X$ and $t \le 0$.
		
		\item For any $t<0$, the map $\phi_t : (\XX_t, d_t) \to (X,d)$ is an isometry.
		
		\item For any $x \in \XX_t$ and $s \leq t$, we have $(\phi_s)_* \nu_{x;s} = \nu'_{\phi_t (x); t-s}$.
	\end{enumerate}
\end{defn}

We consider a sequence of closed Ricci flows with entropy bounded below at the base point. First, we recall the following definition.

\begin{defn}[Ricci flow spacetime] \label{def_RF_spacetime}
An $n$-dimensional Ricci flow spacetime over an interval $I \subset \R$ is a tuple $(\mathcal U, \mathfrak{t}, \partial_{\mathfrak{t}}, g)$ with the following properties:
\begin{enumerate}[label=\textnormal{(\arabic{*})}]
\item $\mathcal U$ is an $(n+1)$-dimensional smooth manifold with smooth boundary $\partial \mathcal U$, and $\partial \mathcal U$ is a disjoint union of smooth manifolds of dimension $n$.

\item $\mathfrak{t} : \mathcal U \to I$ is a smooth function without critical points. For any $t \in I$ we denote by $\mathcal U_t := \mathfrak{t}^{-1} (t) \subset \mathcal U$ the time-$t$-slice of $\mathcal U$.

\item $\t (\partial \mathcal U) \subset \partial I$.

\item $\partial_{\mathfrak{t}}$ is a smooth vector field on $\mathcal U$ that satisfies $\partial_{\mathfrak{t}} \mathfrak{t} \equiv 1$.

\item $g$ is a smooth inner product on the spatial subbundle $\ker (\mathrm{d} \mathfrak{t} ) \subset T \mathcal U$.
For any $t \in I$ we denote by $g_t$ the restriction of $g$ to the time-$t$-slice $\mathcal U_t$.

\item $g$ satisfies the Ricci flow equation: $\mathcal{L}_{\partial_\mathfrak{t}} g = - 2 \Ric (g)$.
Here $\Ric (g)$ denotes the symmetric $(0,2)$-tensor on $\ker (\mathrm{d} \mathfrak{t} )$ that restricts to the Ricci tensor of $(\mathcal U_t, g_t)$ for all $t \in I$.
\end{enumerate}
\end{defn}

Obviously, a conventional Ricci flow $(M,g(t))_{t \in I}$ is a Ricci flow spacetime by setting $\mathcal M=M \times I$, $\t$ to be the projection onto the time factor, and $\partial_{\t}$ to be the unit vector on $I$.

The following structure theorem follows from \cite[Theorems 2.3, 2.4, 2.5, 2.6, 2.46]{bamler2020structure} and \cite[Theorem 9.21]{bamler2023compactness}.

\begin{thm}\label{Fconvergence}
	Let $\XX^i=\{M_i^n,(g_i(t))_{t \in [-L,0]},p_i^*=(p_i,0)\}$ be a sequence of pointed closed Ricci flows with entropy bounded below by $-Y$ at $p_i^*$ (see Definition \ref{def:entropybound}). Suppose $(\XX^\infty, (\mu^\infty_t)_{t \in [-L,0]})$ is a future continuous $\IF$-limit obtained in Theorem \ref{thm:601}. Then the following statements hold.
	
	\begin{enumerate}[label=\textnormal{(\arabic{*})}]
		\item There exists a decomposition
		\begin{equation}\label{decomRS}
			\XX^\infty_{0}=\{p_{\infty}\},\quad \XX^\infty_{(-L,0)}=\RR^\F \sqcup \MS^{\F},\index{$\RR^\F$}\index{$\MS^\F$}
		\end{equation}
		such that $\RR^\F$ is given by an $n$-dimensional Ricci flow spacetime $(\RR^\F, \t, \partial^{\infty}_{\t}, g^{\infty})$ and $\dim_{\MM^*}(\MS^{\F}) \le n-2$, where $\dim_{\MM^*}$ denotes the $*$-Minkowski dimension in \cite[Definition 3.31]{bamler2023compactness}. Moreover, $\RR^\F_t$ is a connected open set and $\mu^{\infty}_t(\MS^{\F}_t)=0$ for any $t \in (-L,0)$.
		
		\item Every tangent flow $(\XX',(\nu_{x_{\infty}';t})_{t \le 0})$ at every point $x \in \XX^\infty$ is a metric soliton. Moreover, $\XX'$ is the Gaussian soliton iff $x \in \RR^\F$. If $x \in \MS^{\F}$, the singular set of $(\XX',(\nu_{x_{\infty}';t})_{t \le 0})$ on each $t <0$ has Minkowski dimension at most $n-4$. In particular, if $n=3$, the metric soliton is a smooth Ricci flow associated with a $3$-dimensional Ricci shrinker. If $n=4$, each slice of the metric soliton is a smooth Ricci shrinker orbifold with isolated singularities.
		
		\item The convergence \eqref{Fconv} is smooth on $\RR^\F$, in the following sense. There exist an increasing sequence $U_1 \subset U_2 \subset \ldots \subset \RR^\F$ of open subsets with $\bigcup_{i=1}^\infty U_i = \RR^\F$, open subsets $V_i \subset M_i \times (-L,0)$, time-preserving diffeomorphisms $\phi_i : U_i \to V_i$ and a sequence $\ep_i \to 0$ such that the following holds:
		\begin{enumerate}[label=\textnormal{(\alph{*})}]
			\item We have
			\begin{align*}
				\Vert \phi_i^* g^i - g^\infty \Vert_{C^{[\ep_i^{-1}]}(U_i)} & \leq \ep_i, \\
				\Vert \phi_i^* \partial_{\t_i} - \partial^\infty_{\t} \Vert_{C^{[\ep_i^{-1}]}(U_i)} &\leq \ep_i, \\
				\Vert w^i \circ \phi_i - w^\infty \Vert_{C^{[\ep_i^{-1}]}(U_i)} &\leq \ep_i,
			\end{align*}
			where $g^i$ is the spacetime metric induced by $g_i(t)$, and $w^i$ is the conjugate heat kernel defined by $\mathrm{d} \nu_{p_i^*;\cdot}=w^i\,\mathrm{d} g^i$, $i \in \N \bigcup \{ \infty \}$.
			
			\item Let $y_\infty \in \RR^\F$ and $y_i \in M_i \times (-L,0)$.
			Then $y_i$ converges to $y_\infty$ within $\CF$ \emph{(cf. Definition \ref{convergenceheatflow2})} if and only if $y_i \in V_i$ for large $i$ and $\phi_i^{-1} (y_i) \to y_\infty$ in $\RR^\F$.
			
			\item If the convergence \eqref{Fconv} is uniform at some time $t \in (-L,0)$, then for any compact subset $K \subset \RR^\F_t$ and for the same subsequence we have
			\[ \sup_{x \in K \cap U_i} d^A_t (\varphi^i_t (\phi_i(x)), \varphi^\infty_t (x) ) \longrightarrow 0. \]
		\end{enumerate}
\item For any $t \in (-L, 0)$, the restriction of $d_t$ on $\RR^\F_t$ agrees with the length metric of $g(t)$.
	\end{enumerate}
\end{thm}

The singular set $\mathcal{S}^\F$ in \eqref{decomRS} has a natural stratification; see \cite[Theorem 1.9]{bamler2020structure}:

\begin{thm}\label{Fstratification}
	There is a stratification of $\mathcal{S}^\F$ 
	\begin{align*}
	    \mathcal{S}^{0,\F}\subset\mathcal{S}^{1,\F}\subset\cdots \subset\mathcal{S}^{n-2,\F}=\mathcal{S}^\F,
	\end{align*}
	such that for each $k=0,\ldots,n-2$,
	\begin{enumerate}[label=\textnormal{\arabic{*}.}]
		\item $\dim_{\mathcal H^*}(\mathcal S^{k,\F}) \le k$, where $\dim_{\mathcal H^*}$ denotes the $*$-Hausdorff dimension in \cite[Definition 3.30]{bamler2023compactness};\index{$\mathcal S^{k,\F}$}
		\item Every point $x\in \mathcal{X}^\F_{<0}\setminus\mathcal{S}^{k-1,\F}$ has a tangent flow $(\XX',(\nu_{x';t})_{t\leq 0})$ that is a metric soliton and satisfies one of the following:
		\begin{enumerate}[label=\textnormal{(\alph{*})}]
			\item $\XX'_{<0}=\XX_{<0}''\times \R^k$ and $(\nu_{x';t})_{t\leq 0}=(\mu_t''\otimes \mu_t^{\R^k})_{t<0}$ for some metric soliton $(\XX'',(\mu_t'')_{t<0})$;
			\item $\XX'_{<0}=\mathcal{X}_{<0}''\times \R^{k-2}$ and $(\nu_{x';t})_{t\leq 0}=(\mu_t''\otimes \mu_t^{\R^{k-2}})_{t<0}$ for some static cone $(\XX'',(\mu_t'')_{t<0})$.
		\end{enumerate}
	\end{enumerate}
\end{thm}

For later use, we record the following splitting result for the $\F$-limit, which is essentially a consequence of \cite[Theorem 15.50]{bamler2020structure}. We sketch the proof for readers' convenience.

\begin{thm}\label{thm:splitting}
Let $\XX^\infty$ be the limit metric flow from Theorem \ref{Fconvergence}, and let $(-T_1,-T_2)\subset [-L,0]$. Suppose that there exist $k$ smooth functions $y_1,\ldots,y_k$ on $\RR^{\F}_{(-T_1,-T_2)}$ such that for all $a,b\in \{1,\ldots,k \}$,
	\begin{align*}
		\la\na y_a,\na y_b\ra =\delta_{ab},\quad \na^2 y_a=0,\quad \partial_\t y_a=0,
	\end{align*}
	then the vector fields $\na y_a$ induce an isometric splitting of the form $\XX^\infty_{(-T_1,-T_2)}=\XX'_{(-T_1,-T_2)}\times \R^k$ for some metric flow $\XX'$ over $(-T_1,-T_2)$.
\end{thm}
\begin{proof}
It suffices to establish the claim on any closed subinterval $[-T_1',-T_2']\subset (-T_1,-T_2)$. Using the cutoff functions constructed in \cite[Lemma 15.27]{bamler2020structure} and the smooth convergence, we obtain $\epsilon_i\to 0$ and functions $u_a^i\in C^\infty(\XX^i)$, $a\in \{1,\ldots,k\}$, such that for all $a,b\in \{1,\ldots,k\}$,
	\begin{align*}
		\int_{-\frac{T_1'+T_1}{2}}^{-\frac{T_2'+T_2}{2}}\int_{M_i}|\square u_a^i|\, \mathrm{d}\nu_{p_i^*}\mathrm{d}t\leq\ep_i,\quad \int_{-\frac{T_1'+T_1}{2}}^{-\frac{T_2'+T_2}{2}}\int_{M_i}|\la \na u_a^i,\na u_b^i\ra-\delta_{ab}|\, \mathrm{d}\nu_{p_i^*;t}\mathrm{d}t\leq \ep_i.
	\end{align*}
Moreover, $u_a^i\to y_a$ on $\RR_{[-T_1',-T_2']}$.
	
By \cite[Proposition 12.1]{bamler2020structure}, there exist $\tilde u_a^i$ on $\XX^i_{[-T'_1, -T'_2]}$ for $a\in \{1,\ldots,k\}$ with $\square \tilde u_a^i=0$ such that, for all $a,b$,
	\begin{align*}
		\int_{-T_1}^{-T_2}\int_{M_i}|\la \na \tilde u^i_a,\na \tilde u^i_b\ra-\delta_{ab}|\, \mathrm{d}\nu_{p_i^*;t}\mathrm{d}t\leq\ep'_i,\quad \tilde u_a^i(p_i^*)=0,
	\end{align*}
	and
	\begin{align*}
		\int_{-T_1}^{-T_2}\int_{M_i}|\na^2\tilde u^i_a|^2\,\mathrm{d}\nu_{p_i^*;t}\mathrm{d}t\leq \ep'_i,
	\end{align*}
where $\epsilon_i'\to 0$ as $i\to\infty$. Furthermore, $\tilde u_a^i\to y_a$ on $\RR_{[-T_1',-T_2']}$. The remainder of the argument follows verbatim from \cite[Theorem 15.50]{bamler2020structure}.
\end{proof}

One can define, even in smooth Ricci flows, the quantitative singular strata as in \cite{cheeger2013lower}. The following definition is from \cite[Definition 2.22]{bamler2020structure}, slightly adapted to our setting.

\begin{defn} \label{Def_quantiSS_weak}\index{$\MS^{\ep, k, \F}_{r_1,r_2}$}
Let $\XX=\{M^n,(g(t))_{t \in \III^{++}}\} \in \MM(n, Y, T)$. For any $\ep>0$ and $0<r_1<r_2<\infty$, we have the following quantitative strata:
	\[  \MS^{\ep, 0, \F}_{r_1,r_2} \subset  \MS^{\ep, 1, \F}_{r_1,r_2} \subset \ldots \subset \MS^{\ep, n-2, \F}_{r_1,r_2} \subset M \times \III^{-} \]
defined as follows:	$x^* \in \MS^{\ep, k, \F}_{r_1,r_2}$ if and only if $\t(x^*)-\ep^{-1}r_2^2 \in \III^-$ and $x^*$ is not $(k+1, \ep, r')$-$\F$-symmetric for any $r' \in [r_1, r_2]$.
	
	Here, a point $x_0^*=(x_0,t_0) \in \XX_{\III^-}$ is called $(k, \ep, r)$-$\F$-symmetric if $\t(x_0^*)-\ep^{-1}r^2>-(1-2\sigma)T$ and there exists a metric flow pair $(\XX', (\nu_{x';t})_{t \le 0})$ over $(-\infty,0]$ that arises as a noncollapsed $\F$-limit of closed Ricci flows as in Theorem \ref{Fconvergence} and satisfies Theorem \ref{Fstratification} 2. (a) or (b) such that the following holds. Consider the metric flow pair
				\begin{align*}
\lc \XX_{[t_0-\ep^{-1} r^2, t_0]}, (\nu_{x_0^*;t})_{[t_0-\ep^{-1} r^2, t_0]} \rc.
			\end{align*}
After a time-shift by $-t_0$ and parabolic rescaling by $r^{-1}$, this metric flow pair has $d_{\F}$-distance smaller than $\ep$ to the metric flow pair $(\XX'_{[-\ep^{-1}, 0]}, (\nu_{x';t})_{t \in [-\ep^{-1}, 0]})$.
\end{defn}

By \cite[Theorems 2.25, 2.28]{bamler2020structure}, we have the following estimates, which can be regarded as parabolic versions of \cite[Theorem 1.3, Corollary 1.11]{cheeger2013lower}.

\begin{thm}\label{quantsizeS}
Let $\XX=\{M^n,(g(t))_{t \in \III^{++}}\} \in \MM(n, Y, T)$ with $x_0^* \in \XX_{\III^-}$. Given $\ep>0$ and $r>0$ with $\t(x_0^*)-2r^2 \in \III^-$, for any $\delta \in (0, \ep)$, there exist $x_1^*, x_2^*,\ldots, x_N^* \in \MS^{\ep, k, \F}_{\delta r, \ep r} \cap P^* (x_0^*; r)$ with $N \leq C(n,Y, \ep) \delta^{-k-\ep}$ and
	\begin{align*}
		\MS^{\ep, k, \F}_{\delta r, \ep r} \cap P^* (x_0^*; r) \subset \bigcup_{i=1}^N P^* (x_i^*; \delta r).
	\end{align*}
Moreover, if $\ep\leq \ep(n,Y)$, then 
	\begin{align*}
		r_{\Rm}\geq  \delta r \quad \text{on} \quad P^* (x_0^*; r) \cap M \times \III^- \setminus \MS^{\ep, n-2, \F}_{\delta r, \ep r},
	\end{align*}
	where $r_{\Rm}$ is the curvature radius from Definition \ref{defncurvatureradius}. The following integral estimate also holds for every sufficiently small $\ep>0$:
		\begin{align*}
		&\int_{[\t(x_0^*)-r^2,\t(x_0^*)+r^2]\cap \III^{-}}\int_{P^*(x_0^*;r)\cap M\times \{t\}}|\Rm|^{2-\ep} \, \mathrm{d}V_{g(t)}\mathrm{d}t \notag\\
		\leq &\int_{[\t(x_0^*)-r^2,\t(x_0^*)+r^2]\cap \III^{-}}\int_{P^*(x_0^*;r)\cap M\times \{t\}}r_{\Rm}^{-4+2\ep} \, \mathrm{d}V_{g(t)}\mathrm{d}t\leq C(n,Y,\ep)r^{n-2+2\ep}. 
	\end{align*} 
\end{thm}

For the rest of the section, let
\begin{align}\label{eq:pGHgeneral}
(M_i\times\III,d_i^*,p_i^*,\t_i)
\xrightarrow[i\to\infty]{\quad\mathrm{pGH}\quad}(Z,d_Z,p_\infty,\t)
\end{align}
be a limit from Theorem \ref{thm:GHlimit-dstar}, where $\XX^i\in\MM(n,T)$. In the part concerning the regular--singular structure, we additionally assume $\XX^i\in\MM(n,Y,T)$ for a fixed $Y$.

Given $z\in Z$, choose $z_i^*\in\XX^i_{\III}$ converging to $z$ in the Gromov--Hausdorff sense. Set
\begin{align*}
J:=\{-(1-\sigma)T\}
\end{align*}
when $T<+\infty$ and $J:=\emptyset$ in the ancient case. By Theorem \ref{thm:601}, after passing to a further subsequence, there exist an $H_n$-concentrated metric flow $\XX^z$, which is future continuous for all $t \in [-T,\t(z)]$ except possibly at $t=-(1-\sigma)T$, and a correspondence $\CF$ such that
\begin{align}\label{Flimit}
(\XX^i,(\nu_{z_i^*;t})_{t\in[-T,\t_i(z_i^*)]})
\xrightarrow[i\to\infty]{\quad\F,\CF,J\quad}
(\XX^z,(\nu_{z;t})_{t\in[-T,\t(z)]}).
\end{align}
For $T<+\infty$, the convergence is uniform at $-(1-\sigma)T$. We call $\XX^z$ a \textbf{metric flow associated with} $z$ and denote its time-function by $\t^z$.

For $x,y\in\XX^z_{\III^+}$ and $\tau \le \min\{\t^z(x),\t^z(y)\}$, set
\begin{align*}
W_{x,y}(\tau):=d_{W_1}^{\XX^z_\tau}(\nu_{x;\tau},\nu_{y;\tau}).
\end{align*}
At a discontinuity time we use the essential left limit
\begin{align*}
W^-_{x,y}(\tau):=\lim_{\rho\nearrow\tau}W_{x,y}(\rho),
\end{align*}
which exists by monotonicity. When $T<+\infty$, at the lower endpoint we set
\begin{align*}
W^-_{x,y}(-(1-\sigma)T):=W_{x,y}(-(1-\sigma)T).
\end{align*}

\begin{defn}\label{defn:dstar-limit}
For $x,y\in\XX^z_{\III^+}$, put
\begin{align*}
t_+:=\max\{\t^z(x),\t^z(y)\},\qquad t_-:=\min\{\t^z(x),\t^z(y)\}.
\end{align*}
We define
\begin{align}\label{eq:dstar-limit-def}
d_z^*(x,y):=\inf_{-(1-\sigma)T\leq\tau\leq t_-}
\max\left\{\sqrt{t_+-\tau},W^-_{x,y}(\tau)\right\}.
\end{align}
For an ancient metric flow, the infimum is taken over $\tau\in(-\infty,t_-]$.
\end{defn}

On a smooth Ricci flow, Definition \ref{defn:dstar-limit} agrees with Definition \ref{defn:dstar-distance}. If $r=d_z^*(x,y)$ and $q=t_+-r^2>-(1-\sigma)T$, then
\begin{align}\label{eq:dstar-limit-crossing1}
\lim_{\tau\nearrow q}d_{W_1}^{\XX^z_\tau}(\nu_{x;\tau},\nu_{y;\tau})\leq r.
\end{align}
If also $r>\sqrt{|\t^z(x)-\t^z(y)|}$, then
\begin{align}\label{eq:dstar-limit-crossing2}
r\leq\lim_{\tau\searrow q}d_{W_1}^{\XX^z_\tau}(\nu_{x;\tau},\nu_{y;\tau}).
\end{align}
At the lower endpoint one always has
\begin{align*}
d_{W_1}^{\XX^z_{-(1-\sigma)T}}(\nu_{x;-(1-\sigma)T},\nu_{y;-(1-\sigma)T})\leq r.
\end{align*}

\begin{prop}\label{prop:dstar-limit-pseudometric}
The function $d_z^*$ is a pseudo-distance on $\XX^z_{\III^+}$. Moreover,
\begin{align*}
|\t^z(x)-\t^z(y)|\leq d_z^*(x,y)^2.
\end{align*}
If the noncollapsed structure theorem applies, then $d_z^*$ is a distance on the regular part $\RR^z$.
\end{prop}
\begin{proof}
The time estimate is immediate. The triangle inequality is proved by the same argument as in Lemma \ref{lem:000a}: choose almost minimizing comparison times for two pairs, pass both Wasserstein estimates to a common earlier time, and use the triangle inequality for $W_1$. The left-envelope convention is compatible with this argument by approximation from earlier continuity times.

We next show that if $x,y\in \RR^z$ with $d_z^*(x,y)=0$, then $x=y$. By \eqref{eq:dstar-limit-def}, $\t^z(x)=\t^z(y)=t<\t^z(z)$. If $d:=d^{z}_t(x,y)>0$, then by \cite[Definition 9.11]{bamler2023compactness}, there exists a sufficiently small constant $r>0$ such that $P(x, r):=\bigcup_{s \in [t-r^2, t]} B_s(x, r)$ and $P(y, r):=\bigcup_{s \in [t-r^2, t]} B_s(y, r)$ are disjoint and both contained in $\RR^z$. Thus, after decreasing $r$ if necessary, Proposition \ref{HncenterScal} gives $d_{W_1}^{\XX^{z}_s}(\nu_{x;s},\nu_{y;s})>r$ for $s$ close to $t$. However, this contradicts $d_z^*(x,y)=0$. Thus, we must have $d^z_t(x,y)=0$, and hence $x=y$.
\end{proof}

\begin{rem}\label{rem:pastc}
	If $\XX^{z}$ is assumed to be past continuous, then $d^*_z$ also defines a metric on $\XX^{z}_{\III^+}$. In fact, we only need to check that $d_z^*$ is positive definite. By the argument in the proof of Proposition \ref{prop:dstar-limit-pseudometric}, if $x,y\in \XX^{z}_{\III^+}$ with $d_z^*(x,y)=0$, then $\t^z(x)=\t^z(y)=t<\t^z(z)$. By past continuity and \cite[(4.22)]{bamler2023compactness}, we have 
	\begin{align*}
		\lim_{s\nearrow t}d_{W_1}^{\XX^{z}_s}(\nu_{x;s},\nu_{y;s})=d^z_t(x,y).
	\end{align*}
	Thus, if $d^z_t(x,y)>0$, then for $s<t$ sufficiently close to $t$, $d_{W_1}^{\XX^{z}_s}(\nu_{x;s},\nu_{y;s})\geq \frac{1}{2}d^z_t(x,y)$, which, by \eqref{eq:dstar-limit-def}, implies $d_z^*(x,y)>0$. This contradicts the assumption $d_z^*(x,y)=0$.
\end{rem}

The analogues of Proposition \ref{distancefunction1}, Proposition \ref{propdistance2}, and Lemma \ref{lem:Hcenterdis} hold for $d_z^*$. In particular, if $w\in\XX^z_s$ is an $H$-center of $x$, then
\begin{align}\label{eq:Hcenterlimit}
d_z^*(x,w)\leq\max\{1,\sqrt H\}\sqrt{\t^z(x)-s}.
\end{align}

In the noncollapsed setting, Theorem \ref{Fconvergence} gives the regular--singular decomposition
\begin{align}\label{RSdecomz}
\XX^z_{\III^+\cap(-\infty,\t(z))}=\RR^z\sqcup\MS^z.
\end{align}
Moreover, we have the following density statement.

\begin{lem}\label{lem:complete1}
Assume $\XX^i\in\MM(n,Y,T)$. Then $\RR^z_{\III}$ is dense in $\XX^z_{\III}$ with respect to $d_z^*$.
\end{lem}
\begin{proof}
Given $x\in\XX^z_{\III}$ with $t=\t^z(x)$, if $t>-(1-2\sigma)T$, we choose $r>0$ small and an $H_n$-center $w\in\XX^z_{t-r^2}$ of $x$. Since $(\XX^z_{t-r^2},d^z_{t-r^2})$ is the metric completion of $(\RR^z_{t-r^2},g^z_{t-r^2})$, choose $w'\in\RR^z_{t-r^2}$ with $d^z_{t-r^2}(w,w')<r$. Then
\begin{align*}
d_z^*(x,w')\leq d_z^*(x,w)+d_z^*(w,w')\leq(\sqrt{H_n}+1)r.
\end{align*}
Letting $r\searrow0$ proves the assertion. If $t=-(1-2\sigma)T$, then, arguing as before, $x$ can be approximated by points in $\RR^z_t$.
\end{proof}

Let $(\widetilde\XX^z_{\III},d_z^*,\t^z)$ be the metric quotient of $(\XX^z_{\III},d_z^*)$ by the relation $x\sim y$ if and only if $d_z^*(x,y)=0$. It is separable by Lemma \ref{lem:complete1}.

\begin{lem}\label{lem:key1}
Let $x,y\in\XX^z_{\III}$ and suppose
\begin{align*}
x_i^*\xrightarrow[i\to\infty]{\CF,J}x,\qquad
y_i^*\xrightarrow[i\to\infty]{\CF,J}y.
\end{align*}
Then
\begin{align}\label{lem:dstar-distance-convergence}
\lim_{i\to\infty}d_i^*(x_i^*,y_i^*)=d_z^*(x,y).
\end{align}
\end{lem}
\begin{proof}
Write $r_i=d_i^*(x_i^*,y_i^*)$ and $r=d_z^*(x,y)$. We first prove the upper limit. Fix $\eta>0$. By Definition \ref{defn:dstar-limit}, choose $\tau$ such that
\begin{align*}
\max\left\{\sqrt{t_+-\tau},W^-_{x,y}(\tau)\right\}<r+\eta.
\end{align*}
If $T<+\infty$ and $\tau=-(1-\sigma)T$, uniform convergence at the lower endpoint gives
\begin{align*}
\limsup_{i\to\infty}r_i\leq r+\eta.
\end{align*}
If $T<+\infty$ and $\tau>-(1-\sigma)T$, we choose a continuity time $\tau'<\tau$ sufficiently close to $\tau$ so that the same maximum with $\tau'$ is smaller than $r+2\eta$. Convergence of the conjugate heat kernel measures at $\tau'$ gives
\begin{align*}
\limsup_{i\to\infty}r_i\leq r+2\eta.
\end{align*}
Letting $\eta\searrow0$ gives $\limsup r_i\leq r$.

For the lower limit, pass to a subsequence such that $r_i\to r'$. Choose comparison times $\tau_i$ satisfying
\begin{align*}
\sqrt{t_{i,+}-\tau_i}\leq r_i+i^{-1},\qquad
d_{W_1}^{\tau_i}(\nu_{x_i^*;\tau_i},\nu_{y_i^*;\tau_i})\leq r_i+i^{-1}.
\end{align*}
After passing to a further subsequence, $\tau_i\to\tau$. For every continuity time $\rho<\tau$, backward monotonicity gives
\begin{align*}
d_{W_1}^{\rho}(\nu_{x_i^*;\rho},\nu_{y_i^*;\rho})\leq r_i+i^{-1}
\end{align*}
for all sufficiently large $i$. Passing to the limit within the correspondence and then letting $\rho\nearrow\tau$ yields
\begin{align*}
W^-_{x,y}(\tau)\leq r',\qquad\sqrt{t_+-\tau}\leq r'.
\end{align*}
At the lower endpoint, the same conclusion follows from uniform convergence there. Hence $r\leq r'$, proving \eqref{lem:dstar-distance-convergence}.
\end{proof}

\begin{thm}\label{thm:idenproof}
For any $z\in Z$, there exists an isometric embedding
\begin{align}\label{eq:dstar-isometric-embedding}
\iota_z:(\widetilde\XX^z_{\III},d_z^*)\longrightarrow(Z,d_Z)
\end{align}
such that $\iota_z(z)=z$ and $\t\circ\iota_z=\t^z$. Moreover, for $y_i^*\in\XX^i_{\III}$ and $y_\infty\in\XX^z_{\III}$,
\begin{align*}
y_i^*\xrightarrow[i\to\infty]{\CF,J}y_\infty
\end{align*}
if and only if $y_i^*\to\iota_z(\widetilde y_\infty)$ in the Gromov--Hausdorff sense, where $\widetilde y_\infty$ is the quotient image of $y_\infty$.
\end{thm}
\begin{proof}
Choose a countable $d_z^*$-dense set $\{x_k\}$ in $\XX^z_{\III}$, with $x_1=z$. By Theorem \ref{representlimit}, choose $x_{k,i}^*\in\XX^i_{\III}$ converging to $x_k$ within $\CF$, with $x_{1,i}^*=z_i^*$ for any $i$, where $z_i^*$ is the sequence in \eqref{Flimit}. After passing to a diagonal subsequence, assume that each $x_{k,i}^*$ converges in the pointed Gromov--Hausdorff sense to $a_k\in Z$. Lemma \ref{lem:key1} gives
\begin{align*}
d_Z(a_k,a_l)=d_z^*(x_k,x_l).
\end{align*}
Thus $\widetilde x_k\mapsto a_k$ extends uniquely to an isometric embedding \eqref{eq:dstar-isometric-embedding}, with the base point mapping to $z$. By our construction, $\t\circ\iota_z(\widetilde x_k)=\t(a_k)=\t^z(\widetilde x_k)$ for any $k$. Thus, the identity for the time functions holds by continuity.

Suppose first that $y_i^*$ converges to $y_\infty$ within $\CF$. Approximate $y_\infty$ by $x_k$ and apply Lemma \ref{lem:key1}; any Gromov--Hausdorff limit of $y_i^*$ has the same distances to the points $a_k$ as $\iota_z(\widetilde y_\infty)$, and hence is equal to it.

Conversely, suppose $y_i^*\to\iota_z(\widetilde y_\infty)$ in the Gromov--Hausdorff sense. Theorem \ref{representlimit} gives $w_i^*$ converging to $y_\infty$ within $\CF$. The first implication shows that $w_i^*$ has the same Gromov--Hausdorff limit, so
\begin{align}\label{eqconv0}
d_i^*(y_i^*,w_i^*)\longrightarrow0.
\end{align}
By Definition \ref{convergenceheatflow2}, we can find $E_i\subset [-T,\t^z(y_{\infty}))$ such that
	\begin{align}\label{eqconv1}
		|E_i|\to 0, \quad \sup_{t\in [-T,\t^z(y_{\infty}))\setminus E_i}d_{W_1}^{A_t}((\varphi_t^i)_*\nu_{w_i^*;t},(\varphi_t^\infty)_*\nu_{y_\infty;t})\to 0,
	\end{align}
	where $\varphi_t$ is the embedding defined in Definition \ref{defncorrespondence}.
	
By \eqref{eqconv0}, for any small $\delta>0$, for all sufficiently large $i$,
	\begin{align*}
d_{W_1}^{\max\{\t_i(y_i^*)-\delta^2, \t_i(w_i^*)-\delta^2, -(1-\sigma)T\}}\lc \nu_{y_i^*;\max\{\t_i(y_i^*)-\delta^2, \t_i(w_i^*)-\delta^2, -(1-\sigma)T\}}, \nu_{w_i^*;\max\{\t_i(y_i^*)-\delta^2, \t_i(w_i^*)-\delta^2, -(1-\sigma)T\}} \rc \le \delta,
	\end{align*}
	which implies
		\begin{align}\label{eqconv2}
\sup_{t \in [-T, \max\{\t_i(y_i^*), \t_i(w_i^*)\}-\delta^2)} d_{W_1}^t (\nu_{y_i^*;t} , \nu_{w_i^*;t}) \le \delta.
	\end{align}
Combining \eqref{eqconv1} and \eqref{eqconv2}, we can find $E_i'\subset [-T,\t^z(y_{\infty}))$ with
	\begin{align*}
		|E_i'| \to 0, \quad \sup_{t\in [-T,\t^z(y_{\infty}))\setminus E'_i}d_{W_1}^{A_t}((\varphi_t^i)_*\nu_{y_i^*;t},(\varphi_t^\infty)_*\nu_{y_\infty;t})\to 0.
	\end{align*}
Therefore, $y_i^* \xrightarrow[i \to \infty]{\CF,J} y_{\infty}$.
\end{proof}

By Proposition \ref{prop:dstar-limit-pseudometric} and Theorem \ref{thm:idenproof}, $\RR^z$ can be regarded as a subset of $Z$ via the map $\iota_z$.

\section{Smooth convergence on the regular part} \label{sec:smooth}

In this section, we consider a Ricci flow limit space $(Z, d_Z, p_{\infty},\t)$ obtained from 
	\begin{equation}\label{eq:conv00}
		(M_i \times \III, d^*_i, p_i^*,\t_i) \xrightarrow[i \to \infty]{\quad \mathrm{pGH} \quad} (Z, d_Z, p_{\infty},\t),
	\end{equation}
where $\XX^i=\{M_i^n,(g_i(t))_{t \in \III^{++}}\} \in \MM(n, Y, T)$ with base point $p_i^* \in \XX^i_\III$.

We first introduce the following definition, which is similar to \cite[Definition 9.20]{bamler2023compactness}.

\begin{defn}[Smooth convergence] \label{def:smoothcv}
The Gromov--Hausdorff convergence \eqref{eq:conv00} is \textbf{smooth} at $z \in Z$ if there exist a constant $r>0$ and a sequence $z_i^* \in M_i \times \III$ converging to $z$ in the Gromov--Hausdorff sense such that for all $i$,
	\begin{align*}
r_{\Rm}(z_i^*) \ge r,
	\end{align*}
where $r_{\Rm}$ denotes the curvature radius defined in Definition \ref{defncurvatureradius}. We denote by $\RR \subset Z$ the set of points at which \eqref{eq:conv00} is smooth.
\end{defn}

The first main result of this section is the following theorem.

\begin{thm}\label{thm:smooth1}
The set $\RR$, which is open in $Z$, can be realized as a Ricci flow spacetime $(\RR, \t, \partial_\t, g^Z_t)$ over $\III$ (see Definition \ref{def_RF_spacetime}). Moreover, there exists an increasing sequence of open subsets $U_1 \subset U_2 \subset \ldots \subset \RR$ such that $\bigcup_{i=1}^\infty U_i = \RR$, and for sufficiently large $i$, there exist open subsets $V_i \subset M_i \times \III$, time-preserving diffeomorphisms $\phi_i : U_i \to V_i$, and a sequence $\ep_i \to 0$ such that the following properties hold:
		\begin{enumerate}[label=\textnormal{(\alph{*})}]
			\item We have
			\begin{align*}
				\Vert \phi_i^* g^i - g^Z \Vert_{C^{[\ep_i^{-1}]}(U_i)} & \leq \ep_i, \\
				\Vert \phi_i^* \partial_{\t_i} - \partial_{\t} \Vert_{C^{[\ep_i^{-1}]}(U_i)} &\leq \ep_i,
			\end{align*}
			where $g^i$ is the spacetime metric induced by $g_i(t)$, and $\partial_{\t_i}$ is the standard time vector field.
			
			\item For $U_i^{(2)}=\{(x,y) \in U_i \times U_i \mid \t(x)> \t(y)+\ep_i\}$, $V_i^{(2)}=\{(x^*,y^*) \in V_i \times V_i \mid \t_i(x^*)> \t_i(y^*)+\ep_i\}$ and $\phi_i^{(2)}:=(\phi_i, \phi_i): U_i^{(2)} \to V_i^{(2)}$, we have
					\begin{align*}
\Vert  (\phi_i^{(2)})^* K^i-K_Z \Vert_{C^{[\ep_i^{-1}]}(U_i^{(2)})} \le \ep_i,
			\end{align*}	
where $K^i$ and $K_Z$\index{$K_Z$} denote the heat kernels on $(M_i \times \III,g_i(t))$ and $(\RR,g^Z)$, respectively.
			
		\item Let $y \in \RR$ and $y_i^* \in M_i \times \III$. Then $y_i^* \to y$ in the Gromov--Hausdorff sense if and only if $y_i^* \in V_i$ for large $i$ and $\phi_i^{-1}(y_i^*) \to y$ in $\RR$.
				\end{enumerate}
\end{thm}

The main idea of the proof is to show that each point $z \in \RR$ admits an open neighborhood $U_z$ such that the statements in Theorem \ref{thm:smooth1} hold on $U_z$. These local neighborhoods are then combined in a standard fashion to construct the desired global structure.

First, we prove

\begin{lem}\label{lem:smooth1}
For any $z \in \RR_{\t < 0}$, there exists an open neighborhood $z \in U_z \subset \RR_{\t < 0}$ such that $U_z$ is realized as a Ricci flow spacetime $(U_z, \t, \partial_{\t}, g_t^z)$ defined on a product domain. That is, the Ricci flow spacetime arises from a conventional Ricci flow on $M' \times I'$, where $M'$ is an open manifold and $I'$ is an open interval. 

Moreover, for sufficiently large $i$, there exist open subsets $V_i \subset M_i \times \III$, time-preserving and $\partial_{\t}$-preserving diffeomorphisms $\phi_i : U_z \to V_i$ (that is, $\t_i \circ \phi_i=\t$ and $(\phi_i)_*(\partial_\t)=\partial_{\t_i}$), and a sequence $\ep_i \to 0$ such that statements (a), (b), and (c) in Theorem \ref{thm:smooth1} hold with $U_i$ replaced by $U_z$, and in (c), the point $y$ is required to lie in $U_z$.
\end{lem}

In the proof of Lemma \ref{lem:smooth1}, we are free to pass to any further subsequence of $\XX^i$. Indeed, if Lemma \ref{lem:smooth1} does not hold for a given sequence $\XX^i$, then, after passing to a subsequence, for all sufficiently large $i$, either a diffeomorphism $\phi_i: U_z \to V_i \subset M_i \times \III$ cannot be found, or such a diffeomorphism exists but no sequence $\ep_i \to 0$ can be chosen so that statements (a), (b), and (c) all hold. However, since we may always extract a further subsequence so that Lemma \ref{lem:smooth1} holds, a diffeomorphism $\phi_i: U_z \to V_i$ can always be found for all sufficiently large $i$. If, for instance, statement (a) fails to hold for such $\phi_i$, then we would have
\begin{align*}
\Vert \phi_i^* g^i - g^z \Vert_{C^{[\ep_i^{-1}]} ( U_z)}+\Vert \phi_i^* \partial_{\t_i} - \partial_{\t} \Vert_{C^{[\ep_i^{-1}]} ( U_z)} \ge \ep>0
\end{align*}	
for some constant $\ep>0$ and all large $i$. This contradicts the possibility of passing to a further subsequence; hence, such a violation cannot occur.

\begin{proof}
For $z \in \RR_{\t < 0}$, Definition \ref{def:smoothcv} provides a sequence $z_i^*=(z_i, t_i)\in M_i \times [-(1-2\sigma)T,0)$ converging to $z$ in the Gromov--Hausdorff sense and a constant $r>0$ such that
	\begin{align*}
r_{\Rm}(z_i^*) \ge r.
	\end{align*}

We choose a small constant $\delta \in (0, 1/10)$ to be determined later and set $w_i^*:=(z_i, t_i+3\delta^2 r^2) \in M_i \times (-(1-2\sigma)T,0)$. By the definition of the curvature radius, we have for any $t \in [t_i,t_i+3\delta^2 r^2]$,
	\begin{align*}
|\Rm_{g_i}(z_i, t)| \le r^{-2}.
	\end{align*}

By Proposition \ref{HncenterScal} (i), Proposition \ref{distancefunction1} (1), and Lemma \ref{lem:Hcenterdis}, we have $d_i^*(z_i^*, w_i^*) \le C_1 \delta r$ for a constant $C_1=C_1(n, Y, \sigma)$. After passing to a subsequence, we may assume that $w_i^* \to w \in Z_{(-(1-2\sigma)T, 0)}$ in the Gromov--Hausdorff sense. Furthermore, by extracting a further subsequence, there exists a correspondence $\CF$ such that
\begin{equation}\label{Flimit1}
	(\XX^i, (\nu_{w_i^*;t})_{t \in [-T, \t(w_i^*)]}) \xrightarrow[i \to \infty]{\quad \IF, \CF, J \quad} (\XX^w, (\nu_{w;t})_{t \in [-T, \t(w)]}).
\end{equation}
Next, we choose a time $s \in [\t(z)+\delta^2 r^2, \t(z)+2\delta^2 r^2]$ such that the convergence \eqref{Flimit1} is uniform at time $s$. By Proposition \ref{existenceHncenter} and Proposition \ref{HncenterScal} (i), we have
	\begin{equation*}
\nu_{w_i^*;s}\lc B_{g_i(s)}(z_i, C_2 \delta r) \rc \ge  \frac{1}{2}
	\end{equation*}
for some constant $C_2=C_2(n)>0$. It follows from the definition of $d_{GW_1}$-convergence that, after passing to a further subsequence, there exists $y_i^*=(y_i, s)$ with $y_i \in B_{g_i(s)}(z_i, C_2 \delta r)$ so that $y_i^*$ strictly converges to a point $y \in \XX^w_s$. By \cite[Theorem 6.13 (b)]{bamler2023compactness}, this implies $y_i^*  \xrightarrow[i \to \infty]{\quad \CF,J \quad} y$.

Now, we can choose $\delta$ sufficiently small so that, on the parabolic ball $P_i:=B_{g_i(s)}(y_i, r/2) \times [s-r^2/4, s]$, the curvature $|\Rm_{g_i}|$ is bounded by $4r^{-2}$, and $z_i \in B_{g_i(s)}(y_i, r/4)$. It follows from \cite[Theorem 9.24]{bamler2023compactness} and the smooth convergence in Theorem \ref{Fconvergence} that, after passing to a further subsequence, $z_i^*  \xrightarrow[i \to \infty]{\quad \CF,J \quad} z_{\infty} \in \RR^w$. By Theorem \ref{thm:idenproof} and the fact that $z_i^*$ converges to $z$ in the Gromov--Hausdorff sense, it follows that $\iota_w (z_{\infty})=z$.

Finally, the conclusion follows by observing that we can find, around $z_{\infty}$, an open set $U \subset \RR^w$ for which the statements in the lemma hold (see Theorem \ref{Fconvergence} (3)). Through the isometric embedding $\iota_w: \RR^w \to Z$, we define $U_z=\iota_w(U)$.
\end{proof}

After possibly shrinking $U_z$, we may find a locally finite cover $\{U_{z_i}\}$ of $\RR_{\t < 0}$. Then, using a standard center-of-mass construction (see, for instance, \cite[Page 1268]{bamler2023compactness}), we can glue these local pieces so that the following result holds.

\begin{lem}\label{lem:smooth2}
The set $\RR_{\t < 0}$, which is open in $Z$, can be realized as a Ricci flow spacetime $(\RR_{\t < 0}, \t, \partial_\t, g^Z_t)$. Moreover, there exists an increasing sequence $U_1 \subset U_2 \subset \ldots \subset \RR_{\t < 0}$ of open subsets with $\bigcup_{i=1}^\infty U_i = \RR_{\t < 0}$. In addition, for sufficiently large $i$, there exist open subsets $V_i \subset M_i \times \III$, time-preserving diffeomorphisms $\phi_i : U_i \to V_i$ and a sequence $\ep_i \to 0$ such that statements (a), (b), and (c) in Theorem \ref{thm:smooth1} hold.
\end{lem}

Next, we extend the Ricci flow spacetime $\RR_{\t < 0}$ to $\RR$ so that Theorem \ref{thm:smooth1} holds. This follows from the next lemma. Indeed, once this lemma is proved, one can find a locally finite cover $\{U_{z_i}\}$ of $\RR$ and then glue them together as in Lemma \ref{lem:smooth2}.

\begin{lem}
For any $z \in \RR_{0}$, there exists an open neighborhood $z \in U_z \subset  \RR$ such that $U_z$ is realized as a Ricci flow spacetime $(U_z, \t, \partial_{\t}, g_t^z)$ defined on a product domain. Moreover, for sufficiently large $i$, there exist open subsets $V_i \subset M_i \times \III$, time-preserving and $\partial_{\t}$-preserving diffeomorphisms $\phi_i : U_z \to V_i$, and a sequence $\ep_i \to 0$ such that statements (a), (b), and (c) in Theorem \ref{thm:smooth1} hold with $U_i$ replaced by $U_z$, and in (c), the point $y$ is required to lie in $U_z$.
\end{lem}

\begin{proof}
For $z \in \RR_{0}$, Definition \ref{def:smoothcv} provides a sequence $z_i^*=(z_i, t_i)\in M_i \times (-(1-2\sigma)T,0]$ converging to $z$ in the Gromov--Hausdorff sense and a constant $r>0$ such that on $B_{g_i(-r^2)}(z_i, r) \times [-r^2, 0]$,
	\begin{equation} \label{smooth001}
r_{\Rm} \ge r.
	\end{equation}
		
We choose $w_i^*:=(z_i, -r^2) \in M_i \times (-(1-2\sigma)T,0)$. After passing to a subsequence, we assume $w_i^* \to w \in \RR_{(-(1-2\sigma)T, 0)}$ in the Gromov--Hausdorff sense. By \eqref{smooth001} and Lemma \ref{lem:smooth2}, we have a product domain:
\begin{align*}
B_{g^Z_{-r^2}}(w, r/2) \times [- r^2, 0) \to \RR_{\t < 0}.
\end{align*}
In addition, under this identification, the curvature of any spacetime point in $B_{g^Z_{-r^2}}(w, r/2) \times [-r^2,0)$ is bounded by $4r^{-2}$. Therefore, the metric $g^Z_t$, restricted to $B_{g^Z_{-r^2}}(w, r/2) \times [-r^2,0)$, can be extended to a Ricci flow spacetime $B_{g^Z_{-r^2}}(w, r/2) \times [-r^2,0]$.

For $x \in B_{g^Z_{-r^2}}(w,r/2)$ and $t \in[-r^2,0)$, let $x^t$ denote the flow of $x$ along $\partial_\t$, normalized by $x^{-r^2}=x$.

\textbf{Claim 1.} $w^t$ converges to $z$ in $d_Z$ as $t \to 0$. Moreover, for any $x \in B_{g^Z_{-r^2}}(w, r/2)$, $x^t$ converges to a point in $Z_0$ as $t \to 0$.

\emph{Proof of Claim 1.} Given $s, t \in [-r^2,0)$ with $s \le t$, it follows from Lemma \ref{lem:smooth2} that $\phi_i^{-1}(z_i, t) \to w^t$ and $\phi_i^{-1}(z_i, s) \to w^s$ as $i \to \infty$. By Proposition \ref{HncenterScal} and Lemma \ref{lem:Hcenterdis}, we have
\begin{align*}
d_i^*((z_i, t), (z_i, s)) \le C(n, Y , \sigma) \sqrt{t-s} \quad \text{and} \quad  d_i^*((z_i, t), z_i^*) \le C(n, Y , \sigma) \sqrt{|t|}.
\end{align*}
By the convergence \eqref{eq:conv00}, we conclude that
\begin{align*}
d_Z(w^t, w^s) \le C(n, Y , \sigma) \sqrt{t-s} \quad \text{and} \quad d_Z(w^t, z) \le C(n, Y , \sigma) \sqrt{|t|}.
\end{align*}
Thus, $w^t \to z$ in $d_Z$ as $t \to 0$. The other conclusion can be proved similarly.

Next, we define a map $\psi: B:=B_{g^Z_{-r^2}}(w, r/4) \to Z_0$ so that $\psi(x)=\lim_{t \to 0} x^t$ for any $x \in B$.

\textbf{Claim 2.} $\psi$ is injective. Moreover, $\psi(B)$ contains an open neighborhood of $z$ in $Z_0$.

\emph{Proof of Claim 2.} For any $a, b \in B$ with $a \ne b$ and $t \in [-r^2,0)$, we can find $a_i^*=(a_i, t), b_i^*=(b_i, t)$ with $a_i, b_i \in  B_{g_i(-r^2)}(z_i, r/2)$ such that $(a_i, -r^2) \to a$ and $(b_i, -r^2) \to b$ in the Gromov--Hausdorff sense. By smooth convergence and distance distortion, we have for large $i$,
\begin{align*}
d_{g_i(t)}(a_i, b_i) \ge c_0 d_{g^Z_{-  r^2}}(a, b)
\end{align*}
for a constant $c_0>0$. Since $r_{\Rm}(a_i^*) \ge r$, it follows from Proposition \ref{equivalenceofballs} that
\begin{align*}
d_i^*(a_i^*, b_i^*) \ge c_1 d_{g^Z_{-  r^2}}(a, b)
\end{align*}
for a constant $c_1>0$. Passing to the limit, we obtain $d_Z(a^t, b^t) \ge c_1 d_{g^Z_{-  r^2}}(a, b)$ and hence $\psi(a) \ne \psi(b)$.

Suppose $y \in Z_0$ with $d_Z(z, y) \le \ep r$ for a small constant $\ep$ to be determined later. We can choose $y_i^*=(y_i, 0) \in M_i \times \III$ so that $y_i^* \to y$. For sufficiently large $i$, we know $d_i^*(y_i^*, z_i^*) \le 2\ep r$. If $\ep \le \ep(n)$, we conclude that $y_i \in B_{g_i(-r^2)}(z_i, r/10)$ by the definition of the $d^*$-distance and \cite[Proposition 9.16 (b)]{bamler2023compactness}. After passing to a subsequence, we assume that $(y_i, -r^2)$ converges to $a \in B$. Then, it is clear that $y=\psi(a)$.

By Claim 1 and Claim 2, we obtain an embedding $B_{g^Z_{-r^2}}(w, c_3 r) \times [-r^2,0]$ into $Z$ for a small constant $c_3=c_3(n)>0$. Its image is an open neighborhood of $z$. Moreover, there exists an embedding $\phi_i^z: U_z:=B_{g^Z_{-r^2}}(w, c_3 r) \times [-r^2,0] \to B_{g^i_{-r^2}}(z_i, r) \times [-r^2, 0]$. It is then straightforward to verify that statements (a), (b), and (c) in Theorem \ref{thm:smooth1} hold.
\end{proof}

Next, we prove that the map $\iota_z$ obtained in Theorem \ref{thm:idenproof} is an isometric embedding of Ricci flow spacetimes.

\begin{prop}\label{prop:embed2}
For any $z \in Z$, let $\XX^z$ denote the metric flow associated with $z$, and let $(\RR^z, \t^z, \partial_{\t^z}, g^z)$ be the Ricci flow spacetime of $\XX^z$ obtained in Theorem \ref{Fconvergence} (1). Then, the time-preserving map 
\begin{align*}
\iota_z:(\RR^z_\III, \t^z) \to (\RR, \t),
\end{align*}
which is the restriction of the map $\iota_z$ from Theorem \ref{thm:idenproof} to $\RR^z_\III$, satisfies the following properties: 
	\begin{enumerate}[label=\textnormal{(\roman{*})}]
	\item $(\iota_z)_*(\partial_{\t^z})=\partial_\t$ and $(\iota_z)^* g^Z=g^z$.
	\item For any $x, y \in \RR^z_\III$, $K_Z(\iota_z(x); \iota_z(y))=K^z(x;y)$, where $K^z$ denotes the heat kernel of $\RR^z$.
	\item If $z \in \RR$, then $K_Z(z; \iota_z(y))=K^z(z;y)$ for any $y \in \RR^z_\III$.
	\end{enumerate}	
\end{prop}

\begin{proof}
Given $z \in Z$, we choose a sequence $z_i^* \in M_i \times \III$ converging to $z$ in the Gromov--Hausdorff sense and a correspondence $\CF$ such that 
\begin{equation*}
	(\XX^i, (\nu_{z_i^*;t})_{t \in [-T, \t(z_i^*)]}) \xrightarrow[i \to \infty]{\quad \IF, \CF, J \quad} (\XX^z, (\nu_{z;t})_{t \in  [-T, \t(z)]}).
\end{equation*}

By Theorem \ref{Fconvergence} (3), there exists an increasing sequence $U^z_1 \subset U^z_2 \subset \ldots \subset \RR^z_\III$ of open subsets with $\bigcup_{i=1}^\infty U^z_i = \RR^z_\III$. In addition, for sufficiently large $i$, there exist open subsets $V^z_i \subset M_i \times \III$, time-preserving diffeomorphisms $\phi^z_i : U^z_i \to V^z_i$ and a sequence $\ep_i \to 0$ such that all statements in Theorem \ref{Fconvergence} (3) hold.

On the other hand, there exists an increasing sequence $U_1 \subset U_2 \subset \ldots \subset \RR$ of open subsets with $\bigcup_{i=1}^\infty U_i = \RR$. In addition, for sufficiently large $i$, there exist open subsets $V_i \subset M_i \times \III$, time-preserving diffeomorphisms $\phi_i : U_i \to V_i$ and a sequence $\ep'_i \to 0$ such that statements (a), (b), and (c) in Theorem \ref{thm:smooth1} hold, with $\ep_i$ replaced by $\ep_i'$.

(i): For any $w \in \RR^z_\III$, we choose a sequence $w_i^* \in M_i \times \III$ so that $w_i^* \xrightarrow[i \to \infty]{\quad \CF,J \quad} w$. By Theorem \ref{Fconvergence} (3)(b), we conclude that 
\begin{align*}
(\phi_i^z)^{-1}(w_i^*) \to w \quad \text{in} \quad \RR^z_\III.
\end{align*}

By Lemma \ref{lem:key1}, $d_i^*(z_i^*, w_i^*)$ is uniformly bounded. Moreover, by the smooth convergence, we know that $r_{\Rm}(w_i^*) \ge r>0$ for a constant $r$. Now, after passing to a subsequence, we assume $w_i^* \to w' \in \RR$ in the Gromov--Hausdorff sense. Then by Theorem \ref{thm:smooth1} (c), we know 
\begin{align*}
(\phi_i)^{-1}(w_i^*) \to w' \quad \text{in} \quad \RR.
\end{align*}

Then, we can find small open neighborhoods $U_{w} \subset \RR^z_\III$ and $U_{w'} \subset \RR$ around $w$ and $w'$, respectively, such that the map $\psi_i:=\phi_i^{-1} \circ \phi_i^z: U_{w} \to U_{w'}$, for sufficiently large $i$, is well defined and a diffeomorphism. Moreover, by Theorem \ref{Fconvergence} (3)(a) and Theorem \ref{thm:smooth1} (a), we may assume, by shrinking the open neighborhoods and taking a further subsequence, that $\psi_i$ converges smoothly to a diffeomorphism $\psi_{\infty}$. In addition, it follows from Theorem \ref{Fconvergence} (3)(a) and Theorem \ref{thm:smooth1} (a) again that
\begin{align*}
(\psi_{\infty})_*(\partial_{\t^z})=\partial_\t \quad \text{and} \quad (\psi_{\infty})^* g^Z=g^z.
\end{align*}

Note that by Theorem \ref{thm:idenproof}, $\iota_z$ agrees with $\psi_{\infty}$ on $U_{w}$. Thus, we have proved that on $\RR^z_\III$,
\begin{align*}
(\iota_z)_*(\partial_{\t^z})=\partial_\t \quad \text{and} \quad (\iota_z)^* g^Z=g^z.
\end{align*}

(ii): For any $x, y \in \RR^z_\III$ with $\t^z(x) > \t^z(y)$, we find an open set $U \subset \RR^z_\III$ containing $x$ and $y$ whose closure $\bar U$ is compact in $\RR^z_\III$. Then for sufficiently large $i$, we have $\bar U \subset U_i^z$. By smooth convergence, we conclude that
\begin{align*}
r_{\Rm}(w^*) \ge r>0
\end{align*}
for a constant $r$ and any $w^* \in D_i:=\phi_i^z(\bar U)$. Since $\bar U$ is compact, Proposition \ref{HncenterScal} (i), Proposition \ref{distancefunction1} (1), and Lemma \ref{lem:Hcenterdis} show that
\begin{align*}
\sup_{w \in \bar U} d_z^*(z, w) <\infty.
\end{align*}

By Lemma \ref{lem:key1}, we conclude that 
\begin{align*}
\sup_{w^* \in D_i} d_i^*(z_i^*, w^*) \le C
\end{align*}
for a constant $C$. After passing to a subsequence, we assume $D_i \to D \subset \RR$ in the Gromov--Hausdorff sense. Arguing as before, $\iota_z \vert_{\bar U}$ is the smooth limit of $\psi_i$. Thus, it follows from Theorem \ref{Fconvergence} (3)(a) and Theorem \ref{thm:smooth1} (b) that
\begin{align*}
K_Z(\iota_z(x); \iota_z(y))=K^z(x; y).
\end{align*}

(iii): If $z \in \RR$, then it follows from Theorem \ref{Fconvergence} (3)(a) that for any $x \in \RR^z_{[-T, \t(z))}$,
\begin{align*}
K^z(z; x)=\lim_{i \to \infty} K^i(z_i^*;x_i^*),
\end{align*}
where $x_i^*=\phi_i^z(x)$. By Theorem \ref{thm:smooth1} (b), we conclude that
\begin{align*}
K^z(z; x)=K_Z(z; \iota_z(x)).
\end{align*}

This completes the proof.
\end{proof}

As a corollary of Proposition \ref{prop:embed2}, we prove:

\begin{cor}\label{cor:comple1}
$\RR_{\III^-}$ is dense in $Z_{\III^-}$ with respect to $d_Z$.
	\end{cor}

\begin{proof}
For any $z \in Z_{\III^-}$, we consider its associated metric flow $\XX^z$. For any $s<\t(z)$, we choose an $H_n$-center $w \in \RR^z_s$. By Lemma \ref{lem:Hcenterdis}, we have
\begin{align*}
d_z^*(z, w) \le  \sqrt{H_n(\t(z)-s)}.
\end{align*}
Then, it follows from Theorem \ref{thm:idenproof} that
\begin{align*}
d_Z(z, \iota_z(w)) \le  \sqrt{H_n(\t(z)-s)}.
\end{align*}
By choosing $s=s_i \nearrow \t(z)$, the conclusion follows.
\end{proof}

Next, we define conjugate heat kernel measures on $\RR$.

\begin{defn}[Conjugate heat kernel measures on $\RR$]
For any $x \in \RR$, we define the \textbf{conjugate heat kernel measure} $\nu_{x;s}$ based at $x$, to be the Borel measure on $(\RR_s, g^Z_s)$ given by
\begin{align*}
\mathrm{d}\nu_{x;s}:=K_Z(x; \cdot) \, \mathrm{d}V_{g^Z_s}
\end{align*}		
for any $s<\t(x)$, and set $\nu_{x;\t(x)}=\delta_x$.
\end{defn}

\begin{lem}\label{lem:pre1}
For any $x \in \RR$, $\nu_{x;s}$ is a probability measure on $(\RR_s, g^Z_s)$.
	\end{lem}

\begin{proof}
Given $x \in \RR$ and $s<\t(x)$, we choose $x_i^* \in M_i \times \III$ so that $x_i^* \to x$ in the Gromov--Hausdorff sense. Then, it follows from Theorem \ref{thm:smooth1} (b) that
	\begin{align*}
		\nu_{x;s}(\RR_s)\le \liminf_{i \to \infty} \nu_{x_i^*;s}(M_i)=1.
	\end{align*}

On the other hand, it follows from Proposition \ref{prop:embed2} (iii) that $K_Z(x;\iota_x(y))=K^x(x;y)$ for any $y \in \RR^x_\III$. Thus, we have
	\begin{align*}
		\nu_{x;s}(\RR_s) \ge \nu_{x;s}(\iota_x(\RR^x_s))=1.
	\end{align*}

Combining these inequalities, we conclude $\nu_{x;s}(\RR_s)=1$.
\end{proof}

The proof of Lemma \ref{lem:pre1} also gives the following corollary.

\begin{cor}\label{cor:pre1}
For any $x \in \RR$ and $s<\t(x)$,
	\begin{align*}
\{y \in \RR_s \mid K_Z(x; y)>0\}=\iota_x( \RR^x_s).
	\end{align*}
	In particular, $\nu_{x;s}\lc \RR_s \setminus \iota_x( \RR^x_s) \rc=0$.
	\end{cor}

\begin{lem}\label{lem:pre2}
For any $x,y \in \RR$ and $s < \min\{\t(x), \t(y)\}$, either $\iota_x(\RR^x_s)=\iota_y(\RR^y_s)$ or $\iota_x(\RR^x_s)\cap \iota_y(\RR^y_s)=\emptyset$. Moreover, if the latter happens, we have $\iota_x(\RR^x_t)\cap \iota_y(\RR^y_t)=\emptyset$ for any $t \in [s, \min\{\t(x), \t(y)\})$.
\end{lem}


\begin{proof}
If $\iota_x(\RR^x_s)\ne \iota_y(\RR^y_s)$ and $\iota_x(\RR^x_s)\cap \iota_y(\RR^y_s)\ne \emptyset$, then we may assume, without loss of generality, that there exists $w \in \iota_y(\RR^y_s)$ which lies in the boundary of $\iota_x(\RR^x_s)$. By Corollary \ref{cor:pre1}, $K_Z(x;w)=0$. Since $K_Z(x;\cdot)$ satisfies the conjugate heat equation, it follows from the strong maximum principle that on $B_{g^Z_s}(w, \delta) \times (s, s+\delta] \subset \RR$, $K_Z(x;\cdot)$ vanishes. However, there exists a point $w' \in B_{g^Z_s}(w, \delta) \times (s, s+\delta] \bigcap \iota_x(\RR^x_{(s,s+\delta]})$ by our assumption. This contradicts Corollary \ref{cor:pre1}.

If $\iota_x(\RR^x_s)\cap \iota_y(\RR^y_s)=\emptyset$ and $\iota_x(\RR^x_t)=\iota_y(\RR^y_t)$ for some $t \in [s, \min\{\t(x), \t(y)\})$, then, for any $w \in \iota_x(\RR^x_t)$, it follows from Proposition \ref{prop:embed2} (ii) that $\nu_{w;s}(\iota_x( \RR^x_s))=1$ and $\nu_{w;s}(\iota_y( \RR^y_s))=1$, which yields a contradiction.

Therefore, the proof is complete.
\end{proof}

\begin{cor}\label{cor:pre2}
For any $x \in \RR$ and $s<\t(x)$, $\iota_x( \RR^x_s)$ is a connected component of $\RR_s$.
\end{cor}

\begin{proof}
Arguing as in the proof of Lemma \ref{lem:smooth1}, one can prove that any $w \in \RR_s$ is contained in $\iota_y(\RR^y_s)$ for some $y \in \RR$ with $\t(y)>s$. Thus, we conclude that
	\begin{align*}
\RR_s=\bigcup_{y \in \RR_{(s,0]}} \iota_y(\RR^y_s).
	\end{align*}
Consequently, the conclusion follows from Lemma \ref{lem:pre2}.
\end{proof}

\begin{lem}\label{lem:pre3}
For any $x \in \RR$, $\nu_{x;s}$ satisfies the reproduction formula. That is, for any $s<t<\t(x)$,
	\begin{align*}
		\nu_{x;s}=\int_{\RR_t} \nu_{y;s} \, \mathrm{d}\nu_{x;t}(y).
	\end{align*}
	\end{lem}

\begin{proof}
Since $\nu_{x;s}$ satisfies the reproduction formula in $\XX^x$, we conclude that for any Borel set $S \subset \RR_s^x$,
	\begin{equation}\label{eq:repr1}
		\nu_{x;s}(S)=\int_{\iota_x(\RR^x_t)} \nu_{y;s}(\iota_x(S)) \, \mathrm{d}\nu_{x;t}(y).
	\end{equation}

On the other hand, for any $y \in \iota_x(\RR^x_t)$, it follows from Proposition \ref{prop:embed2} (ii) that $\nu_{y;s}(\iota_x( \RR^x_s))=1$. Therefore, the conclusion follows from \eqref{eq:repr1} and Corollary \ref{cor:pre1}.
\end{proof}

Next, we define the conjugate heat kernel measure at any $z \in Z \setminus \RR$. Let $\XX^z$ be a metric flow associated with $z$. For any $s \le \t(z)$, the conjugate heat kernel measure $\nu_{z;s}$ on $\XX^z_s$, when restricted to $\RR^z_s$, can be regarded via the map $\iota_z$ as a probability measure on $\RR_s$. More precisely,

\begin{defn}[Conjugate heat kernel measures on $Z \setminus \RR$] \label{def:chks}
For any $z \in Z \setminus \RR$, we define the \textbf{conjugate heat kernel measure} $\nu_{z;s}$ based at $z$, to be the Borel measure on $(\RR_s, g^Z_s)$ given by
\begin{align*}
\mathrm{d}\nu_{z;s}=K_Z(z; \cdot) \, \mathrm{d}V_{g^Z_s},
\end{align*}		
for any $s<\t(z)$, and set $\nu_{z;\t(z)}=\delta_z$. Here, $K_Z(z; \cdot):=K^z(z;\iota_z^{-1}(\cdot))$ on $\iota_z(\RR^z_s)$ and zero on $\RR_s \setminus \iota_z(\RR^z_s)$.	
\end{defn}

We will prove in Lemma \ref{lem:well-define} that $\nu_{z;s}$ is independent of the choice of the associated $\XX^z$.

It is clear that the conjugate heat kernel measure $\nu_{z;s}$ is a probability measure on $(\RR_s, g^Z_s)$ and
\begin{align*}
\nu_{z;s}(\iota_z(\RR^z_s))=1.
\end{align*}		

 Moreover, it satisfies the reproduction formula, as in Lemma \ref{lem:pre3}: for any $s<t<\t(z)$,
	\begin{equation}\label{eq:generalrep}
		\nu_{z;s}=\int_{\RR_t} \nu_{y;s} \, \mathrm{d}\nu_{z;t}(y).
	\end{equation}

In addition, by the same proof as that of Lemma \ref{lem:pre2}, we have

\begin{lem}\label{lem:pre4}
For any $x,y \in Z$ and $s < \min\{\t(x), \t(y)\}$, either $\iota_x(\RR^x_s)=\iota_y(\RR^y_s)$ or $\iota_x(\RR^x_s)\cap \iota_y(\RR^y_s)=\emptyset$. Moreover, if the latter happens, we have $\iota_x(\RR^x_t)\cap \iota_y(\RR^y_t)=\emptyset$ for any $t \in [s, \min\{\t(x), \t(y)\})$. In particular, $\iota_x(\RR^x_s)$ is a connected component of $\RR_s$ for any $x \in Z$.
	\end{lem}

Conversely, we have

\begin{prop}\label{prop:pre5}
For any $x,y \in Z$, if $\max\{\t(x), \t(y)\}-d^2_Z(x, y) > -(1-2\sigma)T$, then $\iota_x(\RR^x_t)=\iota_y(\RR^y_t)$ for any $t \in [-(1-2\sigma)T, \max\{\t(x), \t(y)\}-d^2_Z(x, y))$. Moreover, we have
\begin{align} \label{eq:lowlimit}
\lim_{t \nearrow \max\{\t(x), \t(y)\}-d^2_Z(x, y)} d_{W_1}^{\XX^x_t} \lc \nu_{x;t}, (\iota^{-1}_x)_* (\nu_{y;t}) \rc \le d_Z(x, y),
\end{align}	
where we regard $(\iota^{-1}_x)_* (\nu_{y;t})$ as a probability measure on $\XX^x_t$ by extension from $\RR^x_t$.
\end{prop}

\begin{proof}
We set $r=d_Z(x, y)$ and $t_0=\max\{\t(x), \t(y)\}-r^2$. Then we choose $x^*_i, y^*_i \in M_i \times \III$ so that $x^*_i \to x$ and $y^*_i \to y$ in the Gromov--Hausdorff sense. In particular,
\begin{align*}
\lim_{i \to \infty} d_i^*(x_i^*,y_i^*)=r.
\end{align*}		
Then, for sufficiently large $i$, equation \eqref{eq:dstar-equality1} gives
\begin{align}\label{eq:pre5ax2}
d_{W_1}^{t_i}(\nu_{x_i^*;t_i},\nu_{y_i^*;t_i})\leq r_i,
\end{align}
where $r_i:=d_i^*(x_i^*,y_i^*)$ and $t_i=\max\{\t(x_i^*),\t(y_i^*)\}-r_i^2$. The condition $t_0>-(1-2\sigma)T$ ensures $t_i>-T$ for all sufficiently large $i$.

We take $-(1-2\sigma)T < s' <s <t_0$ so that $\XX^x$ is continuous at time $s'$. Then, by \eqref{eq:pre5ax2},
\begin{align}\label{eq:pre5a2}
d_{W_1}^{s} (\nu_{x_i^*;s},\nu_{y_i^*;s}) \le 2r
\end{align}	
for large $i$. Next, we set $a_i^*, b_i^* \in \XX^i_{s}$ to be $H_n$-centers of $x_i^*$ and $y_i^*$, respectively. Then, by \eqref{eq:pre5a2}, we have
\begin{align*}
d_{g_i(s)}(a_i^*,b_i^*) \le D_0
\end{align*}	
for a constant $D_0$. Thus, for a sufficiently large constant $D_1>0$ to be determined later, we have
\begin{align*}
\quad \nu_{x_i^*;s}\lc B_{g_i(s)}(a_i^*,D_1) \rc \ge \frac 1 2.
\end{align*}	

We fix a correspondence $\CF$ as in Definition \ref{defncorrespondence}. By our assumption, the $\F$-convergence is uniform at $s'$. We choose a compact set $K_{\ep} \subset \XX^x_{s'}$ such that $\nu_{x;s'}(K_{\ep}) \ge 1-\ep$. Then, we define
\begin{align*}
K_{i,\ep}:=(\varphi_{s'}^i)^{-1} \lc B_{A_{s'}}(\varphi_{s'}^\infty(K_{\ep}), \ep) \rc
\end{align*}	
and hence for sufficiently large $i$,
\begin{align}\label{eq:pre5a3-prime}
\nu_{x_i^*;s'}(K_{i,\ep}) \ge 1-2\ep.
\end{align}	
By the reproduction formula and Definition \ref{defnmetricflow} (6), we have for any $w_i^* \in B_{g_i(s)}(a_i^*,D_1)$,
\begin{align}
\nu_{x_i^*;s'}(K_{i,\ep})=&\int_{\XX^i_{s}} \nu_{z^*;s'}(K_{i,\ep}) \, \mathrm{d}\nu_{x_i^*;s}(z^*) \notag \\
\le {}& \nu_{x_i^*;s} \lc \XX^i_s \setminus B_{g_i(s)}(w_i^*,2D_1) \rc \notag\\
&+\Phi \lc \Phi^{-1}(\nu_{w_i^*;s'}(K_{i,\ep}))+2(s-s')^{-\frac12}D_1 \rc
\nu_{x_i^*;s} \lc B_{g_i(s)}(w_i^*,2D_1) \rc. \label{eq:pre5a4}
\end{align}	
Since $\nu_{x_i^*;s} \lc B_{g_i(s)}(w_i^*,2D_1) \rc \ge \nu_{x_i^*;s} \lc B_{g_i(s)}(a_i^*,D_1) \rc \ge 1/2$, we obtain from \eqref{eq:pre5a3-prime} and \eqref{eq:pre5a4} that
\begin{align*}
\Phi \lc \Phi^{-1}(\nu_{w_i^*;s'}(K_{i,\ep}))+2(s-s')^{-\frac 1 2} D_1  \rc \ge 1-4\ep.
\end{align*}	
Thus, for sufficiently large $i$,
\begin{align*}
\nu_{w_i^*;s'}(K_{i,\ep}) \ge 1-\Psi(\ep \vert s-s', D_1),
\end{align*}	
where $\Psi(\ep \vert s-s', D_1)$ denotes a function that goes to $0$ as $\ep \to 0$, while the other arguments are fixed. By the reproduction formula again, we have
\begin{align*}
\nu_{y_i^*;s'}(K_{i,\ep})=\int_{\XX^i_{s}} \nu_{z^*;s'}(K_{i,\ep}) \, \mathrm{d}\nu_{y_i^*;s}(z^*) \ge \lc 1-\Psi(\ep \vert s-s', D_1) \rc \nu_{y_i^*;s} \lc B_{g_i(s)}(a_i^*,D_1) \rc.
\end{align*}	
Thus, we can first choose a sufficiently large $D_1$ so that $\nu_{y_i^*;s} \lc B_{g_i(s)}(a_i^*,D_1) \rc$ is almost $1$ and then choose $\ep$ to be small. In other words, we have shown that the sequence $(\varphi_{s'}^i)_* \nu_{y_i^*;s'}$ is tight. After passing to a subsequence, this sequence converges weakly to a probability measure $\mu^{\infty}$ on $\XX^x_{s'}$. Moreover, by the definition of $K_{i, \ep}$, we conclude that $\mathrm{supp}\,\mu^{\infty} \subset \varphi_{s'}^\infty(\XX^x_{s'})$. By Proposition \ref{takelimitVarW1}, $(\varphi_{s'}^i)_* \nu_{y_i^*;s'}$ converges to $\mu^{\infty}$ in the $d_{W_1}$-sense.

Now, we regard $\mu^{\infty}$ as a probability measure on $\XX^x_{s'}$ and let $\mu_{t}$ be the conjugate heat flow on $\XX^x$ for $t \le s'$ with $\mu_{s'}=\mu^{\infty}$. By \cite[Theorem 6.13]{bamler2023compactness}, we conclude that
\begin{align*}
(\nu_{y_i^*;t})_{t \in [-(1-2\sigma)T, s']}\xrightarrow[i\to\infty]{\quad \CF,J\quad} (\mu_t)_{t \in [-(1-2\sigma)T, s']}.
\end{align*}	
Therefore, it follows from \cite[Theorem 9.21(f)]{bamler2023compactness} and our construction of $\iota_x$ that 
\begin{align*}
(\iota_x)_* (\mu_t \vert_{\RR^x_t})=K_Z(y; \cdot)  \, \mathrm{d}V_{g^Z_t}.
\end{align*}	
for any $t \in [-(1-2\sigma)T, s']$.

Since $s'$ can be chosen as close to $t_0$ as desired, we conclude that $\iota_x(\RR^x_t)=\iota_y(\RR^y_t)$ for any $t \in [-(1-2\sigma)T, t_0)$. 

Moreover, by the $d_{W_1}$-convergence at $s'$, we have
\begin{align*}
\lim_{i \to \infty} d_{W_1}^{s'} (\nu_{x_i^*;s'},\nu_{y_i^*;s'})=d_{W_1}^{\XX^x_{s'}} (\nu_{x;s'}, \mu^{\infty})=d_{W_1}^{\XX^x_{s'}}(\nu_{x;s'}, (\iota^{-1}_x)_* (\nu_{y;s'})).
\end{align*}	
By monotonicity, the asserted estimate follows.
\end{proof}

The proof of Proposition \ref{prop:pre5} also yields the following result:

\begin{lem}\label{lem:pre5}
Let $x,y \in Z$, and suppose that $x_i^*, y_i^* \in M_i \times \III$ converge to $x$ and $y$, respectively, in the Gromov--Hausdorff sense. If there exists $t_0 \in (-(1-2\sigma)T, 0)$ and a constant $D$ such that
\begin{align*}
d_{W_1}^{t_0}(\nu_{x_i^*;t_0}, \nu_{y_i^*;t_0}) \le D.
\end{align*}	
Then $\iota_x(\RR^x_t)=\iota_y(\RR^y_t)$ for any $t \in [-(1-2\sigma)T, t_0)$. Moreover, for any $t_1 \in [-(1-2\sigma)T, t_0)$ such that $\XX^x$ is continuous at $t_1$, we have
\begin{align*}
d_{W_1}^{\XX^x_{t_1}} \lc \nu_{x;t_1}, (\iota^{-1}_x)_* (\nu_{y;t_1}) \rc =\lim_{i \to \infty} d_{W_1}^{t_1}(\nu_{x_i^*;t_1}, \nu_{y_i^*;t_1}).
\end{align*}
\end{lem}

As a corollary of Proposition \ref{prop:pre5}, we have

\begin{cor} \label{cor:continuheat}
For any $x, y \in \RR$ with $r=d_Z(x, y)$ and $\max\{\t(x), \t(y)\}-r^2 >-(1-2\sigma)T$, and for any $-(1-2\sigma)T \le s<s'<\max\{\t(x), \t(y)\}-d^2_Z(x, y)$ and any $w \in \RR_s$, we have
\begin{align*}
\abs{K_Z(x; w)-K_Z(y; w)} \le C(n, Y, s'-s) r.
\end{align*}	
\end{cor}

\begin{proof}
By Proposition \ref{prop:pre5}, the conclusion follows immediately if $w \notin \iota_x(\RR^x_s)$, since $K_Z(x; w)=K_Z(y; w)=0$ in this case.

Next, we assume $w \in \iota_x(\RR^x_s)=\iota_y(\RR^y_s)$. It follows from Theorem \ref{gradientheatkernel} and the smooth convergence in Theorem \ref{thm:smooth1} (b) that for any $z \in \iota_x(\RR^x_{s'})$ and $w \in \iota_x(\RR^x_s)$,
	\begin{align*}
		|\na_z K_Z(z;w)|&\leq C(n)K_Z(z;w)(s'-s)^{-1/2}\sqrt{\log\lc\frac{C(n)\mathrm{exp}(-\NN_{z}(s'-s))}{(s'-s)^{\frac{n}{2}}K_Z(z;w)}\rc}\\
		&\leq  C(n,s'-s)K_Z(z;w)\sqrt{C(n,Y,s'-s)-\log K_Z(z;w)} \leq C(n,Y,s'-s),
	\end{align*}
	where we used Theorem \ref{heatkernelupperbdgeneral} (ii) for the last inequality. Thus, by the definition of $d_{W_1}$-distance,
	\begin{align*}
		|K_Z(x;w)-K_Z(y;w)|&=\abs{\int_{\RR_{s'}} K_Z(\cdot;w)\, \mathrm{d}\nu_{x;s'}-\int_{\RR_{s'}} K_Z(\cdot;w)\, \mathrm{d}\nu_{y;s'}}\\
		&\leq C(n,Y,s'-s) d_{W_1}^{\XX^x_{s'}} \lc \nu_{x;s'}, (\iota^{-1}_x)_* (\nu_{y;s'}) \rc \le C(n,Y,s'-s) r,
	\end{align*}
where we used Proposition \ref{prop:pre5} for the last inequality.
\end{proof}

We next prove:

\begin{lem}\label{lem:well-define}
For any $z \in Z \setminus \RR$, the conjugate heat kernel measure $\nu_{z;s}$ defined in Definition \ref{def:chks} is independent of the associated metric flow $\XX^z$.
	\end{lem}

\begin{proof}
We only need to prove that the conjugate heat kernel $K_Z(z;\cdot)$ is independent of $\XX^z$ for $z \in Z_{\III^-}$. We claim that
	\begin{align}\label{eq:kernellimit}
K_Z(z;\cdot)=\lim_{i \to \infty} K_Z(x_i;\cdot)
	\end{align}
where $x_i \in \RR$ converge to $z$ in $d_Z$. Indeed, it is clear from Corollary \ref{cor:continuheat} that the limit in \eqref{eq:kernellimit} exists and is independent of the choice of $x_i$. 

On the other hand, we consider the associated metric flow $\XX^z$ from which the conjugate heat kernel measure at $z$ is defined. We fix $s<\t^z(z)$ and $w \in \RR^z_{s}$. We choose a sequence $\delta_i \searrow 0$ so that $y_i \in \RR^z_{\t^z(z)-\delta_i^2}$ is an $H_n$-center of $z$. By the same argument as in the proof of Corollary \ref{cor:continuheat}, we obtain
	\begin{align*}
\abs{K^z(z;w)-K^z(y_i;w)} \le C \delta_i
	\end{align*}
for a constant $C$ independent of $i$. Consequently, by Proposition \ref{prop:embed2} (iii), 
	\begin{align*}
K_Z(z;\cdot)=\lim_{i \to \infty} K_Z(\iota_z(y_i);\cdot).
	\end{align*}
This completes the proof.
\end{proof}

Now, we prove the following convergence theorem.

\begin{thm}\label{thm:convextra}
For any $z \in Z$, if $z_i^* \in M_i \times \III$ converge to $z$ in the Gromov--Hausdorff sense, then
\begin{align*}
K^i(z_i^*;\phi_i(\cdot)) \xrightarrow[i \to \infty]{C^{\infty}_\mathrm{loc}} K_Z(z;\cdot) \quad \text{on} \quad \RR_{(-\infty, \t(z))},
\end{align*}	
where $\phi_i$ is from Theorem \ref{thm:smooth1}.
\end{thm}

\begin{proof}
We only need to prove that for any open set $U$ such that $\bar U \subset \RR_{(-\infty, \t(z))}$ is a compact set, we have
\begin{align*}
K^i(z_i^*;\phi_i(\cdot)) \xrightarrow[i \to \infty]{C^{\infty}} K_Z(z;\cdot) \quad \text{on} \quad \bar U.
\end{align*}		

Suppose that the conclusion fails. Then there exist $\delta>0$ and a subsequence, still denoted by $\XX^i$, such that
\begin{align} \label{eq:convexta1}
\Vert K_Z(z;\cdot)-K^i(z_i^*;\phi_i(\cdot)) \Vert_{C^{[\delta^{-1}]} ( \bar U)} \ge \delta.
			\end{align}		
By passing to a further subsequence, there exists a correspondence $\CF$ such that 
\begin{equation*}
	(\XX^i, (\nu_{z_i^*;t})_{t \in [-T, \t(z_i^*)]}) \xrightarrow[i \to \infty]{\quad \IF, \CF, J \quad} (\XX^z, (\nu_{z;t})_{t \in  [-T, \t^z(z)]}),
\end{equation*}
where $\XX^z$ is a metric flow associated with $z$. By Proposition \ref{prop:embed2}, Definition \ref{def:chks} and Lemma \ref{lem:well-define}, we have
\begin{equation*}
K^z(z;\cdot)=K_Z(z;\iota_z(\cdot)).
\end{equation*}
On the other hand, by Theorem \ref{Fconvergence} (3), there exists an increasing sequence $U^z_1 \subset U^z_2 \subset \ldots \subset \RR^z_\III$ of open subsets with $\bigcup_{i=1}^\infty U^z_i = \RR^z_\III$. In addition, for sufficiently large $i$, there exist open subsets $V^z_i \subset M_i \times \III$, time-preserving diffeomorphisms $\phi^z_i : U^z_i \to V^z_i$ and a sequence $\ep_i \to 0$ such that all statements in Theorem \ref{Fconvergence} (3) hold. In particular, we have
\begin{align} \label{eq:convexta2}
\Vert K^z(z;\cdot)-K^i(z_i^*;\phi^z_i(\cdot)) \Vert_{C^{[\ep_i^{-1}]} (U^z_i)} \le \ep_i,
			\end{align}		
for a sequence $\ep_i \to 0$. As in the proof of Proposition \ref{prop:embed2}, the map $\psi_i:=\phi_i^{-1} \circ \phi_i^z$ converges locally and smoothly to $\iota_z$. Note that on $\RR\setminus \RR^z$, we have $K_Z(z;\cdot)=0$, and the limit also holds. Since $K_Z(z;\iota_z(\cdot))=K^z(z;\cdot)$, \eqref{eq:convexta2} contradicts \eqref{eq:convexta1} for all sufficiently large $i$.

This completes the proof.
\end{proof}

Now, we define the isometry between two noncollapsed Ricci flow limit spaces.

\begin{defn}[Isometry]\label{def:iso}
Suppose $(Z, d_Z, z, \t)$ and $(Z', d_{Z'}, z',\t')$ are two pointed noncollapsed Ricci flow limit spaces, with regular parts given by the Ricci flow spacetimes $(\RR, \t, \partial_\t, g^Z)$ and $(\RR', \t', \partial_{\t'}, g^{Z'})$, respectively. 

We say that $(Z, d_Z, z, \t)$ and $(Z', d_{Z'}, z',\t')$ are \textbf{isometric} if there exists a bijection $\phi:Z \to Z'$ satisfying the following conditions:
\begin{enumerate}[label=\textnormal{(\roman{*})}]
\item $\phi(z)=z'$.

\item $\phi$ is time-preserving, that is,  $\t'\circ \phi=\t$.
	
	\item For any $x, y \in Z$, $d_{Z'}(\phi(x), \phi(y))=d_Z(x, y)$.
	
\item \(\phi(\mathcal{R}) = \mathcal{R}'\), and \(\phi\) is an isomorphism of Ricci flow spacetimes between \((\mathcal{R}, \t, \partial_\t, g^Z)\) and \((\mathcal{R}', \t', \partial_{\t'}, g^{Z'})\). That is, the restriction $\phi: \RR \to \RR'$ is a diffeomorphism such that $\phi^* g^{Z'}=g^Z$ and $\phi^* \partial_{\t'}=\partial_\t$.
	\end{enumerate}
\end{defn}

It follows immediately from Theorem \ref{thm:smooth1} and Lemma \ref{lem:well-define} that all noncollapsed Ricci flow limit spaces obtained as pointed Gromov--Hausdorff limits of a given sequence in $\MM(n, Y, T)$ (see Theorem \ref{thm:GHlimit-dstar}) are mutually isometric.

\begin{lem}\label{lem:uniheatequ}
	Let $u\in C^0(\RR_{[a, 0]})\bigcap C^\infty(\RR_{(a,0]})$ be a uniformly bounded function satisfying $\square u=0$ on $\RR_{(a,0]}$. Then for any $z\in \RR_{(a,0]}$, 
	\begin{align}\label{equ:uniheat}
		u(z)=\int_{\RR_{a}}u\,\mathrm{d}\nu_{z;a}.
	\end{align}
\end{lem}
\begin{proof}
Fix $s\in (a,0)$ and $z\in \RR_s$. By the proof of Lemma \ref{lem:smooth1}, there exists a point $y$ with $\t(y)\in (s,0)$ such that $z\in \iota_{y}(\RR^y_s)$. Hence, $\nu_{z;{a}}(\iota_{y}(\RR^y_{a}))=1$. It then follows from \cite[Theorem 15.28(d)]{bamler2020structure} that
	\begin{align}\label{equ:uniheat2}
		u(z)=\int_{\iota_y(\RR^y_{a})}u\,\mathrm{d}\nu_{z;a}=\int_{\RR_{a}}u\,\mathrm{d}\nu_{z;a},
	\end{align} 
which establishes \eqref{equ:uniheat} for all $z\in \RR_{(a, 0)}$.
	
Now consider $z\in \RR_0$. By Corollary \ref{cor:comple1} and its proof, we may find a sequence $z_i\in \RR_{t_i}$ with $t_i\nearrow 0$ such that $z_i \to z$ with respect to $d_Z$. Then \eqref{equ:uniheat} follows from the convergence of the heat kernel measures in \eqref{eq:kernellimit} together with \eqref{equ:uniheat2}. This completes the proof.
\end{proof}

\begin{rem}\label{rem:defheatequS}
In the setting of Lemma \ref{lem:uniheatequ}, we may extend the definition of $u$ to $Z\setminus \RR$ via the integral formula \eqref{equ:uniheat}. By \eqref{eq:kernellimit}, this defines a continuous function on $Z$ that solves $\square u=0$ on $\RR$. Furthermore, combining Lemma \ref{lem:uniheatequ}, \eqref{eq:kernellimit}, and the argument of \cite[Theorem 15.29]{bamler2020structure}, we conclude that the family of conjugate heat kernel measures $(\nu_{z;t})_{z\in Z_{\III^-}, t<\t(z)}$ is uniquely determined by the Ricci flow spacetime $(\RR, \t, \partial_\t, g^Z)$. 

Therefore, for any isometry $\phi$ as in Definition \ref{def:iso}, we also have the following property: for every \(x \in Z\) and every \(s \leq \t(x)\), the pushforward measure satisfies \(\phi_* \nu_{x;s} = \nu_{\phi(x);s}\).
\end{rem}

Next, we define $\Var_{\RR_t}$\index{$\Var_{\RR_t}$} and $d_{W_1}^{\RR_t}$\index{$d_{W_1}^{\RR_t}$} to be the variance and $d_{W_1}$-Wasserstein distance, respectively, with respect to the metric space $(\RR_t, g^Z_t)$. Here, if $x$ and $y$ lie in different connected components of $\RR_t$, we set the distance $d_{g^Z_t}(x, y)=+\infty$.

The following statement follows directly from the fact that any associated metric flow is $H_n$-concentrated.

\begin{prop}\label{prop:004ac}
	For any $z \in Z$, the conjugate heat kernel measure $\nu_{z;s}$ is $H_n$-concentrated; that is, for any $s<\t(z)$,
	\begin{align*}
		\Var_{\RR_s}(\nu_{z;s}) \le H_n (\t(z)-s).
	\end{align*}
\end{prop}

\begin{defn}[Regular $H$-center] \label{def:rhcenter}
For any $z \in Z$, a point $z_1 \in \RR_s$ with $s < \t(z)$ is called a \textbf{regular $H$-center}\index{regular $H$-center} of $z$ for a constant $H>0$ if
 \begin{align*}
	\Var_{\RR_s}(\delta_{z_1},\nu_{z;s})\leq H (\t(z)-s).
	\end{align*}
By Proposition \ref{prop:004ac}, for any $s < \t(z)$, we can always find an $H_n$-center of $z$ in $\RR_s$.
\end{defn}

By the definition of a regular $H$-center, the following statement is immediate.

\begin{lem}\label{lem:004abcx}
Given $x \in Z$, if $z \in \RR_s$ is a regular $H$-center of $x$ with $s<\t(x)$, then
	\begin{align*}
\nu_{x;s}\lc B_{g^Z_s}\lc z, \sqrt{LH(\t(x)-s)} \rc \rc \ge 1-\frac{1}{L}.
	\end{align*}
\end{lem}

\begin{lem}\label{lem:004abc}
	Let $x\in Z$ and let $z\in\RR_s$ be a regular $H$-center of $x$ at a time $s<\t(x)$. Then
	\begin{align*}
		d_Z(x,z) \leq \max\{1,\sqrt H\}\sqrt{\t(x)-s}.
	\end{align*}
\end{lem}
\begin{proof}
This follows directly from the generalization of Lemma \ref{lem:Hcenterdis} to $\XX^x$.
\end{proof}

In general, the Ricci flow spacetime $\RR$ may not be connected. As a corollary of Proposition \ref{prop:pre5}, we prove

\begin{cor} \label{cor:connected}
For any $x, y \in \RR$, if $d_Z(x, y) < \sqrt{\max\{\t(x), \t(y)\}+(1-2\sigma)T}$, then $x$ and $y$ lie in the same connected component of $\RR_{[-(1-2\sigma)T, \max\{\t(x), \t(y)\}]}$. In particular, if $T=+\infty$, then $\RR$ is connected.
\end{cor}

\begin{proof}
By Proposition \ref{prop:pre5}, there exists a time $t \in (-(1-2\sigma)T, \max\{\t(x), \t(y)\}-d^2_Z(x, y))$ such that $\iota_x(\RR^x_t)=\iota_y(\RR^y_t)$. We fix a point $z \in \iota_x(\RR^x_t)=\iota_y(\RR^y_t)$. Since $x \in \RR$, we can choose a time $s$ close to $\t(x)$ and find a regular $H_n$-center $x' \in \RR_s$ such that $x$ and $x'$ can be connected by a curve in $\RR$. On the other hand, since $\RR^x$ is connected, $z$ and $x'$ can be connected by a curve in $\iota_x(\RR^x)$. Therefore, $z$ can be connected to $x$ by a curve in $\RR$. Similarly, $z$ can be connected to $y$ by a curve in $\RR$. It follows that $x$ and $y$ lie in the same connected component of $\RR$.
\end{proof}

Next, we prove monotonicity.

\begin{lem}\label{lem:limitmono}
For any $x, y \in Z$ and $s < \min\{\t(x),\t(y)\}$, the function
\begin{align*}
	s\mapsto d_{W_1}^{\RR_s}(\nu_{x;s},\nu_{y;s})
\end{align*}
is nondecreasing.
\end{lem}

\begin{proof}
Given $s_1 <s_2 \le \min\{\t(x),\t(y)\}$, since $\nu_{x;s}$ (respectively, $\nu_{y;s}$) has full measure on $\iota_x(\RR^x_s)$ (respectively, $\iota_y(\RR^y_s)$), Lemma \ref{lem:pre4} allows us to assume that $\iota_x(\RR^x_s)=\iota_y(\RR^y_s)$ for any $s \in [s_1, s_2]$. For any two probability measures $\mu,\nu$ on $\RR^x_s$, we have $d_{W_1}^{\XX^x_s}(\mu,\nu)=d_{W_1}^{\RR_s}((\iota_x)_*(\mu),(\iota_x)_*(\nu))$.

Suppose $\t(x) \le \t(y)$. Then, we can regard $\nu_{y;s}$ as a conjugate heat flow (see Definition \ref{def:conju}) on $\XX^x$. The desired monotonicity follows from \cite[Proposition 3.16(b)]{bamler2023compactness}.
\end{proof}

Next, we have the following heat kernel estimate, which follows from Theorem \ref{heatkernelupperbdgeneral} (ii) and the same argument as in \cite[Lemma 15.9 (a)]{bamler2020structure}.

\begin{thm}\label{thm:upper1}
For any $x \in Z$ and $s <\t(x)$, we have 
	\begin{align*}
		K_Z(x;y)\leq \frac{C(n, Y, \ep)}{(\t(x)-s)^{n/2}}\exp\lc -\frac{d_{g^Z_s}^2(z,y)}{(4+\epsilon)(\t(x)-s)}\rc
	\end{align*}
	for any $y \in \RR_s$, where $z \in \RR_s$ is any regular $H_n$-center of $x$.
\end{thm}

Using Theorem \ref{thm:upper1}, we can prove, as in Proposition \ref{HncenterScal} (i), the following lemma. Here, for any $x \in \RR$, we denote by $x_t \in \RR_t$ the flow of $x$ with respect to $\partial_\t$.

\begin{lem}\label{lem:center1}
For any $x \in \RR_t$, if $x_s \in \RR_s$ and $|\scal_{g^Z}(x_s)| \le R_0 r^{-2}$ for any $s \in [t-r^2, t]$, then
	\begin{align*}
d_{g^Z_{t-r^2}}(x_{t-r^2}, z) \le C(n, Y, R_0) r,
	\end{align*}
	where $z \in \RR_{t-r^2}$ is any regular $H_n$-center of $x$. In particular, $x_{t-r^2}$ is a regular $H$-center (see Definition \ref{def:rhcenter}) of $x$ for a constant $H=H(n,Y,R_0)>0$.
\end{lem}

Next, we show that there are at most countably many connected components for $\RR_t$.

\begin{prop}\label{prop:connectnumber}
For any $t \in \III$, the time slice $\RR_t$ has at most countably many connected components.
\end{prop}

\begin{proof}
We consider a time $t_0 \in \III$. Suppose $\RR_{t_0}$ has connected components $\{U_{\alpha}\}$ for $\alpha \in \mathcal A$. For each $\alpha \in \mathcal A$, we choose $x_{\alpha} \in U_{\alpha}$ and a small constant $r_{\alpha}>0$ such that
	\begin{align*}
P_{\alpha}:=\{x_t \mid x\in B_{g^Z_{t_0}}(x_{\alpha}, r_{\alpha}),\, t \in [t_0-r_{\alpha}^2, t_0+r_{\alpha}^2] \cap \III\} \subset \RR,
	\end{align*}
and $|\Rm_{g^Z}| \le r_{\alpha}^{-2}$ on $P_{\alpha}$. By the standard distance comparison, there exists $r_{\alpha}'<r_{\alpha}$ such that
	\begin{align*}
P'_{\alpha}:=\{x \mid x\in B_{g^Z_{t}}(x_{\alpha, t}, r'_{\alpha}),\, t \in [t_0-r_{\alpha}^2, t_0+r_{\alpha}^2] \cap \III\} \subset P_{\alpha}.
	\end{align*}
It follows from Proposition \ref{equivalenceofballs} and the smooth convergence in Theorem \ref{thm:smooth1} that there exists $r_{\alpha}''<r_{\alpha}'$ such that
	\begin{align*}
B_Z^*(x_{\alpha}, r_{\alpha}'') \subset P'_{\alpha}.
	\end{align*}
It is clear from the definition that $\{P_{\alpha}\}_{\alpha \in \mathcal A}$ is a pairwise disjoint family. Since $(Z, d_Z)$ is separable, we conclude that the set $\mathcal A$ is countable.
\end{proof}

\begin{defn}[Volume] \label{def:vvvv}
For any set $Q \subset Z$, we define its volume by
	\begin{align*}
|Q|:=\abs{Q \cap \RR}_{g^Z},
	\end{align*}
where $|\cdot|_{g^Z}$ denotes the spacetime volume given by the Ricci flow spacetime $(\RR, \t, \partial_{\t}, g^{Z})$. Moreover, for any $Q \subset Z_t$, we set
	\begin{align*}
|Q|_t:=\abs{Q \cap \RR_t}_{g_t^Z}.
	\end{align*}
\end{defn}

First, we prove the upper volume bound.

\begin{prop}\label{prop:uppervolumebound}
If $T<\infty$, then for any $x \in Z$ and $L>0$, we have
	\begin{align*}
|B^*_Z(x,L \sqrt{T})| \le C(n, \sigma, L) T^{\frac n 2+1}.
	\end{align*}
If $T=+\infty$, we also have for any $L>0$,
\begin{align*}
		|B^*_Z(x,L )| \le C(n) L^{n+2}.
\end{align*}
\end{prop}

\begin{proof}
We prove only the case $T<\infty$; the case $T=+\infty$ is analogous.

Without loss of generality, we may assume $T=1$ and $x \in \RR$. Indeed, there exists a point $x' \in \RR \cap B^*_Z(x, L )$, since otherwise the conclusion holds trivially. Then, we can consider the ball $B^*_Z(x', 2L )$ since $B^*_Z(x, L ) \subset B^*_Z(x', 2L )$.
 
For any spacetime compact set $K \subset B^*_Z(x,L )\cap \RR$ containing $x$, it follows from Theorem \ref{thm:smooth1} that for sufficiently large $i$, $K \subset U_i$. Moreover, if we set $x_i^*=\phi_i(x) \in M_i \times \III$, then it follows from Theorem \ref{thm:smooth1} (c) that 
\begin{align*}
\phi_i(K) \subset B^*_i(x_i^*,2L),
\end{align*}
where $B^*_i(x_i^*,2L)$ denotes the ball with respect to $d_i^*$ on $M_i \times \III$. By Proposition \ref{uppervolumebd}, we conclude that
\begin{align*}
\abs{\phi_i(K)}_{g_i} \le \abs{B^*_i(x_i^*,2L)}_{g_i} \le C(n, \sigma, L).
\end{align*}
Thus, from the smooth convergence, we have
\begin{align*}
|K| \le C(n, \sigma, L).
\end{align*}
By approximation, we conclude
\begin{align*}
\abs{B^*_Z(x,L )} \le C(n, \sigma, L).
\end{align*}
\end{proof}

Next, we prove local volume bounds.

\begin{prop}\label{prop:volumebound}
For any $x \in Z$ and $r>0$ with $\t(x)-r^2 \in \III^-$, we have
	\begin{align*}
0<c(n,Y, \sigma) r^{n+2} \le 	|B^*_Z(x ,r)|\le C(n, \sigma) r^{n+2}.
	\end{align*}
\end{prop}

\begin{proof}
Given $x \in Z$, we consider the associated metric flow $\XX^x$. If $z\in \XX^x_{\t(x)-s^2}$ is a regular $H_n$-center of $x$, then by Lemma \ref{lem:004abcx} and Theorem \ref{thm:upper1}, we have 
\begin{align*}
	|B_{g^x_{\t(x)-s^2}}(z,\sqrt{2H_n}s)|\geq c(n,Y,\sigma) s^n>0.
\end{align*}
 As in Proposition \ref{propdistance3}, that
	\begin{align*}
	|B^*_{\XX^x}(x ,r) \cap \RR^x_{[\t(x)-c_0 r^2, \t(x)-c_1 r^2]}| \ge c(n,Y, \sigma) r^{n+2}>0.
	\end{align*}
	for positive constants $c_0=c_0(n)$ and $c_1=c_1(n)$. Thus, through $\iota_x$, we obtain
	\begin{align*}
	|B^*_{Z}(x ,r) \cap \RR_{[\t(x)-c_0 r^2, \t(x)-c_1 r^2]}| \ge c(n,Y, \sigma) r^{n+2},
	\end{align*}
which implies the lower bound. 

The upper bound follows from the proof of Proposition \ref{prop:uppervolumebound}, using the upper bound in Proposition \ref{propdistance3}.
\end{proof}

We also have the following volume upper bound, which follows directly from the same argument as in the proof of Proposition \ref{prop:uppervolumebound} by using Proposition \ref{propdistance3} (i).

\begin{prop}\label{propvolumeslicelimit}
For any $x \in Z$ and $r>0$ with $\t(x)-r^2 \in \III^-$ and any $t \in \R$, we have
	\begin{equation*}
\abs{B_Z^*(x,r) \bigcap Z_t}_t \leq C(n,\sigma) r^{n}.
	\end{equation*}
\end{prop}

Next, we define the closeness of two noncollapsed Ricci flow limit spaces, which is an approximate version of Definition \ref{def:iso}. 

\begin{defn}[$\ep$-close]\label{defn:close}
Suppose $(Z, d_Z, z, \t)$ and $(Z', d_{Z'}, z',\t')$ are two pointed noncollapsed Ricci flow limit spaces, with regular parts given by the Ricci flow spacetimes $(\RR, \t, \partial_\t, g^Z)$ and $(\RR', \t', \partial_{\t'}, g^{Z'})$, respectively, such that $J$ is a time interval.

We say that $(Z, d_Z, z, \t)$ is \textbf{$\ep$-close} to $(Z', d_{Z'}, z',\t')$ \textbf{over $J$} if there exists an open set $U \subset \RR'_J$ and a smooth embedding $\phi: U \to \RR_J$ satisfying the following properties.
\begin{enumerate}[label=\textnormal{(\alph{*})}]
\item $\phi$ is time-preserving.

\item $U \subset B^*_{Z'}(z', \ep^{-1}) \bigcap \RR'_J$ and $U$ is an $\ep$-net of $B^*_{Z'}(z', \ep^{-1}) \bigcap Z'_J$ with respect to $d_{Z'}$.

\item For any $x, y \in U$, we have
	\begin{align*}
\abs{d_Z(\phi(x), \phi(y))-d_{Z'}(x, y)} \le \ep.
	\end{align*}
	
\item The $\ep$-neighborhood of $\phi(U)$ with respect to $d_Z$ contains $B^*_{Z}(z, \ep^{-1}-\ep) \bigcap Z_J$.

\item There exists $x_0 \in U$ such that $d_{Z'}(x_0, z') \le \ep$ and $d_{Z}(\phi(x_0), z) \le \ep$.

\item On $U$, the following estimates hold:
  \begin{align*}
  	\lVert \phi^* g^Z-g^{Z'}\rVert_{C^{[\ep^{-1}]}(U)}+\lVert \phi^* \partial_\t-\partial_{\t'} \rVert_{C^{[\ep^{-1}]}(U)} \le \ep.
  \end{align*} 
\end{enumerate}
\end{defn}

It is clear from the above definition that if $(Z, d_Z, z, \t)$ is $\ep$-close to $(Z', d_{Z'}, z',\t')$ over $J$, then $(Z', d_{Z'}, z',\t')$ is $\Psi(\ep)$-close to $(Z, d_Z, z, \t)$ over $J$, where $\Psi(\ep) \to 0$ as $\ep \to 0$. 

Next, we introduce the following notation.

\begin{notn}\label{not:2}
For a sequence of noncollapsed Ricci limit spaces 
  \begin{align*}
(Z_i, d_{Z_i}, z_i, \t_i) \in \MM(n, Y), \quad i \in \mathbb N \cup \{\infty\},
  \end{align*} 
we write
	\begin{equation*}
		(Z_i, d_{Z_i}, z_i, \t_i) \xrightarrow[i \to \infty]{\quad \hat C^\infty \quad} (Z_\infty, d_{Z_\infty}, z_\infty,\t_\infty),
	\end{equation*}
	if there exists a sequence $\ep_i \to 0$ such that $(Z_i, d_{Z_i}, z_i, \t_i)$ is $\ep_i$-close to $(Z_\infty, d_{Z_\infty}, z_\infty,\t_\infty)$ over $[-\ep_i^{-1}, \ep_i^{-1}]$.
\end{notn}

In particular, it follows from Theorem \ref{thm:smooth1} that the convergence \eqref{eq:conv00} can be strengthened to
	\begin{equation*} 
(M_i \times \III, d^*_i, p_i^*,\t_i) \xrightarrow[i \to \infty]{\quad \hat C^\infty \quad} (Z, d_Z, p_{\infty},\t).
	\end{equation*}

\section{Extended metric flows} \label{sec:extend}

In this section, we consider a Ricci flow limit space $(Z, d_Z, p_{\infty},\t)$ obtained from 
	\begin{equation} \label{eq:conv0a01}
		(M_i \times \III, d^*_i, p_i^*,\t_i) \xrightarrow[i \to \infty]{\quad \mathrm{pGH} \quad} (Z, d_Z, p_{\infty},\t),
	\end{equation}
where $\XX^i=\{M_i^n,(g_i(t))_{t \in \III^{++}}\} \in \MM(n, Y, T)$ with base point $p_i^* \in \XX^i_\III$. 

First, we define a distance function $d_t^{Z}$ on $Z_t$.

\begin{defn}\label{defntimeslicedist}
	For each $t \in \III^-$, we define the distance at the time-slice $Z_t$ by
	\begin{align*}
		d^{Z}_t(x,y):=\lim_{s \nearrow t} d_{W_1}^{\RR_s}(\nu_{x;s},\nu_{y;s}) \in [0,+\infty]\index{$d^Z_t$}
	\end{align*}
	for any $x,y \in Z_t$, where the limit exists by Lemma \ref{lem:limitmono}. Note that if $d^{Z}_t(x,y)<\infty$, then for any $s<t$,
	\begin{align} \label{eq:support1}
\iota_x(\RR^x_s)=\iota_y(\RR^y_s).
\end{align}	
\end{defn}

\begin{rem} \label{rem:indep}
The definition of $d^{Z}_t$ only uses the conjugate heat kernel measures and the spatial metrics on the regular time slices. It is therefore intrinsic to the resulting extended metric flow.
\end{rem}

 \begin{lem}\label{lem:extend1}
 For any $t \in \III^-$, $(Z_t, d^Z_t)$ is an extended metric space.
 \end{lem}

\begin{proof}
 We first prove the triangle inequality. Given $x, y, z\in Z_t$, for any $s<t$, we have
 	\begin{align*}
d_{W_1}^{\RR_s}(\nu_{x;s},\nu_{z;s}) \le d_{W_1}^{\RR_s}(\nu_{x;s},\nu_{y;s})+d_{W_1}^{\RR_s}(\nu_{y;s},\nu_{z;s}) \le d^Z_t(x, y)+d^Z_t(y, z).
	\end{align*}
Letting $s \nearrow t$, we obtain $d^Z_t(x, z) \le d^Z_t(x, y)+d^Z_t(y, z)$.
 
 In addition, if $d^Z_t(x, y)=0$, then by Lemma \ref{lem:limitmono}, $d_{W_1}^{\RR_s}(\nu_{x;s},\nu_{y;s})=0$ for any $s<t$. Since we can regard $\nu_{y;s}$ as a conjugate heat flow on $\XX^x$, we conclude that $\nu_{x;s}=\nu_{y;s}$. Then we take $w_i \in \RR_{s_i}$ for a sequence $s_i \nearrow t$ so that $w_i$ is a common $H_n$-center of $x$ and $y$. By Lemma \ref{lem:004abc}, $x$ and $y$ are both limits of $w_i$ in $d_Z$. Thus, we conclude that $x=y$. 
 \end{proof}

\begin{lem}\label{dtlem1}
	For any $x,y\in Z_t$,
	\begin{align*}
		d_Z(x,y)\leq d^Z_t(x,y).
	\end{align*}
\end{lem}
\begin{proof}
We assume $d^Z_t(x,y)<\infty$. Then, \eqref{eq:support1} holds for any $s<t$, and we can regard $(\iota_x^{-1})_*\nu_{y;s}$ as a conjugate heat flow on $\XX^x$.

For any $s_i \nearrow t$, we choose $x_i ,y_i \in \RR^x_{s_i}$ to be regular $H_n$-centers of $x$ and $y$, respectively. Then, it follows from the generalization of Proposition \ref{distancefunction1} (1) to $\XX^x$ that
	\begin{align} \label{eq:dt1}
		d^*_x(x_i,y_i)\leq  d_{g^Z_{s_i}}(\iota_x(x_i), \iota_x(y_i)).
	\end{align}

On the other hand, by the definition of the regular $H_n$-center, we have
	\begin{align}\label{eq:dt2}
d_{W_1}^{\RR_{s_i}}(\nu_{x;s_i}, \nu_{y;s_i}) \ge  d_{g^Z_{s_i}}(\iota_x(x_i), \iota_x(y_i))-2 \sqrt{H_n(t-s_i)}.
	\end{align}

Moreover, it follows from Lemma \ref{lem:004abc} that 
	\begin{align}\label{eq:dt3}
		d^*_x(x_i,y_i) = d_Z(\iota_x(x_i), \iota_x(y_i)) \ge d_Z(x, y)-2 \sqrt{H_n(t-s_i)}.
	\end{align}

Combining \eqref{eq:dt1}, \eqref{eq:dt2}, and \eqref{eq:dt3} and letting $i\to\infty$, the conclusion follows.
\end{proof}


%

\begin{lem}\label{takelimitdt1}
Given $x,y\in Z_t$, there exists a sequence $t_i \nearrow t$ such that if $x_i, y_i \in \RR_{t_i}$ are regular $H_n$-centers of $x$ and $y$, respectively, then $d^Z_{t_i}(x_i, y_i)=d_{g^Z_{t_i}}(x_i, y_i)$ and
	\begin{align*}
		d^Z_{t}(x,y)=\lim_{i\to\infty} d^Z_{t_i}(x_i,y_i).
	\end{align*}
\end{lem}
\begin{proof}
We only prove the case $d_t^Z(x, y)<\infty$, since the case $d_t^Z(x, y)=\infty$ can be proved similarly.

Since $d_t^Z(x, y)<\infty$, we regard $(\iota_x^{-1})_*\nu_{y;s}$ as a conjugate heat flow on $\XX^x$. We take a sequence $t_i \nearrow t$ such that $\XX^x$ is continuous at $t_i$. Then, it follows from \cite[Equation (4.22)]{bamler2023compactness} that
	\begin{align*}
		\lim_{s \nearrow t_i} d_{W_1}^{\XX^x_s}(\nu_{\iota_x^{-1}(x_i);s}, \nu_{\iota_x^{-1}(y_i);s})=d_{\XX^x_{t_i}}(\iota_x^{-1}(x_i),\iota_x^{-1}(y_i))=d_{g^Z_{t_i}}(x_i, y_i).
	\end{align*}
In other words, 
	\begin{align}\label{eq:dt11}
d^Z_{t_i}(x_i, y_i)=d_{g^Z_{t_i}}(x_i, y_i).
	\end{align}
	
On the other hand, by the definition of the $H_n$-center, we have
	\begin{align}\label{eq:dt12}
\abs{d_{W_1}^{\RR_{t_i}}(\nu_{x;t_i}, \nu_{y;t_i})-d_{g^Z_{t_i}}(x_i, y_i)} \le 2 \sqrt{H_n(t-t_i)}.
	\end{align}
Combining \eqref{eq:dt11} with \eqref{eq:dt12}, and letting $i \to \infty$, the conclusion follows.
\end{proof}

\begin{prop}[Completeness]\label{prop:com}
	For any $t \in \III^-$, the extended metric space $(Z_t, d_t^Z)$ is complete.
\end{prop}
\begin{proof}
	Suppose $x_i \in Z_t$ is a Cauchy sequence with respect to $d_t^Z$. By Lemma \ref{dtlem1},
	\begin{align*}
		d_Z(x_i, x_j) \le  d_t^Z(x_i,x_j).
	\end{align*}
	In particular, $x_i$ is also a Cauchy sequence with respect to $d_Z$. Since $(Z,d_Z)$ is complete, let $x_i\to x_\infty$ with respect to $d_Z$. Moreover, it is clear that $\t(x_{\infty})=t$ by the continuity of $\t$.
	
	We set $z=x_1$ so that $(\iota_z^{-1})_*\nu_{x_i;s}$ can be regarded as a conjugate heat flow on $\XX^z$ for any $s<t$. In particular,
		\begin{align*}
d_{W_1}^{\XX^z_s}((\iota_z^{-1})_*\nu_{x_i;s},(\iota_z^{-1})_*\nu_{x_j;s}) \le d_t^Z(x_i,x_j).
	\end{align*}
Let $(\iota_z^{-1})_*\nu_{x_i;s}\to\mu_s$ in $d_{W_1}^{\XX^z_s}$. Note that this convergence is uniform in $s$. Passing to the limit in the reproduction formula shows that $\mu_s$ is a conjugate heat flow on $\XX^z_\III$. Moreover, we have, by Proposition \ref{takelimitVarW1} and Proposition \ref{prop:004ac},
	\begin{align*}
		\mathrm{Var}_{\XX^z_s}(\mu_s) \le H_n(t-s).
	\end{align*}

	Now choose a sequence $t_k\nearrow t$ and let $y_k\in\XX^z_{t_k}$ be an $H_n$-center of $\mu_{t_k}$ at $t_k$, i.e.,
	\begin{align*}
		\Var_{\XX^z_{t_k}}(\mu_{t_k},\delta_{y_k})\leq H_n(t-t_k).
	\end{align*}
	We claim that $\iota_z(y_k) \to x_{\infty}$ under $d_Z$. Indeed, if $z_{i,k} \in \iota_z(\RR^{z}_{t_k})$ is a regular $H_n$-center of $x_i$, then
	\begin{align*}
d_Z(x_i, \iota_z(y_k)) \le d_Z(x_i, z_{i,k})+d_Z(z_{i,k}, \iota_z(y_k)) \le  \sqrt{H_n(t-t_k)}+ d^Z_{t_k}(z_{i,k},\iota_z(y_k)),
	\end{align*}	
	where we used Lemma \ref{lem:004abc} and Lemma \ref{dtlem1}. In addition,
	\begin{align*}
 d^Z_{t_k}(z_{i,k},\iota_z(y_k)) \le & d_{W_1}^{\XX^z_{t_k}}( \delta_{\iota_z^{-1}(z_{i,k})},(\iota_z^{-1})_*\nu_{x_i;t_k})+d_{W_1}^{\XX^z_{t_k}}((\iota_z^{-1})_*\nu_{x_i;t_k}, \mu_{t_k})+d_{W_1}^{\XX^z_{t_k}}(\delta_{y_k}, \mu_{t_k}) \\
  \le & d_{W_1}^{\XX^z_{t_k}}((\iota_z^{-1})_*\nu_{x_i;t_k}, \mu_{t_k})+2 \sqrt{H_n(t-t_k)}.
	\end{align*}	
	Combining the above inequalities, we obtain
	\begin{align*}
d_Z(x_i, \iota_z(y_k)) \le 3 \sqrt{H_n(t-t_k)}+ d_{W_1}^{\XX^z_{t_k}}((\iota_z^{-1})_*\nu_{x_i;t_k}, \mu_{t_k})
	\end{align*}		
	and by letting $i \to \infty$,
		\begin{align*}
d_Z(x_\infty, \iota_z(y_k)) \le 3 \sqrt{H_n(t-t_k)}.
	\end{align*}	
Thus, $\iota_z(y_k) \to x_\infty$ in $d_Z$. From this and Proposition \ref{prop:pre5}, we conclude that, for every fixed $s<t$ and all sufficiently large $k$, $s<t_k$ and
		\begin{align*}
\lim_{k \to \infty} d_{W_1}^{\RR^z_s}(\nu_{y_k;s}, (\iota_z^{-1})_* \nu_{x_\infty;s})=0.
	\end{align*}

Now, it follows from the definition of $d_t^Z$ that 
	\begin{align*}
		d_t^Z(x_i, x_{\infty})=&\lim_{s \nearrow t} d_{W_1}^{\RR^z_s} ((\iota_z^{-1})_* \nu_{x_i;s},(\iota_z^{-1})_* \nu_{x_\infty;s})=\lim_{s \nearrow t} \lim_{k \to \infty} d_{W_1}^{\XX^z_s} ((\iota_z^{-1})_* \nu_{x_i;s},\nu_{y_k;s}) \\
		\le& \lim_{s \nearrow t}  d_{W_1}^{\XX^z_s} ((\iota_z^{-1})_* \nu_{x_i;s},\mu_s)+ \lim_{s \nearrow t} \lim_{k \to \infty} d_{W_1}^{\XX^z_s}(\nu_{y_k;s},\mu_s) \\
		\le& \lim_{s \nearrow t}  d_{W_1}^{\XX^z_s} ((\iota_z^{-1})_* \nu_{x_i;s},\mu_s)+\lim_{k \to \infty} d_{W_1}^{\XX^z_{t_k}}(\delta_{y_k},\mu_{t_k}) \\
		\le&\lim_{s \nearrow t}  d_{W_1}^{\XX^z_s} ((\iota_z^{-1})_* \nu_{x_i;s},\mu_s)+\lim_{k \to \infty}  \sqrt{H_n(t-t_k)} \le \lim_{s \nearrow t}  d_{W_1}^{\XX^z_s} ((\iota_z^{-1})_* \nu_{x_i;s},\mu_s),
	\end{align*}
	where we used the monotonicity of $d_{W_1}^{\XX^z_s}(\nu_{y_k;s},\mu_s)$; see \cite[Proposition 3.16 (b)]{bamler2023compactness}.
	Since $(\iota_z^{-1})_* \nu_{x_i;s}$ converges to $\mu_s$ uniformly for $s<t$, the conclusion follows.
\end{proof}

\begin{prop}[$H_n$-concentration I]\label{prop:005b}
	For $x,y \in Z_t$ and $s<t$, we have
	\begin{align*}
		\Var_{\RR_s}(\nu_{x;s},\nu_{y;s}) \le \bigl(d^Z_t(x,y)\bigr)^2+H_n(t-s).
	\end{align*}
\end{prop}

\begin{proof}
	We assume $d^Z_t(x,y) <\infty$ and consider a sequence $t_i \nearrow t$ as in Lemma \ref{takelimitdt1} and let $x_i \in \RR_{t_i}$ and $y_i \in \RR_{t_i}$ be regular $H_n$-centers of $x$ and $y$, respectively. By Lemma \ref{lem:004abc}, $x_i\to x,y_i\to y$ in $d_Z$. Then we have for any $s<t$,
	\begin{align*}
\nu_{\iota_x^{-1}(x_i) ;s} \to \nu_{x;s} \quad \text{and} \quad \nu_{\iota_x^{-1}(y_i) ;s} \to (\iota_x^{-1})_* \nu_{y;s}
	\end{align*}	
in $d_{W_1}^{\XX^x_s}$. By Proposition \ref{takelimitVarW1} and the $H_n$-concentration of $\XX^x$, we have
	\begin{align}\label{Hnconc1}
		\Var_{\RR_s}(\nu_{x;s},\nu_{y;s}) \leq & \liminf_{i \to \infty} \Var_{\XX^x_s}(\nu_{\iota_x^{-1}(x_i) ;s},\nu_{\iota_x^{-1}(y_i);s}) \notag \\
		 \leq& \liminf_{i \to \infty}\lc d_{g^x_{t_i}}(\iota_x^{-1}(x_i),\iota_x^{-1}(y_i))^2+H_n(t_i-s)\rc\nonumber\\
		\leq& \liminf_{i \to \infty} d_{g^Z_{t_i}}(x_i,y_i)^2+H_n(t-s).
	\end{align}
	On the other hand, by Lemma \ref{takelimitdt1}, we have
	\begin{align}\label{Hnconc2c}
		\lim_{i \to \infty} d_{g^Z_{t_i}}(x_i,y_i)= d^Z_t(x,y).
	\end{align}
	Therefore, the conclusion follows from \eqref{Hnconc1} and \eqref{Hnconc2c}.
\end{proof}

In general, the distance $d_{g^Z_t}$ induced by $g^Z_t$, when restricted to $\RR_t$, may not agree with $d^Z_t$. For instance, it is possible that $d^Z_t(x, y)<\infty$, but $x$ and $y$ lie in different connected components of $\RR_t$. Next, we prove that locally, the two distance functions agree.

\begin{prop} \label{prop:dismatch}
For any $w \in \RR_t$, there exists a sufficiently small constant $r>0$ such that for any $x,y \in B_{g^Z_t}(w, r)$,
	\begin{align*}
d_{g^Z_t}(x, y)=d^Z_t(x, y).
	\end{align*}
Moreover, for any $x, y \in \RR_t$, $d^Z_t(x, y) \le d_{g^Z_t}(x, y)$.
\end{prop}

\begin{proof}
We choose a sufficiently small $r>0$ such that the product domain $U=B_{g^Z_t}(w, r) \times [t-r^2,t] \subset \RR$ has the property that $U \cap \RR_s$ is geodesically convex for any $s \in [t-r^2, t]$. Here, being geodesically convex means that any two points in $U \cap \RR_s$ can be connected by a minimal geodesic with respect to $g^Z_s$ and any such minimal geodesic is contained in $U \cap \RR_s$.

For any $x, y \in B_{g^Z_t}(w, r)$, we regard $(\iota_x^{-1})_* \nu_{y;s}$ as a conjugate heat flow on $\XX^x$. We take a sequence $ t_i \nearrow t$ such that $\XX^x$ is continuous at $t_i$. Then, we set $x_i \in U \cap \RR_{t_i}$ and $y_i \in U \cap \RR_{t_i}$ to be the flows of $x$ and $y$ with respect to $\partial_\t$, respectively. By Lemma \ref{lem:center1}, $x_i$ and $y_i$ are regular $H$-centers of $x$ and $y$, respectively, where $H$ is a positive constant. As in the proof of Lemma \ref{takelimitdt1}, we conclude that
	\begin{align*}
d^Z_{t}(x, y)=\lim_{i \to \infty} d_{g_{t_i}^Z}(x_i, y_i).
	\end{align*}
On the other hand, it is clear from our construction that $\lim_{i \to \infty} d_{g_{t_i}^Z}(x_i, y_i)=d_{g^Z_t}(x, y)$. Thus, we obtain
	\begin{align*}
d^Z_{t}(x, y)=d_{g^Z_t}(x, y).
	\end{align*}

Now, since $d_{g^Z_t}$ is a length metric on any connected component of $\RR_t$, we conclude immediately that $d^Z_t(x, y) \le d_{g^Z_t}(x, y)$ for any $x, y \in \RR_t$, by the local isometry.
\end{proof}

Proposition \ref{prop:dismatch} implies, in particular, that $d_{g^Z_t}$ and $d_t^Z$ induce the same topology on $\RR_t$. Moreover, $\RR_t$ is an open set of $Z_t$ with respect to $d^Z_t$. Therefore, we can regard any conjugate heat kernel measure $\nu_{x;t}$ as a probability measure on $Z_t$.

\begin{defn}[Variance and $W_1$-Wasserstein distance]\label{defnvarianceextended}
For any $t \in \III^-$, the variance between two probability measures $\mu_1,\mu_2\in \PP(Z_t)$ is defined by
\begin{equation*}
	\Var_{Z_t}(\mu_1,\mu_2):= \int_{Z_t}\int_{Z_t} d^Z_t(x_1,x_2)^2\,\mathrm{d}\mu_1(x_1)\, \mathrm{d}\mu_2(x_2).
\end{equation*}\index{$\Var_{Z_t}$}
Moreover, the $W_1$-Wasserstein distance between $\mu_1,\mu_2 \in \PP(Z_t)$ is defined by
\begin{align*}
	d_{W_1}^{Z_t}(\mu_1,\mu_2):=\sup \int_{Z_t} f \, \mathrm{d} (\mu_1-\mu_2),
\end{align*}\index{$d_{W_1}^{Z_t}$}
\index{$W_1$-Wasserstein distance}
where the supremum is taken over all bounded $1$-Lipschitz functions $f:Z_t\to\R$.
\end{defn}

Next, we prove

\begin{lem}\label{lem:dw1Z}
For any $x, y \in Z_{\III^-}$, if $d_{W_1}^{Z_{t_0}}(\nu_{x;t_0}, \nu_{y;t_0}) <\infty$ for some $t_0 \in (-(1-2\sigma)T, \min\{\t(x), \t(y)\})$, then $\iota_x(\RR^x_t)=\iota_y(\RR^y_t)$ for any $t \in [-(1-2\sigma)T, t_0)$.
\end{lem}

\begin{proof}
If $\iota_x(\RR^x_{t_0})=\iota_y(\RR^y_{t_0})$, we are done. Otherwise, by the definition of $d_{W_1}^{Z_{t_0}}$, there exist $x' \in \iota_x(\RR^x_{t_0})$ and $y' \in \iota_y(\RR^y_{t_0})$ such that $d_{t_0}^Z(x', y')<\infty$. By \eqref{eq:support1} and the reproduction formula, we conclude that $\iota_x(\RR^x_t)=\iota_y(\RR^y_t)$ for any $t \in [-(1-2\sigma)T, t_0)$.
\end{proof}

 The following result is immediate from Proposition \ref{prop:005b} and Proposition \ref{prop:dismatch}.

\begin{prop}[$H_n$-concentration II]\label{prop:005bb}
	For $x,y \in Z_t$ and $-(1-2\sigma)T<s<t$, we have
	\begin{align*}
		\Var_{Z_s}(\nu_{x;s},\nu_{y;s}) \le \bigl(d^Z_t(x,y)\bigr)^2+H_n(t-s).
	\end{align*}
\end{prop}

Next, we define
\begin{defn}[$H$-center]
For any $z \in Z$, a point $z_1 \in Z_s$ with $s \in (-(1-2\sigma)T, \t(z))$ is called an \textbf{$H$-center} of $z$ for a constant $H>0$ if
 \begin{align*}
	\Var_{Z_s}(\delta_{z_1},\nu_{z;s})\leq H (\t(z)-s).
	\end{align*}
Note that by Proposition \ref{prop:005bb}, for any $s \in (-(1-2\sigma)T, \t(z))$, we can always find an $H_n$-center of $z$ in $Z_s$. Moreover, by Proposition \ref{prop:dismatch}, any regular $H$-center $z_2 \in \RR_s$ of $z$ is also an $H$-center of $z$.
\end{defn}

By the definition of an $H$-center, the following statement is immediate.

\begin{lem}\label{lem:004abcxx}
Given $x \in Z$, if $z \in Z_s$ is an $H$-center of $x$ with $s \in (-(1-2\sigma)T,\t(x))$, then
	\begin{align*}
\nu_{x;s}\lc B_{Z_s} \lc z, \sqrt{LH(\t(x)-s)} \rc \rc \ge 1-\frac{1}{L},
	\end{align*}
	where $B_{Z_s}(z,r):=\{w\in Z_s\mid d^Z_s(z,w)<r\}$ denotes the metric ball with respect to $d^Z_s$.
\end{lem}

We also have the following estimate, which should be compared to Lemma \ref{lem:Hcenterdis} and Lemma \ref{lem:004abc}.

\begin{lem}\label{lem:004abcxee}
	Let $x\in Z$ and let $z\in Z_s$ be an $H$-center of $x$ at a time $s\in (-(1-2\sigma)T,\t(x))$. Then
	\begin{align*}
		d_Z(x,z) \le 3 \sqrt{(H_n+H)(\t(x)-s)}.
	\end{align*}
\end{lem}
\begin{proof}
Let $z' \in \RR_s$ be a regular $H_n$-center of $x$. It follows from Lemma \ref{lem:004abcx}, Proposition \ref{prop:dismatch}, and Lemma \ref{lem:004abcxx} that
	\begin{align*}
d^Z_s(z, z') \le \sqrt{2 H_n(\t(x)-s)}+\sqrt{2 H(\t(x)-s)}.
	\end{align*}
Thus, by Lemma \ref{lem:004abc} and Lemma \ref{dtlem1}, we conclude that
	\begin{align*}
d_Z(x, z) \le d_Z(x, z')+d_Z(z, z') \le  \sqrt{H_n(\t(x)-s)}+ d^Z_s(z, z') \le 3 \sqrt{(H_n+H)(\t(x)-s)}.
	\end{align*}
\end{proof}

In general, it is unclear whether $(Z_t, d_t)$ is separable. For this reason, we have the following definition:

\begin{defn}[Extended metric flow] \label{defn:emf}
An \textbf{extended metric flow}\index{extended metric flow} over a subset $I$ of $\R$ is a tuple of the form 
	\begin{align*}
		\lc Z , \t, (d_t)_{t \in I} , (\nu_{x;s})_{x \in Z, s \in I, s \leq \t (x)} \rc
	\end{align*}
satisfying all conditions (1)-(7) in Definition \ref{defnmetricflow}, except that condition (3) is replaced by the following:
\begin{itemize}
\item $(Z_t, d_t)$ is a complete extended metric space for any $t \in I$.
\end{itemize}
\end{defn}


\begin{thm}[Extended metric flow]\label{thm:em}
	$\lc Z, \t, (d^Z_t)_{t \in \III^-}, (\nu_{z; s})_{s\in \III^-, s\le \t(z)} \rc$ is an $H_n$-concentrated extended metric flow over $\III^-$.
\end{thm}
\begin{proof}
	All items in Definition \ref{defnmetricflow} except item (6) follow from \eqref{eq:generalrep} and Proposition \ref{prop:com}.
	
	For item (6), we consider a function $u_{t_0}=\Phi \circ f_{t_0}$ for some $L^{-\frac 1 2}$-Lipschitz function $f_{t_0}: Z_{t_0} \to \R$ (if $L=0$, then there is no additional assumption on $u_{t_0}$). We define $u_t: Z_t \to \R$ by
	\begin{align*}
u_t(z)=\int_{Z_{t_0}} u_{t_0} \, \mathrm{d}\nu_{z;{t_0}}.
	\end{align*}	
	
For any $x, y\in Z_t$ with $d^Z_t(x, y)<\infty$, we may regard $(\iota_x^{-1})_* \nu_{y;s}$ as a conjugate heat flow on $\XX^x_{s<t}$. We take a sequence $ t_i \nearrow t$ such that $\XX^x$ is continuous at $t_i$ and choose $x_i, y_i \in \RR_{t_i}$ to be regular $H_n$-centers of $x$ and $y$, respectively. As in the proof of Lemma \ref{takelimitdt1}, we have
	\begin{equation}\label{eq:extend1}
d_{\XX^x_{t_i}}(\iota_x^{-1}(x_i), \iota_x^{-1}(y_i))=d_{g^Z_{t_i}}(x_i, y_i)
	\end{equation}	
	and
	\begin{equation}\label{eq:extend2}
\lim_{i \to \infty} d_{g^Z_{t_i}}(x_i, y_i)=d^Z_t(x, y).
	\end{equation}		
	
Since $\XX^x$ is a metric flow, we obtain from \eqref{eq:extend1}
		\begin{equation}\label{eq:extend3}
\abs{f_{t_i}(x_i)-f_{t_i}(y_i)} \le (t_i-t_0+L)^{-\frac 1 2} d_{g^Z_{t_i}}(x_i, y_i),
	\end{equation}	
	where $u_s=\Phi \circ f_s$ for any $s \in [t_0, t]$. On the other hand, by Lemma \ref{lem:004abc} on $\XX^x$, we have for any $s<t$
			\begin{align*}
\nu_{\iota_x^{-1}(x_i);s} \xrightarrow{i \rightarrow\infty} \nu_{x;s}
	\end{align*}	
	in $d_{W_1}^{\XX^x_s}$. Since 
	\begin{align*}
u_{t_i}(x_i)=\int_{Z_{t_0}} u_{t_0} \, \mathrm{d}\nu_{x_i;t_0},
	\end{align*}
	this implies $u_t(x)=\lim_{i \to \infty}u_{t_i}(x_i)$. Similarly, we have $u_t(y)=\lim_{i \to \infty}u_{t_i}(y_i)$.
	
Thus, by passing to the limit and using \eqref{eq:extend2} and \eqref{eq:extend3}, we obtain
		\begin{align*}
\abs{f_{t}(x)-f_{t}(y)} \le (t-t_0+L)^{-\frac 1 2} d^Z_t (x, y).
	\end{align*}	
This proves item (6). Finally, Proposition \ref{prop:005bb} implies that the extended metric flow
\begin{align*}
\lc Z, \t, (d^Z_t)_{t \in \III^-}, (\nu_{z;s})_{s\in \III^-,\,s\le \t(z)} \rc
\end{align*}
is $H_n$-concentrated.
\end{proof}

\begin{rem}
If two noncollapsed Ricci flow limit spaces $(Z, d_Z, z, \t)$ and $(Z', d_{Z'}, z', \t')$ are isometric (see Definition \ref{def:iso}), then they are also isometric as extended metric flows (see \cite[Definition 3.7]{bamler2023compactness}), by Remark \ref{rem:defheatequS} and Definition \ref{defntimeslicedist}.
\end{rem}

The proof of \cite[Proposition 3.16]{bamler2023compactness} also gives the following monotonicity.

\begin{lem}\label{lem:limitmono1}
For any $x, y \in Z$ and $s \in (-(1-2\sigma)T, \min\{\t(x),\t(y)\}]$, the function
\begin{align*}
	s\mapsto d_{W_1}^{Z_s}(\nu_{x;s},\nu_{y;s}) \quad \text{is nondecreasing.}
\end{align*}
\end{lem}

Next, we prove past continuity.

\begin{prop} \label{prop:pastcont}
For any $x,y\in Z_t$ with $t\in\III^-$,
\begin{align*}
\lim_{s \nearrow t} d_{W_1}^{Z_s}(\nu_{x;s},\nu_{y;s})=d^Z_t(x, y).
\end{align*}
\end{prop}

\begin{proof}
By Lemma \ref{takelimitdt1}, there exists a sequence $t_i \nearrow t$ such that if $x_i, y_i \in \RR_{t_i}$ are regular $H_n$-centers of $x$ and $y$, respectively, then
	\begin{align*}
		d^Z_{t}(x,y)=\lim_{i\to\infty} d^Z_{t_i}(x_i,y_i).
	\end{align*}
Since
\begin{align*}
\abs{d_{W_1}^{Z_{t_i}}(\nu_{x;t_i},\nu_{y;t_i})-d^Z_{t_i}(x_i,y_i)} \le 2 \sqrt{H_n(t-t_i)},
\end{align*}
	we conclude that
	\begin{align*}
\lim_{i \to \infty} d_{W_1}^{Z_{t_i}}(\nu_{x;t_i},\nu_{y;t_i}) =	d^Z_{t}(x,y),
\end{align*}
which, when combined with Lemma \ref{lem:limitmono1}, yields the conclusion.	
\end{proof}

Next, we have the following characterization of $d_Z$.

\begin{prop}\label{prop:chara}
For any $x,y\in Z_{\III^-}$ with $t_0=\t(x)\geq\t(y)$, set
\begin{align*}
r:=d_Z(x,y),\qquad q:=t_0-r^2.
\end{align*}
If $q\in\III^-$, then
\begin{align}\label{eq:charaleft}
\lim_{t\nearrow q}d_{W_1}^{Z_t}(\nu_{x;t},\nu_{y;t})\leq r.
\end{align}
If in addition $q<\t(y)$, equivalently $r>\sqrt{\t(x)-\t(y)}$, then
\begin{align}\label{eq:chararight}
r\leq\lim_{t\searrow q}d_{W_1}^{Z_t}(\nu_{x;t},\nu_{y;t}).
\end{align}
\end{prop}

\begin{proof}
By Proposition \ref{prop:pre5}, $\iota_x(\RR^x_t)=\iota_y(\RR^y_t)$ for every $t\in[-(1-2\sigma)T,q)$. Proposition \ref{prop:dismatch} and \eqref{eq:lowlimit} therefore give
\begin{align*}
\lim_{t\nearrow q}d_{W_1}^{Z_t}(\nu_{x;t},\nu_{y;t})
\leq\lim_{t\nearrow q}d_{W_1}^{\XX^x_t}\left(\nu_{x;t},(\iota_x^{-1})_*\nu_{y;t}\right)
\leq r,
\end{align*}
which proves \eqref{eq:charaleft}.

Assume now that $q<\t(y)$ and suppose that \eqref{eq:chararight} fails. By monotonicity, there are $\delta>0$ and $t_1\in(q,\t(y))$ such that
\begin{align}\label{eq:characontra}
d_{W_1}^{Z_{t_1}}(\nu_{x;t_1},\nu_{y;t_1})\leq r-\delta.
\end{align}
Lemma \ref{lem:dw1Z} implies that $\iota_x(\RR^x_t)=\iota_y(\RR^y_t)$ for every $t<t_1$. Choose $x_i^*,y_i^*\in M_i\times\III$ converging to $x,y$, respectively, and choose
\begin{align*}
q<t_3<t_2<t_1
\end{align*}
so that $\XX^x$ is continuous at $t_3$. In particular, the distance $d^Z_{t_3}$ on $\iota_x(\RR^x_{t_3})$ agrees with $d_{g^Z_{t_3}}$.

We claim that
\begin{align} \label{eq:bounded}
	d_{W_1}^{t_2}(\nu_{x_i^*;t_2},\nu_{y_i^*;t_2}) \le D
\end{align}
for a constant $D$. To see this, let $z \in \iota_x(\RR^x_{t_2})$ be a regular $H_n$-center of $x$, and suppose that $z_i^*=(z_i, t_2) \in M_i \times \{t_2\}$ converges to $z$ in the Gromov--Hausdorff sense. Then, by Theorem \ref{heatkernelupperbdgeneral} (ii) and Theorem \ref{thm:convextra}, $z_i^*$ is an $H$-center of $x_i^*$ for some constant $H>0$ independent of $i$. Similarly, let $w \in \iota_x(\RR^x_{t_2})$ be a regular $H_n$-center of $y$, and suppose that $w_i^*=(w_i, t_2) \in M_i \times \{t_2\}$ converges to $w$ in the Gromov--Hausdorff sense. Then $w_i^*$ is also an $H$-center of $y_i^*$. Therefore, we have
\begin{align*} 
	d_{W_1}^{t_2}(\nu_{x_i^*;t_2},\nu_{y_i^*;t_2}) \le 2\sqrt{H(t_0-t_2)}+d_{g_i(t_2)}(z_i, w_i).
\end{align*}
By smooth convergence (Theorem \ref{thm:smooth1}), the distance $d_{g_i(t_2)}(z_i, w_i)$ remains uniformly bounded. This establishes \eqref{eq:bounded}.

Lemma \ref{lem:pre5}, continuity at $t_3$, and \eqref{eq:characontra} now yield
\begin{align}\label{eq:charalim}
\lim_{i\to\infty}d_{W_1}^{t_3}(\nu_{x_i^*;t_3},\nu_{y_i^*;t_3})
=d_{W_1}^{Z_{t_3}}(\nu_{x;t_3},\nu_{y;t_3})
\leq r-\delta.
\end{align}
On the other hand,
\begin{align*}
\sqrt{\max\{\t_i(x_i^*),\t_i(y_i^*)\}-t_3}
\longrightarrow\sqrt{t_0-t_3}<r.
\end{align*}
Using the fixed comparison time $t_3$ in Definition \ref{defn:dstar-distance}, \eqref{eq:charalim} gives a constant $\eta>0$ such that
\begin{align*}
d_i^*(x_i^*,y_i^*)\leq r-\eta
\end{align*}
for all sufficiently large $i$. This contradicts $d_i^*(x_i^*,y_i^*)\to d_Z(x,y)=r$ and proves \eqref{eq:chararight}.
\end{proof}

Combining Lemma \ref{lem:limitmono1} and Proposition \ref{prop:chara}, the following corollary is immediate.

\begin{cor}\label{cor:limit1}
	If $x_i \to x$ in $d_Z$, then for any $s<\t(x)$,
	\begin{align*}
\lim_{i \to \infty} d_{W_1}^{Z_s}(\nu_{x_i;s},\nu_{x;s})=0.
	\end{align*}
\end{cor}

We also have the following result.

\begin{lem}\label{takelimitdt}
	If $x_i,y_i\in Z_{t_i}$ and $x,y\in Z_t$ satisfy $x_i\to x,y_i\to y$ in $d_Z$, then
	\begin{align*}
		d^Z_{t}(x,y)\leq \liminf_{i\to\infty}d^Z_{t_i}(x_i,y_i).
	\end{align*}
\end{lem}
\begin{proof}
	For any $s<t$, since $x_i\to x,y_i\to y$ in $d_Z$, we have, by Corollary \ref{cor:limit1}, $\nu_{x_i;s}\to \nu_{x;s},\nu_{y_i;s}\to \nu_{y;s}$ in $d_{W_1}^{Z_s}$, and hence $\lim_{i\to\infty}d_{W_1} ^{Z_s}(\nu_{x_i;s},\nu_{y_i;s})=d_{W_1} ^{Z_s}(\nu_{x;s},\nu_{y;s})$. By Lemma \ref{lem:limitmono1} and Proposition \ref{prop:pastcont}, $d_{W_1} ^{Z_s}(\nu_{x_i;s},\nu_{y_i;s})\leq d^Z_{t_i}(x_i,y_i)$ for large $i$, and therefore
	\begin{align*}
		d_{W_1} ^{Z_s}(\nu_{x;s},\nu_{y;s})\leq 
		\liminf_{i\to\infty}d^Z_{t_i}(x_i,y_i),
	\end{align*}
	which, by Proposition \ref{prop:pastcont}, implies
	\begin{align*}
		d^Z_{t}(x,y)\leq 
		\liminf_{i\to\infty}d^Z_{t_i}(x_i,y_i).
	\end{align*}
\end{proof}

We end this section by proving the following result.

\begin{prop} \label{prop:dismatch1}
For all but countably many times $t \in \III^-$, we have
	\begin{align*}
d^Z_t=d_{g^Z_t}
	\end{align*}
on each connected component of $\RR_t$.
\end{prop}

\begin{proof}
We choose $\{t_k\}_{k \in \mathbb N}=(-(1-2\sigma)T, 0) \cap \mathbb Q$. For each $t_k$, it follows from Proposition \ref{prop:connectnumber} that each $\RR_{t_k}$ has at most countably many connected components, denoted by $U_{k,j}$. By Corollary \ref{cor:pre2}, there exist $z_{k,j} \in \RR$ such that
	\begin{align*}
U_{k,j}=\iota_{z_{k,j}}(\RR^{z_{k,j}}_{t_k}).
	\end{align*}
 For the associated metric flow $\XX^{z_{k,j}}$, it follows from \cite[Corollary 4.11]{bamler2023compactness} that there exists a countable set $J_{k,j} \subset \III$ such that $\XX^{z_{k,j}}$ is continuous at every time $t \notin J_{k,j}$. Thus, it follows from \cite[Equation (4.22)]{bamler2023compactness} and Definition \ref{defntimeslicedist} that
 	\begin{align} \label{eq:dism2}
d^Z_{t}=d_{g^Z_t}
	\end{align}
on $\iota_{z_{k,j}}(\RR^{z_{k, j}}_{t})$, for any $t \in [-(1-2\sigma)T, \t(z_{k,j})) \setminus J_{k,j}$.

We set $J=\bigcup_{k,j} J_{k,j} \bigcup \{0\} \bigcup \{-(1-2\sigma)T\}$, which is a countable set. For any $t \in \III \setminus J$ and any connected component $U$ of $\RR_t$, there exists $z \in Z$ such that $\iota_z(\RR^z_t)=U$. We choose $t_k \in (t, \t(z))$ and $j$ so that
 	\begin{align*}
\iota_z(\RR^z_{t_k})=U_{k,j}=\iota_{z_{k,j}}(\RR^{z_{k, j}}_{t_k}).
	\end{align*}
Thus, it follows from Lemma \ref{lem:pre2} that
 	\begin{align*}
\iota_z(\RR^z_{t})=\iota_{z_{k,j}}(\RR^{z_{k, j}}_{t}).
	\end{align*}
Consequently, it follows from \eqref{eq:dism2} that on $U$,
 	\begin{align*}
d^Z_{t}=d_{g^Z_t}.
	\end{align*}
\end{proof}

\section{Ricci shrinker spaces and tangent flows} \label{sec:tangent}

As in the preceding section, we consider a Ricci flow limit space $(Z, d_Z, p_{\infty},\t)$ obtained from 
	\begin{equation} \label{eq:conv001}
		(M_i \times \III, d^*_i, p_i^*,\t_i) \xrightarrow[i \to \infty]{\quad \mathrm{pGH} \quad} (Z, d_Z, p_{\infty},\t),
	\end{equation}
where $\XX^i=\{M_i^n,(g_i(t))_{t \in \III^{++}}\} \in \MM(n, Y, T)$ with base point $p_i^* \in \XX^i_\III$. 

First, we define the Nash entropy $\mathcal N_z(\tau)$ at a point $z \in Z_{\III^-}$, which is a direct generalization of Definition \ref{defnentropy}.

\begin{defn}[Nash entropy]
	For $z\in Z_{\III^-}$, we write $K(z; \cdot)=(4\pi(\t(z)-\t(\cdot)))^{-n/2}e^{-f_z(\cdot)}$, where $f_z\in C^\infty(\RR_{(-(1-2\sigma)T, \t(z))})$. Then the Nash entropy at $z$ is defined as
	\begin{align*}
		\NN_z(\tau):=\int_{\RR_{\t(z)-\tau}}f_z\,\mathrm{d}\nu_{z;\t(z)-\tau}-\frac{n}{2}
	\end{align*}
	for any $\tau \in (0, \t(z)+(1-2\sigma)T)$.
\end{defn}

\begin{lem}\label{lem:nashconv1}
Suppose $z_i^* \in M_i \times \III$ converge to $z \in Z_{\III^-}$ in the Gromov--Hausdorff sense. Then for any $\tau \in (0, \t(z)+(1-2\sigma)T)$,
	\begin{align*}
		\lim_{i \to \infty} \NN_{z_i^*}(\tau) =\NN_{z}(\tau).
	\end{align*}
\end{lem}

\begin{proof}
Suppose otherwise. There exist $\delta>0$ and a subsequence $\{i_j\}$ such that
	\begin{align} \label{eq:nashcont11}
		\abs{\NN_{z_{i_j}^*}(\tau) -\NN_{z}(\tau)} \ge \delta.
	\end{align}
After passing to a further subsequence if necessary, there exist a metric flow and a correspondence $\CF$ such that 
\begin{equation*}
	(\XX^{i_j}, (\nu_{z_{i_j}^*;t})_{t \in [-T, \t(z_{i_j}^*)]}) \xrightarrow[j \to \infty]{\quad \IF, \CF, J \quad} (\XX^z, (\nu_{z;t})_{t \in  [-T, \t(z)]}).
\end{equation*}
In particular, we have
		\begin{align*}
		z_{i_j}^* \xrightarrow[j \to \infty]{\quad \CF,J \quad} z.
	\end{align*}
For any $\tau \in (0, \t(z)+(1-2\sigma)T)$, since $\nu_{z;\t(z)-\tau}$ has full measure on $\iota_z(\RR^z_{\t(z)-\tau})$, we conclude from Lemma \ref{lem:well-define} that
	\begin{align*}
		\NN_z(\tau):=\int_{\iota_z(\RR^z_{\t(z)-\tau})}f_z\,\mathrm{d}\nu_{z;\t(z)-\tau}-\frac{n}{2}.
	\end{align*}
However, it follows from \cite[Theorem 2.10]{bamler2020structure} that
 	\begin{align*}
\lim_{j \to \infty} \NN_{z_{i_j}^*}(\tau)=\NN_z(\tau)
	\end{align*}
 which contradicts \eqref{eq:nashcont11}. 
\end{proof}

\begin{prop}\label{limitnash}
For any $z \in Z_{\III^-}$, $\NN_z(\tau)$ is nonincreasing for $\tau \in (0, \t(z)+(1-2\sigma)T)$. If $\NN_z(\tau)$ is constant for all admissible $\tau$, then $\XX^z_{(-(1-2\sigma)T, \t(z))}$ is a metric soliton. Moreover, on $\iota_z\lc \RR^z_{(-(1-2\sigma)T, \t(z))} \rc$,
	\begin{align*}
\Ric(g^Z)+\na^2 f_{z}=\frac{g^Z}{2 \tau_z},
	\end{align*}
	where $\tau_z=\t(z)-\t(\cdot)$.
\end{prop}

\begin{proof}
From the convergence \eqref{eq:conv001}, there exists $z_i^* \in M_i \times \III$ so that $z_i^* \to z$ in the Gromov--Hausdorff sense. Then, the fact that $\NN_z(\tau)$ is nonincreasing follows immediately from Lemma \ref{lem:nashconv1}.

If $\NN_z(\tau)$ is constant, then Proposition \ref{propNashentropy} and Lemma \ref{lem:nashconv1} give a sequence $\delta_i \to 0$ such that
	\begin{align*}
\int_{-(1-2\sigma)T+\delta_i}^{\t_i(z_i^*)-\delta_i} \int_{M_i} \abs {\Ric(g_i)+\na^2 f_{z_i^*}-\frac{g_i}{2(\t_i(z_i^*)-\t_i)}}^2 \, \mathrm{d}\nu_{z_i^*;t} \mathrm{d}t \le \delta_i.
	\end{align*}
	
By the smooth convergence in Theorem \ref{thm:smooth1}, we conclude that
	\begin{align*}
\Ric(g^Z)+\na^2 f_{z}=\frac{g^Z}{2 \tau_z}
	\end{align*}
holds on $\iota_z\lc \RR^z_{(-(1-2\sigma)T, \t(z))} \rc$. By the high codimension of the singular set of $\XX^z$, it can be proved (see \cite[Theorem 15.69]{bamler2020structure}) that $\XX^z_{(-(1-2\sigma)T, \t(z))}$ is a metric soliton.
\end{proof}

Next, we define

\begin{defn}[Curvature radius]\label{curvatureradiuslimit}
For any $z \in \RR$, the curvature radius $r_{\Rm}(z)$ is defined to be the supremum of all $r>0$ such that $B_{g^Z_t}(z,r)$ is relatively compact in $\RR_t$, and the product domain 
	\begin{align*}
B_{g^Z_t}(z,r)\times [\t(z)-r^2, \t(z)+r^2] \cap \III
	\end{align*}
is defined in $\RR$ with the curvature bound $|\Rm| \le r^{-2}$. Moreover, we set $r_{\Rm}=0$ on $Z\setminus \mathcal R$.
\end{defn}

The following lemma is immediate from Definition \ref{def:smoothcv} and Theorem \ref{thm:smooth1}.

\begin{lem} \label{lem:curcon}
Suppose that $z_i^* \in M_i \times \III$ converge to $z \in Z$ in the Gromov--Hausdorff sense and $\t(z)-r^2_{\Rm}(z) \in \III$. Then
	\begin{align*}
r_{\Rm}(z) = \lim_{i \to \infty} r_{\Rm}(z_i^*).
	\end{align*}
\end{lem}

Passing to the limit in Proposition \ref{pro:LiprRm} and using Lemma \ref{lem:curcon}, we immediately obtain the following result.

\begin{prop}\label{prop:LiprRmlimit}
For any $x,y\in Z$ with $\t(x)-r^2_{\Rm}(x) \in \III$ and $\t(y)-r^2_{\Rm}(y) \in \III$, we have
\begin{align*}
	|r_{\Rm}(x)-r_{\Rm}(y)|\leq C(n, Y, \sigma T) d_Z(x,y).
\end{align*}
\end{prop}

Next, we prove

\begin{prop}\label{prop:limitepsilon}
There exists a constant $\ep=\ep(n)>0$ such that if $\NN_z(r^2) \ge -\ep$, then
	\begin{align*}
r_{\Rm}(z) \ge \ep r.
	\end{align*}
\end{prop}

\begin{proof}
We choose $\ep=\ep_n/2$, where $\ep_n$ is the same constant as in Theorem \ref{epregularityNash}. From the convergence \eqref{eq:conv001}, there exists $z_i^* \in M_i \times \III$ so that $z_i^* \to z$ in the Gromov--Hausdorff sense. Then it follows from Lemma \ref{lem:nashconv1} that
	\begin{align*}
\NN_{z_i^*}(r^2) \ge -\ep_n
	\end{align*}
for large $i$. By Theorem \ref{epregularityNash}, we conclude that $r_{\Rm}(z_i^*) \ge \ep_n r$, which yields the conclusion by Lemma \ref{lem:curcon}.
\end{proof}

\begin{defn}[Tangent flow]\index{tangent flow}
	For any $z \in Z_{\III^-}$, a \textbf{tangent flow} at $z$ is a pointed parabolic metric space $(Z',d_{Z'},z',\t')$, which is a pointed Gromov--Hausdorff limit of $(Z, r_j^{-1} d_Z, z, r_j^{-2}(\t-\t(z)))$ for a sequence $r_j \searrow 0$. 
	\end{defn}

Suppose $(Z',d_{Z'},z',\t')$ is a tangent flow at $z$, which is obtained from the pointed Gromov--Hausdorff limit of $(Z, r_j^{-1} d_Z, z, r_j^{-2}(\t-\t(z)))$. Then $(Z',d_{Z'},z',\t')$ is a Ricci flow limit space over $\R$ or $\R_-$. Indeed, by the convergence \eqref{eq:conv001}, there exists $z_i^* \in M_i \times \III$ so that $z_i^* \to z$ in the Gromov--Hausdorff sense. For each $j$, we can find $i_j$ so that if we set
	\begin{align*}
& g'_j(t)=r_j^{-2} g_{i_j}\lc r_j^2t+\t_{i_j}(z_{i_j}^*) \rc, \quad \t_j':=r_j^{-2}(\t_{i_j}-\t_{i_j}(z^*_{i_j})), \notag\\
& T_j=r_j^{-2}(T+\t_{i_j}(z^*_{i_j})), \quad T_j'=-r_j^{-2} \t_{i_j}(z_{i_j}^*), \quad \III^{++}_j=[-T_j, T_j'], \notag\\
& \III_j^{+}=[-T_j+\sigma(T_j'+T_j), T_j'],\quad \III_j=[-T_j+2\sigma(T_j'+T_j), T_j'],
	\end{align*}
	then, after a time translation, $\{M_{i_j}, (g_j'(t))_{t \in \III_j^{++}}\} \in \MM(n, Y, T_j'+T_j)$. Thus, after passing to a subsequence, we have the following convergence (see Remark \ref{rem:general} and Notation \ref{not:2})
	\begin{align} \label{eq:rescaled}
(M_{i_j} \times \III_j, d^{\prime,*}_j, z_{i_j}^*,\t'_j)\xrightarrow[j \to \infty]{\quad \hat C^\infty \quad} (Z',d_{Z'},z',\t'),
	\end{align}
	where $d^{\prime,*}_j$ is the $d^*$-distance induced by $g_j'(t)$. Consequently, we conclude that $(Z',d_{Z'},z',\t')$ is a noncollapsed Ricci flow limit space over $\R$ if $\t(z) \in (-(1-2\sigma)T, 0)$ or over $\R_-$ if $\t(z)=0$.
	
As in Definition \ref{def:smoothcv}, we denote by $\RR'$ the set of points at which \eqref{eq:rescaled} is smooth. Then $\RR'$ is realized as a Ricci flow spacetime $(\RR', \t', \partial_{\t'}, g^{Z'}_t)$.

\begin{prop}\label{prop:004}
For any tangent flow $(Z',d_{Z'},z',\t')$ at $z \in Z_{\III^-}$, we have on $\RR'_{(-\infty,0)}$,
	\begin{align} \label{eq:riccishrinker}
\Ric(g^{Z'})+\na^2 f_{z'}=\frac{g^{Z'}}{2 \tau},
	\end{align}
where $\tau(\cdot)=-\t'(\cdot)$. Moreover, $\RR'_t$ is connected for any $t \in (-\infty,0)$.
\end{prop}

\begin{proof}
We assume that the convergence \eqref{eq:rescaled} holds. For any $\tau>0$, it follows from Lemma \ref{lem:nashconv1} that 
	\begin{align*}
\NN_{z'}(\tau)=\lim_{j \to \infty} \NN_{z} (\tau r_j^2).
	\end{align*}
Since the last limit is independent of $\tau$, we conclude that $\NN_{z'}(\tau)$ is constant for $\tau>0$. Thus, by Proposition \ref{limitnash}, we have on $\iota_{z'}(\RR^{z'}_{(-\infty,0)})$, 
	\begin{align}\label{eq:tang0}
\Ric(g^{Z'})+\na^2 f_{z'}=\frac{g^{Z'}}{2 \tau}.
	\end{align}
To finish the proof, we only need to prove that $\RR'_t$ is connected for any $t \in (-\infty,0)$. 

Suppose $\RR'_{t_0}$ is disconnected for some $t_0\in (-\infty,0)$. We fix $x_0 \in \iota_{z'}(\RR^{z'}_{t_0})$ and $y_0 \in \RR'_{t_0}$ so that $y_0$ lies in a different connected component than $x_0$. By Corollary \ref{cor:connected}, there exists a curve $\gamma(s)$, $s \in [0, L]$, contained in $\RR'$ such that $\gamma(0)=z_0 \in \iota_{z'}(\RR^{z'}_{t_1})$ for some $t_1<t_0$, and $\gamma(L)=y_0$. Then we set
	\begin{align*}
s_0=\sup\{s' \in [0, L] \mid \gamma(s) \in \iota_{z'}(\RR^{z'}_{\t'(\gamma(s))}) \text{ for any } s \in [0, s'] \}
	\end{align*}
and $t_2:=\t'(\gamma(s_0))$. It is clear that $t_2 \in (t_1, t_0]$ and
	\begin{align} \label{eq:tang1}
\lim_{s \nearrow s_0}f_{z'}(\gamma(s))=+\infty.
	\end{align}
On the other hand, it follows from \eqref{eq:tang0} that 
	\begin{align} \label{eq:tang2}
\partial_{t'} f_{z'}=|\na f_{z'}|^2 \quad \text{and} \quad f_{z'}-\tau(|\na f_{z'}|^2+\scal_{g^{Z'}})=\NN_{z'}(1).
	\end{align}
	Since $\scal_{g^{Z'}} \ge 0$, we conclude from \eqref{eq:tang2} that
	\begin{align*}
\abs{\diff{}{s} f_{z'}(\gamma(s))} \le C_0(f_{z'}(\gamma(s))+1)
	\end{align*}	
	for some constant $C_0>0$ and any $s \in [0, s_0)$.
By integration, we conclude that $\lim_{s \nearrow s_0}f_{z'}(\gamma(s))$ must be finite, which contradicts \eqref{eq:tang1}.

Consequently, we have proved that $\RR'_t$ is connected for any $t \in (-\infty,0)$, and hence $\RR'_t=\iota_{z'}(\RR^{z'}_t)$ for any $t \in (-\infty,0)$.
\end{proof}

By Proposition \ref{limitnash} and Proposition \ref{prop:004}, the metric flow $\XX^{z'}$ associated with $z'$ is a metric soliton so that $\iota_{z'}(\RR^{z'}_t)=\RR'_t$ for any $t<0$. Moreover, since any metric soliton is continuous, by Remark \ref{rem:pastc}, the map $\iota_{z'}$ is injective on $\XX^{z'}$. Thus, by Theorem \ref{thm:idenproof}, $\iota_{z'}$ is an isometric embedding from $\XX^{z'}$ to $Z'$. Moreover, since $\XX^{z'}$ is continuous, it follows from Definition \ref{defntimeslicedist} that the following result holds.

\begin{cor}\label{cor:agree2}
For any $t<0$, the metric $d^{Z'}_t$, when restricted to $\RR'_t$, agrees with the Riemannian distance $d_{g^{Z'}_t}$.
\end{cor}

\begin{defn}[Tangent metric soliton] \label{def:tms}\index{tangent metric soliton}
The metric soliton $\XX^{z'}$ is called a \textbf{tangent metric soliton} at $z$.
\end{defn}

We have the following fundamental estimates for $f_{z'}$, which are well-known for smooth Ricci shrinkers. The estimates can be proved in the same way as \cite[Theorem 1.1]{CMZ24}. Note that the lower bound is improved due to Theorem \ref{heatkernelupperbdgeneral} (ii).

\begin{lem}\label{lem:metrices1}
For any tangent flow $(Z',d_{Z'},z',\t')$ at $z \in Z_{\III^-}$, we have for any $x \in \RR'_{(-\infty,0)}$,
	\begin{align*}
\frac{ d^2_{g^{Z'}_{\t'(x)}}(x, p_{\t'(x)})}{(4+\ep)\tau(x) }-C(n, \ep) \le f_{z'}(x)-\NN_{z'}(1) \le \frac{1}{4 \tau(x)}\lc d_{g^{Z'}_{\t'(x)}}(x, p_{\t'(x)})+C(n)\sqrt{\tau(x)} \rc^2,
	\end{align*}
	where $p_{-1} \in \RR'_{-1}$ is a regular $H_n$-center of $z'$ and $p_t \in \RR'_t$ is the flow of $\partial_{\t'}-\na f_{z'}$ from $p_{-1}$.
\end{lem}

We also need the following local noncollapsing theorem in \cite[Lemma 8.1]{CMZ24}, which was originally proved in \cite[Theorem 22]{li2020heat} for smooth Ricci shrinkers. 

\begin{lem}\label{no-local-collapsing}
Let $(Z',d_{Z'},z',\t')$ be a tangent flow at $z\in Z_{\III^-}$. For any $x\in\RR'_t$ with $t<0$, if $\scal_{g^{Z'}_t} \le \Lambda r^{-2}$ on $B_{g^{Z'}_t}(x, r)$, then
	\begin{align*}
\abs{B_{g^{Z'}_t}(x, r)}_t \ge c(n, Y,\Lambda) r^n>0.
	\end{align*}
\end{lem}

\begin{defn}[Regular and singular sets]\label{defnregularsingular}
For the Ricci flow limit space $(Z,d_Z,\t)$, a point $z\in Z_{(-T,0)}$ is \textbf{regular} if there exists a tangent flow at $z$ that is isometric (see Definition \ref{def:iso}) to $(\R^{n} \times \R,d_E^*,(\vec 0, 0),\t)$, where $d_E^*$ denotes the induced $d^*$-distance on $\R^{n} \times \R$ (see Example \ref{ex:euclidean}). Similarly, a point $z\in Z_0$ is \textbf{regular} if there exists a tangent flow at $z$ that is isometric to $(\R^{n} \times \R_{-},d_E^*, (\vec 0,0),\t)$. Any point in $Z_{\III^-}$ that is not regular is called \textbf{singular}.
\end{defn}

\begin{thm}\label{thm:tworegular}
Let $\RR^* \subset Z_{\III^-}$ denote the set of regular points. Then $\RR^*=\RR_{\III^-}$.
\end{thm}
\begin{proof}
Let $z \in \RR_{\III^-}$, and suppose that $(Z',d_{Z'},z',\t')$ is a tangent flow at $z$. By Lemma \ref{lem:curcon}, we conclude that $z' \in \RR'$ and $r_{\Rm}(z')=+\infty$. By Definition \ref{curvatureradiuslimit}, this implies that $(\RR', g^{Z'}_t)$ is given by the conventional Ricci flow $(\R^n \times \R, g_E)$ or $(\R^{n} \times \R_{-}, g_E)$ so that $z'$ corresponds to $(\vec 0, 0)$. By Proposition \ref{equivalenceofballs}, both $\R^n \times \R$ and $\R^{n} \times \R_{-}$, when equipped with $d_E^*$, are complete. Thus, we conclude that $(Z',d_{Z'},z',\t')$ is isometric to $(\R^{n} \times \R, d_E^* ,(\vec 0, 0),\t)$ or $(\R^{n} \times \R_{-}, d_E^* , (\vec 0,0),\t)$.

Conversely, if $z \in \RR^*$, then, by Lemma \ref{lem:nashconv1}, we can find a small $r>0$ such that $\NN_z(r^2) \ge -\ep$, where $\ep$ is the same constant in Proposition \ref{prop:limitepsilon}. Then, from Lemma \ref{lem:curcon} and Proposition \ref{prop:limitepsilon} we obtain $z \in \RR$.
\end{proof}

Definition \ref{defnregularsingular} and Theorem \ref{thm:tworegular} give rise to the following regular--singular decomposition:
	\begin{align} \label{decomRS2}
Z_{\III^-}=\RR_{\III^-} \sqcup \MS,
	\end{align}
	where $\MS$\index{$\MS$} denotes the set of singular points.

Next, we introduce a class that contains all tangent flows.

\begin{defn}[Ricci shrinker space] \label{def:rss}\index{Ricci shrinker space}
A pointed parabolic metric space $(Z',d_{Z'},z',\t')$ with $\t'(z')=0$ is called an $n$-dimensional \textbf{Ricci shrinker space} with entropy bounded below by $-Y$ if it satisfies $\R_- \subset \mathrm{image}(\t')$ and arises as the pointed Gromov--Hausdorff limit of a sequence of Ricci flows in $\MM(n, Y, T_i)$ with $T_i\to +\infty$ (see Remark \ref{rem:general}). Moreover, $\NN_{z'}(\tau)$ remains constant for all $\tau>0$.
	\end{defn}

As above, we denote by $\RR'$ the regular set, which is realized as a Ricci flow spacetime $(\RR', \t', \partial_{\t'}, g^{Z'}_t)$. With identical proofs, one can show Proposition \ref{prop:004}, Corollary \ref{cor:agree2}, Lemma \ref{lem:metrices1} and Lemma \ref{no-local-collapsing} also hold for Ricci shrinker spaces.

We make the following definitions:
\begin{defn}[Static/quasi-static cone] \label{def:stacone}\index{static cone} \index{quasi-static cone}
Let $(Z',d_{Z'},z',\t')$ be a Ricci shrinker space.
\begin{itemize}
	
\item It is called a \textbf{static cone} if the Ricci curvature vanishes on $\RR'_{-1}$ and the \textbf{arrival time} \index{arrival time}
	\begin{align} \label{eq:arrivaltime}
t_a:=\sup \{\t'(x) \mid x \in \mathrm{spine}(Z')\}=+\infty.
	\end{align}

\item It is called a \textbf{quasi-static cone} if the Ricci curvature vanishes on $\RR'_{-1}$ and $t_a<+\infty$.
\end{itemize}

The definition of the spine and its properties are provided in Appendix \ref{app:D}.
\end{defn}

Note that since $(\RR'_{(-\infty, 0)}, g^{Z'}_t)$ is self-similar, the Ricci curvature vanishes on $\RR'_{(-\infty, 0)}$ for a static cone.

\begin{defn}[Noncollapsed and collapsed Ricci shrinker space] \label{def:ncf}
	A Ricci shrinker space $(Z',d_{Z'},z',\t')$ is called \textbf{noncollapsed}\index{noncollapsed} if for some base point $p \in \RR'_{-1}$,
	\begin{align}\label{defnnoncollapse}
		\liminf_{r\to\infty}\frac{\abs{\RR'_{-1}\bigcap B_{g^{Z'}_{-1}}(p,r)}_{-1}}{r^n}>0.
	\end{align}
	Otherwise, $(Z',d_{Z'},z',\t')$ is called \textbf{collapsed}\index{collapsed}.
\end{defn}

Proposition \ref{prop:volume}(i) and Lemma \ref{no-local-collapsing} imply that every static cone is noncollapsed and that its asymptotic volume ratio lies in $[C(n, Y)^{-1}, C(n, Y)]$ for a constant $C(n, Y)>1$.

\begin{thm}\label{metricsolitonbdscalcomplete}
Suppose that a Ricci shrinker space $(Z',d_{Z'},z',\t')$ has uniformly bounded scalar curvature on $\RR'_{-1}$. Then
	\begin{align*}
\iota_{z'} \lc \XX^{z'}_{(-\infty, 0)} \rc=Z'_{(-\infty, 0)}.
	\end{align*}
\end{thm}

\begin{proof}
By the self-similarity of $(\RR'_{(-\infty, 0)}, g^{Z'}_t)$, we conclude that the scalar curvature is uniformly bounded on $\RR'_J$ for any compact interval $J \subset (-\infty, 0)$.

It remains to show that $\iota_{z'}(\XX^{z'}_J)$ is complete with respect to $d_{Z'}$, for any compact interval $J \subset (-\infty, 0)$. In other words, we need to prove that $\XX^{z'}_J$ is complete with respect to $d^*_{z'}$ (see Definition \ref{defn:dstar-limit}). In the following, we use $d_t$ to denote the distance function at time $t$ on $\XX^{z'}_t$.

Consider a Cauchy sequence $x_i\in \XX^{z'}_J$. Since $\RR^{z'}_J$ is dense in $\XX^{z'}_J$, we may assume that all $x_i \in \RR^{z'}_J$. Since $\sqrt{|\t^{z'}(x)-\t^{z'}(y)|} \le d_{z'}^*(x, y)$, after passing to a subsequence, we assume that $\t^{z'}(x_i) \to t_0 \in J$. Since $x_i$ is a Cauchy sequence, for any $t<t_0$,
	\begin{align}\label{complsoliton2}
		\lim_{i,j\to\infty}d_{W_1}^{\XX^{z'}_{t}}(\nu_{x_i;t},\nu_{x_j;t})=0.
	\end{align}
We fix a time $t_1<t_0$ to be determined later and set $z_i \in \XX^{z'}_{t_1}$ to be an $H_n$-center of $x_i$. By \eqref{complsoliton2}, we have for large $i,j$,
	\begin{align*}
	\lim_{i,j\to\infty}	d_{t_1}(z_i,z_j)\leq \lim_{i,j\to\infty} \lc d_{W_1}^{\XX^{z'}_{t_1}}(\nu_{x_i;t_1},z_i)+d_{W_1}^{\XX^{z'}_{t_1}}(\nu_{x_j;t_1},z_j)+d_{W_1}^{\XX^{z'}_{t_1}}(\nu_{x_i;t_1},\nu_{x_j;t_1}) \rc \leq 2\sqrt{H_n(t_0-t_1)}.
	\end{align*}
In particular, $\{z_i\}$ is uniformly bounded with respect to $d_{t_1}$.

Because the scalar curvature is uniformly bounded on $\RR'_J$ for any compact interval $J \subset (-\infty, 0)$, we conclude from Lemma \ref{no-local-collapsing} that 
	\begin{align}\label{complsoliton2a}
	|B_{\t^{z'}(x_i)}(x_i, 1)|_{\t^{z'}(x_i)} \ge c_0>0
	\end{align}
for a constant $c_0$. Then, it follows from \eqref{complsoliton2a} and \cite[Theorem 2.31, Lemma 15.27 (a)]{bamler2020structure} (see also Corollary \ref{cor:coverslice}) that we can find $y_i \in B_{\t^{z'}(x_i)}(x_i, 1) \cap \RR^{z'}$ such that
	\begin{align}\label{complsoliton2b}
r_{\Rm}(y_i) \ge c_1>0
	\end{align}
for a constant $c_1$.

Next, we fix $t_1$ so that $t_0-t_1 \le c_1^2/10$. Then it follows from \eqref{complsoliton2b} and Lemma \ref{lem:center1} that the $H_n$-center of $y_i$ in $\XX^{z'}_{t_1}$, denoted by $w_i$, satisfies
	\begin{align*}
d_{t_1}(w_i, y_{i,t_1}) \le C_2
	\end{align*}
for a constant $C_2>0$, where $y_{i,t} \in \XX^{z'}_t$ denotes the flow line of $\partial_{\t^{z'}}$ from $y_i$. Moreover, it follows from the monotonicity that
	\begin{align*}
d_{t_1}(w_i, z_i) \le d_{W_1}^{\XX^{z'}_{t_1}}(\nu_{x_i;t_1},\nu_{y_i; t_1})+2\sqrt{H_n(\t^{z'}(x_i)-t_1)} \le 1+2\sqrt{H_n(\t^{z'}(x_i)-t_1)}.
	\end{align*}
Thus, we conclude that $\{w_i\}$ and hence $\{y_{i,t_1}\}$ are uniformly bounded with respect to $d_{t_1}$. In particular, Lemma \ref{lem:metrices1} shows that $f_{z'}(y_{i,t_1})$ is uniformly bounded in $i$.

On the other hand, since $\partial_{\t'} f_{z'}=|\na f_{z'}|^2 \le \tau^{-1}(f_{z'}+C(n, Y))$, we obtain that $f_{z'}(y_i)$ is uniformly bounded in $i$. Now, we set $x_{i;s}\in \XX^{z'}_s$ to be the flow line of $\partial_{\t^{z'}}-\na f_{z'}$ from $x_i$ and define $x_i' =x_{i;t_0}$. By the definition of a metric soliton, we conclude that $x_i' \in \RR^{z'}_{t_0}$, and $\{x_i'\}$ are uniformly bounded with respect to $d_{t_0}$. Since $(\XX^{z'}_{t_0},d_{t_0})$ is complete, after passing to a subsequence we may assume that $x_i'$ converges to $x_{\infty}$ in $d_{t_0}$.

We first assume $\t^{z'}(x_i) \ge t_0$. It is clear from Lemma \ref{lem:metrices1} that for any $s \in [t_0, \t^{z'}(x_i)]$,
	\begin{align}\label{equ:RSS1}
|\na f_{z'}|^2(x_{i;s})+\scal_{g^{z'}}(x_{i;s}) \le C_3
	\end{align}
for a constant $C_3$. On the other hand, the heat kernel satisfies:
	\begin{align}\label{equ:RSS2}
K_{Z'}(x_i;x_i')
&\ge \frac{1}{\lc 4\pi|t_0-\t^{z'}(x_i)| \rc^{\frac n 2}} \notag\\
&\quad {}\times \exp \lc-\frac{1}{2\sqrt{\t^{z'}(x_i)-t_0}}
\int_{t_0}^{\t^{z'}(x_i)} \sqrt{\t^{z'}(x_i)-s}
\lc |\na f_{z'}|^2(x_{i;s})+\scal_{g^{z'}}(x_{i;s}) \rc \,\mathrm{d}s \rc.
\end{align}
Indeed, this estimate follows from the corresponding estimate for the closed Ricci flow (see \cite[Corollary 9.4]{perelman2002entropy}) and the smooth convergence.

Therefore, if we denote an $H_n$-center in $\XX^{z'}_{t_0}$ of $x_i$ by $z'_i$, then it follows from Theorem \ref{thm:upper1}, \eqref{equ:RSS1} and \eqref{equ:RSS2} that 
	\begin{align*}
d_{t_0}(x_i', z_i') \le C_4 (\t^{z'}(x_i)-t_0),
	\end{align*}
and hence by Proposition \ref{distancefunction1} (1) and Lemma \ref{lem:004abc} that
	\begin{align}\label{complsoliton2e}
d_{z'}^*(x_i, x_i') \le C_5 \sqrt{\t^{z'}(x_i)-t_0}.
	\end{align}
If $\t^{z'}(x_i) < t_0$, we can also obtain \eqref{complsoliton2e} in a similar way.

Now, it follows from Proposition \ref{distancefunction1} (1) and \eqref{complsoliton2e} that $x_i \to x_{\infty}$ in $d_{z'}^*$. This completes the proof.
\end{proof}

\begin{rem}
By the same argument as in the proof of Theorem \ref{metricsolitonbdscalcomplete}, one can show that for any Ricci shrinker space $(Z',d_{Z'},z',\t')$, 
	\begin{align*}
d^{Z'}_t (x, y)=+\infty
	\end{align*}
	for any $t<0$, whenever $x \in \iota_{z'} \lc \XX^{z'}_{t} \rc$ and $y \in Z'_{t} \setminus \iota_{z'} \lc \XX^{z'}_{t} \rc$. In general, we conjecture that the conclusion of Theorem \ref{metricsolitonbdscalcomplete} remains valid even without the assumption on the scalar curvature.
\end{rem}

Theorem \ref{metricsolitonbdscalcomplete} applies, in particular, to static or quasi-static cones. In fact, we have the following characterization.

\begin{thm}\label{staticcone}
Let $(Z',d_{Z'},z',\t')$ be a Ricci shrinker space that is a static or quasi-static cone. Then $(\RR'_{(-\infty, t_a]},g^{Z'})$ is isometric to $(\RR'_{-1} \times (-\infty, t_a])$, where $t_a$ is defined in \eqref{eq:arrivaltime}.
\end{thm}

\begin{proof}
By Lemma \ref{lem:realstatic}, we know that $\Ric \equiv 0$ on $\RR'_{(-\infty, t_a]}$. Now, it follows from \cite[Theorem 2.16, Theorem 15.60, Claim 22.7]{bamler2020structure} that 
	\begin{align*}
\partial_{\t',x}K_{Z'}(x;y)+\partial_{\t',y} K_{Z'}(x;y)=0
	\end{align*}
	for $x, y \in \RR'_{(-\infty, t_a]}$. In other words, if we denote the flow induced by $\partial_{\t'}$ by $\boldsymbol{\varphi}^t$, then $K_{Z'}(x;y)=K_{Z'}(\boldsymbol{\varphi}^t(x);\boldsymbol{\varphi}^t(y))$ for any $x, y, \boldsymbol{\varphi}^t(x), \boldsymbol{\varphi}^t(y) \in \RR'_{(-\infty, t_a]}$. Thus, one can follow the same argument as in the proof of \cite[Theorem 15.60]{bamler2020structure} to show that the Nash entropy $\NN_{\boldsymbol{\varphi}^t(x)}(\tau)$ is constant as long as $\boldsymbol{\varphi}^t(x) \in \RR'$. Proposition \ref{prop:limitepsilon} then implies that $\boldsymbol{\varphi}^t(x) \in \RR'_{(-\infty, t_a]}$ for any $t \in (-\infty, t_a-\t'(x)]$ as long as $x \in \RR'_{(-\infty, t_a]}$.

Consequently, we conclude that $(\RR'_{(-\infty, t_a]},g^{Z'})$ is isometric to $(\RR'_{-1} \times (-\infty, t_a], g^{Z'}_{-1})$.
\end{proof}

Combining Corollary \ref{cor:agree2}, Theorem \ref{metricsolitonbdscalcomplete} and Theorem \ref{staticcone}, we have

\begin{cor}
With the above assumptions, for any $t \in (-\infty, t_a]$, the distance induced by $g^{Z'}_t$ on $\RR'_t$ agrees with $d^{Z'}_t$. Moreover, $(Z'_t, d^{Z'}_t)$ is the completion of $(\RR'_t, g^{Z'}_t)$.
\end{cor}

In the setting of Theorem \ref{staticcone}, there exists a flow induced by $\partial_{\t'}$ on $\RR'$. More precisely, for any $x \in \RR'_{(-\infty, t_a]}$, we define $\boldsymbol{\varphi}^t(x) \in \RR'_{\t'(x)+t}$ to be the flow line of $\partial_{\t'}$ from $x$, where $t \in (-\infty, t_a-\t'(x)]$.

\begin{prop}\label{prop:staticcone1}
With the above assumptions, $\boldsymbol{\varphi}^t$ can be defined on $Z'$ so that the following statements hold.
	\begin{enumerate}[label=\textnormal{(\roman{*})}]
\item For any $x, y \in Z'_{(-\infty, t_a]}$, $d_{Z'}(x, y)=d_{Z'}(\boldsymbol{\varphi}^t(x), \boldsymbol{\varphi}^t(y))$ for all $t \in (-\infty, t_a-\max\{\t'(x), \t'(y)\}]$.

\item For any $x, y \in Z'_s$, $d^{Z'}_s(x, y)=d^{Z'}_{s+t}(\boldsymbol{\varphi}^t(x), \boldsymbol{\varphi}^t(y))$ for all $s \le t_a$ and $t \in (-\infty, t_a-s]$.

\item For any $x \in Z'_{(-\infty, t_a]}$ and $\tau>0$, $\NN_x(\tau)=\NN_{\boldsymbol{\varphi}^t(x)}(\tau)$ for all $t \in (-\infty, t_a-\t'(x)]$.
	\end{enumerate}	
\end{prop}

\begin{proof}
As in the proof of Theorem \ref{staticcone}, we have 
	\begin{align} \label{eq:con00a}
K_{Z'}(x;y)=K_{Z'}(\boldsymbol{\varphi}^t(x);\boldsymbol{\varphi}^t(y))
	\end{align}
for any $x, y \in \RR'_{(-\infty, t_a]}$ and $t \in (-\infty, t_a-\max\{\t'(x), \t'(y)\}]$. Thus, by Definition \ref{defntimeslicedist}, we conclude that
	\begin{align} \label{eq:con1}
d^{Z'}_s(x, y)=d^{Z'}_{s+t}(\boldsymbol{\varphi}^t(x), \boldsymbol{\varphi}^t(y))
	\end{align}
for any $x, y \in \RR'_s$ with $s \le t_a$ and $t \le t_a-s$. Using the variational definition of the spacetime distance together with the invariance of the conjugate heat kernels, this implies
	\begin{align} \label{eq:con2}
d_{Z'}(x, y)=d_{Z'}(\boldsymbol{\varphi}^t(x), \boldsymbol{\varphi}^t(y))
	\end{align}
for any $x, y \in \RR'_{(-\infty, t_a]}$ and $t \in (-\infty, t_a-\max\{\t'(x), \t'(y)\}]$. 

Next, for any $w \in Z'_{(-\infty, t_a]}$, we choose a sequence $w_i \in \RR'_{(-\infty, \t(w)]}$ such that $w_i \to w$ in $d_{Z'}$. Then, for any $t \in (-\infty, t_a-\t'(w)]$, $\{\boldsymbol{\varphi}^t(w_i)\}$ is a Cauchy sequence by \eqref{eq:con2} with respect to $d_{Z'}$. We define
	\begin{align*}
\boldsymbol{\varphi}^t(w)=\lim_{i \to \infty} \boldsymbol{\varphi}^t(w_i).
	\end{align*}
It is clear that the definition of $\boldsymbol{\varphi}^t(w)$ is independent of the choice of $\{w_i\}$. Moreover, it follows from \eqref{eq:kernellimit} and \eqref{eq:con00a} that
	\begin{align} \label{eq:con00b}
K_{Z'}(x;y)=K_{Z'}(\boldsymbol{\varphi}^t(x);\boldsymbol{\varphi}^t(y))
	\end{align}
for any $x \in Z'_{(-\infty, t_a]}$, $y \in \RR'_{(-\infty, \t'(x))}$ and $t \in (-\infty, t_a-\t'(x)]$.

(i): This follows from \eqref{eq:con2} by taking the limit.

(ii): For any $x, y \in Z'_s$ with $s \le t_a$, since $\RR'_{s'}$ is connected for any $s'<s$, the argument of Lemma \ref{takelimitdt1} gives a sequence $s_i \nearrow s$ such that if $x_i, y_i \in \RR'_{s_i}$ are regular $H_n$-centers of $x$ and $y$, respectively, then
	\begin{align} \label{eq:con3}
d^{Z'}_{s}(x, y)=\lim_{i \to \infty} d_{s_i}^{Z'}(x_i, y_i).
	\end{align}
Since $\boldsymbol{\varphi}^t(x_i)$ and $\boldsymbol{\varphi}^t(y_i)$ converge to $\boldsymbol{\varphi}^t(x)$ and $\boldsymbol{\varphi}^t(y)$, respectively, it follows from Lemma \ref{takelimitdt} that
	\begin{align*}
\liminf_{i \to \infty} d_{s_i+t}^{Z'}(\boldsymbol{\varphi}^t(x_i), \boldsymbol{\varphi}^t(y_i)) \ge d^{Z'}_{s+t}(\boldsymbol{\varphi}^t(x), \boldsymbol{\varphi}^t(y)).
	\end{align*}
Combining this with \eqref{eq:con1} and \eqref{eq:con3}, we obtain
	\begin{align*}
d^{Z'}_{s}(x, y) \ge d^{Z'}_{s+t}(\boldsymbol{\varphi}^t(x), \boldsymbol{\varphi}^t(y)).
	\end{align*}
The reverse inequality also holds since $\boldsymbol{\varphi}^t$ is the inverse map of $\boldsymbol{\varphi}^{-t}$.

(iii): This is immediate from \eqref{eq:con00b} and Theorem \ref{staticcone}.
\end{proof}

Next, we prove the following bi-Lipschitz estimate.

\begin{lem}\label{lem:bilip1}
With the above assumptions, for any $x \in Z'_{(-\infty, t_a]}$ and $t \in (-\infty, t_a-\t'(x)]$,
	\begin{align*}
|t|^{\frac 1 2} \le d_{Z'}(x, \boldsymbol{\varphi}^t(x)) \le C(n, Y) |t|^{\frac 1 2}.
	\end{align*}
\end{lem}

\begin{proof}
The first inequality is immediate, so we focus on proving the second. Without loss of generality, assume $t<0$ and $x \in \RR'_{(-\infty, t_a]}$. The general case follows by approximation.

Since $\Ric(g^{Z'})=0$, it follows from Lemma \ref{lem:center1} that $\boldsymbol{\varphi}^t(x)$ is a regular $H$-center of $x$ for some $H=H(n, Y)>0$. Therefore, by Lemma \ref{lem:004abc}, we obtain
	\begin{align*}
d_{Z'}(x, \boldsymbol{\varphi}^t(x)) \le C(n, Y) \sqrt{|t|},
	\end{align*}
which completes the proof.
\end{proof}

For general Ricci shrinker spaces, we have the following result.

\begin{thm}\label{dichotomy}
Let $(Z',d_{Z'},z',\t')$ be a collapsed Ricci shrinker space. Then $\mathrm{image}(\t')=\R_-$.
\end{thm}

\begin{proof}
Suppose $(Z',d_{Z'},z',\t')$ is collapsed and $Z'_{(0,\infty)}$ is nonempty. We fix a point $q \in \RR'_0$ with $r_{\Rm}(q) \ge \delta>0$. In particular, there exists a product domain $B_{g^{Z'}_0}(q, \delta) \times [-\delta^2, 0] \subset \RR'$ on which the curvature is bounded by $\delta^{-2}$.

We choose a sequence $t_i \nearrow 0$ and define $q_i \in \RR'_{t_i}$ as the flow of $\partial_{\t'}$ from $q$. By distance comparison, we obtain $B_{g^{Z'}_{t_i}}(q_i, \delta/2) \subset B_{g^{Z'}_0}(q, \delta)$ for sufficiently large $i$.

Next, we set $q'_i \in \RR'_{-1}$ to be the flow of $\tau(\partial_{\t'}-\na f_{z'})$ from $q_i$. It is clear from our construction that $d_{g^{Z'}_{-1}}(q'_i,p_{-1})$ is uniformly bounded, where $p_{-1}$ is a regular $H_n$-center of $z'$. From the Ricci shrinker equation \eqref{eq:riccishrinker}, the flow of $\tau(\partial_{\t'}-\na f_{z'})$ from $\RR'_{t_i}$ to $\RR'_{-1}$ is an isometry with respect to the metrics $|t_i|^{-1}g^{Z'}_{t_i}$ and $g^{Z'}_{-1}$. Thus, we conclude that
	\begin{align*}
|\Rm(g^{Z'})| \le |t_i| \delta^{-2}
	\end{align*}
on $B_{g^{Z'}_{-1}}(q'_i, |t_i|^{-1/2}\delta/2)$. Combined with Lemma \ref{no-local-collapsing}, we conclude that \eqref{defnnoncollapse} holds. However, this contradicts our assumption.

This completes the proof.
\end{proof}	

In the special case of tangent flows, we have

\begin{thm}\label{dichotomytang}
Let $(Z',d_{Z'},z',\t')$ be a tangent flow at a point $z \in Z_{(-(1-2\sigma)T, 0)}$, where $(Z, d_Z, p_{\infty},\t)$ is the noncollapsed Ricci flow limit space from \eqref{eq:conv001}. Then $\mathrm{image}(\t')=\R$ if $(Z',d_{Z'},z',\t')$ is noncollapsed, and $\mathrm{image}(\t')=\R_-$ if collapsed.
\end{thm}

\begin{proof}
The collapsed case follows from Theorem \ref{dichotomy}, so we focus on the noncollapsed case.

Suppose $(Z',d_{Z'},z',\t')$ is noncollapsed. It follows from Theorem \ref{thm:spacemink} (iii) (see also \cite[Corollary 6.24]{li2024heat}) that
	\begin{align} \label{eq:dic1}
\int_{B_{Z'_{-1}}(p, r)} r^{-4+\ep}_{\Rm}  \, \mathrm{d}V_{g^{Z'}_{-1}} \le C(n, Y, \ep) r^{n-2+\ep}
	\end{align}
for every sufficiently small $\ep>0$ and any $r \ge 1$, where $p \in \RR'_{-1}$ is a regular $H_n$-center of $z'$. By \eqref{defnnoncollapse} and \eqref{eq:dic1}, there exists a sequence $r_i \to \infty$ such that
	\begin{align*}
C_0 r_i^n \delta_i^{-4+\ep} \le C(n, Y, \ep) r_i^{n-2+\ep},
	\end{align*}
for a constant $C_0>0$, where $\delta_i=\sup_{B_{Z'_{-1}}(p, r_i)} r_{\Rm}$. Thus, we conclude that there exists a sequence $x_i\in\RR'_{-1}$ such that $r_{\Rm}(x_i) \to +\infty$. 

Suppose that $(Z',d_{Z'},z',\t')$ is obtained as in \eqref{eq:rescaled}. We can find $q_{k,i_j}^* \in M_{i_j} \times \III_j$ such that $q_{k,i_j}^*$ converges to $x_k$ in the Gromov--Hausdorff sense, as $j\to\infty$. By Lemma \ref{lem:curcon}, we conclude that
\begin{align*}
	r_{\Rm}(q_{k,i_j}^*) \ge \frac{1}{2}r_{\Rm}(x_k),
\end{align*}
for sufficiently large $j$. Since $T_j'\to\infty$ as $j\to\infty$ and $r_{\Rm}(x_k)$ can be arbitrarily large, we conclude that $\RR'_t$ is nonempty for any $t \in \R$. In particular, $\mathrm{image}(\t')=\R$.
\end{proof}	

\begin{rem}
By the same proof, the conclusion of Theorem \ref{dichotomytang} also holds for Ricci shrinker spaces obtained from the ancient blow-down procedure described in Remark \ref{rem:general}.
\end{rem}

Let $(Z',d_{Z'},z',\t')$ be a Ricci shrinker space. We define the flow $\boldsymbol{\psi}^s$ on $\RR'_{(-\infty, 0)}$ generated by $X=\tau(\partial_{\t'}-\na f_{z'})$ with $\boldsymbol{\psi}^0=\mathrm{id}$. \cite[Theorem 15.69]{bamler2020structure} proves that $\boldsymbol{\psi}^s(x) \in \RR'_{(-\infty, 0)}$ if $x \in \RR'_{(-\infty, 0)}$.

We first prove:

\begin{lem}\label{selfsimilarityshrinker}
For any $x, y \in \RR'_{(-\infty, 0)}$ and $s \in \R$, we have
	\begin{align}\label{selfsimilarity1}
		d_{Z'}(\boldsymbol{\psi}^s(x), \boldsymbol{\psi}^s(y))=e^{-\frac{s}{2}}d_{Z'}(x,y).
	\end{align} 
\end{lem}

\begin{proof}
First, we have
	\begin{align}\label{self1}
\mathcal L_X g^{Z'}=\tau \lc \mathcal L_{\partial_{\t'}} g^{Z'}-2\na^2 f_{z'} \rc=-g^{Z'}.
	\end{align}
Moreover, it is clear that for any $x \in \RR'_{(-\infty, 0)}$,
	\begin{align}\label{self2}
\t'(\boldsymbol{\psi}^s(x))=e^{-s} \t'(x).
	\end{align}
	
On the other hand, it follows from \cite[Theorem 15.69]{bamler2020structure} that 
	\begin{align*}
		X_{x}K_{Z'}(x;y)+X_{y}K_{Z'}(x;y)=\frac{n}{2}K_{Z'}(x;y)
	\end{align*}
for any $x, y \in \RR'_{(-\infty, 0)}$. Therefore, we obtain
	\begin{align}\label{self4}
	K_{Z'}(\boldsymbol{\psi}^s(x) ; \boldsymbol{\psi}^s(y))=e^{\frac{n}{2} s}K_{Z'}(x ; y)
	\end{align}
for any $x, y \in \RR'_{(-\infty, 0)}$ and $s \in \R$. Combining \eqref{self1}, \eqref{self2} and \eqref{self4}, we have
	\begin{align*}
d_{W_1}^{\RR'_t}(\nu_{x;t}, \nu_{y;t})=e^{\frac{s}{2}}d_{W_1}^{\RR'_{t'}} \lc \nu_{\boldsymbol{\psi}^s(x);t'}, \nu_{\boldsymbol{\psi}^s(y);t'} \rc,
	\end{align*}
for any $t \le \min\{\t'(x), \t'(y)\}$, where $t'=e^{-s}t$. Since $\iota_{z'}(\RR^{z'})=\RR'_{(-\infty, 0)}$, it follows from Definition \ref{defn:dstar-limit} that
	\begin{align*}
d^*_{z'}(\iota_{z'}^{-1}(\boldsymbol{\psi}^s(x)), \iota_{z'}^{-1}(\boldsymbol{\psi}^s(y)))=e^{-\frac{s}{2}}d^*_{z'}(\iota_{z'}^{-1}(x),\iota_{z'}^{-1}(y)).
	\end{align*}
By Theorem \ref{thm:idenproof}, this implies \eqref{selfsimilarity1}.
\end{proof}

Next, we prove

\begin{lem}\label{selfsimilarityshrinker1}
For any $x, y \in \RR'_t$ with $t<0$ and any $s \in \R$, we have
	\begin{align}\label{selfsimilarity2}
		d^{Z'}_{e^{-s}t}(\boldsymbol{\psi}^s(x), \boldsymbol{\psi}^s(y))=e^{-\frac{s}{2}}d^{Z'}_{t}(x,y).
	\end{align}
\end{lem}
\begin{proof}
It follows from \eqref{self1} and \eqref{self2} that
	\begin{align*}
	d_{g^{Z'}_{e^{-s}t}}(\boldsymbol{\psi}^s(x), \boldsymbol{\psi}^s(y))=e^{-\frac{s}{2}}d_{g^{Z'}_t}(x,y).
	\end{align*}
Consequently, the conclusion \eqref{selfsimilarity2} follows from Corollary \ref{cor:agree2}.
\end{proof}

Now, we can extend $\boldsymbol{\psi}^s$ to all $Z'_{(-\infty, 0]}$.

\begin{prop}\label{selfsimilarall}
$\boldsymbol{\psi}^s$ can be defined on $Z'_{(-\infty, 0]}$ so that the following statements hold.
	\begin{enumerate}[label=\textnormal{(\roman{*})}]
\item For any $x, y \in Z'_{(-\infty, 0]}$, $d_{Z'}(\boldsymbol{\psi}^s(x), \boldsymbol{\psi}^s(y))=e^{-\frac s 2} d_{Z'}(x, y)$ for any $s \in \R$.

\item For any $x, y \in Z'_t$ with $t \le 0$, $d^{Z'}_{e^{-s}t}(\boldsymbol{\psi}^s(x), \boldsymbol{\psi}^s(y))=e^{-\frac{s}{2}}d^{Z'}_{t}(x,y)$ for any $s \in \R$.

\item For any $x \in Z'_t$ with $t \le 0$ and $\tau>0$, $\NN_x(\tau)=\NN_{\boldsymbol{\psi}^s(x)}(e^{-s}\tau)$ for all $s \in \R$.
	\end{enumerate}	
\end{prop}

\begin{proof}
For any $w \in Z'_{(-\infty, 0]}$, we choose a sequence $w_i \in \RR'_{(-\infty, 0)}$ such that $w_i \to w$ in $d_{Z'}$. Then, for any $s \in \R$, $\{\boldsymbol{\psi}^s(w_i)\}$ is a Cauchy sequence by Lemma \ref{selfsimilarityshrinker} with respect to $d_{Z'}$. We define
	\begin{align*}
\boldsymbol{\psi}^s(w)=\lim_{i \to \infty} \boldsymbol{\psi}^s(w_i).
	\end{align*}
It is clear that the definition of $\boldsymbol{\psi}^s(w)$ is independent of the choice of $\{w_i\}$. 

(i): This follows from \eqref{selfsimilarity1} by taking the limit.

(ii): For any $x, y \in Z'_t$ with $t \le 0$, it follows from Lemma \ref{takelimitdt1} that there exists a sequence $t_i \nearrow t$ such that if $x_i, y_i \in \RR'_{t_i}$ are regular $H_n$-centers of $x$ and $y$, respectively, then
	\begin{align} \label{eq:con3a}
d^{Z'}_{t}(x, y)=\lim_{i \to \infty} d_{t_i}^{Z'}(x_i, y_i).
	\end{align}
Since $\boldsymbol{\psi}^s(x_i)$ and $\boldsymbol{\psi}^s(y_i)$ converge to $\boldsymbol{\psi}^s(x)$ and $\boldsymbol{\psi}^s(y)$, respectively, it follows from Lemma \ref{takelimitdt} that
	\begin{align*}
\liminf_{i \to \infty} d_{e^{-s} t_i}^{Z'}(\boldsymbol{\psi}^s(x_i), \boldsymbol{\psi}^s(y_i)) \ge d^{Z'}_{e^{-s}t}(\boldsymbol{\psi}^s(x), \boldsymbol{\psi}^s(y)).
	\end{align*}
Combining this with Lemma \ref{selfsimilarityshrinker1} and \eqref{eq:con3a}, we obtain
	\begin{align*}
e^{-\frac{s}{2}}d^{Z'}_{t}(x,y) \ge d^{Z'}_{e^{-s}t}(\boldsymbol{\psi}^s(x), \boldsymbol{\psi}^s(y)).
	\end{align*}
The reverse inequality also holds, since $\boldsymbol{\psi}^s$ is the inverse map of $\boldsymbol{\psi}^{-s}$.

(iii): This is immediate from \eqref{eq:kernellimit}, \eqref{self4}, the definition of $\boldsymbol{\psi}^s$ and \cite[Theorem 15.69]{bamler2020structure}.
\end{proof}

\section{Stratification and dimension of the singular set} \label{sec:singular}

First, we introduce the following definition for Ricci shrinker spaces.

\begin{defn}[$k$-splitting] \index{$k$-splitting} \label{def:ksplitting}
A Ricci shrinker space $(Z',d_{Z'},z',\t')$ is called \textbf{$k$-splitting} if $\RR'_{-1}$ splits off an $\R^k$-factor isometrically.
\end{defn}

We first prove:

\begin{prop}\label{splittingforpositivetime}
Let $(Z',d_{Z'},z',\t')$ be a Ricci shrinker space. If $(Z',d_{Z'},z',\t')$ is $k$-splitting, then $\RR'_{(-\infty, 0)}=\RR'' \times \R^k$ decomposes isometrically as a product of Ricci flow spacetimes, where $\RR''$ is another Ricci flow spacetime of dimension $n-k$ over $(-\infty, 0)$.
\end{prop}

\begin{proof}
By assumption, there exist $k$ smooth maps $\{y_i\}_{1\leq i\leq k}$ on $\RR'_{-1}$ satisfying
	\begin{align*}
	\la \na y_i, \na y_j \ra=\delta_{ij} \quad \text{and} \quad \na^2 y_i=0 \quad \mathrm{on}\quad \RR'_{-1}.
	\end{align*}
Moreover, we have	
	\begin{align*}
\int_{\RR'_{-1}} y_i \, \mathrm{d}\nu_{z',-1}=0.
	\end{align*}	
	
By the self-similarity of $(\RR'_{(-\infty, 0)}, g^{Z'})$, these functions extend smoothly to $\RR'_{(-\infty, 0)}$ such that
	\begin{align}\label{eq:analy1}
	\la \na y_i, \na y_j \ra=\delta_{ij},\quad \na^2 y_i=0,\quad \partial_{\t'} y_i=0, \quad \mathrm{on}\quad \RR'_{(-\infty, 0)}.
	\end{align}
Indeed, let $\boldsymbol{\psi}^s$ be the map in Proposition \ref{selfsimilarall}. Define
	\begin{align}\label{eq:analy2}
y_i(\boldsymbol{\psi}^s(x))=e^{-\frac s 2}y_i(x), \quad \text{for} \quad x\in  \RR'_{-1}.
	\end{align}
A direct computation confirms that \eqref{eq:analy1} holds.

By assumption, the flow of $\nabla y_i$ at $t=-1$ preserves the regular part $\RR_{-1}'$. Hence, by \eqref{eq:analy2} and Proposition \ref{selfsimilarall}, the flow of $\nabla y_i$ preserves the regular part $\RR'_t$ for all $t<0$. Let $\boldsymbol{\phi}^s$ denote the flow generated by ${\nabla y_1, \ldots, \nabla y_k}$ for $s \in \R^k$ on $\RR'_{(-\infty,0)}$. Then, for every $s \in \R^k$ and $x \in \RR'_{(-\infty,0)}$, we have $\boldsymbol{\phi}^s(x) \in \RR'_{(-\infty,0)}$.

Consequently, we conclude that $\RR'_{(-\infty, 0)}=\RR'' \times \R^k$, where $\RR''$ is another Ricci flow spacetime of dimension $n-k$ over $(-\infty, 0)$, and the splitting is isometric.

\end{proof}

\begin{rem} \label{rem:split}
To reach the same conclusion as in Proposition \ref{splittingforpositivetime}, it is enough to assume the existence of smooth functions $\{y_i\}$ $(1 \le i \le k)$ on $\RR'_{-1}$ satisfying
\begin{align}\label{eq:analy3}
\langle \nabla y_i, \nabla y_j \rangle = \delta_{ij}, \quad \nabla^2 y_i = 0,
\end{align}
rather than assuming an $\R^k$-splitting a priori. Indeed, Lemma \ref{lem:RCD} shows that
\begin{align*}
(\iota_{z'}(\XX^{z'}_{-1}), d^{Z'}_{-1}, \nu_{z';-1})
\end{align*}
is an $\mathrm{RCD}(1/2, \infty)$-space. The existence of the functions $y_i$ in \eqref{eq:analy3} ensures that the eigenspace of the weighted Laplacian $\Delta_{f_{z'}(-1)}$ corresponding to the eigenvalue $1/2$ has dimension at least $k$. Consequently, by \cite{GKKO}, the vector fields $\nabla y_i$ generate splitting directions that preserve $\RR'_{-1}$. Alternatively, the same conclusion follows from Theorem \ref{thm:splitting}.
\end{rem}

For a $k$-splitting Ricci shrinker space $(Z',d_{Z'},z',\t')$, we define the flow $\boldsymbol{\phi}^s$ for $s \in \R^k$ on $\RR'_{(-\infty, 0)}$ induced by the splitting in Proposition \ref{splittingforpositivetime} with $\boldsymbol{\phi}^0=\mathrm{id}$. More precisely, for any $x \in \RR'_{(-\infty, 0)}$, we write $x=(x',s') \in \RR'' \times \R^k$; then $\boldsymbol{\phi}^s(x)$ is defined as $(x', s'+s)$.

Next, we prove

\begin{prop}\label{prop:splitk}
With the above assumptions, $\boldsymbol{\phi}^s$ can be defined on $Z'_{(-\infty, 0]}$ so that the following statements hold.
	\begin{enumerate}[label=\textnormal{(\roman{*})}]
\item For any $x, y \in Z'_{(-\infty, 0]}$, $d_{Z'}(x, y)=d_{Z'}(\boldsymbol{\phi}^s(x), \boldsymbol{\phi}^s(y))$ for all $s \in \R^k$.

\item For any $x, y \in Z'_t$ with $t \le 0$, $d^{Z'}_t(x, y)=d^{Z'}_t(\boldsymbol{\phi}^s(x), \boldsymbol{\phi}^s(y))$ for all $s \in \R^k$.

\item For any $x \in Z'_{(-\infty, 0]}$ and $\tau>0$, $\NN_x(\tau)=\NN_{\boldsymbol{\phi}^s(x)}(\tau)$ for all $s \in \R^k$.
	\end{enumerate}	
\end{prop}

\begin{proof}
It is clear that for any $x \in \RR'_t$ with $t<0$, $\boldsymbol{\phi}^s(x) \in \RR'_t$. Moreover, 
	\begin{align} \label{eq:kcon00a}
K_{Z'}(x;y)=K_{Z'}(\boldsymbol{\phi}^s(x) ; \boldsymbol{\phi}^s(y))
	\end{align}
for any $x, y \in \RR'_{(-\infty, 0)}$ and $s \in \R^k$. Thus, by Definition \ref{defntimeslicedist}, we conclude that
	\begin{align} \label{eq:kcon1}
d^{Z'}_t(x, y)=d^{Z'}_{t}(\boldsymbol{\phi}^s(x), \boldsymbol{\phi}^s(y))
	\end{align}
for any $x, y \in \RR'_t$ with $t<0$ and $s \in \R^k$. Using the variational definition of the spacetime distance together with the invariance of the conjugate heat kernels, this implies
	\begin{align} \label{eq:kcon2}
d_{Z'}(x, y)=d_{Z'}(\boldsymbol{\phi}^s(x), \boldsymbol{\phi}^s(y))
	\end{align}
for any $x, y \in \RR'_{(-\infty, 0)}$ and $s \in \R^k$. 

For any $w \in Z'_{(-\infty, 0]}$, we choose a sequence $w_i \in \RR'_{(-\infty, 0)}$ such that $w_i \to w$ in $d_{Z'}$. Then, for any $s \in \R^k$, $\{\boldsymbol{\phi}^s(w_i)\}$ is a Cauchy sequence by \eqref{eq:kcon2}. We define
	\begin{align*}
\boldsymbol{\phi}^s(w)=\lim_{i \to \infty} \boldsymbol{\phi}^s(w_i).
	\end{align*}
It is clear that the definition of $\boldsymbol{\phi}^s(w)$ is independent of the choice of $\{w_i\}$. Moreover, it follows from \eqref{eq:kernellimit} and \eqref{eq:kcon00a} that
	\begin{align} \label{eq:kcon00b}
K_{Z'}(x;y)=K_{Z'}(\boldsymbol{\phi}^s(x) ; \boldsymbol{\phi}^s(y))
	\end{align}
for any $x \in Z'_{(-\infty, 0]}$, $y \in \RR'_{(-\infty, 0)}$ and $s \in \R^k$.

(i): This follows from \eqref{eq:kcon2} by taking the limit.

(ii): For any $x, y \in Z'_t$ with $t \le 0$, it follows from Lemma \ref{takelimitdt1} that there exists a sequence $t_i \nearrow t$ such that if $x_i, y_i \in \RR'_{t_i}$ are regular $H_n$-centers of $x$ and $y$, respectively, then
	\begin{align} \label{eq:kcon3}
d^{Z'}_{t}(x, y)=\lim_{i \to \infty} d_{t_i}^{Z'}(x_i, y_i).
	\end{align}
Since $\boldsymbol{\phi}^s(x_i)$ and $\boldsymbol{\phi}^s(y_i)$ converge to $\boldsymbol{\phi}^s(x)$ and $\boldsymbol{\phi}^s(y)$, respectively, it follows from Lemma \ref{takelimitdt} that
	\begin{align*}
\liminf_{i \to \infty} d_{t_i}^{Z'}(\boldsymbol{\phi}^s(x_i), \boldsymbol{\phi}^s(y_i)) \ge d^{Z'}_{t}(\boldsymbol{\phi}^s(x), \boldsymbol{\phi}^s(y)).
	\end{align*}
Combining this with \eqref{eq:kcon1} and \eqref{eq:kcon3}, we obtain
	\begin{align*}
d^{Z'}_{t}(x, y) \ge d^{Z'}_{t}(\boldsymbol{\phi}^s(x), \boldsymbol{\phi}^s(y)).
	\end{align*}
The reverse inequality also holds since $\boldsymbol{\phi}^s$ is the inverse map of $\boldsymbol{\phi}^{-s}$.

(iii): This is immediate from \eqref{eq:kcon00b} and Proposition \ref{splittingforpositivetime}.
\end{proof}

Next, we prove the following bi-Lipschitz estimate.

\begin{lem}
For any $x \in Z'_{(-\infty, 0]}$ and $s \in \R^k$,
	\begin{align*}
0 \le c(n)|s| \le d_{Z'}(x, \boldsymbol{\phi}^s(x)) \le  |s|.
	\end{align*}
\end{lem}

\begin{proof}
The second inequality follows from Lemma \ref{dtlem1} and Proposition \ref{prop:dismatch}, so we prove only the first inequality.

Set $r=d_{Z'}(x,\boldsymbol{\phi}^s(x))$. Then it follows from Proposition \ref{prop:chara} that
	\begin{align*}
\lim_{t \nearrow \t'(x)-r^2} d^{Z'_t}_{W_1}(\nu_{x;t}, \nu_{\boldsymbol{\phi}^s(x);t}) \le r.
	\end{align*}
	For any $t<\t'(x)-r^2$, if we set $w \in \RR'_t$ to be an $H_n$-center of $x$, it is clear that $\boldsymbol{\phi}^s(w)$ is an $H_n$-center of $\boldsymbol{\phi}^s(x)$. Thus, we obtain
	\begin{align*}
d^{Z'_t}_{W_1}(\nu_{x;t}, \nu_{\boldsymbol{\phi}^s(x);t}) \ge d^{Z'}_t(w, \boldsymbol{\phi}^s(w))-2\sqrt{H_n |\t'(x)-t|}=|s|-2\sqrt{H_n |\t'(x)-t|}.
	\end{align*}
Consequently, we obtain
		\begin{align*}
|s| \le C(n) r+r \le C(n) r,
	\end{align*}
which completes the proof.
\end{proof}

For later applications, we also need the following characterization of the potential function.

\begin{prop}\label{prop:splitkpoten}
With the above assumptions, for any $s \in \R^k$, we have
	\begin{align*}
f_{z'}(x)-f_{\boldsymbol{\phi}^s(z')}(x)=\frac{1}{2\tau} \la \vec{x}, s \ra+\frac{c(s)}{\tau} 
	\end{align*}
for any $x \in \RR'_{(-\infty, 0)}$, where $\tau=-\t'$, $\vec{x}$ is the component of $x$ in $\R^k$ with respect to the decomposition $\RR'_{(-\infty, 0)}=\RR''\times \R^k$, and $c(s)$ is a constant independent of $x$ and $\tau$.
\end{prop}

\begin{proof}
Since $K_{Z'}(z'; x)=K_{Z'}(\boldsymbol{\phi}^s(z'); \boldsymbol{\phi}^s(x))$ for any $x \in \RR'_{(-\infty, 0)}$, we conclude that $f_{z'}\lc\boldsymbol{\phi}^{-s}(x) \rc=f_{\boldsymbol{\phi}^s(z')}(x)$.

By the Ricci shrinker equation
	\begin{align*}
\Ric(g^{Z'})+\na^2 f_{z'}=\frac{g^{Z'}}{2 \tau}
	\end{align*}
and the fact that $\RR'_{(-\infty, 0)}=\RR''\times \R^k$, we know that
	\begin{align*}
f_{z'}(x)=h(x'')+\frac{|\vec{x}-v|^2}{4\tau},
	\end{align*}
	for some $v \in \R^k$, where $x=(x'', \vec{x}) \in \RR''\times \R^k$ for $x \in \RR'_{(-\infty, 0)}$, and $h(x'')$ is a function on $\RR''$. Thus, we conclude that
	\begin{align*}
f_{z'}(x)-f_{\boldsymbol{\phi}^s(z')}(x)=f_{z'}(x)-f_{z'}(\boldsymbol{\phi}^{-s}(x))=\frac{|\vec{x}-v|^2}{4\tau}-\frac{|\vec{x}-s-v|^2}{4\tau}=\frac{1}{2\tau} \la \vec{x}, s \ra+\frac{c(s)}{\tau}. 
	\end{align*}
\end{proof}

\begin{defn}[$k$-symmetric]\label{defnsymmetricsoliton}\index{$k$-symmetric}
A Ricci shrinker space $(Z',d_{Z'},z',\t')$ is called \textbf{$k$-symmetric} if one of the following holds:
		\begin{enumerate}[label=\textnormal{(\arabic*)}]
		\item  $(Z',d_{Z'},z',\t')$ is $k$-splitting and is not a static cone.
	
		\item $(Z',d_{Z'},z',\t')$ is a static cone that is $(k-2)$-splitting.
	\end{enumerate} 
\end{defn}
The number $k$ in Definition \ref{defnsymmetricsoliton} represents the number of directions in which the tangent flow is invariant, as established by Proposition \ref{prop:staticcone1} and Proposition \ref{prop:splitk}. Notably, in item (2), the tangent flow is invariant along the time direction, and since we view the time direction as two dimensions in the parabolic setting, it contributes two to the count. In addition, if $(Z',d_{Z'},z',\t')$ is a $(k-2)$-splitting static cone, the map $\boldsymbol{\varphi}^t$ (for $t\in \R$) defined in Proposition \ref{prop:staticcone1} and $\boldsymbol{\phi}^s$ (for $s\in \R^{k-2}$) defined in Proposition \ref{prop:splitk} commute, since they do so on the regular part.

As in the last section, we consider a Ricci flow limit space $(Z, d_Z, p_{\infty},\t)$ obtained from 
	\begin{equation} \label{eq:conv0101}
		(M_i \times \III, d^*_i, p_i^*,\t_i) \xrightarrow[i \to \infty]{\quad \mathrm{pGH} \quad} (Z, d_Z, p_{\infty},\t),
	\end{equation}
where $\XX^i=\{M_i^n,(g_i(t))_{t \in \III^{++}}\} \in \MM(n, Y, T)$ with base point $p_i^* \in \XX^i_\III$.

Recall that we have the following regular--singular decomposition from \eqref{decomRS2}:
	\begin{align*}
Z_{\III^-}=\RR_{\III^-} \sqcup \MS.
	\end{align*}
In particular, any tangent flow at $z \in \RR_{(-(1-2\sigma)T,0)}$ is $(n+2)$-symmetric and any tangent flow at $z\in \RR_0$ is $n$-symmetric.

Thus, we have the following natural stratification of $\MS$:
	\begin{equation}\label{defnstratification}\index{$\MS^k$}
		\mathcal S^0 \subset \mathcal S^1 \subset \cdots \subset \mathcal S^{n+1}=\mathcal S,
	\end{equation}
	where $z \in \MS^k$ if and only if no tangent flow at $z$ is $(k+1)$-symmetric.

The next result shows that $\MS\setminus \MS^{n-2}$ is in fact empty.

\begin{thm}\label{thm:singular}
In the same setting as above, we have
	\begin{align*}
\MS=\MS^{n-2}.
	\end{align*}
Moreover, no tangent flow at any singular point is a static or quasi-static cone that is $(n-2)$- or $(n-3)$-splitting.
\end{thm}

\begin{proof}
Given a tangent flow $(Z',d_{Z'},z',\t')$ at a point $z \in \MS$, we consider its tangent metric soliton $\XX^{z'}$ (see Definition \ref{def:tms}), which can be regarded as a tangent metric flow of $\XX^z$. Then it follows from Proposition \ref{prop:004} that $\iota_{z'}(\RR^{z'})=\RR'_{(-\infty,0)}$.

It follows from \cite[Theorem 2.8]{bamler2020structure} that either the Ricci curvature vanishes on $\RR^{z'}_{-1}$, in which case $\RR^{z'}_{-1}$ splits off an $\R^k$ for some $k \le n-4$, or the scalar curvature is positive on $\RR^{z'}_{-1}$, in which case $\RR^{z'}_{-1}$ splits off an $\R^k$ for some $k \le n-2$.

Consequently, $\MS=\MS^{n-2}$. The last conclusion also follows.
\end{proof}

By Theorem \ref{thm:singular}, we can refine the stratification \eqref{defnstratification} as follows:
 	\begin{equation*}
		\mathcal S^0 \subset \mathcal S^1 \subset \cdots \subset \mathcal S^{n-2}=\mathcal S.
	\end{equation*}

To control the size of each stratum $\MS^k$, we next recall the following definition of the Minkowski content and dimension.

\begin{defn}\label{defn:Minkowski-dstar}
	For a subset $Z_1\subset Z$, let $B_Z^*(Z_1,r)$ denote the $r$-neighborhood of $Z_1$ with respect to $d_Z$. For any $s>0$, we define the $s$-dimensional (upper) \textbf{Minkowski content}\index{Minkowski content} of $Z_1$ as
		\begin{align*}
\MMM^s(Z_1):=\sup_{L>0}\,\limsup_{r \to 0} \frac{|B_Z^*(Z_1, r) \cap B_Z^*(x_0, L)|}{r^{n+2-s}},
	\end{align*} 
	where $x_0$ is a fixed point in $Z$. The common value of $\inf \{s \ge 0 \mid \MMM^s(Z_1)=0\}=\sup \{s \ge 0 \mid \MMM^s(Z_1)=+\infty\}$ is called the (upper) \textbf{Minkowski dimension}\index{Minkowski dimension} of $Z_1$, denoted by $\dim_{\MMM} Z_1$\index{$\dim_{\MMM}$}.
\end{defn}

Next, we define the quantitative singular sets. The concept of $\ep$-closeness can be found in Definition \ref{defn:close}.

\begin{defn}\label{def:epsym}
	A point $z\in Z_{\III^-}$ is called \textbf{$(k,\ep,r)$-symmetric}\index{$(k,\ep,r)$-symmetric} if $\t(z)-\ep^{-1} r^2 \in \III^-$, and there exists a $k$-symmetric Ricci shrinker space $(Z',d_{Z'},z',\t')$ such that 
	  \begin{align*}
(Z, r^{-1} d_Z, z, r^{-2}(\t-\t(z))) \quad \text{is $\ep$-close to} \quad (Z',d_{Z'},z',\t') \quad \text{over} \quad [-\ep^{-1}, \ep^{-1}].
  \end{align*} 

Furthermore, if $k \in \{n-3, n-2\}$, then the model space $(Z',d_{Z'},z',\t')$ cannot be a quasi-static cone. If $k \ge n-1$, then the model space $(Z',d_{Z'},z',\t')$ is isometric to $(\R^{n} \times (-\infty, t_a], d_E^*, (\vec 0,0),\t)$ for some constant $t_a \in [0, +\infty]$.
\end{defn}

\begin{defn} \index{$\MS^{\ep,k}_{r_1,r_2}$}
	For $\ep > 0$ and $0<r_1<r_2<\infty$, the quantitative singular strata
	\[  \MS^{\ep,0}_{r_1,r_2} \subset \MS^{\ep,1}_{r_1,r_2} \subset   \ldots \subset  \MS^{\ep,n-2}_{r_1,r_2} \subset Z_{\III^-} \]
	are defined as follows:
	$z \in  \MS^{\ep,k}_{r_1,r_2}$ if and only if $\t(z)-\ep^{-1} r_2^2 \in \III^-$ and for all $r \in [r_1, r_2]$, $z$ is not $(k+1,\ep,r)$-symmetric. 
\end{defn}

The following identity is clear from the above definitions and Theorem \ref{thm:singular}: for any $L>1$,
\begin{align}\label{eq:siden1}
\MS^{k}=\bigcup_{\ep \in (0, L^{-1})} \bigcap_{0<r<\ep L} \MS^{\ep,k}_{r, \ep L}.
\end{align}

Note that the quantitative singular set $\MS^{\ep,k}_{r_1,r_2}$ can be defined in $M \times \III^-$ for any $\XX=\{M^n,(g(t))_{t \in \III^{++}}\} \in \MM(n, Y, T)$, even though $\XX$ contains no singular set.

Next, we compare the quantitative singular sets in Definition \ref{Def_quantiSS_weak} and $\MS^{\ep,k}_{r_1,r_2}$ for the top stratum.

\begin{lem}\label{comparediffquanset}
Given $\XX=\{M^n,(g(t))_{t \in \III^{++}}\} \in \MM(n, Y, T)$, for any $\ep\in (0,1)$, if $\ep'\leq \ep'(n,Y,\sigma,\ep)$ and $r_1<r_2 \ep'$, then
	\begin{align}\label{comparediffquanset1}
\MS^{\ep, n-2}_{r_1,r_2 \ep} \subset \MS^{\ep',n-2,\F}_{r_1, r_2 \ep'} \quad \text{and} \quad \MS^{\ep, n-2,\F}_{r_1,r_2 \ep} \subset \MS^{\ep',n-2}_{r_1, r_2 \ep'}.
	\end{align}
\end{lem}
\begin{proof}
We prove only the first inclusion; the second is analogous.

Suppose that the first inclusion in \eqref{comparediffquanset1} fails. Then, for a fixed $\ep>0$, we can find a sequence $\XX^i=\{M_i^n,(g(t))_{t \in \III^{++}}\} \in \MM(n, Y, T_i)$ and points $z_i^* \in M_i \times \III^-$ such that $z^*_i \in \MS^{\ep,n-2}_{r_1^i, r_2^i \ep}$, but $z^*_i \notin \MS^{i^{-2}, n-2, \F}_{r_1^i, r_2^i i^{-2}}$, where $r_1^i<r_2^i i^{-2}$.

From Definition \ref{Def_quantiSS_weak}, there exists $s_i \in [r_1^i, r_2^i i^{-2}]$ such that, after passing to a subsequence, we have
	\begin{align*}
		(M_i \times \III, s_i^{-1} d^*_i, z^*_i,s_i^{-2}(\t_i-\t_i(z_i^*))) \xrightarrow[i \to \infty]{\quad \hat C^\infty \quad} (Z, d_Z, z,\t).
	\end{align*}
Moreover, the Ricci flow limit space $(Z, d_Z, z,\t)$ satisfies $\R_- \subset \mathrm{image}(\t)$ and Proposition \ref{prop:004} on $\RR_{(-\infty, 0)}$. By considering the associated metric flow $\XX^z$, we conclude that either $\RR_{-1}$ splits off an $\R^{n-1}$, or the Ricci curvature vanishes on $\RR_{-1}$ and $\RR_{-1}$ splits off an $\R^{n-3}$. In both cases, it is clear that $(Z, d_Z, z,\t)$ is isometric to $(\R^{n} \times (-\infty, t_a] ,d_E^*, (\vec 0,0),\t)$ or $(\R^{n} \times \R_{-},d_E^*, (\vec 0,0),\t)$. Thus, $z_i^*$ is $(n-1, \ep, s_i)$-symmetric, for sufficiently large $i$. However, this implies $z_i^* \notin  \MS^{\ep,n-2}_{r_1^i, r_2^i \ep}$, which is a contradiction.
\end{proof}

By Proposition \ref{propdistance2}, Theorem \ref{quantsizeS} and Lemma \ref{comparediffquanset}, the following theorem is immediate.

\begin{thm}\label{thm:quant-size-dstar-limit}
Let $\XX=\{M^n,(g(t))_{t \in \III^{++}}\} \in \MM(n, Y, T)$ with $x_0^* \in \XX_{\III^-}$. Given $\ep>0$ and $r>0$ with $\t(x_0^*)-2 r^2 \in \III^-$, for any $\delta \in (0, \ep)$, there exist $x_1^*, x_2^*,\ldots, x_N^* \in B^*(x^*_0,r)$ with $N \leq C(n,Y,\sigma, \ep) \delta^{-n+2-\ep}$ and
	\begin{align*}
		\MS^{\ep,n-2}_{\delta r, \ep r} \cap B^* (x_0^*, r) \subset \bigcup_{i=1}^N B^* (x_i^*, \delta r).
	\end{align*}
Moreover, if $\ep\leq \ep(n,Y,\sigma)$, then 
	\begin{align*}
		r_{\Rm}\geq  \delta r \quad \text{on} \quad B^* (x_0^*, r)  \setminus \MS^{\ep,n-2}_{\delta r, \ep r}.
	\end{align*}
\end{thm}

Next, we prove 

\begin{thm}\label{eq:quant-size-dstar-limit1}
Let $(Z, d_Z, p_{\infty}, \t)$ be the Ricci flow limit space from \eqref{eq:conv0101} with $x_0 \in Z_{\III^-}$. Given $\ep>0$ and $r>0$ with $\t(x_0)-2 r^2 \in \III^-$, for any $\delta \in (0, \ep)$, there exist $x_1, x_2,\ldots, x_N \in B_Z^*(x_0,1.1 r)$ with $N \leq C(n,Y,\sigma, \ep) \delta^{-n+2-\ep}$ and
	\begin{equation}\label{eq:coverlimit}
		\MS^{\ep,n-2}_{\delta r, \ep r} \cap B_Z^* (x_0, r) \subset \bigcup_{j=1}^N B_Z^* (x_j, \delta r).
	\end{equation}
Moreover, if $\ep\leq \ep(n,Y,\sigma)$, then 
	\begin{align} \label{eq:currlimit}
		r_{\Rm}\geq  \delta r \quad \text{on} \quad B_Z^* (x_0, r) \setminus \MS^{\ep,n-2}_{\delta r, \ep r}.
	\end{align} 
\end{thm}

\begin{proof}
Let $\MS^{\cdot,\cdot,i}_{\cdot,\cdot}$ denote the corresponding quantitative singular set in $M_i\times\III^-$ associated with \eqref{eq:conv0101}. Then, we choose a sequence $x_i^* \in M_i \times \III^{-}$ such that $x_i^*$ converge to $x_0$ in the Gromov--Hausdorff sense.

We may further assume $\delta \in (0, \ep/2)$, since otherwise \eqref{eq:coverlimit} follows from a standard covering argument using Proposition \ref{prop:volumebound}. By Theorem \ref{thm:quant-size-dstar-limit}, for each $i$, there exist $x_{i,1}^*, x_{i,2}^*,\ldots, x_{i,N_i}^* \in  B_i^* (x_i^*, 1.01r)$ with 
	\begin{align*}
N_i \leq C(n,Y, \sigma, \ep) \delta^{-n+2-\ep}
	\end{align*}
and
	\begin{equation}\label{eq:cover0}
		\MS^{\ep/2,n-2,i}_{\delta r, \ep r/2} \cap B_i^* (x_i^*, 1.01r) \subset \bigcup_{j=1}^{N_i} B_i^* (x_{i,j}^*, \delta r/2).
	\end{equation}
After passing to a subsequence, we may assume that $N_i=N$ is constant. After passing to a further diagonal subsequence, we may assume that $x_{i,j}^*$ converge to $x_j \in B_Z^*(x_0,1.1 r)$ as $i \to \infty$, for any $1 \le j \le N$.

We claim that
	\begin{equation} \label{eq:cover1}
		\MS^{\ep,n-2}_{\delta r, \ep r} \cap B_Z^* (x_0, r) \subset \bigcup_{j=1}^N B_Z^* (x_j, \delta r).
	\end{equation}
Indeed, if \eqref{eq:cover1} fails, then one can find $y \in \MS^{\ep, n-2}_{\delta r, \ep r} \cap B_Z^* (x_0, r)$ so that $d_Z(y, x_j) \ge \delta r$ for any $1 \le j \le N$. We choose a sequence $y_i^* \in M_i \times \III^{-}$ converging to $y$ in the Gromov--Hausdorff sense. For sufficiently large $i$, it is clear that $y_i^* \in \MS^{\ep/2,n-2,i}_{\delta r, \ep r/2} \cap B_i^* (x_i^*, 1.01r)$. Then, it follows from \eqref{eq:cover0} that we can find $x_{i,j_i}^*$ with $d_i^*(y_i^*, x_{i,j_i}^*) < \delta r/2$. After passing to a subsequence, we conclude that
	\begin{align*}
d_Z(y, x_j) \le \delta r/2
	\end{align*} 
 for some $1 \le j \le N$, which is a contradiction.

For the final assertion, we consider $z \in B_Z^* (x_0, r) \cap Z_{\III^-} \setminus \MS^{\ep,n-2}_{\delta r, \ep r}$ so that $\ep$ is sufficiently small. We then choose a sequence $z_i^* \in M_i \times \III^{-}$ converging to $z$ in the Gromov--Hausdorff sense. Then, for sufficiently large $i$, $z_i^* \in B_i^* (x_i^*, 1.1 r) \cap \XX^i_{\III^-} \setminus \MS^{2\ep,n-2,i}_{\delta r, 2\ep r}$. Thus, by Theorem \ref{thm:quant-size-dstar-limit}, we have
	\begin{align*}
r_{\Rm}(z_i^*) \ge \delta r.
	\end{align*} 
Consequently, the conclusion follows from Lemma \ref{lem:curcon}.
\end{proof}

We obtain the following volume estimates:
\begin{cor}\label{cor:006}
Given $x_0 \in Z$, $\ep>0$ and $r>0$ with $\t(x_0)-2 r^2 \in \III^-$, for any $\delta \in (0, \ep)$,
	\begin{align*}
\abs{B^*_{Z} \lc \MS^{\ep, n-2}_{\delta r,\ep r}, \delta r \rc \cap B_Z^* (x_0, r)} \le C(n,Y,\sigma, \ep) \delta^{4-\ep} r^{n+2},
	\end{align*}
where $B^*_{Z} \lc \MS^{\ep,k}_{\delta r, \ep r}, \delta r \rc$ denotes the $\delta r$-neighborhood of $\MS^{\ep,k}_{\delta r, \ep r}$ with respect to $d_Z$, and $|\cdot|$ denotes the volume (see Definition \ref{def:vvvv}).
\end{cor}

\begin{proof}
Given $x_0 \in Z_{\III^-}$ and a constant $r>0$ with $\t(x_0)-2 r^2>-(1-2\sigma)T$. It follows from \eqref{eq:coverlimit} that for any $\delta \in (0, \ep)$, there exist $x_1, x_2,\ldots, x_N \in B_Z^*(x_0,1.1 r)$ with $N \leq C(n,Y, \sigma,\ep) \delta^{-n+2-\ep}$ and
	\begin{align*}
		\MS^{\ep, n-2}_{\delta r, \ep r} \cap B_Z^* (x_0, 1.01 r) \subset \bigcup_{j=1}^N B_Z^* (x_j, \delta r).
	\end{align*}
In particular, we have
	\begin{align*}
		B^*_{Z} \lc \MS^{\ep, n-2}_{\delta r, \ep r}, \delta r \rc \cap B_Z^* (x_0, r) \subset \bigcup_{j=1}^N B_Z^* (x_j, 2 \delta r).
	\end{align*}
By Proposition \ref{prop:volumebound}, this implies
	\begin{align*}
		\abs{B^*_{Z} \lc \MS^{\ep, n-2}_{\delta r, \ep r}, \delta r \rc \cap B_Z^* (x_0, r)} \le \sum_{j=1}^N \abs{B_Z^* (x_j, 2\delta  r)} \le C(n, Y, \sigma) N \delta^{n+2} r^{n+2} \le C(n,Y,\sigma, \ep) \delta^{4-\ep} r^{n+2}.
	\end{align*}
\end{proof}

The proof of Corollary \ref{cor:006}, with Proposition \ref{propvolumeslicelimit}, also gives the following result.

\begin{cor}\label{cor:coverslice}
Given $x_0\in Z$, $\ep>0$, and $r>0$ with $\t(x_0)-2r^2\in\III^-$, for any $\delta \in (0, \ep)$ and any $t\in \R$,
	\begin{align*}
\abs{B^*_{Z} \lc \MS^{\ep, n-2}_{\delta r, \ep r}, \delta r \rc \cap B_Z^* (x_0, r) \cap Z_t}_t \le C(n,Y,\sigma, \ep) \delta^{2-\ep} r^{n}.
	\end{align*}
\end{cor}

By Theorem \ref{thm:singular}, Definition \ref{defn:Minkowski-dstar}, \eqref{eq:siden1} and Corollary \ref{cor:006}, the following result is immediate.

\begin{cor}\label{cor:minkow}
	We have
	\begin{align*}
		\dim_{\MMM} \mathcal S \le n-2.
	\end{align*}
\end{cor}

Next, we prove the following integral estimates.

\begin{thm}\label{thm:integral-Rm-dstar-limit}
Let $(Z,d_Z,p_\infty,\t)$ be the Ricci flow limit space from \eqref{eq:conv0101}, and let $x_0\in Z$ and $r>0$ satisfy $\t(x_0)-2r^2\in\III^-$. Then, for any $\ep>0$, we have
	\begin{align} \label{eq:spacetimeint1}
		\int_{B^*_Z(x_0,r)\cap \RR}|\Rm|^{2-\ep} \, \mathrm{d}V_{g^Z_t}\mathrm{d}t \leq \int_{B^*_Z(x_0,r)\cap \RR}r_{\Rm}^{-4+2\ep} \, \mathrm{d}V_{g^Z_t}\mathrm{d}t\leq C(n,Y,\sigma,\ep)r^{n-2+2\ep}.		
	\end{align} 
Moreover, for any $t \in \R$,
	\begin{align}\label{eq:sliceint1}
		\int_{B^*_Z(x_0,r)\cap \RR_t}|\Rm|^{1-\ep} \, \mathrm{d}V_{g^Z_t} \leq \int_{B^*_Z(x_0,r)\cap \RR_t}r_{\Rm}^{-2+2\ep} \, \mathrm{d}V_{g^Z_t} \leq C(n,Y,\sigma,\ep)r^{n-2+2\ep}.		
	\end{align} 

\end{thm}
\begin{proof}
Without loss of generality, we assume $r=1$. It follows from \eqref{eq:currlimit} and Corollary \ref{cor:006} that
	\begin{align} \label{eq:currlimits}
	\abs{\{r_{\Rm} <  \delta \} \cap B_Z^* (x_0, 1)} \le \abs{\MS^{\ep,n-2}_{\delta , \ep}\cap B_Z^* (x_0, 1)} \le C(n,Y,\sigma, \ep) \delta^{4-\ep}.
	\end{align} 
Thus, it follows from \eqref{eq:currlimits} and Proposition \ref{prop:volumebound} that
	\begin{align*}
\int_{B^*_Z(x_0,1)\cap \RR}r_{\Rm}^{-4+2\ep} \, \mathrm{d}V_{g^Z_t}\mathrm{d}t =&\int_{B^*_Z(x_0,1)\cap \{r_{\Rm} \ge 1\}} 1\,\mathrm{d}V_{g^Z_t}\mathrm{d}t+\sum_{k \ge 1}\int_{B^*_Z(x_0,1)\cap \{2^{-k}\le r_{\Rm} < 2^{-k+1}\}} r_{\Rm}^{-4+2\ep} \, \mathrm{d}V_{g^Z_t}\mathrm{d}t \\
\le & C(n, Y)+C(n,Y,\sigma, \ep)\sum_{k \ge 1} 2^{k(4-2\ep)} 2^{(1-k)(4-\ep)} \le C(n,Y,\sigma,\ep).
	\end{align*} 
Consequently, the proof of \eqref{eq:spacetimeint1} is complete. The proof of \eqref{eq:sliceint1} follows similarly from
	\begin{align} \label{eq:currlimitslice}
	\abs{\{r_{\Rm} <  \delta \} \cap B_Z^* (x_0, 1) \cap Z_t}_t \le \abs{\MS^{\ep,n-2}_{\delta , \ep}\cap B_Z^* (x_0, 1)\cap Z_t}_t \le C(n,Y,\sigma, \ep) \delta^{2-\ep},
	\end{align} 
	where the last inequality is from Corollary \ref{cor:coverslice}.
\end{proof}

For later applications, we construct cutoff functions by smoothing $\eta(r_{\Rm}/r)$, as in \cite[Lemma 15.27 (b)]{bamler2020structure}, where $\eta$ is a fixed cutoff function with $\eta=0$ on $[0, 1.1]$ and $\eta=1$ on $[1.9, \infty]$. Note that item (6) below follows from \eqref{eq:currlimits} and \eqref{eq:currlimitslice}.

\begin{prop}\label{cutoffProp}
Let $(Z, d_Z,  \t)$ be the Ricci flow limit space from \eqref{eq:conv0101}. There is a family of smooth functions $\{\eta_r\in C^\infty(\RR)\}_{r>0}$ taking values in $[0,1]$ such that the following holds:
	\begin{enumerate}[label=\textnormal{(\arabic{*})}]
		\item $r_{\Rm}\geq r$ on $\{\eta_r>0\}$.
		\item $\eta_r=1$ on $\{r_{\Rm}\geq 2r\}$.
		\item $r|\na \eta_r|+r^2|\partial_\t\eta_r|+r^2|\na^2\eta_r| \leq C_0$ for some dimensional constant $C_0$.
		\item For any $z\in Z$ with $t=\t(z)$, $L<\infty$ and $r>0$, the set $\{\eta_r>0\}\cap B_{g^Z_t}(z,L)$ is relatively compact in $\RR_t$.
		\item For any $ L<\infty$, $z \in Z$ and $t\in \III$, the set $\{\eta_r>0\}\cap B_Z^*(z ,  L)\cap \RR_t$ is relatively compact in $\RR_t$.
		\item Given $A>1$, $z \in Z$, and $L>0$ with $\t(z)-2L^2\in \III^-$, and for any $\ep\in (0,1)$, there exists a constant $C=C(n,Y,\sigma, A, L,\ep)$ such that the following holds:
		\begin{align*}
			\iint_{(\spt\,\eta_r)^c\cap \RR_{[\t(z)-L^2, \t(z)+L^2]} \cap B^*_Z(z, A)} \, \mathrm{d}V_{g^Z_t}\mathrm{d}t\leq Cr^{4-\ep}.
		\end{align*}
		Moreover, for any $t \in [\t(z)-L^2, \t(z)+L^2]$, we have
		\begin{align*}
			\int_{(\spt\,\eta_r)^c\cap \RR_t \cap B^*_Z(z, A)} \, \mathrm{d}V_{g^Z_t} \leq Cr^{2-\ep}.
		\end{align*}
	\end{enumerate}
\end{prop}

In practice, we can slightly modify the cutoff functions above so that the resulting functions have compact support. We fix a point $z \in Z_{\III^-}$, and let $\eta_r$ be the cutoff functions in Proposition \ref{cutoffProp}. For any $A \gg  1$ and $r \ll 1$, we set $\kappa_{r, A}$ to be a smoothing of the characteristic function $\chi_{r, A}$ of $B^*_Z(z, 1.5 A) \cap \{r_{\Rm}>1.5 r\}$. Indeed, we only need to mollify $\chi_{r, A}$ on $B_{g^Z_{\t(x)}}(x, r) \times [\t(x)-r^2, \t(x)+r^2]$ for any $x \in \partial \lc B^*_Z(z, 1.5 A) \cap \{r_{\Rm}>1.5 r\}\rc$. Since $r_{\Rm}(x) \ge 1.5 r$, this can be done by first performing the standard convolution process on a collection of parabolic balls of size $cr$ and then patching the local data together by a partition of unity. We define
		\begin{align} \label{eq:cutoffnon}
\eta_{r, A}:=\kappa_{r, A} \eta_r.\index{$\eta_{r, A}$}
		\end{align}
The following proposition follows directly from Proposition \ref{cutoffProp} and our construction. 

\begin{prop}\label{cutoffProp1}
For any $z \in Z_{\III^-}$, the family of smooth cutoff functions $\{\eta_{r, A}\}$ defined in \eqref{eq:cutoffnon} satisfies the following properties for $r \le r(n, Y, \sigma)$.
	\begin{enumerate}[label=\textnormal{(\arabic{*})}]
		\item $r_{\Rm}\geq r$ and $d_Z(z, \cdot) \le 2A$ on $\{\eta_{r, A}>0\}$.
		\item $\eta_{r, A}=1$ on $\{r_{\Rm}\geq 2r\} \cap B_Z^*(z, A)$.
		\item $\displaystyle r|\na \eta_{r, A}|+r^2|\partial_\t\eta_{r, A}|+r^2|\na^2\eta_{r, A}|\leq C(n)$.
		\item For any $L$ with $\t(z)-2L^2 \in \III^-$ and any $\ep \in (0, 1)$, there exists a constant $C=C(n, Y,\sigma, L, A, \ep)>0$ such that
		 		\begin{align*}
			\iint_{\RR_{[\t(z)-L^2, \t(z)+L^2]} \cap \{0<\eta_{r, A}<1\}} \, \mathrm{d}V_{g^Z_t}\mathrm{d}t\leq Cr^{4-\ep}.
		\end{align*}
			Moreover, for any $t \in [\t(z)-L^2, \t(z)+L^2]$, we have
		\begin{align*}
			\int_{\RR_t \cap \{0<\eta_{r, A}<1\}} \, \mathrm{d}V_{g^Z_t} \leq Cr^{2-\ep}.
		\end{align*}
	\end{enumerate}
\end{prop}

Next, we consider a Ricci shrinker space $(Z',d_{Z'},z',\t')$ such that $p \in \RR'_{-1}$ is a regular $H_n$-center of $z'$.

\begin{lem}\label{lem:contain}
There exists a constant $C=C(n)>0$ such that for any $w \in B_{Z'_{-1}}(p, r)$ with $r \ge 1$, 
		\begin{align*}
\boldsymbol{\psi}^s(w) \in B_{Z'}^*(z', C  r)
		\end{align*}
for any $s \in [0, 1]$, where $\boldsymbol{\psi}^s$ is the map from Proposition \ref{selfsimilarall}.
\end{lem}

\begin{proof}
For any $s \in [0, 1]$, it follows from the self-similarity (see Proposition \ref{selfsimilarall}) that $\boldsymbol{\psi}^s(p)$ is an $H$-center of $z'$ for a constant $H=H(n)>0$. Thus, it follows from Lemma \ref{lem:004abcxee} that
	\begin{align} \label{containeq1}
		d_{Z'}(\boldsymbol{\psi}^s(p),z') \le C(n).
	\end{align}
On the other hand, it follows from Proposition \ref{selfsimilarall} that for any $w \in B_{Z'_{-1}}(p, r)$ and $s \in [0, 1]$,
		\begin{align*}
		d^{Z'}_{-e^{-s}}(\boldsymbol{\psi}^s(p),\boldsymbol{\psi}^s(w))=e^{-s/2}d^{Z'}_{-1}(p,w) \le  r,
	\end{align*}
which, when combined with Lemma \ref{dtlem1} and \eqref{containeq1}, implies
		\begin{align*}
	d_{Z'}(\boldsymbol{\psi}^s(w), z') \le  \lc C(n)+r \rc.
	\end{align*}
	This completes the proof.
\end{proof}

Combining Corollary \ref{cor:006}, Theorem \ref{thm:integral-Rm-dstar-limit} and Lemma \ref{lem:contain}, the following result is immediate from Proposition \ref{selfsimilarall}.

\begin{thm} \label{thm:spacemink}
With the above assumptions, the following statements hold.
\begin{enumerate}[label=\textnormal{(\roman{*})}]
\item For any $t<0$, the Minkowski dimension of $\MS \cap Z'_{t}$ with respect to $d^{Z'}_t$ is at most $n-4$.

\item For any $\ep>0$ and $r \ge 1$, we have
	\begin{align*}
	\abs{\{r_{\Rm} < \delta r \} \cap B_{Z'_{-1}} (p, r)}_{-1} \le C(n,Y, \ep) \delta^{4-\ep} r^{n+2}.
	\end{align*} 
\item For any $\ep>0$ and $r \ge 1$, we have
	\begin{align*}
		\int_{B_{Z'_{-1}} (p, r) \cap \RR'_{-1}}|\Rm|^{2-\ep} \, \mathrm{d}V_{g^{Z'}_{-1}} \le \int_{B_{Z'_{-1}} (p, r) \cap \RR'_{-1}}r_{\Rm}^{-4+2\ep} \, \mathrm{d}V_{g^{Z'}_{-1}}\leq C(n,Y, \ep)r^{n-2+2\ep}.
	\end{align*} 
	\end{enumerate}	
\end{thm}

\begin{proof}
(i): Without loss of generality, we assume $t=-1$. For any $w \in \MS \cap B_{Z'_{-1}}(p, r)$ with $r \ge 1$, it follows from Proposition \ref{selfsimilarall} that $\boldsymbol{\psi}^s(w) \in \MS$ for any $s \in [0, 1]$. Moreover, if $x \in B_{Z'_{-1}}(w, \delta r)$, then $\boldsymbol{\psi}^s(x) \in B_{Z'_{-e^{-s}}}(\boldsymbol{\psi}^s(w), e^{-s/2}\delta r)$. By Lemma \ref{dtlem1}, this implies that for any $s \in [0, 1]$,
	\begin{align*}
d_{Z'}(\boldsymbol{\psi}^s(x),\boldsymbol{\psi}^s(w)) \le  e^{-s/2} \delta r \le   \delta r.
	\end{align*} 
Thus, by Proposition \ref{selfsimilarall} and Lemma \ref{lem:contain}, we have
	\begin{align*}
\abs{B_{Z'_{-1}}\lc \MS \cap B_{Z'_{-1}}(p, r), \delta r \rc}_{-1} \le  \abs{B_{Z'}^*\lc B_{Z'}^*(z', C(n)  r) \cap \MS, \delta r \rc},
	\end{align*} 
	where $B_{Z'_{-1}}\lc \MS \cap B_{Z'_{-1}}(p, r), \delta r \rc$ denotes the $\delta r$-neighborhood of $\MS \cap B_{Z'_{-1}}(p, r)$ in $(Z'_{-1},d^{Z'}_{-1})$. Therefore, it follows from Corollary \ref{cor:006} that with respect to $d^{Z'}_{-1}$,
	\begin{align*}
		\dim_{\MMM}\lc  \mathcal S \cap Z'_{-1} \rc \le n-4.
	\end{align*}
(ii): By an argument similar to the proof of (i), we obtain
	\begin{align*}
\abs{\{r_{\Rm} < \delta r \} \cap B_{Z'_{-1}} (p, r)}_{-1} \le \abs{\{r_{\Rm} < 2\delta r \} \cap B^*_{Z'}(z', C(n) r)}.
	\end{align*} 
Using \eqref{eq:currlimits}, we obtain
	\begin{align} \label{eq:inde}
\abs{\{r_{\Rm} < \delta r \} \cap B_{Z'_{-1}} (p, r)}_{-1} \le C(n,Y, \sigma,\ep) \delta^{4-\ep} r^{n+2}.
	\end{align} 
	
We now claim that the constant $C$ can be chosen independently of $\sigma$. Indeed, by Remark \ref{rem:indep}, the left-hand side of \eqref{eq:inde} does not depend on the choice of $d_{Z'}$. Hence, if the Ricci shrinker space $(Z', d_{Z'}, z', \t')$ arises as the pointed Gromov--Hausdorff limit of a sequence $\XX^i \in \MM(n, Y, T_i)$ as in Remark \ref{rem:general}, we may assume without loss of generality that all $d^*$-distances are defined using a fixed parameter, say $\sigma = 1/100$.

(iii): This follows from (ii) and integration, as in the proof of Theorem \ref{thm:integral-Rm-dstar-limit}.
\end{proof}

Using Theorem \ref{thm:spacemink}, one can construct a family of cutoff functions on each negative time slice of $Z'$, similar to Proposition \ref{cutoffProp} and Proposition \ref{cutoffProp1}.

\begin{prop}\label{cutoffProp1shrinker}
There exists a family of smooth cutoff functions $\{\eta_{r, A} \in C^\infty(\RR'_{-1})\}$ taking values in $[0,1]$ such that the following holds:
	\begin{enumerate}[label=\textnormal{(\arabic{*})}]
		\item $r_{\Rm}\geq r$ and $d^{Z'}_{-1}(p, \cdot) \le 2A$ on $\{\eta_{r, A}>0\}$.
		\item $\eta_{r, A}=1$ on $\{r_{\Rm}\geq 2r\} \cap B_{Z'_{-1}}(p, A)$.
		\item $\displaystyle r|\na \eta_{r, A}|+r^2|\na^2\eta_{r, A}|\leq C(n)$.
		\item For any $\ep \in (0, 1)$, there exists a constant $C=C(n, Y, A, \ep)>0$ such that
		 		\begin{align*}
			\int_{\RR'_{-1}\cap \{0<\eta_{r, A}<1\}} \, \mathrm{d}V_{g^{Z'}_{-1}} \leq Cr^{4-\ep}.
		\end{align*}
	\end{enumerate}
\end{prop}

\section{Application: the first singular time of the Ricci flow} \label{sec:first}

In this section, we present several applications of the preceding results in a particular setting.

Let $\XX=\{M^n,(g(t))_{t\in[-T,0)}\}$ be a closed Ricci flow such that $0$ is the first singular time. We assume that $T<\infty$ and that the entropy of $\XX$ is bounded below by $-Y$.

We consider the $d^*$-distance on $\XX_{[-0.99T,0)}$, defined as in Definition \ref{defn:dstar-distance}. For convenience, we fix $\sigma=1/100$ throughout this section.

We then define
\begin{align*}
(Z,d_Z,\t)
\end{align*}
to be the metric completion of $\XX_{[-0.98T,0)}$ with respect to $d^*$. By construction,
\[
(Z_{[-0.98T,0)},d_Z)=(\XX_{[-0.98T,0)},d^*),
\]
that is, the completion only adds the points in $Z_0$.

It is clear that $Z$ has bounded diameter with respect to $d_Z$. Indeed, for any $x^*,y^*\in \XX_{[-0.99T,0)}$, it follows from Definition \ref{defn:dstar-distance} that
\begin{align}\label{eq:diam}
\mathrm{diam}_{d_Z}(Z)\le \max\bigl\{d_{W_1}^{-0.99T}(\nu_{x^*;-0.99T},\nu_{y^*;-0.99T}),\sqrt{T}\bigr\} \le \max\bigl\{\mathrm{diam}_{g(-0.99T)}(M),\sqrt{T}\bigr\}.
\end{align}

Moreover, $(Z,d_Z,\t)$ is a noncollapsed Ricci flow limit space over $I=[-0.98T,0]$ in the sense of Definition \ref{defnRicciLimitSpace}. Indeed, fix a sequence $\ep_i\searrow0$ and a base point $p^*\in M\times[-0.98T,0)$. Consider the smooth closed Ricci flows
\begin{align*}
\XX^i=\{M,(g(t))_{t\in[-T,-99\ep_iT]}\}.
\end{align*}
On each $\XX^i$ define $d_i^*$ by Definition \ref{defn:dstar-distance}, using the same lower comparison time $-0.99T$. Hence $d_i^*$ is exactly the restriction of the distance on the original flow. Put
\begin{align*}
I^{++}_i=[-(1-\ep_i)T,-99\ep_iT], \quad I^+_i=[-0.99T,-99\ep_iT], \quad I_i=[-(0.98+\ep_i)T,-99\ep_iT].
\end{align*}
The natural inclusions
\begin{align*}
(M\times I_i,d_i^*)\longrightarrow (Z,d_Z)
\end{align*}
are isometric and their images exhaust a dense subset of $Z$. They therefore exhibit the pointed Gromov--Hausdorff convergence
\begin{align*}
(M\times I_i,d_i^*,p^*,\t)\xrightarrow[i\to\infty]{\mathrm{pGH}}(Z,d_Z,p^*,\t).
\end{align*}
The entropy bound is uniform along this sequence, so all noncollapsed structural results from Sections 4--8 apply. The limit is independent of the choice of $\{\ep_i\}$ and of the auxiliary base point.

As in Definition \ref{defntimeslicedist}, for each $t \in (-0.98T,0]$, we define a distance function on $Z_t$ by
\begin{align} \label{eq:ztdistance}
d^{Z}_t(x,y):=\lim_{s \nearrow t} d_{W_1}^{s}(\nu_{x;s},\nu_{y;s}) \in [0,+\infty]
\end{align}
for any $x,y \in Z_t$. Note that if $t \in (-0.98T,0)$, then $(Z_t,d_t^Z)$ is simply $(M,d_t)$.

We first prove the following.

\begin{lem}\label{lem:twospaces}
$(Z_0,d_0^Z)$ is isometric to Bamler's asymptotic boundary $(M_0,d_{M_0})$ introduced in \cite[Section 2.6]{bamler2020structure}.
\end{lem}

\begin{proof}
Recall that $(M_0,d_{M_0})$, as defined in \cite[Definition 2.32]{bamler2020structure}, consists of all conjugate heat flows $(\mu_t)_{t\in[-T,0)}$ on $\XX$ satisfying $\Var_t(\mu_t)\le H_n|t|$. For any $(\mu_t^1),(\mu_t^2)\in M_0$, the distance is defined by
\begin{align} \label{eq:twospaces1}
d_{M_0}((\mu_t^1),(\mu_t^2))=\lim_{t\nearrow 0} d_{W_1}^t(\mu_t^1,\mu_t^2).
\end{align}

For any $z\in Z_0$, $\nu_{z;t}$ is a conjugate heat flow on $\XX_{[-0.98T,0)}$. By the reproduction formula, we may regard it as a conjugate heat flow on $\XX$. Moreover, since $\nu_{z;t}$ is $H_n$-concentrated (see Proposition \ref{prop:004ac}), we have $(\nu_{z;t})\in M_0$.

Thus, we obtain a map
\begin{align*}
\mathfrak{i}: Z_0\to M_0,\qquad z\mapsto (\nu_{z;t}).
\end{align*}

By \eqref{eq:ztdistance} and \eqref{eq:twospaces1}, the map $\mathfrak{i}$ is an isometric embedding. To complete the proof, it remains to show that $\mathfrak{i}$ is surjective.

Take any $(\mu_t)\in M_0$. Fix the sequence $t_i=-2^{-i}$, and let $z_i^*=(z_i,t_i)$ be an $H_n$-center of $(\mu_{t_i})$. Set
\[
K_i=B_{t_i}(z_i,\sqrt{4H_n|t_i|}).
\]

For any $i<j$, by the reproduction formula, \cite[Theorem 4.1]{bamler2020entropy}, and Proposition \ref{existenceHncenter}, we have
\begin{align}\label{eq:twospaces2}
\frac{3}{4}\le \mu_{t_i}(K_i)
=& \int_{\XX_{t_j}} \nu_{x^*;t_i}(K_i)\, \mathrm{d}\mu_{t_j}(x^*) \notag\\
\le &\ \mu_{t_j}(\XX_{t_j}\setminus K_j)
+\Phi\!\left( \Phi^{-1}(\nu_{z_j^*;t_i}(K_i))+(t_j-t_i)^{-1/2}(4H_n|t_j|)^{1/2}\right)\mu_{t_j}(K_j) \notag\\
\le &\ \frac{1}{4}+\Phi\!\left(\Phi^{-1}(\nu_{z_j^*;t_i}(K_i))+2\sqrt{H_n}\right),
\end{align}
where we used the fact that $t_j-t_i\ge |t_j|$, and $\Phi$ is the function defined in Definition \ref{defnmetricflow}(6).

It follows from \eqref{eq:twospaces2} that $\nu_{z_j^*;t_i}(K_i)\ge c_0>0$, where $c_0$ is a universal constant. By Proposition \ref{existenceHncenter} again, this implies that $z_i^*$ is an $H(n)$-center of $z_j^*$. Hence, by Lemma \ref{lem:Hcenterdis},
\begin{align}\label{eq:twospaces3}
d_Z(z_i^*,z_j^*)\le C(n,Y)\sqrt{t_j-t_i}.
\end{align}
In particular, $\{z_i^*\}$ is a Cauchy sequence with respect to $d_Z$. By the completeness of $(Z,d_Z)$, we obtain some $z\in Z_0$ such that $z_i^*\to z$ in $(Z,d_Z)$.

Since $z_i^*$ is an $H(n)$-center of $\nu_{z_j^*;t_i}$ for every $i<j$, convergence of the conjugate heat kernel measures at the fixed time $t_i$ and lower semicontinuity of the variance show that $z_i^*$ is an $H'(n)$-center of $\nu_{z;t_i}$. By monotonicity, this implies
\begin{align*}
d_{W_1}^t(\mu_t,\nu_{z;t})\le C(n)\sqrt{|t_i|}, \quad \forall t \in [-T, t_i].
\end{align*}
Since $i$ is arbitrary, we conclude that $\nu_{z;t}=\mu_t$ for all $t<0$.

Therefore, $\mathfrak{i}(z)=(\mu_t)$, and hence $\mathfrak{i}$ is surjective. This completes the proof.
\end{proof}
Next, we prove

\begin{thm}\label{intefirst}
For any $\ep>0$,
	\begin{align}\label{eq:firstspacetimein}
\int_{-T}^0 \int_M |\Rm|^{2-\ep}  \, \mathrm{d}V_{g(t)} \mathrm{d}t \le \int_{-T}^0 \int_M r_{\Rm}^{-4+2\ep}  \, \mathrm{d}V_{g(t)} \mathrm{d}t \le C_{\ep}.
	\end{align} 
Moreover, for any $t \in [-T, 0)$,
		\begin{align}\label{eq:firstsliceint}
\int_M |\Rm|^{1-\ep}  \, \mathrm{d}V_{g(t)} \le  \int_M r_{\Rm}^{-2+2\ep}  \, \mathrm{d}V_{g(t)}  \le C_{\ep}.
	\end{align} 
Here, the constant $C_{\ep}$ depends on $\ep$ and the Ricci flow $\XX$.
\end{thm}

\begin{proof}
We set $r_0=\sqrt{T}/20$ and assume that $\{x_i\}_{1 \le i \le N}$ is a maximal $r_0$-separated subset of $Z_{[-T/4, 0]}$. It is clear from \eqref{eq:diam}, Proposition \ref{prop:uppervolumebound} and Proposition \ref{prop:volumebound} that $N$ is finite. Moreover, $\{B_Z^*(x_i,2r_0)\}_{1 \le i \le N}$ cover $Z_{[-T/4, 0]}$.

By Theorem \ref{thm:integral-Rm-dstar-limit}, we obtain
	\begin{align*}
\int_{-T/4}^0 \int_M r_{\Rm}^{-4+2\ep}  \, \mathrm{d}V_{g(t)} \mathrm{d}t \le \sum_{i=1}^N  \int_{B^*_Z(x_i,2r_0) \cap \RR} r_{\Rm}^{-4+2\ep} \, \mathrm{d}V_{g^Z_t}\mathrm{d}t \le C_{\ep},
	\end{align*} 
which completes the proof of \eqref{eq:firstspacetimein}. The proof of \eqref{eq:firstsliceint} follows similarly from \eqref{eq:sliceint1} in Theorem \ref{thm:integral-Rm-dstar-limit}.
\end{proof}

By considering the Ricci flow on the standard $S^2$, it is clear that the constant $\ep$ in \eqref{eq:firstspacetimein} cannot be $0$.

Next, we show that the volume of $M$ at time $t$ has a limit as $t$ approaches $0$. 

\begin{prop} \label{prop:volume0}
With the above assumptions, we have
	\begin{align}\label{eq:volume00a0}
\lim_{t \nearrow 0} |M|_t=V_0 \in [0, +\infty).
	\end{align} 
\end{prop}

\begin{proof}
We have
	\begin{align*}
\diff{}{t} |M|_t=-\int_M \scal  \, \mathrm{d}V_{g(t)} \le  C(n, T) |M|_t
	\end{align*} 
for $t \in [-T/2, 0)$, which implies that
	\begin{align}\label{eq:volume00a1}
|M|_t \le C(n, T) |M|_{-T/2} \quad \text{for all } t \in [-T/2, 0).
	\end{align} 
	
For any $-T/2<t_1<t_2<0$, it follows from Theorem \ref{intefirst} and \eqref{eq:volume00a1} that
	\begin{align}\label{eq:volume00a2}
\abs{|M|_{t_2}-|M|_{t_1}} \le \int_{t_1}^{t_2} \int_M |\scal|  \, \mathrm{d}V_{g(t)}\mathrm{d}t \le C(n)\lc \int_{t_1}^{t_2} \int_M |\Rm|^{\frac 3 2}  \, \mathrm{d}V_{g(t)}\mathrm{d}t \rc^{\frac 2 3} \abs{M \times [t_1, t_2]}^{\frac 1 3} \le C (t_2-t_1)^{\frac 1 3},
	\end{align} 
where $C$ depends on the Ricci flow. It follows that $\lim_{t \nearrow 0} |M|_t$ exists, which, by \eqref{eq:volume00a1} again, must be finite.
\end{proof}

Using Theorem \ref{intefirst} and the same argument as above (see \eqref{eq:volume00a2}), we obtain the following corollary.

\begin{cor} \label{cor:volume0}
For any $\ep>0$, there exists a constant $C_{\ep}$ depending on $\ep$ and the Ricci flow $\XX$ such that for any $t \in [-T, 0)$,
	\begin{align*}
\abs{|M|_t-V_0} \le C_{\ep} |t|^{\frac 1 2-\ep}.
	\end{align*} 
\end{cor}

Note that if the regular part $\RR_0$ of $Z$ at time $0$ is nonempty, then the limit $V_0$ in \eqref{eq:volume00a0} is positive by smooth convergence. On the other hand, we prove the following volume estimate if $\RR_0=\emptyset$, which improves Corollary \ref{cor:volume0} and is analogous to \cite[Corollary 6.25]{li2024heat}.

\begin{prop} \label{prop:volumesingular}
Suppose that $\RR_0=\emptyset$. Then for any $\ep>0$, there exists a constant $C_{\ep}$ depending on $\ep$ and the Ricci flow $\XX$ such that for any $t \in [-T, 0)$,
	\begin{align*}
|M|_t \le C_{\ep} |t|^{1-\ep}.
	\end{align*} 
\end{prop}

\begin{proof}
We only need to prove the conclusion for $t$ close to $0$. 

It follows from the definition of the curvature radius that $r_{\Rm} < 2 \sqrt{|t|}$ on $M \times \{t\}$, since otherwise $\RR_0$ is not empty. Then by \eqref{eq:currlimitslice}, we have
	\begin{align*}
|M|_t \le \abs{ \left\{r_{\Rm} <  2 \sqrt{|t|} \right \}}_t \le C(n,Y, T, \ep) |t|^{1-\ep/2},
	\end{align*} 
which completes the proof.
\end{proof}


As a corollary of Proposition \ref{prop:volumesingular}, we obtain the following dichotomy, depending on whether $\RR_0$ is empty or not.

\begin{cor} \label{cor:volume}
For the limit $V_0$ in \eqref{eq:volume00a0}, $V_0>0$ if and only if $\RR_0 \ne \emptyset$.
\end{cor}

Thus, we have the following definition.

\begin{defn}
With the above assumptions, $\XX$ is called \textbf{noncollapsed}\index{noncollapsed} at the first singular time if $V_0>0$, and \textbf{collapsed}\index{collapsed} if $V_0=0$.
\end{defn}

The term ``collapsed'' is justified by the following lemma.

\begin{prop}
$\XX$ is collapsed at $0$ if and only if any tangent flow at $z \in Z_0$ is collapsed (see Definition \ref{def:ncf}).
\end{prop}

\begin{proof}
If every tangent flow at every point of $Z_0$ is collapsed, then in particular $\RR_0=\emptyset$. Hence, $\XX$ is collapsed at $0$.

Conversely, suppose for contradiction that $\XX$ is collapsed at $0$, but there exists a tangent flow $(Z', d_{Z'}, z', \t')$ at some point $z \in Z_0$ that is noncollapsed. We may assume that this tangent flow is obtained as the pointed Gromov--Hausdorff limit of  $(Z, r_j^{-1} d_Z, z, r_j^{-2}\t)$ for a sequence $r_j \searrow 0$. Then, by the same argument as in the proof of Theorem \ref{dichotomytang}, there exists a sequence $x_i \in \RR'_{-1}$ such that $r_{\Rm}(x_i) \to +\infty$. 

By smooth convergence on the regular part and Theorem \ref{thm:twoside}, we can find points $y_j^*=(y_j, -r_j^2) \in M \times [-T, 0)$ such that $r_j^{-1}r_{\Rm}(y_j^*) \to +\infty$. This implies that $\RR_0$ is nonempty, contradicting the assumption that $\XX$ is collapsed at time $0$.
\end{proof}

Next, we prove

\begin{prop} 
We have
\begin{align} \label{eq:volume00c1}
\abs{\RR_0}_{0}=V_0,
	\end{align} 	
	where $\abs{\RR_0}_{0}$ denotes the volume of $\RR_0$ with respect to $g^Z_0$.
\end{prop}

\begin{proof}
From the smooth convergence of $\RR_t$ to $\RR_0$ along $\partial_\t$, we obtain 
\begin{align*}
    |\RR_0|_0 \le V_0.
\end{align*}

On the other hand, by \eqref{eq:currlimitslice}, we have
	\begin{align} \label{eq:weakcurve}
\abs{\{r_{\Rm} <  \delta \}}_t \le C(n,Y, T) \delta
	\end{align} 
for all $t \in [-T/100, 0]$ and small $\delta>0$. By smooth convergence again, for any fixed $\delta$,
\begin{align*}
|\RR_0|_0 \ge \abs{\{r_{\Rm} \ge \delta/2\}}_0 \ge \lim_{t \nearrow 0} \abs{\{r_{\Rm} \ge \delta\}}_t \ge V_0-C(n, Y, T) \delta,
\end{align*}
where we used \eqref{eq:weakcurve} for the last inequality. Taking $\delta \to 0$ yields the other inequality $ |\RR_0|_0 \ge V_0$.
\end{proof}

We end this section by proving the following result.

\begin{thm}\label{thm:slicecodim2}
For every sufficiently small $\delta>0$ and $\ep>0$, we have
	\begin{align*}
\abs{\left\{y\in Z_0 \mid d_0^Z(y,\MS)<\delta \right\} }_0\leq C_\ep \delta^{2-\ep},
	\end{align*} 
	where $C_{\ep}$ depends on $\ep$ and the Ricci flow $\XX$.
\end{thm}

\begin{proof}
It follows from Lemma \ref{dtlem1} that
	\begin{align*}
\left\{y\in Z_0 \mid d_0^Z(y,\MS)<\delta \right\} \subset S':=\left\{y\in Z_0 \mid d_Z(y,\MS)< \delta \right\}.
	\end{align*} 
Since $r_{\Rm}=0$ on $\MS$, we obtain from Proposition \ref{prop:LiprRmlimit} that any $y \in S'$ satisfies
	\begin{align*}
r_{\Rm}(y) < C(n, Y) \delta,
	\end{align*} 
which implies
	\begin{align*}
S' \subset \left\{y\in Z_0 \mid r_{\Rm}(y) < C(n, Y)\delta\right\}.
	\end{align*} 

The conclusion follows from \eqref{eq:weakcurve}.
\end{proof}

\section{Almost splitting maps} \label{sec:almost}

In this section, we consider a closed Ricci flow $\XX=\{M^n,(g(t))_{t\in I}\}$ with a fixed spacetime point $x_0^*=(x_0,t_0)\in \XX$. Moreover, we set
\begin{align*}
\mathrm{d}\nu_t=\mathrm{d}\nu_{x^*_0;t}=(4\pi\tau)^{-n/2}e^{-f} \mathrm{d}V_{g(t)},
\end{align*}
where $\tau=t_0-t$.

We use the following definition of almost splitting maps, which is similar to \cite[Definition 5.7]{bamler2020structure}.

\begin{defn}[$(k,\ep, r)$-splitting map]\label{defnsplittingmap}\index{$(k, \ep, r)$-splitting map}
	A map $\vec u=(u_1,\ldots,u_k)$ is called a \textbf{$(k,\ep, r)$-splitting map at $x_0^*$} if $t_0-10r^2 \in I$, and for all $i,j\in \{1, \ldots, k\}$, the following properties hold:
	\begin{enumerate}[label=\textnormal{(\roman{*})}]
		\item $u_i(x^*_0)=0$.
		\item $\square u_i=0$ on $M\times [t_0-10r^2,t_0]$.
		\item $\displaystyle \int_{t_0-10 r^2}^{t_0-r^2/10} \int_{M} |\na^2 u_i|^2 \, \mathrm{d}\nu_t \mathrm{d}t \le \ep$.		
		\item $\displaystyle \int_{t_0-10 r^2}^{t_0-r^2/10} \int_{M} \la \na u_i,\na u_j \ra-\delta_{ij} \, \mathrm{d}\nu_t \mathrm{d}t=0$.
	\end{enumerate}

\end{defn}

In the following, we assume throughout that $\ep$ is sufficiently small, say $\ep \le 10^{-3}$.

\begin{prop}\label{timesliceL2estimateofgradient}
	Let $\vec u=(u_1,\ldots,u_k)$ be a $(k,\ep, r)$-splitting map at $x_0^*$. Then, for any $i,j\in \{1, \ldots, k\}$
	and for all $t\in [t_0-10r^2,t_0-r^2/10]$, we have
	\begin{align*}
	\abs{\int_{M} \la \na u_i ,\na u_j \ra-\delta_{ij} \, \mathrm{d}\nu_t }\leq 2\ep \quad \text{and} \quad	\int_{M} \big|\la \na u_i ,\na u_j \ra-\delta_{ij}\big|\, \mathrm{d}\nu_t\leq 50 \epsilon^{\frac{1}{2}}.
	\end{align*}
	Moreover, for all $t\in [t_0-\frac{10}{p-1}r^2, t_0)$ and $p \ge 2$,
	\begin{align}\label{esti4}
		\lc\int_{M} |\nabla u_i|^p \, \mathrm{d}\nu_{t}\rc^{1/p}\leq 1+\epsilon^{\frac{1}{2}}.
	\end{align}
\end{prop}
\begin{proof}
Without loss of generality, we assume $t_0=0$ and $r=1$. We set
\begin{align*}
I_{ij}(t)=\int_{M} \la \nabla u_i,\nabla u_j\ra \, \mathrm{d}\nu_t-\delta_{ij}.
	\end{align*}

Since
\begin{align*}
\diff{}{t}I_{ij}(t)=-2\int_{M} \la \nabla ^2u_i,\nabla^2 u_j\ra \, \mathrm{d}\nu_t,
	\end{align*}
we obtain for any $-10\leq s,t \leq -1/10$,
\begin{align}
|I_{ij}(t)-I_{ij}(s)| \le & 2 \int_{-10}^{-1/10} \int_M |\na^2 u_i||\na^2 u_j| \, \mathrm{d}\nu_t \mathrm{d}t \notag\\
\le &  2 \lc \int_{-10}^{-1/10}\int_M|\na^2 u_i|^2 \, \mathrm{d}\nu_t \mathrm{d}t\rc^{\frac 1 2}\lc \int_{-10}^{-1/10}\int_M|\na^2 u_j|^2 \, \mathrm{d}\nu_t \mathrm{d}t \rc^{\frac 1 2} \le 2\ep. \label{esti2}
	\end{align}

Then, it follows from (iv) in Definition \ref{defnsplittingmap} that for all $t\in [-10,-1/10]$,
	\begin{align}\label{esti1}
|I_{ij}(t)| \le 2\ep.
	\end{align}
	Applying Theorem \ref{poincareinequ} to $ \la \nabla u_i,\nabla u_j\ra-\delta_{ij}$, we have for all $t\in [-10,-1/10]$,
	\begin{align*}
		\int_{M} \left| \la \nabla u_i,\nabla u_j\ra-\delta_{ij}-I_{ij}(t) \right|\mathrm{d}\nu_t&\leq \sqrt{\pi |t|} \int_{M} \left|\nabla ( \la \nabla u_i,\nabla u_j\ra-\delta_{ij})\right|\, \mathrm{d}\nu_t\\
		&\leq \sqrt{10 \pi}\int_{M} |\nabla ^2u_i||\nabla u_j|+|\nabla ^2u_j||\nabla u_i| \, \mathrm{d}\nu_t.
	\end{align*}
	Integrating in time, we get
	\begin{align*}
		& \int_{-10}^{-1/10}\int_{M} \left| \la \nabla u_i,\nabla u_j\ra-\delta_{ij}-I_{ij}(t)\right|\, \mathrm{d}\nu_t \mathrm{d}t\\
		\leq & \sqrt{10 \pi} \lc\int_{-10}^{-1/10}\int_{M} |\nabla ^2u_i|^2+|\nabla^2 u_j|^2 \, \mathrm{d}\nu_t \mathrm{d}t\rc^{1/2}\lc\int_{-10}^{-1/10}\int_{M} |\nabla u_i|^2+|\nabla u_j|^2 \, \mathrm{d}\nu_t \mathrm{d}t \rc^{1/2}	\leq 20 \sqrt{\pi} \epsilon^{1/2},
	\end{align*}
	where we used Definition \ref{defnsplittingmap} (iv) to obtain
		\begin{align*}
\int_{-10}^{-1/10}\int_{M} |\nabla u_i|^2+|\nabla u_j|^2 \, \mathrm{d}\nu_t \mathrm{d}t \le 20.
	\end{align*}
		
	Combining with \eqref{esti1}, we have
	\begin{align}\label{esti3}
\int_{-10}^{-1/10}\int_{M} \big|\la \nabla u_i,\nabla u_j\ra-\delta_{ij}\big|\mathrm{d}\nu_t \mathrm{d}t\leq 20\ep+20 \sqrt{\pi} \epsilon^{1/2} \le 40 \epsilon^{1/2}.
	\end{align}	
	Since 
	\begin{align*}
	\diff{}{t}\int_{M} \big| \la \nabla u_i,\nabla u_j\ra-\delta_{ij}\big| \, \mathrm{d}\nu_t&\leq \int_{M} \left|\square ( \la \nabla u_i,\nabla u_j\ra-\delta_{ij})\right|\mathrm{d}\nu_t\\
		&\leq 2\int_{M} |\nabla^2 u_i||\nabla ^2 u_j|\mathrm{d}\nu_t,
	\end{align*}
	we obtain, as in \eqref{esti2}, that for any $s,t \in [-10, -1/10]$,
		\begin{align*}
\abs{\int_{M} \big| \la \nabla u_i,\nabla u_j\ra-\delta_{ij}\big| \, \mathrm{d}\nu_t-\int_{M} \big| \la \nabla u_i,\nabla u_j\ra-\delta_{ij}\big| \, \mathrm{d}\nu_s} \le 2\ep.
	\end{align*}
	Combining this with \eqref{esti3}, we conclude that for all $t\in [-10,-1/10]$, 
	\begin{align*}
		\int_{M} \big|\la \nabla u_i,\nabla u_j\ra-\delta_{ij}\big| \, \mathrm{d}\nu_t\leq 40 \epsilon^{1/2}+2\ep \le 50 \ep^{\frac 1 2}.
	\end{align*}

For the last statement \eqref{esti4}, we apply Theorem \ref{hypercontractivity} to $|\nabla u_i|$ with $0<\tau_1\leq \frac{10}{p-1}$ and $\tau_2=10$ so that
	\begin{align*}
		\lc\int_{M} |\nabla u_i|^p\,\mathrm{d}\nu_{-\tau_1}\rc^{1/p}\leq \lc\int_{M} |\nabla u_i|^2\,\mathrm{d}\nu_{-10}\rc^{1/2}\leq 1+\epsilon^{1/2},
	\end{align*}
	for any $p \ge 2$, where we used the estimate \eqref{esti1} for the last inequality.
\end{proof}

\begin{prop}[Gradient estimate]\label{pointwisegradientestimatesplittingmap}
	Let $\vec{u}=(u_1,\ldots,u_k)$ be a $(k,\ep, r)$-splitting map at $x_0^*$ and $i\in \{1,\ldots,k\}$. Then there exists a constant $\hat C=\hat C(n)>0$ such that on $M\times [t_0-r^2,t_0)$,
		\begin{align}\label{gradientesti2}
			|\na u_i|^2(x, t) \leq 1+\hat C \ep^{1/8} \exp \lc \frac{\hat C d^2_t(x, z)}{r^2}\rc,
		\end{align}
		where $(z, t)$ is an $H_n$-center of $x_0^*$.
\end{prop}
\begin{proof}
Without loss of generality, we assume $t_0=0$ and $r=1$.

We consider $x^*=(x,t)$ with $t \in [-1,0)$. By the reproduction formula and $\square(|\nabla u_i|^2-1)=-2|\nabla^2u_i|^2$, we have
	\begin{align}\label{grad001}
		(|\nabla u_i|^2-1)(x^*)&=\int_{M} (|\nabla u_i|^2-1)\, \mathrm{d}\nu_{x^*;-2}+\int_{-2}^t\int_{M} \square (|\nabla u_i|^2-1) \, \mathrm{d}\nu_{x^*;s}\mathrm{d}s\nonumber\\
		&\leq \int_{M} \big||\nabla u_i|^2-1\big| \, \mathrm{d}\nu_{x^*;-2}.
	\end{align}
Take an $H_n$-center $(z,t)$ of $x_0^*$. Then $d^{t}_{W_1}(\nu_t,\delta_z)\leq \sqrt{H_n|t|} \le C(n)$. By Proposition \ref{changeofbasepoint1}, for any small constant $\alpha \in (0, 1)$,
	\begin{align}\label{grad002}
		\mathrm{d}\nu_{x^*;-2}\leq  e^{C(n, \alpha)\big( d_{W_1}^{t}(\nu_t,\delta_x)\big)^2}e^{\alpha f}\, \mathrm{d}\nu_{-2}\leq C(n, \alpha) e^{C(n, \alpha) d_t^2(x,z)}e^{\alpha f} \, \mathrm{d}\nu_{-2}.
	\end{align}
Combining \eqref{grad002} with \eqref{grad001}, we have
	\begin{align}\label{grad003x}
\abs{(|\nabla u_i|^2-1)(x^*)} \le & C(n, \alpha) e^{C(n, \alpha) d_t^2(x,z)} \int_M \big||\nabla u_i|^2-1\big| e^{\alpha f} \, \mathrm{d}\nu_{-2}.
	\end{align}
On the other hand, we have
	\begin{align*}
		\int_{M} \big||\na u_i|^2-1\big|^2 \, \mathrm{d}\nu_{-2} 		\leq  \lc\int_{M} \big||\na u_i|^2-1\big|\, \mathrm{d}\nu_{-2}\rc^{1/2}\lc\int_{M} \big||\na u_i|^2-1\big|^{3}\, \mathrm{d}\nu_{-2}\rc^{1/2} \le C \ep^{1/4},
	\end{align*}
for a universal constant $C>0$, where we used \eqref{esti4} with $p=6$. Combining this with \eqref{grad003x} and using Proposition \ref{integralbound} (with $\alpha=\alpha(n)$), we obtain
	\begin{align*}
\abs{(|\nabla u_i|^2-1)(x^*)} \le  C(n) e^{C(n) d_t^2(x,z)} \ep^{1/8}.
	\end{align*}	
This completes the proof of \eqref{gradientesti2}.
\end{proof}

We next prove that the almost splitting map is locally Lipschitz with respect to the spacetime distance.

\begin{prop}
Suppose $\XX \in \MM(n, T)$. Let $\vec{u}=(u_1,\ldots,u_k)$ be a $(k,\ep, r)$-splitting map at $x_0^*$ such that $\t(x_0^*)-10r^2 \in \III^-$. Then there exists a small constant $\hat c=\hat c(n)>0$ such that for any $z^*$ with $d^*(x_0^*, z^*) \le \hat c r$, we have
	\begin{align*}
\abs{\vec u(z^*)} \le C(n) d^*(x_0^*, z^*).
	\end{align*}
\end{prop}
\begin{proof}
Without loss of generality, we assume $t_0=0$ and $r=1$. We also assume $\t(z^*)\leq t_0$; the other case is analogous.

Define $s:=100 d^*(x_0^*, z^*)$, and let $(z_0, -s^2)$ and $(z_1, -s^2)$ be $H_n$-centers of $x_0^*$ and $z^*$, respectively. 

By Definition \ref{defn:dstar-distance}, we have
	\begin{align*}
s \ge d_{W_1}^{-s^2}(\nu_{x_0^*; -s^2}, \nu_{z^*; -s^2}) \ge d_{-s^2}(z_0, z_1)-2s \sqrt{H_n}
	\end{align*}
and hence
	\begin{align} \label{eq:lip1a}
d_{-s^2}(z_0, z_1) \le C(n) s.
	\end{align}

Choose $\hat c$ sufficiently small so that $\frac{1}{10 s^2} \ge \max\{1, \hat C(n)\}$, where $\hat C(n)$ is the same constant as in Proposition \ref{pointwisegradientestimatesplittingmap}. Since $\vec u$ solves the heat equation, we have $\vec u(x^*_0)=\int_{M} \vec u(\cdot,-s^2)\, \mathrm{d}\nu_{-s^2}$. Thus, by Proposition \ref{pointwisegradientestimatesplittingmap}, we obtain
	\begin{align}
		|\vec u(x^*_0)-\vec u(z_0,-s^2)|&\leq \int_{M} |\vec u(y,-s^2)-\vec u(z_0,-s^2)|\, \mathrm{d}\nu_{-s^2}(y) \notag \\
		&\leq \int_{M} \lc 1+C(n)\exp\lc \frac{d_{-s^2}^2(y,z_0)}{6s^2} \rc  \rc d_{-s^2}(y,z_0)\, \mathrm{d}\nu_{-s^2}(y)\notag \\
		&= \sum_{k=0}^\infty \int_{\{k s \le d_{-s^2}(y, z_0) \le (k+1)s \}} \lc 1+C(n)\exp \lc \frac{d_{-s^2}^2(y,z_0)}{6s^2} \rc \rc d_{-s^2}(y,z_0)\, \mathrm{d}\nu_{-s^2}(y) \notag \\
		&\leq \sum_{k=0}^\infty \lc 1+C(n) e^{(k+1)^2/6}\rc (k+1) s e^{-k^2/5} \le C(n) s, \label{eq:lip0}
	\end{align}
 where we used Theorem \ref{heatkernelupperbdgeneral} (i) for the third inequality. Since $\vec u(x^*_0)=0$, we conclude 
	\begin{align} \label{eq:lip1}
		|\vec u(z_0,-s^2)| \le C(n) s.
	\end{align}

By using \eqref{eq:lip1a} and the same argument as \eqref{eq:lip0}, we obtain
	\begin{align} \label{eq:lip2}
|\vec u(z^*)-\vec u(z_1,-s^2)| \le C(n) s.
	\end{align}

Combining \eqref{eq:lip1} and \eqref{eq:lip2}, we conclude from Proposition \ref{pointwisegradientestimatesplittingmap} that
	\begin{align*}
|\vec u(z^*)| \le |\vec u(z^*)-\vec u(z_1,-s^2)|+|\vec u(z_0, -s^2)-\vec u(z_1,-s^2)|+|\vec u(z_0,-s^2)| \le C(n) s.
	\end{align*}

This completes the proof.
\end{proof}

For the rest of the section, we consider a Ricci flow limit space $(Z, d_Z, p_\infty,\t) \in \MM(n, Y, T)$; see Notation \ref{not:1}.

Next, we introduce the following quantitative concept of splitting on $Z$. 

\begin{defn}\label{defnksplitting}\index{$(k, \ep, r)$-splitting map}
	A point $z\in Z_{\III^-}$ is called \textbf{$(k,\ep,r)$-splitting} if $\t(z)-10  r^2 \in \III^-$ and there exists a noncollapsed Ricci flow limit space $(Z',d_{Z'},z',\t')$ such that its regular part $\RR'_{[-10, 0]}$ splits off an $\R^k$ as a Ricci flow spacetime. Moreover, 
	  \begin{align*}
(Z, r^{-1} d_Z, z, r^{-2}(\t-\t(z))) \quad \text{is $\ep$-close to} \quad (Z',d_{Z'},z',\t') \quad \text{over} \quad [-10, 0].
  \end{align*} 
\end{defn}

Moreover, we generalize Definition \ref{defnsplittingmap} on $Z$.

\begin{defn}
A map $\vec u=(u_1,\ldots,u_k)$ is called a \textbf{$(k,\ep, r)$-splitting map at $z \in Z_{\III^-}$} if $\t(z)-10 r^2 \in \III^-$, and $\vec u$ is obtained as the limit of a sequence of $(k, \ep, r)$-splitting maps $\vec u^i=(u^i_1,\ldots,u^i_k)$ at $z_i^*$ with $z_i^* \to z$ in the Gromov--Hausdorff sense. Note that $\vec u$ is defined on $Z_{(\t(z)-10 r^2, 0]}$ by the reproduction formula and Theorem \ref{thm:convextra}.
\end{defn}

Passing to the limit, all of the above propositions and corollaries hold for almost splitting maps on $Z$.

We end this section by proving the following result.

\begin{prop}\label{equivalencesplitsymetric}
Let $(Z, d_Z, \t) \in \MM(n, Y, T)$. For any $\ep > 0$, if $\delta \leq \delta(n, Y, \ep)$ and $z$ is a $(k, \delta, r)$-splitting point, then there exists a map $\vec{u} = (u_1, \ldots, u_k)$ defined on $Z_{(\t(z) - 9 r^2, 0]}$ such that for any $x \in B^*_Z(z, \ep^{-1} r) \cap Z_{[\t(z)-\delta r^2, \t(z)+\delta r^2]}$ and $s \in [\ep r,  r/2]$, there exists a matrix $T_{x, s}$ satisfying $\|T_{x, s} - \mathrm{Id}\| \le \ep$, for which the rescaled map $\vec{u}_{x, s} := T_{x, s}\big(\vec{u} - \vec{u}(x)\big)$ is a $(k, \ep, s)$-splitting map at $x$.
\end{prop}
\begin{proof}
Without loss of generality, we assume $r=1$ and $\t(z)=0$.

Assume that the conclusion is false. Then we can find Ricci flow limit spaces $(Z^l,d_{Z^l},\t_l) \in \MM(n, Y)$ such that there exists $z_l\in Z^l$ that is $(k, l^{-2},1)$-splitting. 

After passing to a subsequence, we assume
	\begin{align*}
		(Z^l,d_{Z^l}, \t_l, z_l)\xrightarrow[l\to\infty]{\quad \mathrm{\hat C^\infty}\quad} (Z,d_{Z},\t,z).
	\end{align*} 
After passing to a diagonal subsequence, we may assume that $(Z^l,d_{Z^l}, \t_l, z_l)=(M_l \times \III_l, d_l^*, \t_l, z_l^*) \in \MM(n, Y, T_l)$.

By Definition \ref{defnksplitting}, we conclude that the regular part $\RR_{[-10 ,0]}=\RR' \times \R^k$. Moreover, it follows from Lemma \ref{lem:uniheatequ} that for any $w=(w', \vec b) \in \RR' \times \R^k$ and any $s \in [-10, \t(w))$, we have
	\begin{align} \label{eq:productmeasure}
\nu_{w;s}=\nu'_{w';s} \otimes \nu^{\R^k}_{\vec b ;s},
	\end{align} 
where $\nu^{\R^k}_{\vec b ;s}$ is the standard Gaussian measure on $\R^k$ defined by
	\begin{align*}
\nu^{\R^k}_{\vec b ;s}=(4\pi(\t(w)-s))^{-\frac{k}{2}} \exp \lc -\frac{|\vec b-\vec x|^2}{4(\t(w)-s)} \rc \, \mathrm{d}\vec x.
	\end{align*} 
	
	Let $(y_1,\ldots,y_k)$ denote the corresponding coordinate functions satisfying $\int_{\RR_{-10}} y_i \, \mathrm{d}\nu_{z; -10}=0$. By solving the corresponding heat equation, we assume that $y_i$ satisfies $\square y_i =0$ on $\RR_{[-10, 0]}$ and is defined on $Z_{(-10, 0]}$. 
	
	By \eqref{eq:productmeasure}, we conclude that for any $w \in \RR_{(-10, 0]}$
	\begin{align} \label{eq:extralimitint1}
		\int_{\RR_{-9}} \lc y_i-y_i(w) \rc^2\,\mathrm{d}\nu_{w;-9}=2(\t(w)+9), \quad \int_{\RR_{-9}}\la \na y_i,\na y_j\ra \, \mathrm{d}\nu_{w;-9}=\delta_{ij}.
	\end{align}		
	Passing to the limit (see \eqref{eq:kernellimit}), we conclude that \eqref{eq:extralimitint1} also holds for $w \in Z_{(-10, 0]}$. 
	
Using the cutoff functions in Proposition \ref{cutoffProp1} and the smooth convergence, we can find smooth functions $\vec u^l=(u^l_1, \ldots, u^l_k)$ on $M_l \times [-9, 0]$ with $\square \vec u^l=0$ so that $\vec u^l$ converge smoothly to $\vec y$ on $\RR_{(-9, 0)}$.
	
According to our assumption, there exist $x_l^* \in B^*(z_l^*, \ep^{-1}) \cap M_l \times [-l^{-2}, l^{-2}]$ and $s_l \in [\ep, 1/2]$ for which $\vec u^l$ does not satisfy the conclusion. After passing to a subsequence, we assume that $s_l \to s_\infty \in [\ep, 1/2]$, and $x_l^*$ converge to $x_\infty \in \overline{B^*_Z(z, \ep^{-1})} \cap Z_0$.

Applying \eqref{eq:extralimitint1} to $x_\infty$, we obtain, by smooth convergence, that
	\begin{align*} 
	\abs{\int_{M_l} (u_i^l-a_i^l(x_l^*))^2 \, \mathrm{d}\nu_{x_l^*;-9}-18} \xrightarrow[l \to \infty]{} 0,
		\end{align*}	
	where
		\begin{align*}
a_i^l(x_l^*):=\int_{M_l} u_i^l \, \mathrm{d}\nu_{x_l^*;-9}.
		\end{align*}
Moreover, the following estimate holds:
		\begin{align*} 
	 \abs{ \int_{M_l} \la \na u^l_i, \na u^l_j\ra \, \mathrm{d}\nu_{x_l^*;-9}-\delta_{ij}} \xrightarrow[l \to \infty]{} 0
		\end{align*}
		for any $1 \le i, j \le k$. Thus, the corresponding frequency function $F_{u_i^l-a_i^l(x_l^*)}$ with respect to $x_l^*$ (see Definition \ref{def:fre}) satisfies
	\begin{align*}
\abs{F_{u_i^l-a_i^l(x_l^*)}(-9)-\frac{1}{2}} \xrightarrow[l \to \infty]{} 0.
	\end{align*}
Thus, by the same argument as in the proof of Proposition \ref{propagationofeigenvalues} (see also Corollary \ref{cor:exisplitting}), we conclude that for sufficiently large $l$, there exists a matrix $T_{x_l^*, s_l}$ satisfying $\|T_{x_l^*, s_l} - \mathrm{Id}\| \to 0$, such that the map $T_{x_l^*, s_l}\lc \vec{u}^l-\vec{u}^l(x_l^*)\rc$ is a $(k, \ep, s_l)$-splitting map at $x_l^*$.

This contradicts our assumption and completes the proof.
\end{proof}

\section{Further discussions} \label{sec:further}

In this section, we discuss how the results of this paper extend to Ricci flows with bounded curvature. We also examine certain special properties that arise in the K\"ahler setting.

\subsection*{Complete Ricci flows with bounded curvature}

For a fixed constant $T \in (0, +\infty]$ and a small parameter $\sigma \in (0, 1/100]$, we define as before
	\begin{align*}
\III^-=(-(1-2\sigma)T,0], \quad \III=[-(1-2\sigma)T,0], \quad \III^{+}=[-(1-\sigma)T, 0], \quad \III^{++}=[-T,0].
	\end{align*}

\begin{defn}
For a fixed constant $T\in(0,+\infty]$, the moduli space $\widetilde\MM(n,T)$\index{$\widetilde\MM(n,T)$} consists of all $n$-dimensional complete Ricci flows
\begin{align*}
\XX=\{M^n,(g(t))_{t\in\III^{++}}\}
\end{align*}
with bounded curvature on every compact time subinterval of $\III^{++}$. For $Y>0$, the subspace $\widetilde\MM(n,Y,T)\subset\widetilde\MM(n,T)$\index{$\widetilde\MM(n,Y,T)$} consists of the flows whose entropy is bounded below by $-Y$.
\end{defn}

It is clear that $\MM(n,T)\subset\widetilde\MM(n,T)$ and $\MM(n,Y,T)\subset\widetilde\MM(n,Y,T)$. As shown in \cite[Appendix A]{bamler2025fundamental}, the results of \cite{bamler2020entropy}, \cite{bamler2023compactness}, and \cite{bamler2020structure} remain valid for Ricci flows in $\widetilde \MM(n, Y, T)$. For related heat kernel estimates on complete noncompact Ricci flows, we refer to \cite{li2025RFnoncompact}, where full details are provided. For example, the upper bound of the heat kernel in Theorem \ref{heatkernelupperbdgeneral} appears as \cite[Theorem 11.4]{li2025RFnoncompact}; Theorem \ref{hypercontractivity} on hypercontractivity corresponds to \cite[Theorem 12.1]{li2025RFnoncompact}; and the integral estimates in Proposition \ref{integralbound} are established in \cite[Section 13]{li2025RFnoncompact}.

Consequently, the arguments in Sections \ref{secpreliminary}--\ref{sec:f} apply to $\widetilde\MM(n,T)$. In particular, the analogue of Theorem \ref{thm:GHlimit-dstar} is the following.

\begin{thm}\label{thm:GHlimit-dstar-rs}
Consider a sequence
\begin{align*}
\XX^i=\{M_i^n,(g_i(t))_{t\in\III^{++}}\}\in\widetilde\MM(n,T)
\end{align*}
with base points $p_i^*\in\XX^i_{\III}$. When $T=+\infty$, assume additionally that $\limsup_{i\to\infty}\t_i(p_i^*)>-\infty$. After passing to a subsequence, we obtain
\begin{align}\label{eq:conrs1}
(M_i\times\III,d_i^*,p_i^*,\t_i)
\xrightarrow[i\to\infty]{\quad\mathrm{pGH}\quad}
(Z,d_Z,p_\infty,\t),
\end{align}
where $(Z,d_Z,\t)$ is a complete, separable, locally compact parabolic metric space over $\III$.
\end{thm}

If, in addition, the sequence in Theorem \ref{thm:GHlimit-dstar-rs} lies in $\widetilde\MM(n,Y,T)$ for a fixed $Y$, then all results in Sections \ref{sec:smooth}--\ref{sec:singular} hold for the limit. We summarize the principal conclusions below.

\begin{thm}\label{Fconvergencers}
Suppose $(Z,d_Z,p_\infty,\t)$ is obtained by Theorem \ref{thm:GHlimit-dstar-rs} from a sequence in $\widetilde\MM(n,Y,T)$ for a fixed $Y$. Then the following statements hold.
	
	\begin{enumerate}[label=\textnormal{(\arabic{*})}]
		\item There exists a decomposition
		\begin{align*}
		Z_{\III^-}=\RR_{\III^-} \sqcup \MS
		\end{align*}
		such that $\RR$ is given by an $n$-dimensional Ricci flow spacetime $(\RR, \t, \partial_{\t}, g^{Z})$ and $\dim_{\MMM} \mathcal S \le n-2$, where $\dim_{\MMM}$ denotes the Minkowski dimension in Definition \ref{defn:Minkowski-dstar}. Moreover, $\RR$ is open, $\RR_{\III^-}$ is dense in $Z_{\III^-}$, and $\RR$ is connected when $T=+\infty$.
		
	\item For each $t \in \III^-$, there exists an extended distance $d^Z_t$ on $Z_t$ such that $d^Z_t$, when restricted to $\RR_t$, agrees with $d_{g^Z_t}$ locally. In addition, each point $z \in Z$ is assigned a conjugate heat kernel measure $\nu_{z;s}$ such that $(Z, \t, (d^Z_t)_{t \in \III^-}, (\nu_{z;s})_{z\in Z, \, -(1-2\sigma)T \le s \le \t(z)})$ is an extended metric flow in the sense of Definition \ref{defn:emf}. Moreover, for any $x,y\in Z$ with $\t(x)\geq\t(y)$, set $r=d_Z(x,y)$ and $q=\t(x)-r^2$. If $q\in\III^-$, then
\begin{align*}
\lim_{t\nearrow q}d_{W_1}^{Z_t}(\nu_{x;t},\nu_{y;t})\leq r.
\end{align*}
If in addition $q<\t(y)$, then
\begin{align*}
r\leq\lim_{t\searrow q}d_{W_1}^{Z_t}(\nu_{x;t},\nu_{y;t}).
\end{align*}
		
		\item Every tangent flow $(Z',d_{Z'}, \t',z')$ at a point $z \in Z$, when restricted to the regular part $\RR'_{(-\infty, 0)}$, satisfies the equation
			\begin{align*}
\Ric(g^{Z'})+\na^2 f_{z'}=\frac{g^{Z'}}{2 \tau}.
	\end{align*}
Moreover, each $\RR'_{t}$ is connected for any $t \in (-\infty, 0)$.
		
		\item The convergence \eqref{eq:conrs1} is smooth on $\RR$ in the following sense. There exist an increasing sequence $U_1 \subset U_2 \subset \ldots \subset \RR$ of open subsets with $\bigcup_{i=1}^\infty U_i = \RR$, open subsets $V_i \subset M_i \times \III$, time-preserving diffeomorphisms $\phi_i : U_i \to V_i$ and a sequence $\ep_i \to 0$ such that the following holds:
		\begin{enumerate}[label=\textnormal{(\alph{*})}]
			\item We have
			\begin{align*}
				\Vert \phi_i^* g^i - g^Z \Vert_{C^{[\ep_i^{-1}]}(U_i)} & \leq \ep_i, \\
				\Vert \phi_i^* \partial_{\t_i} - \partial_{\t} \Vert_{C^{[\ep_i^{-1}]}(U_i)} &\leq \ep_i,
			\end{align*}
			where $g^i$ is the spacetime metric induced by $g_i(t)$, and $\partial_{\t_i}$ is the standard time vector field.
			
			\item For $U_i^{(2)}=\{(x,y) \in U_i \times U_i \mid \t(x)> \t(y)+\ep_i\}$, $V_i^{(2)}=\{(x^*,y^*) \in V_i \times V_i \mid \t_i(x^*)> \t_i(y^*)+\ep_i\}$ and $\phi_i^{(2)}:=(\phi_i, \phi_i): U_i^{(2)} \to V_i^{(2)}$, we have
					\begin{align*}
\Vert  (\phi_i^{(2)})^* K^i-K_Z \Vert_{C^{[\ep_i^{-1}]}(U_i^{(2)})} \le \ep_i,
			\end{align*}	
where $K^i$ and $K_Z$ denote the heat kernels on $(M_i \times \III, g_i(t))$ and $(\RR, g^Z)$, respectively.
			
		\item Let $y \in \RR$ and $y_i^* \in M_i \times \III$. Then $y_i^* \to y$ in the Gromov--Hausdorff sense if and only if $y_i^* \in V_i$ for large $i$ and $\phi_i^{-1}(y_i^*) \to y$ in $\RR$.
		
		\item If $y_i^* \in M_i \times \III$ converge to $y \in Z$ in the Gromov--Hausdorff sense, then
\begin{align*}
K^i(y_i^*;\phi_i(\cdot)) \xrightarrow[i \to \infty]{C^{\infty}_\mathrm{loc}} K_Z(y;\cdot) \quad \text{on} \quad \RR_{(-\infty, \t(y))}.
\end{align*}	
		\item For each $t \in \III^-$, the time slice $\RR_t$ has at most countably many connected components.	
		\item For all but countably many times $t \in \III^-$, we have
	\begin{align*}
d^Z_t=d_{g^Z_t}
	\end{align*}
on each connected component of $\RR_t$.	
				\end{enumerate}
		\item Given $x_0 \in Z$ and $r>0$ with $\t(x_0)-2r^2 \in \III^-$, for any $\ep>0$, we can find a constant $C=C(n,Y, \sigma,\ep)>0$ such that
	\begin{align*}
		\int_{B^*_Z(x_0,r)\cap \RR}r_{\Rm}^{-4+2\ep} \, \mathrm{d}V_{g^Z_t}\mathrm{d}t\leq C r^{n-2+2\ep}.
	\end{align*} 		
	Moreover, for any $t \in \R$,
	\begin{align*}
\int_{B^*_Z(x_0,r)\cap \RR_t}r_{\Rm}^{-2+2\ep} \, \mathrm{d}V_{g^Z_t} \leq C r^{n-2+2\ep}.		
	\end{align*} 
	\end{enumerate}
\end{thm}

All results in Section \ref{sec:almost} also hold for $\widetilde \MM(n, Y, T)$, except that we need to modify the definition of the $(k, \ep, r)$-splitting map slightly as follows.

\begin{defn}[$(k,\ep, r)$-splitting map] \label{def:splitnon}
Given $\XX=\{M^n,(g(t))_{t \in \III^{++}}\} \in \widetilde \MM(n, Y, T)$ with $x_0^*=(x_0, t_0) \in \XX$, a map $\vec u=(u_1,\ldots,u_k)$ is called a \textbf{$(k,\ep, r)$-splitting map at $x_0^*$} if $t_0-10 r^2 \in \III^{-}$, and for all $i,j \in \{1,\ldots,k\}$, the following properties hold:
	\begin{enumerate}[label=\textnormal{(\roman{*})}]
		\item $u_i(x^*_0)=0$.
		\item $\square u_i=0$ on $M\times [t_0-10r^2,t_0]$.
		\item $\displaystyle \int_{t_0-10 r^2}^{t_0-r^2/10} \int_{M} |\na^2 u_i|^2 \, \mathrm{d}\nu_t \mathrm{d}t \le \ep$.		
		\item $\displaystyle \int_{t_0-10 r^2}^{t_0-r^2/10} \int_{M} \la \na u_i,\na u_j \ra-\delta_{ij} \, \mathrm{d}\nu_t \mathrm{d}t=0$.
        \item For any compact interval $J \subset [t_0-10r^2,t_0)$, there exists a constant $m>0$ and $z \in M$ such that on $M\times J$,
    \begin{align*}
        |\na u_i(x,t)|\leq m \lc d^m_t(z, x)+1 \rc.
    \end{align*}
	\end{enumerate}
\end{defn}
Note that (v) above obviously holds for any closed Ricci flow. Thus, Definition \ref{def:splitnon} agrees with Definition \ref{defnsplittingmap}.

\begin{exmp}[Tangent flow at infinity]
Let $\XX=\{M^n, (g(t))_{t \in (-\infty, 0]}\}$ be a complete Ricci flow with bounded curvature on any compact time interval of $(-\infty, 0]$. Moreover, we assume that $\XX$ has entropy bounded below by $-Y$ at a spacetime point $p^*$. Note that it follows from \cite{CMZ23} that this assumption implies that $\XX$ has entropy bounded below by $-Y$ at any spacetime point.

We consider the $d^*$-distance on $\XX$, defined as in Definition \ref{defn:dstar-distance}. For a sequence $r_i\to+\infty$, set
	\begin{align*}
g_i(t)=r_i^{-2} g(r_i^2 t),\quad \t_i=r_i^{-2} \t \quad d^{*}_i=r_i^{-1} d^*.
	\end{align*}
Then, after passing to a subsequence, we have the pointed Gromov--Hausdorff convergence
	\begin{align*}
(M \times (-\infty,0], d^{*}_i, p^*,\t_i) \xrightarrow[i \to \infty]{\quad \mathrm{pGH} \quad} (Z,d_{Z},z,\t).
	\end{align*}
where $(Z,d_{Z},z,\t)$ is a noncollapsed Ricci flow limit space over $(-\infty,0]$, which is called a \textbf{tangent flow at infinity}\index{tangent flow at infinity}. Note that $(Z,d_{Z},z,\t)$ depends on the sequence $\{r_i\}$, but is independent of $p^*$. 

It follows from Lemma \ref{lem:nashconv1} that $\NN_z(\tau)$ is constant. Thus, by Proposition \ref{prop:004} and Corollary \ref{cor:agree2}, we conclude that on the regular part $\RR_{(-\infty,0)}$,
	\begin{align*}
\Ric(g^{Z})+\na^2 f_{z}=\frac{g^{Z}}{2 \tau},
	\end{align*}
where $\tau(\cdot)=-\t(\cdot)$. Moreover, $\RR_t$ is connected for any $t \in (-\infty,0)$, and the metric $d^Z_t$ on $\RR_t$ agrees with the Riemannian distance $d_{g^Z_t}$. In addition, there exists a flow $\boldsymbol{\psi}^s$ on $Z_{(-\infty, 0)}$ so that the statements of Proposition \ref{selfsimilarall} hold.
\end{exmp}

\subsection*{K\"ahler Ricci flows}

We now consider the subspace $\widetilde{\mathcal{KM}}(n, Y, T)$ of $\widetilde \MM(n, Y, T)$ which consists of all K\"ahler Ricci flows. In particular, $n=2m$ is even.

We use the following definition, analogous to Definition \ref{def:rss}.

\begin{defn}[K\"ahler Ricci shrinker space] 
A pointed parabolic metric space $(Z',d_{Z'},z',\t')$ with $\t'(z')=0$ is called an $m$-dimensional \textbf{K\"ahler Ricci shrinker space} with entropy bounded below by $-Y$ if it satisfies $\R_- \subset \mathrm{image}(\t')$ and arises as the pointed Gromov--Hausdorff limit of a sequence of K\"ahler Ricci flows in $\widetilde{\mathcal{KM}}(2m, Y, T_i)$ with $T_i\to +\infty$. Moreover, $\NN_{z'}(\tau)$ remains constant for all $\tau>0$.
	\end{defn}

It is clear from the smooth convergence that any K\"ahler Ricci shrinker space $(Z',d_{Z'},z',\t')$ has a complex structure $J$ on the regular part $\RR'$.

\begin{lem} \label{lem:krs}
Any $m$-dimensional K\"ahler Ricci shrinker space $(Z',d_{Z'},z',\t')$ is $2k$-symmetric for some $k \in \{0, \ldots, m+1\}$.
\end{lem}

\begin{proof}
By Proposition \ref{splittingforpositivetime} (see also Remark \ref{rem:split}), we only need to prove that on $\RR'_{-1}$, if $\na y$ for some smooth function $y$ induces a splitting direction, then so does $J \na y$.

Indeed, define $y'=2g^{Z'}(\na f_{z'}, J \na y)$. A direct computation yields
	\begin{align*}
\na y'=2 (\na^2 f_{z'})^{\#}( J \na y)+2(J \na^2 y)^{\#}(\na f_{z'})=J \na y-2\Ric(g^{Z'})^{\#}( J \na y)=J \na y,
	\end{align*}
where we used $\Ric(g^{Z'})^{\#}(J \na y)=0$. Since $J\na y$ is a nonvanishing parallel vector field, it is clear that $\na y'$ induces another splitting direction.
\end{proof}

By using Lemma \ref{lem:krs}, we have the following stratification of singular points (see also \cite{JM23}).

Let $(Z, d_Z, \t)$ be a noncollapsed \textbf{K\"ahler Ricci flow limit space}, which is the pointed Gromov--Hausdorff limit of a sequence of K\"ahler Ricci flows in $\widetilde{\mathcal{KM}}(2m, Y, T)$. Then the singular set satisfies the following refined stratification:
 	\begin{equation*}
		\mathcal S^0 \subset \mathcal S^2 \subset \cdots \subset \mathcal S^{2(m-1)}=\mathcal S.
	\end{equation*}

\newpage

\appendixpage
\addappheadtotoc
\appendix
\section{Change of basis for conjugate heat kernel measures} 
\label{app:A} 

In this appendix, we derive two versions of estimates for the conjugate heat kernel measures. The proofs follow from slight modifications of \cite[Proposition 8.1]{bamler2020structure}.

\begin{prop}\label{changeofbasepoint}
	Let $\{M^n,(g(t))_{t\in I}\}$ be a closed Ricci flow. Suppose that $s,t_0,t_1\in I$ satisfy the following conditions for some constants $D$ and $Y$:
	\begin{itemize}
		\item $s<t_1\leq t_0$, and $\scal(\cdot,s)\geq -D(t_1-s)^{-1}$;
		\item $t_0-s \le D(t_1-s)$;
		\item $d_{W_1}^{t_1}(\nu_{x_0,t_0;t_1},\delta_{x_1})\leq D\sqrt{t_1-s}$;
		\item $\mathcal N_{x_0,t_0}(t_0-s)\geq -Y$ or $\mathcal{N}_{x_1,t_1}(t_1-s)\geq -Y$.
	\end{itemize}
	Define $\mathrm{d}\nu_{x_i,t_i;t}=(4\pi\tau_i)^{-n/2}e^{-f_i}dV_{g(t)}$ for $i=0,1$, where $\tau_i=t_i-t$. Then,
	\begin{align*} 
		\nu_{x_1,t_1;s}\leq C(n,Y,D)e^{C(n,Y,D)\sqrt{f_0(\cdot,s)+C(n,Y,D)}} \nu_{x_0,t_0;s}.
	\end{align*}
\end{prop}

\begin{proof}
Throughout the proof, all constants $C_i$ are positive and depend on $n$, $Y$, and $D$. Without loss of generality, we assume $s=0$ and $t_1=1\leq t_0$. 

By Proposition \ref{propNashentropy1}, we have
	\begin{align}\label{app:comp2}
		|\NN_0^{*}(x_0,t_0)-\mathcal{N}_0^{*}(x_1,1)|&\leq C(n) d^{1}_{W_1}(\nu_{x_0,t_0;1},\delta_{x_1})+\frac{n}{2}\log (t_0)\leq C_1.
	\end{align}

Moreover, we set $(z_1,1)$ to be an $H_n$-center of $(x_0,t_0)$. Then,
	\begin{align}\label{distesti}
		d_1(x_1,z_1)\leq d_{W_1}^{1}(\nu_{x_0,t_0;1},\delta_{x_1})+d_{W_1}^{1}(\nu_{x_0,t_0;1},\delta_{z_1})\leq C_2.
	\end{align}
	Fix $y_0 \in M$ and let $u:=K(\cdot,1;y_0,0)$. By the reproduction formula, it suffices to prove
	\begin{equation}\label{8inequality}
		u(x_1)\leq C(n,Y,D)e^{C(n,Y,D)\sqrt{f_0(y_0,0)+C(n,Y,D)}}\int_M u \, \mathrm{d}\nu_{x_0,t_0;1}.
	\end{equation}
	By Theorem \ref{heatkernelupperbdgeneral} (i) and Proposition \ref{gradientheatkernel}, we have
		\begin{align*}
u\leq C_3\exp\lc-\NN_0^*(\cdot,1)\rc \quad \text{and} \quad \frac{|\na u|}{u}\leq C_3\sqrt{\log\lc\frac{C_3\exp(-\NN_0^*(\cdot,1))}{u}\rc}.
	\end{align*}
	
We set $v:=C_3^{-1}u\exp(\NN_0^*(\cdot,1))/2$ and $w:=\sqrt{-\log v}$. Then we obtain
	\begin{align}\label{changebase2}
		|\na w|\leq C_4.
	\end{align} 
	Thus, to prove \eqref{8inequality}, we only need to show 
	\begin{align}\label{changebase3}
		v(x_1)\exp\lc-\NN_0^*(x_1,1)\rc\leq C(n,Y,D)e^{C(n,Y,D)\sqrt{f_{0}(y_0,0)+C(n,Y,D)}}\int_M v\exp\lc-\NN_0^*(\cdot,1)\rc \, \mathrm{d}\nu_{x_0,t_0;1}.
	\end{align}
	
By \eqref{app:comp2} and our assumption, we have
	\begin{align}\label{changebase5}
		-\NN_0^*(x_1,1)\leq C_5.
	\end{align}
	Moreover, by \eqref{distesti} and \eqref{changebase2}, we obtain 
	\begin{align*}
		w(x_1)\geq \lc w(z_1)-C_4d_1(x_1,z_1)\rc_+ \geq \lc w(z_1)-C_6\rc_+,
	\end{align*}
	which implies
	\begin{align}\label{changebase6}
		v(x_1)=\exp(-w(x_1)^2)\leq \exp \lc-(w(z_1)-C_6)_+^2\rc.
	\end{align}

Similarly, for any $y \in B:= B_1(z_1,\sqrt{2H_n(t_0-1)})$, we have
		\begin{align}\label{changebase6X}
		v(x_1) \leq \exp \lc-(w(y)-C_4d_1(x_1,y))_+^2\rc \leq C_7\exp\lc C_7w(z_1)\rc\exp\lc-w(y)^2\rc.
	\end{align}
	
We define $L \ge 0$ such that
	\begin{align*}
w(z_1)= L\sqrt{f_0(y_0,0)+L}.
	\end{align*}
If $L \le C(n, Y, D)$, then it follows from \eqref{changebase5} and \eqref{changebase6X} that \eqref{changebase3} holds. Indeed, since $\nu_{x_0,t_0;1}(B)\geq 1/2$, we have
	\begin{align*}
\int_M v\exp\lc-\NN_0^*(\cdot,1)\rc \, \mathrm{d}\nu_{x_0,t_0;1} \ge \int_M v\,\mathrm{d}\nu_{x_0,t_0;1} \ge \frac{1}{2} \min_{y \in B} \exp\lc-w(y)^2\rc,
	\end{align*}
which implies \eqref{changebase3}.

Otherwise, it follows from \eqref{changebase6} that for some constant $C_8$,
		\begin{align}\label{changebase6X1}
v(x_1)\exp\lc-\NN_0^*(x_1,1)\rc \le C_8 \exp \lc-(L\sqrt{f_0(y_0,0)+L}-C_6)^2 \rc.
	\end{align}
On the other hand, we have
		\begin{align}\label{changebase6X2}
\int_M v\exp\lc-\NN_0^*(\cdot,1)\rc \, \mathrm{d}\nu_{x_0,t_0;1}=C_3^{-1} \int_M u \, \mathrm{d}\nu_{x_0,t_0;1}=C_3^{-1} (4\pi t_0)^{-\frac n 2}e^{-f_0(y_0, 0)} \ge C_9 e^{-f_0(y_0, 0)}
	\end{align}
for some constant $C_9$. Combining \eqref{changebase6X1} and \eqref{changebase6X2}, we conclude that \eqref{changebase3} also holds if $L \gg C(n, Y, D)$.

This completes the proof.
\end{proof}

\begin{prop}\label{changeofbasepoint1}
	Let $\{M^n,(g(t))_{t\in I}\}$ be a closed Ricci flow. Suppose $s,t^*,t_0,t_1\in I$ satisfy:
	\begin{itemize}
		\item $s<t^*\leq \min\{t_0,t_1\}$, and $\scal(\cdot,s)\geq -A(t^*-s)^{-1}$;
		\item for any constants $-\infty <\alpha_1<\alpha_0<1$ and some $\theta=\theta(n, A)>0$, we have
		\begin{align*}
			t_0-t^*\leq A(t^*-s), \ t_1-t^*\leq \theta \frac{\alpha_0-\alpha_1}{1-\alpha_0}(t^*-s).
		\end{align*}
	\end{itemize}
	Let $\mathrm{d}\nu_{x_i,t_i;t}=(4\pi\tau_i)^{-n/2}e^{-f_i} \, \mathrm{d}V_{g(t)}$ for $i=0,1$, where $\tau_i=t_i-t$. Assume that
		\begin{align*}
		d_{W_1}^{t^*}(\nu_{x_0,t_0;t^*},\nu_{x_1,t_1;t^*})=D\sqrt{t^*-s}.
		\end{align*}
		Then
		\begin{align*}
		e^{\alpha_1 f_1}\mathrm{d}\nu_{x_1,t_1;s}\leq e^{(\alpha_0-\alpha_1)\mathcal{N}_{x_0,t_0}(t_0-s)}C(n,A,\alpha_0,\alpha_1)e^{C(n,A,\alpha_0,\alpha_1)D^2} e^{\alpha_0f_0}\mathrm{d}\nu_{x_0,t_0;s}.
	\end{align*}
\end{prop}

\begin{proof}
In the proof, all constants $C_i$ are positive and depend on $n$ and $A$; the constants $L_i$ are positive constants depending on $n$, $A$, $\alpha_0$ and $\alpha_1$.

Without loss of generality, we assume $s=0,t^{*}=1$ and hence
	\begin{align*}
		d^{1}_{W_1}(\nu_{x_0,t_0;1},\nu_{x_1,t_1;1})= D,\quad \scal\geq -A,\quad t_0\leq A+1.
	\end{align*}
By Proposition \ref{propNashentropy1}, we have
	\begin{align*}
		|\mathcal{N}_0^{*}(x_0,t_0)-\mathcal{N}_0^{*}(x_1,t_1)|\leq C(n,A) D+\frac{n}{2}\log \max\{t_0,t_1\}\leq L_1(D+1).
	\end{align*}
	
	Let $(z_i,1)$ be an $H_n$-center of $(x_i,t_i)$ for $i=0,1$. Then by Proposition \ref{propNashentropy1} again, we know that
	\begin{equation}\label{app:exx202}
		\mathcal{N}_0^{*}(z_1,1)\geq \mathcal{N}_0^{*}(x_1,t_1)-C(n, A) \sqrt{H_n(t_1-1)}\geq\mathcal{N}_0^{*}(x_0,t_0)-L_2(D+1).
	\end{equation}
Moreover, we have
	\begin{align*}
d_1(z_1,z_0)\leq D+2\sqrt{H_n(t_1-1)} \le D+L_2.
	\end{align*}
	
We set $\lambda:= \frac{1-\alpha_0}{1-\alpha_1}<1$ and $u(\cdot):=K(\cdot,1;y_0,0)$ for a fixed $y_0\in M$. By the reproduction formula, we only need to prove
	\begin{equation}\label{9inequality}
		\int_M u\,\mathrm{d}\nu_{x_1,t_1;1}\leq e^{(1-\lambda)\mathcal{N}_{x_0,t_0}(t_0-s)}C(n,A,\alpha_0,\alpha_1)e^{C(n,A,\alpha_0,\alpha_1)D^2}\lc\int_M u\,\mathrm{d}\nu_{x_0,t_0;1}\rc^{\lambda}.
	\end{equation}
As in the proof of Proposition \ref{changeofbasepoint}, we can find $C_1$ such that $v:= C_1^{-1}u\exp(\mathcal{N}_0^*(\cdot,1))\leq 1/2$, and $w:=\sqrt{-\log v}$ satisfies:
	\begin{align}\label{app:exx203}
		|\nabla w|\leq C_1.
	\end{align}
	Thus, \eqref{9inequality} becomes
	\begin{align}
		&\int_M v \exp(-\mathcal{N}_0^*(\cdot,1))\, \mathrm{d}\nu_{x_1,t_1;1} \notag \\
		\leq  &e^{(1-\lambda)\mathcal{N}_{x_0,t_0}(t_0-s)}C(n,A,\alpha_0,\alpha_1)e^{C(n,A,\alpha_0,\alpha_1)D^2}\lc\int_M v \exp(-\mathcal{N}_0^*(\cdot,1))\, \mathrm{d}\nu_{x_0,t_0;1}\rc^\lambda. \label{app:exx203a1}
	\end{align}
	
	We set $B:= B_1\lc z_0,\sqrt{2H_n(t_0-1)}\rc$ and choose $y_1 \in \bar B$ so that $u(y_1)=\inf_{y \in B} u(y)$. In particular, since $\nu_{x_0,t_0;1}(B)\geq 1/2$, we have
		\begin{align}\label{app:exx203a}
\lc\int_M v \exp(-\mathcal{N}_0^*(\cdot,1))\, \mathrm{d}\nu_{x_0,t_0;1}\rc^\lambda \ge &  2^{-\lambda} \exp\lc-\lambda w(y_1)^2-\lambda \mathcal{N}_0^*(y_1,1) \rc \notag \\
\ge & 2^{-\lambda} \exp\lc-\lambda w(y_1)^2-\lambda \mathcal{N}_0^*(x_0,t_0)-C_2 \rc,
	\end{align}
	where we used Proposition \ref{propNashentropy1} for the last inequality.
	
Next, it follows from \eqref{app:exx202} and Proposition \ref{propNashentropy1} that
	\begin{align}\label{app:exx204}
		0 \ge \mathcal{N}^*_0(\cdot,1) \ge \mathcal{N}_0^{*}(x_0,t_0)-L_2(D+1)-C_3 d_1(z_1, \cdot).
	\end{align}
	
Moreover, for any $q\in M$, we have by \eqref{app:exx203}
	\begin{align}
		w(q)\geq & w(y_1)-C_1 d_1(q, y_1) \notag\\
		\geq & w(y_1)-C_1\lc d_1(q,z_1)+d_1(z_1,z_0)+d_1(z_0,y_1)\rc \notag\\
		 \ge &  w(y_1)-C_1d_1(q,z_1)-L_3(D+1). \label{app:exx204a}
	\end{align}	
	
By Theorem \ref{heatkernelupperbdgeneral} (i), we see that for any $l>0$,
	\begin{equation}\label{8inequality2}
		\nu_{x_1,t_1;1}\lc M\setminus B_1(z_1,l)\rc\leq C(n) \exp \lc -\frac{l^2}{5(t_1-1)} \rc \leq C(n) \exp \lc -\frac{l^2 \lambda}{5 \theta(1-\lambda)} \rc.
	\end{equation}

Consequently, we have
	\begin{align}
		&\int_M v \exp\lc-\mathcal{N}_0^*(\cdot,1)\rc \, \mathrm{d}\nu_{x_1,t_1;1} \notag\\
		=&\sum_{j=0}^{\infty}\int_{B_1(z_1,j+1)\setminus B_1(z_1,j)}v \exp\lc-\mathcal{N}_0^*(\cdot,1)\rc \, \mathrm{d}\nu_{x_1,t_1;1}\notag\\
		\leq& e^{ -\mathcal{N}_0^{*}(x_0,t_0)+ L_2(D+1)} \sum_{j=0}^{\infty}\int_{B_1(z_1,j+1)\setminus B_1(z_1,j)}v e^{ C_3d_1(\cdot,z_1)} \, \mathrm{d}\nu_{x_1,t_1;1}\notag\\
		\leq& C(n)e^{ -\mathcal{N}_0^{*}(x_0,t_0)+ L_2(D+1)}\sum_{j=0}^{\infty} \exp \lc -\lc w(y_1)-C_1(j+1)-L_3(D+1) \rc_+^2+C_3(j+1)-\frac{j^2 \lambda}{5 \theta(1-\lambda)} \rc, \label{app:exx205}
	\end{align}
	where we used \eqref{app:exx204}, \eqref{app:exx204a} and \eqref{8inequality2}. If $\theta \le C_1^{-2}/10$, then we have
		\begin{align}
&  \exp \lc -\lc w(y_1)-C_1(j+1)-L_3(D+1) \rc_+^2+C_3(j+1)-\frac{j^2 \lambda}{5 \theta(1-\lambda)} \rc\notag\\
\le & e^{L_4(D^2+1)}\exp \lc -\lc w(y_1)-C_1(j+1)-L_3(D+1) \rc_+^2-\frac{\lambda}{1-\lambda} \lc C_1(j+1)+L_3(D+1) \rc^2-j \rc \notag\\
\le & e^{L_4(D^2+1)}\exp \lc -\lambda \lc \lc w(y_1)-C_1(j+1)-L_3(D+1) \rc_++ C_1(j+1)+L_3(D+1) \rc^2-j \rc \notag\\
 \le & \exp \lc L_4(D^2+1)-j-\lambda w(y_1)^2 \rc, \label{app:exx206}
	\end{align}
	where we used $x^2+\frac{\lambda}{1-\lambda} y^2 \ge \lambda(x+y)^2$ for the second inequality. Combining \eqref{app:exx203a}, \eqref{app:exx205} and \eqref{app:exx206}, we obtain \eqref{app:exx203a1}.
\end{proof}

\section{Comparison of spacetime distances}\label{app:B}

Let $\XX=\{M^n,(g(t))_{t\in\III^{++}}\}\in\MM(n,T)$. We now give an alternative definition of $d^*$.

\begin{defn}\label{defn:dstar-alpha}
Fix $\alpha >0$. For $x^*,y^*\in\XX_{\III^+}$, set $t_+=\max\{\t(x^*),\t(y^*)\}$ and $t_-=\min\{\t(x^*),\t(y^*)\}$. Define
\begin{align*}
d^{*,\alpha}(x^*,y^*):=\inf_{-(1-\sigma)T\leq\tau\leq t_-}
\max\left\{\sqrt{t_+-\tau},\ \alpha d_{W_1}^{\tau}(\nu_{x^*;\tau},\nu_{y^*;\tau})\right\}.
\end{align*}
For an ancient flow, the infimum is taken over $\tau\in(-\infty,t_-]$.
\end{defn}

The following lemma shows that all these distances are equivalent to $d^*=d^{*,1}$.

\begin{lem}
	For any $x^*,y^*\in \XX_{\III^{+}}$ and $\alpha>0$, 
	\begin{align*}
		\min\{\alpha,1\}d^*(x^*,y^*)\leq d^{*,\alpha}(x^*,y^*)\leq \max\{\alpha,1\} d^*(x^*,y^*).
	\end{align*}
\end{lem}
\begin{proof}
	Without loss of generality, we assume $t=\t(x^*)\geq s=\t(y^*)$. We set $r=d^*(x^*,y^*)$ and $\tau=\max\{-(1-\sigma)T, t-r^2\}$. By \eqref{eq:dstar-equality1}, 
	\begin{align*}
		d_{W_1}^{\tau}(\nu_{x^*;\tau},\nu_{y^*;\tau}) \le r,
	\end{align*}
	which, by Definition \ref{defn:dstar-alpha}, implies $d^{*,\alpha}(x^*,y^*)\leq \max\{\alpha,1\} r$. 
	
	On the other hand, set $r'=d^{*,\alpha}(x^*,y^*)$ and $\tau'=\max\{-(1-\sigma)T, t-(r')^2\}$. By Definition \ref{defn:dstar-alpha}, 
	\begin{align*}
		\alpha d_{W_1}^{\tau'}(\nu_{x^*;\tau'},\nu_{y^*;\tau'})\le r',
	\end{align*}
	which, by \eqref{eq:defn-dstar}, implies $d^*(x^*,y^*)\leq \max\{\alpha^{-1},1\}r'$. This completes the proof.
\end{proof}

Definition \ref{defn:dstar-distance} uses the $W_1$-Wasserstein distance. We compare it with the analogous construction based on $W_2$.

\begin{defn}\label{def:dstarW2}
For $x^*,y^*\in\XX_{\III^+}$, set $t_+=\max\{\t(x^*),\t(y^*)\}$ and $t_-=\min\{\t(x^*),\t(y^*)\}$. Define
\begin{align*}
d_2^*(x^*,y^*):=\inf_{-(1-\sigma)T\leq\tau\leq t_-}
\max\left\{\sqrt{t_+-\tau},\ d_{W_2}^{\tau}(\nu_{x^*;\tau},\nu_{y^*;\tau})\right\}.
\end{align*}
For an ancient flow, the infimum is taken over $\tau\in(-\infty,t_-]$.
\end{defn}

The proof of Lemma \ref{lem:000a}, using the monotonicity of $W_2$, shows that $d_2^*$ is a distance on $\XX_{\III^+}$ and that the time-function is $2$-H\"older with respect to $d_2^*$.

\begin{prop}\label{prop:W1W2comparison}
For all $x^*,y^*\in\XX_{\III^+}$,
\begin{align}\label{eq:w1w2}
d^*(x^*,y^*)\leq d_2^*(x^*,y^*)
\leq(1+2\sqrt{H_n})d^*(x^*,y^*).
\end{align}
\end{prop}
\begin{proof}
The first inequality follows from $d_{W_1}\leq d_{W_2}$.

Set $r=d^*(x^*,y^*)$ and assume $t=\t(x^*)\geq s=\t(y^*)$. Put
\begin{align*}
\tau:=\max\{t-r^2,-(1-\sigma)T\}.
\end{align*}
Equation \eqref{eq:dstar-equality1} gives
\begin{align*}
d_{W_1}^{\tau}(\nu_{x^*;\tau},\nu_{y^*;\tau})\leq r.
\end{align*}
Using Lemma \ref{lem:var} and the $H_n$-concentration of the conjugate heat kernels,
\begin{align*}
d_{W_2}^{\tau}(\nu_{x^*;\tau},\nu_{y^*;\tau})
&\leq d_{W_1}^{\tau}(\nu_{x^*;\tau},\nu_{y^*;\tau})
+\sqrt{\Var_\tau(\nu_{x^*;\tau})}+\sqrt{\Var_\tau(\nu_{y^*;\tau})}\\
&\leq r+\sqrt{H_n(t-\tau)}+\sqrt{H_n(s-\tau)}\leq(1+2\sqrt{H_n})r.
\end{align*}
Since also $\sqrt{t-\tau}\leq r$, Definition \ref{def:dstarW2} proves the second inequality.
\end{proof}

\begin{rem}
The same definitions and comparison hold on a metric-flow limit after replacing the Wasserstein profiles by their essential left limits, as in Definition \ref{defn:dstar-limit}.
\end{rem}

\section{Eigenvalues and almost splitting} 
\label{app:C} 

In this section, we consider a closed Ricci flow $\XX=\{M^n,(g(t))_{t\in I}\}$. All time subintervals considered below are assumed to be contained in $I$.

We fix a spacetime point $x_0^*=(x_0,t_0)\in \XX$ and set
\begin{align*}
\mathrm{d}\nu_t=\mathrm{d}\nu_{x^*_0;t}=(4\pi\tau)^{-n/2}e^{-f} \mathrm{d}V_{g(t)},
\end{align*}
where $\tau=t_0-t$. We consider the weighted Laplacian $\Delta_f=\Delta-\la \na \cdot, \na f \ra$. It is clear that $\Delta_f$ is self-adjoint with respect to $\mathrm{d}\nu_t$.

\begin{defn}\label{def:fre}
Given a subinterval $J \subset I$ containing $t_0$ and a smooth function $u$ on $M \times J$, we define
	\begin{align*}
	I_u(t):=\int_M u^2\,\mathrm{d}\nu_t,\quad \ D_u(t):=\int_M |\nabla u|^2\,\mathrm{d}\nu_t, \quad F_u(t):=\frac{\tau\int_M |\nabla u|^2\,\mathrm{d}\nu_t}{\int_M u^2\,\mathrm{d}\nu_t}=\frac{\tau D_u(t)}{I_u(t)}
	\end{align*}
for $t \in J$ and $t<t_0$. $F_u$ is called the \textbf{frequency}\index{frequency} of $u$, which is well defined for $I_u(t) \ne 0$.
\end{defn}

The following frequency estimate essentially follows from \cite{colding2024eigenvalue}.

\begin{lem}[Frequency estimate]\label{frequencyestimate}
Suppose $\square u=0$ on $M \times J$. Then the following evolution of the frequency holds for $t \in J$:
	\begin{align}\label{evolutionoffrequency}
		\diff{}{t} F_u(t)=-\frac{F_u}{\tau}+\frac{2}{\tau}F_u^2-\frac{2\tau\int_M |\nabla ^2 u|^2\,\mathrm{d}\nu_t}{I_u(t)}.
	\end{align}
	In particular, we have
	\begin{equation}\label{evolutioninequalityoffrequency}
		\diff{}{t}F_u(t)\leq \frac{F_u(t)}{\tau}(2F_u(t)-1).
	\end{equation}
	If $F_u(t_0-r^2)\leq \dfrac{1}{2}+\ep$ for some $r>0$, then for any $t \in [t_0-r^2, t_0-4\ep r^2]$, we have
	\begin{align*}
		F_u(t)\leq \frac 1 2+\frac{2\ep r^2}{\tau}.
	\end{align*}
\end{lem}
\begin{proof}	
	Since 
	\begin{align*}
\diff{}{t}I_u(t)=\int_M \square u^2\,\mathrm{d}\nu_t=\int_M 2u\square u-2|\nabla u|^2 \, \mathrm{d}\nu_t=-2D_u(t)
	\end{align*}
	and 
	\begin{align*}
\diff{}{t} D_u(t)=-2\int_M |\nabla^2u|^2\,\mathrm{d}\nu_t,
	\end{align*}
	we have
	\begin{align*}
		\diff{}{t} F_u(t)=\diff{}{t}\lc\frac{\tau D_u(t)}{I_u(t)}\rc&=-\frac{D_u(t)}{I_u(t)}-\frac{2\tau\int_M |\nabla ^2 u|^2\,\mathrm{d}\nu_t}{I_u(t)}+\frac{2\tau \lc\int_M |\nabla u|^2\,\mathrm{d}\nu_t\rc^2}{\lc\int_M u^2\,\mathrm{d}\nu_t\rc^2}\nonumber\\
		&=-\frac{F_u}{\tau}+\frac{2}{\tau}F_u^2-\frac{2\tau\int_M |\nabla ^2 u|^2 \mathrm{d}\nu_t}{I_u(t)},
	\end{align*}
	which gives \eqref{evolutionoffrequency}. Moreover, \eqref{evolutioninequalityoffrequency} follows immediately.
	
Without loss of generality, we assume $t_0=0$ and $r=1$. If $F_u(-1)\leq\frac 1 2+\ep$, by integrating \eqref{evolutioninequalityoffrequency}, we obtain
	\begin{align*}
F_u(t)\leq\frac{1}{2-\tau^{-1}\lc\dfrac{2F_u(-1)-1}{F_u(-1)}\rc}\leq \frac{1}{2-\dfrac{4\epsilon}{\tau}}.
	\end{align*}
Thus, if $4\ep \leq \tau \leq 1$, then
	\begin{align*}
	F_u(t)\leq \frac{1}{2}+2\tau^{-1}\epsilon.
	\end{align*}
\end{proof}

We now discuss the relationship between eigenvalues of $\Delta_f$ and splitting maps. Denote by $0<\lambda_1(t)\leq\lambda_2(t)\leq\ldots$ the eigenvalues of $-\Delta_f$ at time $t$, counted with multiplicities. Recall that by Theorem \ref{poincareinequ}, $\tau\lambda_1(t)\geq 1/2$ for any $t<t_0$.

The next proposition describes the propagation of the eigenvalues.

\begin{prop}[Propagation of eigenvalues]\label{propagationofeigenvalues}
	If $(r^2\lambda_k)(t_0-r^2) \leq \frac{1}{2}+\epsilon$, then for any $t \in [t_0-r^2, t_0-4\ep r^2]$,
	\begin{align}\label{eq:eigenes1}
		(\tau\lambda_k)(t)\leq \frac{1}{2}+\frac{2\ep r^2}{\tau}.
	\end{align}
	Moreover, we can find $\vec u=(u_1,\ldots,u_k):M\times [t_0-r^2,t_0]\to\R^k$ such that the following holds.
	
	For any $\delta \in [4 \ep, 1]$ and $i, j \in \{1, \ldots, k\}$,
	\begin{enumerate}[label=\textnormal{(\roman{*})}]
		\item On $M\times [t_0-r^2,t_0]$, $\square u_i=0$ and $ u_i(x_0^*)=0$.
		\item $\displaystyle \int_{t_0-r^2}^{t_0-\delta r^2}\int_M |\nabla^2 u_i|^2\, \mathrm{d}\nu_t \mathrm{d}t\leq 3 \delta^{-1} \ep$.
		\item For any $t \in [t_0-r^2, t_0-\delta r^2]$, we have $\displaystyle \abs{\int_M \la \nabla u_i,\nabla u_j\ra \, \mathrm{d}\nu_t-\delta_{ij}\lambda_i(t_0-r^2)}\leq 6 \delta^{-1} \ep$.
	\end{enumerate}
\end{prop}
\begin{proof}
Without loss of generality, we assume $r=1$ and $t_0=0$. Choose $\phi_i(-1)$, $i=1,\ldots,k$, to be $L^2$-orthonormal eigenfunctions (with respect to $\mathrm{d}\nu_{-1}$) corresponding to $\lambda_i(-1)$.

Next, we solve 
	\begin{align*}
		\square u_i=0, \quad u_i=\phi_i(-1)\quad \mathrm{at}\quad t=-1.
	\end{align*}
	Denote $I_i(t):= I_{u_i}(t), D_i(t):=D_{u_i}(t)$ and $F_i(t)=F_{u_i}(t)$. By Lemma \ref{frequencyestimate},
	\begin{align}\label{evofre1}
		\diff{}{t} F_i(t)=-\frac{F_i}{\tau}+\frac{2}{\tau}F_i^2-\frac{2\tau\int |\nabla ^2 u_i|^2\,\mathrm{d}\nu_t}{I_i(t)},
	\end{align}
	and
	\begin{equation} \label{evofre0}
		\diff{}{t} F_i(t)\leq \frac{F_i(t)}{\tau}(2F_i(t)-1).
	\end{equation}
	
By the Gram-Schmidt process, it follows from \eqref{evofre0} by using the same argument as in \cite[Equation (3.34)]{colding2024eigenvalue} that
	\begin{align*}
		\lc \diff{}{t}(\tau\lambda_k) \rc(-1) \leq \frac{(\tau\lambda_k)(-1)}{\tau}\big(2(\tau\lambda_k)(-1)-1\big).
	\end{align*}
	Since the argument works at any time, we can conclude that for any $t\in [-1,0)$,
	\begin{align*}
		\diff{}{t} (\tau\lambda_k)(t)\leq \frac{(\tau\lambda_k)(t)}{\tau}\big(2(\tau\lambda_k)(t)-1\big).
	\end{align*}
	By integration, we obtain \eqref{eq:eigenes1}.
	
	By Lemma \ref{frequencyestimate}, $F_i(t)\in [1/2,1/2+2\tau^{-1}\epsilon]$ for any $\tau \in [4\ep,1]$. Thus, we can integrate \eqref{evofre1} to get
	\begin{align}\label{eigenevolve1}
		\int_{-1}^{-\delta}\frac{2\tau\int_M |\nabla ^2 u_i|^2\, \mathrm{d}\nu_t}{I_i(t)}\mathrm{d}t&=F_i(-\delta)-F_i(-1)+\int_{-1}^{-\delta}\lc\frac{F_i(t)}{\tau}(2F_i(t)-1)\rc \mathrm{d}t\nonumber\\
		&\leq 2 \delta^{-1} \ep+4\ep\int_{-1}^{-\delta}\tau^{-2}\,\mathrm{d}t\leq 6 \delta^{-1} \ep.
	\end{align}
	Since $\diff{}{t}\log I_i(t)=-2\tau^{-1}F_i(t)\leq -\tau^{-1}$, we obtain for $t\in [-1,-\delta]$
	\begin{align*}
		I_i(t) \le  I_i(-1)e^{-\int_{-1}^{t} |s|^{-1}\,\mathrm{d}s} =\tau
	\end{align*}
	Combining this with \eqref{eigenevolve1}, we have
	\begin{align}\label{hesianesti1}
		\int_{-1}^{-\delta}\int_M |\nabla^2 u_i|^2\,\mathrm{d}\nu_t \mathrm{d}t\leq 3 \delta^{-1} \ep.
	\end{align}
	Since $\diff{}{t}\int_M u_i\,\mathrm{d}\nu_t=\int_M \square u_i\,\mathrm{d}\nu_t=0$ and $\int \phi_i(-1)\, \mathrm{d}\nu_{-1}=0$, we see that for all $t\in [-1,0]$, 
	\begin{align*}
		\int_M u_i\,\mathrm{d}\nu_t=0.
	\end{align*}
Moreover, for any $1\leq i,j \leq k$, since $\int_M \la\nabla u_i,\nabla u_j\ra \, \mathrm{d}\nu_{-1}=\delta_{ij}\lambda_i(-1)$ and 
	\begin{align*}
\diff{}{t}\int_M \la\nabla u_i, \nabla u_j\ra \, \mathrm{d}\nu_t=-2\int_M \la\nabla^2 u_i, \nabla^2 u_j\ra \, \mathrm{d}\nu_t,
	\end{align*}
we have for any $t\in [-1,-\delta]$,
	\begin{align*}
		&\left|\int_M \la \nabla u_i,\nabla u_j\ra \, \mathrm{d}\nu_t-\delta_{ij}\lambda_i(-1)\right| \notag \\
		\le & 2 \int_{-1}^t \int_M |\na^2 u_i| |\na^2 u_j| \, \mathrm{d}\nu_s \mathrm{d}s \notag \\
		\le & 2\lc  \int_{-1}^t \int_M |\na^2 u_i|^2 \, \mathrm{d}\nu_s \mathrm{d}s\rc^{\frac 1 2} \lc \int_{-1}^t \int_M |\na^2 u_j|^2 \, \mathrm{d}\nu_s \mathrm{d}s\rc^{\frac 1 2}
		\leq 6 \delta^{-1} \ep,
	\end{align*}
	where we used \eqref{hesianesti1} for the last inequality.
	
	This completes the proof.
\end{proof}

For the map $\vec u$, we can modify it by a positive definite matrix so that the following holds (see Definition \ref{defnsplittingmap}).

\begin{cor} \label{cor:exisplitting}
If $(r^2\lambda_k)(t_0-r^2) \leq \frac{1}{2}+\epsilon$, then there exists a $(k,C \ep,r/\sqrt{10})$-splitting map at $x_0^*$, where $C$ is a universal constant.
\end{cor}

The following proposition shows that, under the assumption of almost self-similarity, the existence of a $(k,\epsilon,r)$-splitting map is equivalent to the smallness of $r^2\lambda_k(-r^2)-\frac 1 2$.

\begin{prop}
Suppose that $\vec{u}=(u_1,\ldots,u_k)$ is a $(k,\epsilon, r)$-splitting map at $x_0^*$ with $\ep \le \ep(n)$, and 
	\begin{align*}
\WW_{x_0^*}(r^2/10)-\WW_{x_0^*}(10 r^2) \le \delta.
	\end{align*}
Then there exists a constant $C=C(n)>0$ such that
	\begin{align*}
		(\tau\lambda_k)(t_0-r^2) \leq \frac{1}{2}+C(\epsilon+\delta^{\frac{1}{2}}).
	\end{align*}
\end{prop}
\begin{proof}
Without loss of generality, we assume $r=1$ and $t_0=0$. In the proof, the constant $C$ denotes a universal constant, which can be different line by line.

By our assumption, we know that
	\begin{align*}
		2\int_{-10}^{-1/10}\int_M \tau\abs{\Ric+\nabla ^2f-\frac{g}{2\tau} }^2\,\mathrm{d}\nu_t \mathrm{d}t\leq \delta.
	\end{align*}
	Let $\{\phi_i(t)\}$ be a sequence $L^2$-orthonormal eigenfunctions corresponding to eigenvalues $\lambda_i(t)$. For a smooth function $u$ with decomposition $u=\sum_{l=1}^{\infty}a_l\phi_l$ at $t$, we have
	\begin{align}\label{eigenidentity1}
		\int_M \frac{1}{2\tau}|\nabla u|^2-(\Delta_f u)^2\,\mathrm{d}\nu_t=\sum_{l=1}^{\infty}\lambda _l(\frac{1}{2\tau}-\lambda_l)a_l^2\leq 0.
	\end{align}
By Bochner's formula, we have
	\begin{align*}
		\int_M \lc \frac{g}{2\tau}-\nabla^2f-\Ric\rc (\nabla u,\nabla u)\, \mathrm{d}\nu_t=\int_M |\nabla ^2 u|^2+\frac{1}{2\tau}|\nabla u|^2-(\Delta_f u)^2 \, \mathrm{d}\nu_t.
	\end{align*}
	Applying this to $u_i(t)=\sum_{l=1}^{\infty} a_l^i(t)\phi_l(t)$, we get
	\begin{align*}
		\int_M \lc \frac{g}{2\tau}-\nabla^2f-\Ric\rc (\nabla u_i,\nabla u_i)\, \mathrm{d}\nu_t=\int_M |\nabla ^2 u_i|^2+\frac{1}{2\tau}|\nabla u_i|^2-(\Delta_f u_i)^2\,\mathrm{d}\nu_t.
	\end{align*}
	By Proposition \ref{timesliceL2estimateofgradient}, we have
	\begin{align*}
		&\left|\int_{-3}^{-1/10}\int_M \lc \frac{g}{2\tau}-\nabla^2f-\Ric\rc(\nabla u_i,\nabla u_i)\, \mathrm{d}\nu_t \mathrm{d}t \right| \\
		\leq & \lc\int_{-3}^{-1/10}\int \abs{\frac{g}{2\tau}-\nabla^2f-\Ric}^2\,\mathrm{d}\nu_t \mathrm{d}t\rc^{1/2}\lc\int_{-3}^{-1/10}|\nabla u_i|^4\,\mathrm{d}\nu_t \mathrm{d}t\rc^{1/2}\leq C\delta^{1/2}.
	\end{align*}
	Combining this with \eqref{eigenidentity1} and Definition \ref{defnsplittingmap} (iii), we obtain
	that for all $1\leq i\leq k$,
	\begin{align*}
		\int_{-3}^{-1/10}\sum_{l=1}^{\infty}\lambda_l(\lambda_l-\frac{1}{2\tau})(a_l^i)^2(t)\,\mathrm{d}t\leq C(\epsilon+\delta^{1/2}).
	\end{align*}
	
	In particular, we have
	\begin{align*}
		\int_{-3}^{-1/10}(\lambda_k-\frac{1}{2\tau})\sum_{l=k}^{\infty}\lambda_l (a_l^i)^2(t)\,\mathrm{d}t\leq \int_{-3}^{-1/10}\sum_{l=k}^{\infty} \lambda_l(\lambda_l-\frac{1}{2\tau})(a_l^i)^2(t)\, \mathrm{d}t\leq C(\epsilon+\delta^{1/2}).
	\end{align*}
Thus, we can find $s_1\in [-3,-2]$ such that for all $1\leq i\leq k$,
	\begin{equation}\label{eigenestimate1}
		\big(\lambda_k(s_1)-\frac{1}{2\tau(s_1)}\big)\sum_{l=k}^{\infty}\lambda_l(s_1) (a_l^i)^2(s_1)\leq C(\epsilon+\delta^{1/2}).
	\end{equation}	
	Since $\big|\int_M \la \nabla u_i,\nabla u_j\ra \mathrm{d}\nu_{s_1}-\delta_{ij}\big|\leq 2\epsilon$ by Proposition \ref{timesliceL2estimateofgradient}, we have at time $s_1$,
	\begin{align}\label{eigenidentity5}
		\abs{\delta_{ij}-\sum_{l=1}^{\infty}\lambda_l a_l^ia_l^j}\leq 2\epsilon.
	\end{align}
	If there exists $1\leq i_0\leq k$ such that the following holds: for some small dimensional constant $c_0>0$ to be determined later, 
	\begin{align*}
		\sum_{l=k}^{\infty}\lambda_l(s_1)(a_l^{i_0})^2(s_1)\geq c_0, 
	\end{align*}
	then by \eqref{eigenestimate1}, we obtain $\lambda_k(s_1)-\dfrac{1}{2\tau(s_1)}\leq Cc_0^{-1}(\ep+\delta^{1/2})$. Therefore, the conclusion follows from Proposition \ref{propagationofeigenvalues}. Now we assume that for all $1\leq i\leq k$, 
	\begin{align}\label{eigenidentity8}
\sum_{l=k}^{\infty}\lambda_l(s_1)(a_l^i)^2(s_1)\leq c_0 \quad \text{and hence} \quad \sum_{l=1}^{k-1}\lambda_l(s_1)(a_l^i)^2(s_1)\geq 1-c_0-2\ep.
	\end{align}
At time $s_1$, by \eqref{eigenidentity5}, we have for $i\neq j$, 
	\begin{align}\label{eigenidentity9}
\abs{\sum_{l=1}^{k-1}\lambda_l a_l^ia_l^j} \leq 2\ep+\big|\sum_{l=k}^{\infty}\lambda_la_l^ia_l^j\big|\leq 2\ep+\lc\sum_{l=k}^{\infty}\lambda_l(a_l^i)^2\rc^{1/2} \lc\sum_{l=k}^{\infty}\lambda_l(a_l^j)^2\rc^{1/2}\leq 2\ep+c_0.
	\end{align}
We define an inner product for $(k-1)$-tuples as follows: for $\vec a=(a_1,\ldots,a_{k-1}),\vec b=(b_1,\ldots,b_{k-1})$, set
	\begin{align*}
		\la \vec a, \vec b \ra=\sum_{l=1}^{k-1}\lambda_l(s_1) a_lb_l.
	\end{align*}
	Thus, for $a^i=(a_1^i,\ldots, a_{k-1}^i)$, we see that for all $1\leq i\neq j\leq k$,
	\begin{align*}
		1-c_0-2\ep \leq 	\la a^i, a^i \ra \leq 1+2\ep, \quad \abs{\la a^i, a^j \ra}\leq 2 \ep+c_0.
	\end{align*}
	Thus, if $c_0$ and $\ep$ are smaller than some constant depending on $n$, the cardinality of such a family $\{a^i\}$ is at most $k-1$. This contradicts \eqref{eigenidentity8} and \eqref{eigenidentity9}.
	
	This completes the proof.
\end{proof}

\section{Spines of Ricci shrinker spaces} 
\label{app:D} 

Let $(Z',d_{Z'},z',\t')$ be an $n$-dimensional Ricci shrinker space with entropy bounded below by $-Y$ (see Definition \ref{def:rss}). We denote by $\RR'$ the regular set, which is realized as a Ricci flow spacetime $(\RR', \t', \partial_{\t'}, g^{Z'}_t)$. For simplicity, we set $f=f_{z'}$ and $\nu_t=\nu_{z'; t}$.

First, we prove the following lemma. Here, $\XX^{z'}$ is the associated metric flow at $z'$.

\begin{lem}\label{lem:RCD}
$\lc \iota_{z'}(\XX^{z'}_{-1}), d^{Z'}_{-1}, \nu_{-1} \rc$ is an $\rcd(1/2, \infty)$-space.
\end{lem}

\begin{proof}
By Proposition \ref{prop:004}, the following Ricci shrinker equation holds on $\RR'_{(-\infty, 0)}$:
	\begin{align} \label{eq:soliton1}
\Ric(g^{Z'})+\na^2 f=\frac{g^{Z'}}{2\tau},
	\end{align}
where $\tau:=-\t'$. In addition, thanks to Corollary \ref{cor:agree2}, $\RR'_t$ is connected for any $t<0$, and the distance $d^{Z'}_t$ on $\RR'_t$ agrees with the distance induced by $g^{Z'}_t$. 

Furthermore, the Minkowski dimension of the singular set $\iota_{z'}(\XX^{z'}_{-1}) \setminus \RR'_{-1}$ is at most $n-4$ (see Theorem \ref{thm:spacemink} (i)). By combining the Ricci shrinker equation \eqref{eq:soliton1} with the high codimension of the singular set, one can then derive the desired conclusion using the same argument as in \cite[Proposition A.16]{li2024rigidity}.

For the reader's convenience, we sketch the proof below. For simplicity, we set $(X, d, \mu)=\lc \iota_{z'}(\XX^{z'}_{-1}), d^{Z'}_{-1}, \nu_{-1} \rc$. We define the Sobolev space $W^{1, 2}(X, \mu)$ to be the subspace of $L^2(X, \mu)$ consisting of functions $u$ for which
	\begin{align*}
\|u\|^2_{W^{1, 2}} =\|u\|^2_{L^2}+\inf_{u_i} \liminf_{i \to \infty} \|h_i\|^2_{L^2}<\infty,
	\end{align*}
where the infimum is taken over all upper gradients $h_i$ of the function $u_i$ with $\|u_i-u\|_{L^2} \to 0$. The same argument as in \cite[Corollary 2.12]{chen2017space} shows that $C^{\infty}_c(\RR'_{-1})$ is dense in $W^{1, 2}$, because the singular set has codimension greater than $2$. We then consider the standard nonnegative symmetric bilinear form:
	\begin{align*}
\mathcal D(u, v):=\int_{\RR'_{-1}} \la \na u, \na v\ra \, \mathrm{d}\mu
	\end{align*}
for $u, v \in W^{1,2}$. It can be proved (see \cite[Corollary 2.15]{chen2017space}) that $\mathcal D$ is an irreducible, strongly local, and regular Dirichlet form. Moreover, if we denote by $\Delta_f$ the unique generator associated with $\mathcal D$, then the following Bakry-\'Emery condition holds:
	\begin{align*}
\frac{1}{2} \int |\na u|^2 \Delta_f v \, \mathrm{d}\mu \ge \frac{1}{2} \int v|\na u|^2 \, \mathrm{d}\mu+\int v\la \na u, \na \Delta_f u \ra  \, \mathrm{d}\mu
	\end{align*}
for any $u \in D(\Delta_f)$ with $\Delta_f u \in W^{1,2}$ and  $v \in L^{\infty} \cap D(\Delta_f)$ with $v \ge 0$ and $\Delta_f v \in L^\infty$. Here, $D(\Delta_f)$ denotes the domain of $\Delta_f$.

With these facts established, we see that conditions (i) and (iv) in \cite[Definition 3.1]{gigli2018lecture} are satisfied. Moreover, condition (ii) in \cite[Definition 3.1]{gigli2018lecture} is trivially satisfied since $\mu$ is a probability measure. Finally, condition (iii) in \cite[Definition 3.1]{gigli2018lecture} is immediate from Corollary \ref{cor:agree2}.
\end{proof}

Next, we prove

\begin{lem}\label{lem:appD2}
Suppose $u$ is a smooth function on $\RR'_{t}$ for $t<0$ such that
	\begin{align*}
\int_{\RR'_{t}} u^2+|\na u|^2 \, \mathrm{d}\nu_{t}<\infty
	\end{align*}
and $\Delta_f u+\frac{u}{2|t|}=0$, where $\Delta_f=\Delta_{g^{Z'}}-\la \na f, \na  \ra$ at $t$. Then $\na u$ induces a splitting factor $\R$ on $\RR'_t$.
\end{lem}

\begin{proof}
Without loss of generality, we assume $t=-1$.

Using the notation from the proof of Lemma \ref{lem:RCD}, our assumptions imply that $u \in W^{1,2}$. The conclusion then follows directly from \cite[Proposition 3.2]{GKKO}, since $\lc \iota_{z'}(\XX^{z'}_{-1}), d^{Z'}_{-1}, \nu_{z';-1} \rc$ is an $\rcd(1/2, \infty)$-space.
\end{proof}

We call 
	\begin{align*}
\boldsymbol{\mu}:=\NN_{z'}(1)
	\end{align*}
the \textbf{entropy} of the Ricci shrinker space.

Next, we show

\begin{lem}
For any $x \in Z'$,
	\begin{align}\label{eq:appd001}
\lim_{\tau \to +\infty}\NN_{x}(\tau)= \boldsymbol{\mu}.
	\end{align}
In particular, for any $\tau>0$,
	\begin{align} \label{eq:appd002}
\NN_{x}(\tau) \ge \boldsymbol{\mu}.
	\end{align}
\end{lem}

\begin{proof}
We only prove \eqref{eq:appd001}, from which \eqref{eq:appd002} follows by monotonicity.

After taking the limit for Proposition \ref{nashlip1}, we obtain for any $x \in Z'$,
	\begin{align*}
\abs{\NN_{x}(\tau)-\NN_{z'}(\tau)} \le \frac{C(n)}{\sqrt{\tau}} d_{Z'}(x, z').
	\end{align*}
Letting $\tau \to +\infty$, we obtain \eqref{eq:appd001} and hence complete the proof.
\end{proof}

We have the following definition.

\begin{defn}
The \textbf{spine}\index{spine} of a Ricci shrinker space $(Z',d_{Z'},z',\t')$ is defined by
	\begin{align*}
\mathrm{spine}(Z'):=\{x \in Z' \mid \NN_{x}(\tau)=\boldsymbol{\mu},\, \forall \tau>0 \}.
	\end{align*}
	Moreover, we define the \textbf{arrival time}:
		\begin{align*}
t_a:=\sup \{\t'(x) \mid x \in \mathrm{spine}(Z')\} \in [0, \infty].
	\end{align*}
The \textbf{dimension} of $\mathrm{spine}(Z')$ is defined to be the unique integer $k \in [0, n+2]$ such that $(Z',d_{Z'},z',\t')$ is $k$-symmetric, but not $(k+1)$-symmetric (see Definition \ref{defnsymmetricsoliton}).
\end{defn}

We next prove the static principle.

\begin{lem} \label{lem:realstatic}
Suppose $y \in \mathrm{spine}(Z')$ with $\t'(y) \ne 0$. Then $(Z',d_{Z'},z',\t')$ is a static or quasi-static cone (see Definition \ref{def:stacone}). In this case, the Ricci curvature vanishes on $\RR'_{(-\infty, t_a]}$.
\end{lem}

\begin{proof}
Without loss of generality, we assume $\t'(y)>-1$.

We set $f'=f_{y}$. Then it follows from Proposition \ref{prop:004} that on $\RR'_{(-\infty, \t'(y))}$,
	\begin{align*}
\Ric(g^{Z'})+\na^2 f'=\frac{g^{Z'}}{2(\t'(y)-\t')},
	\end{align*}
	which, when combined with $\Ric(g^{Z'})+\na^2 f=\frac{g^{Z'}}{2|\t'|}$, implies
	\begin{align} \label{eq:appd003}
\Ric+\na^2 u=0
	\end{align}
on $\RR'_{-1}$, where
	\begin{align*}
u:=\frac{(\t'(y)+1)f'-f}{\t'(y)}.
	\end{align*}
Applying $\delf$ to \eqref{eq:appd003}, we obtain on $\RR'_{-1}$,
	\begin{align*}
\delf(\na^2 u)=\na \lc \Delta_f u+\frac{1}{2} u\rc=0.
	\end{align*}
Thus, it follows that
	\begin{align*}
\Delta_f u+\frac{1}{2} u \equiv c
	\end{align*}
 for a constant $c$, since $\RR'_{-1}$ is connected. Define $u':=u-2c$. Then we have
 	\begin{align} \label{eq:appd005}
\Delta_f u'+\frac{1}{2} u'=0.
	\end{align}
 
  On the other hand, since all $f$, $f'$, $|\na f|^2$ and $|\na f'|^2$ increase at most quadratically (see \eqref{eq:tang2} and Lemma \ref{lem:metrices1}), we conclude that $|u'|+|\na u'|$ belongs to $L^2(Z'_{-1}, \nu_{-1})$. Consequently, it follows from \eqref{eq:appd005} and Lemma \ref{lem:appD2} that $\na^2 u' \equiv 0$ on $\RR'_{-1}$. Combined with \eqref{eq:appd003}, it follows that $\Ric \equiv 0$ on $\RR'_{-1}$. Thus, $(Z',d_{Z'},z',\t')$ is a static or quasi-static cone.
  
  By the same argument, one concludes that $\Ric \equiv 0$ on $\RR'_{(-\infty, t_a)}$. By taking the limit, we also obtain $\Ric \equiv 0$ on $\RR'_{(-\infty, t_a]}$.  
 \end{proof}

Suppose that $\mathrm{spine}(Z')$ has dimension $k$ and $(Z',d_{Z'},z',\t')$ is a static cone. Then it follows from Proposition \ref{prop:staticcone1} and Proposition \ref{prop:splitk} that there exist maps $\boldsymbol{\varphi}^t$ for $t \in \R$ and $\boldsymbol{\phi}^s$ for $s \in \R^{k-2}$. Next, we prove

\begin{prop}\label{prop:appD1}
With the above assumptions, we have
	\begin{align*}
\mathrm{spine}(Z')=\{ \boldsymbol{\varphi}^t \circ \boldsymbol{\phi}^s (z') \mid t\in \R,\,s \in \R^{k-2}\}.
	\end{align*}
\end{prop}

\begin{proof}
We set $S:=\{ \boldsymbol{\varphi}^t \circ \boldsymbol{\phi}^s (z') \mid t\in \R,\,s \in \R^{k-2}\}$. For any $y \in S$, it follows from Proposition \ref{prop:staticcone1} (iii) and Proposition \ref{prop:splitk} (iii) that $y \in \mathrm{spine}(Z')$.

Conversely, suppose $x \in \mathrm{spine}(Z')$. We define $x':=\boldsymbol{\varphi}^{-\t'(x)}(x) \in Z'_0$. It is clear that $x' \in \mathrm{spine}(Z')$ and we only need to prove that $x' \in S$.

Set $f':=f_{x'}$. Then, by the equation $\Ric(g^{Z'})+\na^2 f'=g^{Z'}/2$ on $\RR'_{-1}$, we have
	\begin{align*}
\na^2 u\equiv 0
	\end{align*}
on $\RR'_{-1}$, where $u=f-f'$. If $u$ is a constant $c$, then we have
	\begin{align}\label{eq:appd006}
(4\pi)^{\frac n 2}=\int_{\RR'_{-1}} e^{-f}  \, \mathrm{d}V_{g^{Z'}_{-1}}=\int_{\RR'_{-1}} e^{-f'-c}  \, \mathrm{d}V_{g^{Z'}_{-1}}=e^{-c} (4\pi)^{\frac n 2},
	\end{align}
which implies that $c=0$. Then we have $f=f'$, meaning that $x'=z' \in S$.

If $u$ is not a constant, then $\na u$ induces a splitting factor $\R$ in $\RR'_{-1}$. On the other hand, by our assumption, we have a decomposition $\RR'_{-1}=\RR''_{-1} \times \R^{k-2}$. For any $w \in \RR'_{-1}$, we denote its components in the above decomposition by $(w_1, w_2)$. By the Ricci shrinker equation, we have
	\begin{align*}
f(w)=h_1(w_1)+\frac{|w_2-v_1|^2}{4}
	\end{align*}
	and
	\begin{align*}
f'(w)=h_2(w_1)+\frac{|w_2-v_2|^2}{4}
	\end{align*}	
for any $w \in \RR'_{-1}$, where $v_1, v_2 \in \R^{k-2}$ are constant vectors. Since $\na u$ must be parallel to $\R^{k-2}$ (otherwise the dimension of the spine exceeds $k$), we obtain
	\begin{align*}
u(w)=\la w_2, v \ra+c'
	\end{align*}	
 for some $v \in \R^{k-2}$ and $c' \in \R$. Thus, by Proposition \ref{prop:splitkpoten}, there exists $s_0 \in \R^{k-2}$ such that 
 	\begin{align*}
f'=f_{\boldsymbol{\phi}^{s_0}(z')}+c''
	\end{align*}	
for some constant $c''$. By the same argument as in \eqref{eq:appd006}, we conclude that $c''=0$ and hence $x'=\boldsymbol{\phi}^{s_0}(z')$. Thus, $x' \in S$ and the proof is complete.
 \end{proof}

The proof of Proposition \ref{prop:appD1} also gives the following results.

\begin{prop}
Suppose that $\mathrm{spine}(Z')$ has dimension $k$ and $(Z',d_{Z'},z',\t')$ is a quasi-static cone. Then
	\begin{align*}
\mathrm{spine}(Z')=\{ \boldsymbol{\varphi}^t \circ \boldsymbol{\phi}^s (z') \mid t\in (-\infty, t_a],\,s \in \R^{k}\}.
	\end{align*}
\end{prop}

\begin{prop}
Suppose $\mathrm{spine}(Z')$ has dimension $k$ and $(Z',d_{Z'},z',\t')$ is neither a static cone nor a quasi-static cone. Then
	\begin{align*}
\mathrm{spine}(Z')=\{\boldsymbol{\phi}^s (z') \mid s \in \R^{k}\}.
	\end{align*}
	In particular, $\mathrm{spine}(Z') \subset Z'_0$.
\end{prop}

\printindex

\bibliographystyle{alpha}
\bibliography{lojaref}

\vspace{10pt}

Hanbing Fang, Mathematics Department, Stony Brook University, Stony Brook, NY 11794, United States; Email: hanbing.fang@stonybrook.edu;\\

Yu Li, Institute of Geometry and Physics, University of Science and Technology of China, No. 96 Jinzhai Road, Hefei, Anhui Province, 230026, China; Hefei National Laboratory, No. 5099 West Wangjiang Road, Hefei, Anhui Province, 230088, China; Email: yuli21@ustc.edu.cn. \\

\end{CJK}
\end{document}